\documentclass[11pt,a4paper,final,oneside,openright,reqno]{amsart}

 \usepackage{latexsym, amsfonts, amsthm, amsmath, amssymb, tikz, enumitem, booktabs, float,color,multicol,epsf, url,epsfig,epstopdf, array, setspace, graphicx, inputenc, csquotes, xcolor,tikz-cd}

\usepackage[margin=1in,footskip=0.25in]{geometry}
\usepackage{titlesec}
\usepackage{amsfonts}
\usepackage{amsmath}
\usepackage{amssymb}
\usepackage{url}
\usepackage{hyperref}
\usepackage{color}
\usepackage{lipsum}
\usepackage{mathtools}
 \usepackage[T1]{fontenc}
\usepackage{tikz-cd}
\tikzcdset{scale cd/.style={every label/.append style={scale=#1},
    cells={nodes={scale=#1}}}}
 \usepackage{url}
 \usepackage{hyperref}
\usepackage{pdfpages}
\usepackage{graphicx}
\usepackage{titletoc}
\usepackage{marginnote}
\usepackage{enumitem}
\usepackage[noabbrev,capitalize]{cleveref}
\usepackage[all]{hypcap}
\usepackage{cleveref}

\titlecontents{section}
  [0em] 
  {\vspace{0.5em}} 
  {\thecontentslabel\quad} 
  {} 
  {\hfill\contentspage} 
  [\vspace{0.5em}] 

\titlecontents{subsection}
  [1.5em] 
  {\vspace{0.3em}} 
  {\thecontentslabel\quad} 
  {} 
  {\hfill\contentspage} 
  [\vspace{0.3em}] 

\let\oldcup\cup
\let\oldcap\cap
\usepackage{mathabx}
\let\cup\oldcup
\let\cap\oldcap

\titleformat{\section}
{\normalfont\Large\bfseries}{\thesection}{1em}{}

\numberwithin{equation}{section}

\newcommand{\mylabel}[2]{#2\def\@currentlabel{#2}\label{#1}}

\def\H{\operatorname{H}}
\def\F{\operatorname{F}}
\def\M{\operatorname{M}}
\def\S{\operatorname{S}}
\def\C{\operatorname{C}}
\def\Heta{H_\eta}
\def\Hdeta{H^*_\eta}
\def\HC{H_C}
\def\HdC{H_C^*}
\def\PhiH{\Phi_H}

\def\PsiH{\Psi_{H^*}}
\def\psilH{\psi^l_{H^*}}
\def\varphieta{\varphi_\eta}
\def\PsiHeta{\Psi_{H^*_\eta}}
\def\PsiHnu{\Psi_{H^*_\nu}}
\def\PhiHeta{\Phi_{H_\eta}}
\def\PsiC{\Psi_C}
\def\Honenull{H^1_0(\mathbb{T},\mathbb{R}^{2n})}
\def\Hone{{H^1(\mathbb{T},\mathbb{R}^{2n})}}
\def\Honehalf{{H^{1/2}(\mathbb{T},\mathbb{R}^{2n})}}
\def\modulifloerhat{\hat{\mathcal{M}}_F}
\def\modulifloer{\mathcal{M}_F}

\def\codim{\operatorname{codim}}
\def\rest{\operatorname{rest}}

\def\sign{\operatorname{sign}}
\newcommand{\indsys}{\operatorname{ind}_{\operatorname{sys}}^{S^1}}
\newcommand{\indfr}{\operatorname{ind}_{\operatorname{FR}}}

\newcommand{\unit}{1\!\!1}
\newcommand{\crit}{\operatorname{crit}}
\newcommand{\ind}{\operatorname{ind}}

\newcommand{\nullity}{\operatorname{null}}

\newcommand{\defeq}{\mathrel{\mathop:}=}

\newtheorem*{theorem*}{Theorem}
\newtheorem{theorem}{Theorem}[section]
\newtheorem*{theoremA}{Theorem A}
\newtheorem{corollaryB}{Corollary B.\ignorespaces}
\newtheorem{theoremC}{Theorem C.\ignorespaces}
\newtheorem*{corollaryD}{Corollary D}
\newtheorem{corollary}[theorem]{Corollary}
\newtheorem{lemma}[theorem]{Lemma}
\newtheorem{proposition}[theorem]{Proposition}
\theoremstyle{definition}
\newtheorem{remark}[theorem]{Remark}
\newtheorem*{remark*}{Remark}

\author{Stefan Matijevi{\'c}}
\address{Fakult\"{a}t f\"{u}r Mathematik,
Ruhr-Universit\"at Bochum, Bochum, Germany}
\email{stefan.matijevic@rub.de}

\theoremstyle{definition}

\title{Positive ($S^1$-equivariant) symplectic homology of convex domains, higher capacities, and Clarke's duality}

\begin{document}

\begin{abstract}
We prove that the filtered positive ($S^1$-equivariant) symplectic homology of a convex domain is naturally isomorphic to the filtered singular ($S^1$-equivariant) homology induced by Clarke's dual functional associated with the convex domain. As a result, we prove that the Gutt–Hutchings capacities coincide with the spectral invariants introduced by Ekeland–Hofer for convex domains. From this identification, it follows that Besse convex domains can be characterized by their Gutt–Hutchings capacities, which implies that the interiors of Besse-type convex domains encode information about the Reeb flow on their boundaries. Moreover, as a corollary of the aforementioned isomorphism, we deduce that the barcode entropy associated with the singular homology induced by Clarke's dual functional provides a lower bound for the topological entropy of the Reeb flow on the boundary of a convex domain in $\mathbb{R}^{2n}$. In particular, this barcode entropy coincides with the topological entropy for convex domains in $\mathbb{R}^4$.

\end{abstract}

\maketitle

\tableofcontents

\newpage

\section*{Introduction}

We say that $C \subset \mathbb{R}^{2n}$ is a strongly convex domain if it is the closure of a bounded open convex set with a smooth boundary that has positive sectional curvature everywhere.

Let $C$ be a strongly convex domain. The standard symplectic form 
\[
\omega_0 \defeq \sum_{j=1}^n dx_j \wedge dy_j
\]
restricted to $\partial C$ has a 1-dimensional kernel (this is true for every hypersurface). Consequently, we can naturally define a line bundle $\mathcal{L}_{\partial C}$, known as the characteristic line bundle (see \cite{HZ94}) 
\[
\mathcal{L}_{\partial C} = \{ (x, v) \in T\partial C \mid \omega_0(v, \eta) = 0 \text{ for all } \eta \in T_x \partial C \}.
\]

An embedded circle $S \subset \partial C$ is a closed characteristic if $TS \subset \mathcal{L}_{\partial C}$.

Let $\gamma: [0,T] \to \mathbb{R}^{2n}$ be a parametrized loop. The action of $\gamma$ is given by 
\[
\mathcal{A}(\gamma) = \int_{\gamma} \lambda
\]
where $\lambda$ is any primitive of the standard symplectic form $\omega_0$. Since Stokes' theorem shows that the action of $\gamma$ corresponds to the symplectic area enclosed by $\gamma$, the action value is independent of the parametrization of $\gamma$ up to a sign. For a chosen parametrization $\gamma: [0,T] \to \mathbb{R}$, we have 
\[
\mathcal{A}(\gamma) = \frac{1}{2} \int_0^T \langle \dot{\gamma}(t), J_0 \gamma(t) \rangle \, dt.
\]

Here $J_0$ is the standard complex structure, defined as 
\[
J_0: \mathbb{R}^{2n} \to \mathbb{R}^{2n}, \quad J_0(x,y) = (-y,x),
\]
which corresponds to multiplication by $i$ under the identification $\mathbb{R}^{2n} \cong \mathbb{C}^n$ given by $(x,y) \mapsto x + iy$.

The spectrum of $\partial C$, denoted by $\sigma(\partial C)$, is defined as 
\[
\sigma(\partial C) \defeq \{ |\mathcal{A}(\gamma)| \mid \gamma \text{ is a closed characteristic} \}.
\]

Lines that pass through the origin are transverse to $\partial C$ if and only if the origin lies in the interior of $C$.
Hence, by translating $C$ we can assume that $(C, \lambda_0)$ is a Liouville domain where 
\[
\lambda_0 = \frac{1}{2} \sum_{j=1}^n (x_j dy_j - y_j dx_j),
\]
is the standard Liouville form, which is a primitive of $\omega_0$.

The Reeb vector field $R$ is defined on $\partial C$ by the system of equations 
\[
i_R \lambda_0 = 1, \quad i_R \omega_0 = 0.
\]

It is clear from the definition of $R$ that the closed Reeb orbits are exactly the closed characteristics. Moreover, the spectrum $\sigma(\partial C)$ corresponds to the set of periods of closed Reeb orbits.

If we denote by 
\[
H_C: \mathbb{R}^{2n} \to \mathbb{R}
\]
a positively 2-homogeneous function such that $H_C|_{\partial C} = 1$, then \[R=J_0\nabla \HC\]
on $\partial C$. Hence, the closed Reeb orbits are solutions of the equation 
\[
y: \mathbb{R}/T\mathbb{Z} \to \partial C, \quad \dot{y}(t) = J_0 \nabla H_C(y(t)), \quad T > 0.
\]

By reparametrizing on $\mathbb{T}:=\mathbb{R}/\mathbb{Z}$, we can describe the set of Reeb orbits as 
\[
\mathcal{R}(\partial C) \defeq \{ y: \mathbb{T} \to \partial C \mid \dot{y} = T J_0 \nabla H_C(y(t)), \quad T > 0 \}.
\]

Therefore, we have that 
\[
\sigma(\partial C) = \{ \mathcal{A}(y) \mid y \in \mathcal{R}(\partial C) \}.
\]

Another way to study periodic Reeb orbits (closed characteristics) of $\partial C$ and $\sigma(\partial C)$ when $C$ is convex is through Clarke's dual functional.

Consider the Sobolev space
\[
\Honenull = \left\{ x \in \Hone \mid \int_\mathbb{T} x(t) \, dt = 0 \right\}
\]
and the smooth hypersurface
\[
\Lambda = \{ x \in \Honenull \mid \mathcal{A}(x) = 1 \}.
\]

By translation, we can assume that the interior of $C$ contains the origin. We denote by  
\[
H_C^*: \mathbb{R}^{2n} \to \mathbb{R}, \quad H_C^*(x) = \max_{y \in \mathbb{R}^{2n}} (\langle x, y \rangle - H_C(y))
\]  
the Fenchel conjugate of $H_C$. The Clarke dual functional is defined as  
\[
\PsiC: \Lambda \to \mathbb{R}, \quad \PsiC(x) = \int_{\mathbb{T}} H_C^*(-J_0 \dot{x}(t)) \, dt.
\]

Since $C$ is a strongly convex domain, $\PsiC$ is $C^{1,1}$. For a general approach to Clarke's duality, see \cite{Eke90}.

As shown in \cite{Cla79} (see also \cite{HZ94,AO08}), there is a one-to-one correspondence between critical points of $\PsiC$ and Reeb orbits. More precisely, there is a bijection 
\[
    \mathcal{P}: \mathcal{R}(\partial C) \to \crit(\PsiC), \quad \mathcal{P}(y) = \frac{1}{\sqrt{\mathcal{A}(y)}} \pi(y),
\]

where $\pi(y) = y - \int_\mathbb{T} y(t) \, dt$. Additionaly it holds $\PsiC(\mathcal{P}(y)) = \mathcal{A}(y)$ for all $y\in \mathcal{R}(\partial C)$.

Therefore, 
\[
\sigma(\partial C) = \{ \PsiC(x) \mid x \in \crit(\PsiC) \}.
\]

For an arbitrary bounded convex domain $C$, the boundary of $C$ is a Lipschitz submanifold of $\mathbb{R}^{2n}$. Moreover, the positive outer normal cone, denoted bt $N_{\partial C}(x)$, is defined and non-empty for every $x\in \partial C$. Hence, generalized closed characteristics on $\partial C$ can be defined as closed Lipshitz curves $\gamma:\mathbb{T}\to \partial C$ such that \[\dot{\gamma}(t)\in J_0N_{\partial C}(\gamma(t)),\quad a.e.\] Therefore the spectrum of $\partial C$ is well-defined as well (see \cite{Cla81, AO14}). In addition, although $\PsiC$ is not $C^1$, it is locally Lipschitz. Thus, weak critical points are well-defined for this functional. See Clarke's book \cite{Cla83} for the definition of weak critical points and non-smooth analysis in general. As shown in \cite{Cla81, AO14}, there is a surjection from generalized closed characteristics to weak critical points of $\PsiC$, and the action of generalized closed characteristics equals the value of $\PsiC$ at the corresponding weak critical points. Therefore,  
\[
\sigma(\partial C) = \{ \PsiC(x) \mid x \in \crit(\PsiC) \},  
\]  
holds in the non-smooth setting as well. For insights into the relationship between billiard trajectories and the Ekeland-Hofer-Zehnder capacity on convex sets, established through Clarke's duality, see \cite{AO08, AO14}.

Clarke's dual functional is invariant under the $S^1$-action, where the action on $\Lambda$ is given by  
\[
\theta \cdot x = x(\cdot - \theta), \quad \theta \in \mathbb{T}, \ x \in \Lambda.
\]  

Using the fact that the Fadell–Rabinowitz index can be associated with $S^1$-spaces, Ekeland and Hofer, in \cite{EH87}, introduced spectral invariants for bounded convex domains. Let us first review the definition of the index.

\textbf{Fadell-Rabinowitz index}

Let $\mathbb{F}$ be a field, and let $X$ be a topological space equipped with an $S^1$-action. Following Borel's approach, we define the $S^1$-equivariant cohomology with $\mathbb{F}$-coefficients as
\[
\H^*_{S^1}(X; \mathbb{F}) = \H^*(X \times_{S^1} ES^1; \mathbb{F}),
\]
where $ES^1 \to BS^1$ is the universal $S^1$-bundle over the classifying space, and \[X \times_{S^1} ES^1 \defeq (X \times ES^1)/S^1.\]

For instance, we can take $ES^1\defeq S^\infty \subset \mathbb{C}^{\infty}$ and $BS^1\defeq\mathbb{C}P^{\infty}$. The cohomology of the classifying space $BS^1$ is given by the ring $\H^*(BS^1; \mathbb{F}) \cong \mathbb{F}[e]$, where $e$ is the generator of $H^2(BS^1; \mathbb{F})$. The quotient projection $\pi_2: X \times_{S^1} ES^1 \to BS^1$, defined by $\pi_2([x, y]) = [y]$, induces a cohomology ring homomorphism $\pi^*_2: \H^*(BS^1; \mathbb{F}) \to \H^*_{S^1}(X; \mathbb{F})$. The Fadell-Rabinowitz index of $X$ is defined as
\[
\indfr(X; \mathbb{F}) = \sup_{k \geq 0} \{k+1 \mid \pi^*_2 e^k \neq 0\}.
\]
Note that if $X$ is non-empty, then $\indfr(X; \mathbb{F}) \geq 1$. By convention, we set $\indfr(\varnothing; \mathbb{F}) = 0$, so $\indfr(X; \mathbb{F}) = 0$ if and only if $X = \varnothing$.

\textbf{Spectral invariants}

Let $C \subset \mathbb{R}^{2n}$ be a bounded convex domain,  i.e., the closure of a bounded open convex set. The $k$-th spectral invariant is defined as
\[
s_k(C; \mathbb{F}) = \inf \{L > 0 \mid \indfr(\{\Psi_{C_0} < L\}; \mathbb{F}) \geq k\}, \quad k \geq 1,
\]
where $C_0$ is any translation of $C$ whose interior contains the origin. 

Let $C_0$ and $C_0'$ be two bounded convex domains whose interiors contain the origin. If there exists a symplectic isotopy $\Phi_t: \mathbb{R}^{2n} \to \mathbb{R}^{2n}$ from $C_0$ to $C_0'$, which preserves convexity and keeps the origin inside the interior of the convex sets (i.e., $\Phi_0 = \text{id}$, $\Phi_1(C_0) = C_0'$, and $\Phi_t(C)$ is convex with $0 \in \Phi_t(C)$ for every $t \in [0,1]$), one can show that for every $k \in \mathbb{N}$, it holds that
\[
\inf \{L > 0 \mid \indfr(\{\Psi_{C_0} < L\}; \mathbb{F}) \geq k\} = \inf \{L > 0 \mid \indfr(\{\Psi_{C_0'} < L\}; \mathbb{F}) \geq k\}.
\]
In particular, this holds for translations, and therefore the sequence of numbers $(s_k(C; \mathbb{F}))_{k \in \mathbb{N}}$ is well-defined.

The first spectral invariant is the minimum of the spectrum of $\partial C$, while the others also belong to the spectrum. Moreover, they are continuous with respect to the Hausdorff metric, that is, for any sequence of bounded convex domains $C_i$ Hausdorff-converging to a bounded convex domain $C$, it holds that $s_k(C_i) \to s_k(C)$. For proofs of these properties, see \cite{EH87}. Furthermore, as we will see, this sequence is also invariant under symplectomorphisms of the interiors.

In the referenced work, these invariants were first used to establish multiplicity results for periodic trajectories of convex energy surfaces. Recently, it has been shown in \cite{GGM21} that these invariants characterize strongly convex Besse domains and strongly convex Zoll domains as a particular case. Moreover, in \cite{MR23}, the functional $\PsiC$ was used to explore the structure of strongly convex Besse domains.

\textbf{Gutt-Hutchings capacities }

Building on the works of Cieliebak, Floer, and Hofer \cite{FH94,CFH95}, Viterbo developed (positive) symplectic homology in \cite{Vit99}. The $S^1$-equivariant version of this homology was also originally defined by Viterbo in the same paper. Following a suggestion by Seidel in \cite{Sei08}, Bourgeois and Oancea introduced an alternative definition using family Floer homology in \cite{BO17}. Using positive $S^1$-equivariant symplectic homology, Gutt and Hutchings in \cite{GH18} introduced a sequence of symplectic capacities denoted by $c_k^{GH}$. 

We will not define positive ($S^1$-equivariant) symplectic homology here and refer the reader to \cite{BO17} for various definitions and to \cite{GH18} for the properties of positive $S^1$-equivariant symplectic homology and the definition of $c_k^{GH}$ in the Liouville setting.

If $W \subset \mathbb{R}^{2n}$ is a nice star-shaped domain (i.e., the closure of a bounded open star-shaped set with a smooth boundary such that $\partial W$ is transverse to the radial vector field), then $(W, \lambda_0)$ is a Liouville domain.

According to \cite[Remark 4.8]{GH18}, in this case we have
\[
c_k^{GH}(W; \mathbb{F}) = \inf\{L > 0 \mid i_L^{S^1,+} \neq 0\},
\]
where
\[
i_L^{S^1,+} : \S\H^{S^1,+,<L}_{n-1+2k}(W; \mathbb{F}) \to \S\H^{S^1,+}_{n-1+2k}(W; \mathbb{F}).
\]

Here, $\S\H^{S^1,+,<L}_{n-1+2k}(W ;\mathbb{F})$ denotes the filtered positive $S^1$-equivariant homology, while the group $\S\H^{S^1,+}_{n-1+2k}(W; \mathbb{F})$ is the full positive $S^1$-equivariant homology. The map $i_L^{S^1,+}$ refers to the inclusion map. These capacities are extended continuously to bounded star-shaped domains that are not nice. In order to keep the exposition simple, all results in this paper use $\mathbb{Z}_2$ coefficients. Our first result is the following.

\begin{theoremA}\hypertarget{TheoremA}{}
    Let $C \subset \mathbb{R}^{2n}$ be a bounded convex domain. Then, for all $k \in \mathbb{N}$,
    \[
    c^{GH}_k(C) = s_k(C).
    \]
\end{theoremA}

This theorem can be proven for an arbitrary field. Notably, it follows from this result that the spectral invariants $s_k$ are monotonic with respect to symplectic embeddings of the interiors of bounded convex domains. In particular, they remain invariant under symplectomorphisms of the interiors of bounded convex domains, which is not immediately evident from the definition.

\begin{remark*}
From Theorem \hyperlink{TheoremA}{A} and \cite[Theorem 1.2]{GR24}, we have that  
\[
c^{EH}_k(C) = s_k(C), \quad k \in \mathbb{N},
\]
where $C \subset \mathbb{R}^{2n}$ is a bounded convex domain, and $c_k^{EH}$ denotes the Ekeland-Hofer capacities defined in \cite{EH90}.
\end{remark*}

We now discuss some applications of the previous theorem.

\textbf{Strongly convex domains of Besse-type}  

Let $C \subset \mathbb{R}^{2n}$ be a bounded convex domain with smooth boundary such that the origin is contained in the interior. We say, that $C$ is Besse if all the Reeb orbits of $(\partial C,\lambda_0)$ are closed. In such cases, according to Wadsley’s theorem \cite{Wad75}, it follows that the Reeb orbits share a common period. When this period is minimal, we refer to $C$ as Zoll.

If $\tau > 0$ is a common period for the Reeb orbits on $\partial C$, then all $\tau$-periodic closed Reeb orbits possess the same Conley-Zehnder index $\mu \in \mathbb{N}$, in which case $C$ is called $(\tau, \mu)$-Besse. From \cite[Theorem 1.2]{GGM21}, and the previous theorem, the following two results follow.

\begin{corollaryB}\hypertarget{CorollaryB2}{}
    A strongly convex domain $C \subset \mathbb{R}^{2n}$ is Besse if and only if $c^{GH}_i(C) = c^{GH}_{i+n-1}(C)$ for some integer $i \geq 1$. In this scenario, $C$ is $(\tau, \mu)$-Besse, where $\tau = c_i^{GH}(C) = c^{GH}_{i+n-1}(C)$ and $\mu = 2(i-1) + n$.     
\end{corollaryB}

\begin{corollaryB}\hypertarget{CorollaryB3}{}
    A strongly convex domain $C \subset \mathbb{R}^{2n}$ is Zoll if and only if $c^{GH}_1(C) = c^{GH}_{n}(C)$.
\end{corollaryB}

Since we now, in the strongly convex category, have a characterization of Besse domains in terms of capacities, it follows that the interior of a strongly convex Besse-type domain encodes information about the Reeb flow on the boundary.

Consider the ellipsoid  
\[
E(a_1,\dots,a_n)=\left\{z=(z_1,\dots,z_n)\in \mathbb{R}^{2n}\mid \sum\limits_{i=1}^{n}\frac{|z_i|^2}{a_i}\leq \frac{1}{\pi}\right\},
\]
where $ 0<a_1\leq \dots \leq a_n<\infty$ . The Reeb flow on its boundary is given by  
\[
\varphi_R^t: \partial E(a_1,\dots,a_n)\to \partial E(a_1,\dots,a_n), \quad \varphi^t_R(z)=(e^{J_0 2\pi t/a_1}z_1,\dots, e^{J_0 2\pi t/a_n}z_n).
\]
This ellipsoid is Besse if and only if the ratios $a_i/a_k$ are rational for every $ i,k\in \{1,\dots ,n\}$. When this condition holds, the minimal period $\tau_0>0$ of the Reeb flow is the least common multiple of $ a_1,\dots, a_n $. For any $ \tau>0 $ that is a multiple of $ \tau_0 $, the Reeb orbits with period $ \tau $ have Conley–Zehnder index $
\mu=2\left(\frac{\tau}{a_1}+\dots+\frac{\tau}{a_n}\right)-n.$ In particular, the ellipsoid is $(\tau,\mu)$-Besse.

\begin{corollaryB}
    Let $C \subset \mathbb{R}^{2n}$ be a strongly convex domain whose interior is symplectomorphic to the interior of a rational ellipsoid. Then $C$ is Besse of the same type as the corresponding rational ellipsoid. In particular, if the ellipsoid is a ball, then $C$ is Zoll.
\end{corollaryB}

This rigidity phenomenon does not hold for certain classes of irrational ellipsoids. The Anosov-Katok construction (see \cite{AK70,Kat73}) produces strongly convex domains whose interior is symplectomorphic to that of an irrational ellipsoid, but whose boundary dynamics differ significantly (for example, they exhibit dense orbits). See \cite[Appendix A]{ABE23} for further details on these results.

\textbf{Generalized Zoll convex domains}

Theorem \hyperlink{TheoremA}{A} allows us to extend the definition of Zoll convex domains to the non-smooth convex domains in a way that is natural from a symplectic point of view. 

Indeed, let $C \subset \mathbb{R}^{2n}$ be a bounded convex domain. For such a domain it is possible to define the space of generalized systoles as the space of generalized closed characteristis of minimal action. In particular, if the interior of $C$ contains the origin, we can define the space of generalized systoles as  
\[
\operatorname{Sys}(\partial C) = \left\{\gamma: \mathbb{T} \to \partial C \text{ absolutely continuous} \mid \dot{\gamma}(t) \in T_{\min} J_0 \partial \HC(\gamma(t)) \ \text{a.e.} \right\},
\]
where $\partial H_C$ denotes the subdifferential of $\HC$ (see \cite{Cla83}) and $T_{\min} = \min \sigma(\partial C)$.  

On this space, we have a natural $S^1$-action given by  
\[
\theta \cdot \gamma = \gamma(\cdot - \theta), \quad \theta \in \mathbb{T}, \  \gamma \in \operatorname{Sys}(\partial C).
\]
Taking the $L^2$-projection of this space onto the space of curves in $\mathbb{R}^{2n}$ with zero mean, we define the space of centralized systoles (see \cite{Mat25}).

We introduce the \textit{systolic $S^1$-index} of $C$, denoted by $\indsys(C)$, as the Fadell–Rabinowitz index of the space of centralized systoles in a suitable topology (see \cite[Definition 1.1]{Mat25}).

Using Theorem \hyperlink{TheoremA}{A}, we show that it holds  
\[
\indsys(C) = \max\{k \in \mathbb{N} \mid c^{GH}_k(C) = c^{GH}_1(C)\}.
\]  

For a more precise statement, see \cite[Theorem 1.2]{Mat25}. From this identity, it follows that $\indsys$ is a symplectic invariant (see \cite[Corollary 1.3]{Mat25}), i.e., it is invariant under symplectomorphisms of the interiors of bounded convex domains.  

We define a \textit{generalized Zoll convex domain} as a bounded convex domain $C \subset \mathbb{R}^{2n}$ that satisfies  
\[
\indsys(C) \geq n.
\]  

We show that this definition is equivalent to the standard definition of Zoll convex domains in the smooth setting (see \cite[Theorem 1.5]{Mat25}). Moreover, since $\indsys$ is a symplectic invariant, we obtain non-smooth examples of bounded convex domains whose interiors are symplectomorphic to the interior of a ball and are therefore generalized Zoll domains (see \cite[Corollary 1.6]{Mat25}). Examples of such domains can be found in \cite{Tra95, Sch05, LMS13, Rud22,ORS23}. In particular, this work demonstrates that the condition of strong convexity can be replaced by convexity in Corollary \hyperlink{CorollaryB2}{B.2} and in Corollary \hyperlink{CorollaryB3}{B.3} for the Zoll case.

These results are not a straightforward application of Theorem \hyperlink{TheoremA}{A} and we refer the reader to \cite{Mat25} for further details.

The result of Theorem \hyperlink{TheoremA}{A} is a corollary of a more precise statement. Since $\Lambda$ is $S^1$-contractible and $\mathbb{Z}_2$ is a field, it is straightforward to show that $s_k(C)$ can be expressed as
\[
s_k(C) = \inf\{L > 0 \mid (i^C_L)_* \neq 0\},
\]
where
\[
(i^C_L)_* : \H^{S^1}_{2k-2}(\{\Psi_C < L\}) \to \H_{2k-2}^{S^1}(\Lambda)
\]
and $i^C_L: \{\Psi_C < L\} \to \Lambda$ is the inclusion. Since both $c_k^{GH}$ and $s_k$ arise from $S^1$-equivariant homologies and have analogous definitions, it is natural to ask if they essentially derive from the same object.

We provide a positive answer to this question by showing that the positive ($S^1$-equivariant) homology of non-degenerate strongly convex domain whose interior contains the origin is naturally isomorphic to the singular ($S^1$-equivariant) homology induced by the Clarke's dual functional.

Let $L_1, L_2 \notin \sigma(\partial C)$ with $L_1 < L_2$. We denote by
\[
i^{+}_{L_2,L_1}: \S\H^{+,<L_1}_{*}(C) \to \S\H^{+,<L_2}_{*}(C)
\]
the inclusion in positive symplectic homology, and by
\[
i^{S^1,+}_{L_2,L_1}: \S\H^{S^1,+,<L_1}_{*}(C) \to \S\H^{S^1,+,<L_2}_{*}(C)
\]
the inclusion in positive $S^1$-equivariant homology. On the other hand,
\[
i^C_{L_2,L_1}: \{\Psi_C < L_1\} \to \{\Psi_C < L_2\}
\]
denotes the inclusion of sets.

We say that $C \subset \mathbb{R}^{2n}$ is non-degenerate if all closed Reeb orbits on the boundary are transversely non-degenerate.

\begin{theoremC}\hypertarget{TheoremC1}{}
    Let $C$ be a non-degenerate, strongly convex domain whose interior contains the origin. For every $L \in \mathbb{R} \setminus \sigma(\partial C)$ there exists an isomorphism
    \[
    D_L: \H_*(\{\Psi_C < L\}) \to \S\H^{+,<L}_{*+n+1}(C).
    \]

    Moreover, this isomorphism is natural, that is, for every $L_1 < L_2$ not in the spectrum of $\partial C$, the following diagram commutes:

    \[
    \begin{tikzcd}
    \H_*(\{\Psi_C < L_2\}) \arrow{r}{D_{L_2}} & \S\H^{+,<L_2}_{*+n+1}(C) \\
    \H_*(\{\Psi_C < L_1\}) \arrow[swap]{u}{(i^C_{L_2,L_1})_*} \arrow{r}{D_{L_1}} & \S\H^{+,<L_1}_{*+n+1}(C) \arrow[swap]{u}{i^{+}_{L_2,L_1}}
    \end{tikzcd}.
    \]
\end{theoremC}

\begin{theoremC}\hypertarget{TheoremC2}{}
    Let $C$ be a non-degenerate, strongly convex domain whose interior contains the origin. For every $L \in \mathbb{R} \setminus \sigma(\partial C)$ there exists an isomorphism
    \[
    D^{S^1}_L: \H^{S^1}_*(\{\Psi_C < L\}) \to \S\H^{S^1,+,<L}_{*+n+1}(C), \quad L \in \mathbb{R} \setminus \sigma(\partial C).
    \]
    Moreover, this isomorphism is natural, that is, for every $L_1 < L_2$ not in the spectrum of $\partial C$, the following diagram commutes:

    \[
    \begin{tikzcd}
    \H^{S^1}_*(\{\Psi_C < L_2\}) \arrow{r}{D^{S^1}_{L_2}} & \S\H^{S^1,+,<L_2}_{*+n+1}(C) \\
    \H^{S^1}_*(\{\Psi_C < L_1\}) \arrow[swap]{u}{(i^C_{L_2,L_1})_*} \arrow{r}{D^{S^1}_{L_1}} & \S\H^{S^1,+,<L_1}_{*+n+1}(C) \arrow[swap]{u}{i^{S^1,+}_{L_2,L_1}}
    \end{tikzcd}.
    \]
\end{theoremC}

\textbf{The topological entropy of the Reeb flow on the boundary of a strongly convex domain}

Theorem \hyperlink{TheoremC1}{C.1} establishes an isomorphism between the persistence modules $(\S\H^{+,<L}(C), i^{+}_{L_2,L_1})$ and $(\H(\{\Psi_C < L\}), (i^C_{L_2,L_1})_*)$ (see \cite{FLS23} or \cite{CGGM24} for definitions of persistence modules, barcodes, and barcode entropy). Let $\hslash_+(\lambda_0)$ denote the barcode entropy of the persistence module $(\S\H^{+,<L}(C), i^{+}_{L_2,L_1})$, and let $\hslash(\Psi_C)$ represent the barcode associated with the persistence module $(\H(\{\Psi_C < L\}), (i^C_{L_2,L_1})_*)$. Theorem \hyperlink{TheoremC1}{C.1} implies that

\[
\hslash_+(\lambda_0) = \hslash(\Psi_C).
\]

Let $(W, \lambda)$ be a non-degenerate Liouville domain, and let $\varphi$ be the Reeb flow associated with the contact form $\lambda|_{\partial C}$. Recent work \cite{FLS23} shows that

\[
\hslash(\lambda) \leq h_{\text{top}}(\varphi),
\]

where $\hslash(\lambda)$ is the barcode entropy of the symplectic homology of $(W, \lambda)$, and $h_{\text{top}}$ is the topological entropy. Moreover, according to \cite[Proposition 4.11]{FLS23}, we have

\[
\hslash_+(\lambda) = \hslash(\lambda),
\]

where $\hslash_+(\lambda)$ is the barcode entropy of the positive symplectic homology. Combining these results and \cite[Theorem B]{CGGM24}, we obtain the following.

\begin{corollaryD}
    Let $C \subset \mathbb{R}^{2n}$ be a non-degenerate, strongly convex domain whose interior contains the origin. Then
    \[
    \hslash(\Psi_C) \leq h_{\text{top}}(\varphi).
    \]
    Moreover, the equality holds for $n=2$.
\end{corollaryD}

\textbf{Idea of the proofs.}

Theorems \hyperlink{TheoremC1}{C.1} and \hyperlink{TheoremC2}{C.2} extend the work of Abbondandolo and Kang in \cite{AK22}. There, they constructed a chain isomorphism from the Morse complex of Clarke's dual action functional  
\[
\PsiH: \Honenull \to \mathbb{R}, \quad \PsiH(x) = -\int\limits_\mathbb{T} x^* \lambda_0 + \int\limits_\mathbb{T} H^*(-J_0 \dot{x}(t)) \, dt,
\]
restricted to a specific finite-dimensional manifold, to the Floer complex of a convex and asymptotically quadratic Hamiltonian $H: \mathbb{T} \times \mathbb{R}^{2n} \to \mathbb{R}$.

Here, $H^*$ denotes the Fenchel conjugate of $H$. The restriction is necessary due to regularity issues with  $\PsiH$, and we refer to it as the reduced dual action functional, denoted by $\psi_{H^*}$.  

We denote the chain isomorphism by  
\begin{equation}\label{chainisomorphismintroduction}
    \Theta: \C\M_{*}(\psi_{H^*}) \to \C\F_{*+n}(H).
\end{equation}

The critical points of $\psi_{H^*}$ and the $1$-periodic orbits of $H$ are in bijective correspondence. Under this correspondence, the Morse index of the critical points of $\psi_{H^*}$ matches the Conley-Zehnder index of the corresponding $1$-periodic orbits of $H$, shifted by $n$. This explains the grading shift in \eqref{chainisomorphismintroduction}. For further details, we refer the reader to \cite{AK22}. Similar types of isomorphisms in different settings can be found in \cite{AS06, Hec13, Dju17}.

\begin{itemize}
    \item Theorem \hyperlink{TheoremC1}{C.1}:
\end{itemize}

Our first step is to show that the isomorphism $\Theta$ is functorial. Specifically, if we have $H_\alpha \leq H_\beta$, and a chain map induced by a non-increasing homotopy from $H_\beta$ to $H_\alpha$ denoted by $i^F_{(H^{\alpha\beta}_s, J^{\alpha\beta}_s)}$, along with a chain map induced by a non-increasing homotopy from $\psi_{H^*_\alpha}$ to $\psi_{H^*_\beta}$ denoted by $i^M_{(\psi^{\alpha\beta}_s, g^{\alpha\beta}_s)}$, then it holds that $i^F_{(H^{\alpha\beta}_s, J^{\alpha\beta}_s)} \circ \Theta_\alpha$ and $\Theta_\beta \circ i^M_{(\psi^{\alpha\beta}_s, g^{\alpha\beta}_s)}$ are chain-homotopic. Here,
\[
\Theta_\alpha: \C\M_{*}(\psi_{H_\alpha^*}) \to \C\F_{*+n}(H_\alpha), \ \text{and} \ \Theta_\beta: \C\M_{*}(\psi_{H_\beta^*}) \to \C\F_{*+n}(H_\beta)
\]
are the aforementioned chain isomorphisms. See Theorem \ref{theoremE1} for a precise statement and proof. Once the chain isomorphism from \cite{AK22} is upgraded, we define a family of smooth convex functions (denoted by $\mathcal{F}_{\operatorname{lin}}(C)$) of the form $\varphi: \mathbb{R}_{\geq 0} \to \mathbb{R}$, which are linear at infinity. Here, $C \subset \mathbb{R}^{2n}$ is a non-degenerate, strongly convex domain whose interior contains the origin. The slope $\eta$ at infinity of a function $\varphi\in  \mathcal{F}_{\operatorname{lin}}(C)$ satisfies $\eta \notin \sigma(\partial C)$ and $\eta > \min \sigma(\partial C)$. We write $\varphi_\eta$ to emphasize the slope of $\varphi$ at infinity and denote by $T_\eta$ the minimal element of $\sigma(\partial C)$ that is greater than $\eta$. Let $H_\eta=\varphi_\eta\circ \HC$ and $\varepsilon=\frac{1}{2}\min \sigma(\partial C)$.

Using the isomorphism from \cite{AK22} and the functoriality property, we show that for every $L \in (\eta, T_\eta)$, there exists an isomorphism
\[
D^{H_\eta}_L: \H_{*}(\{\Psi_{H^*_\eta} < L\}, \{\Psi_{H^*_\eta} < \varepsilon\}) \to \S\H^{+,<L}_{*+n}(C),
\]
which is functorial with respect to inclusions in homology. On the other hand, by employing Morse-theoretic methods and tools from algebraic topology, we prove that for every $L \in (\eta, T_\eta)$, there exists an isomorphism
\[
D^{H^*_\eta}_L: \H_*(\{\PsiC < L\}) \to \H_{*+1}(\{\Psi_{H_\eta^*} < L\}, \{\Psi_{H_\eta^*} < \varepsilon\}),
\]
which is also functorial with respect to inclusions in homology. By composing these two isomorphisms for $L>\sigma(\partial C)$ and $L\notin \sigma(\partial C)$ and taking the trivial isomorphism when $L < \sigma(\partial C)$, we obtain the desired result.
\begin{itemize}
    \item Theorem \hyperlink{TheoremC2}{C.2}:
\end{itemize}

By adapting the isomorphism from \cite{AK22} to the $S^1$-equivariant setting (see Theorem \ref{theoremE2}) and following the same methods as in the proof of Theorem \hyperlink{TheoremC1}{C.1}, we prove the claim.

\begin{itemize}
    \item Theorem \hyperlink{TheoremA}{A}:
\end{itemize}

Let $C \subset \mathbb{R}^{2n}$ be a non-degenerate, strongly convex domain whose interior contains the origin. The isomorphism from Theorem \hyperlink{TheoremC2}{C.2} induces an isomorphism in the limit:
\[
D^{S^1}: H_*^{S^1}(\Lambda) \to \S\H^{S^1,+}_{*+n+1}(C).
\]
Indeed, since we can understand $H_*^{S^1}(\Lambda)$ as the direct limit over $L$ of $H_*^{S^1}(\{\Psi_C < L\})$, this isomorphism is defined as
\[
D^{S^1}(x) = i^{S^1,+}_L \circ D^{S^1}_L(x), \quad x \in H_*^{S^1}(\{\Psi_C < L\}).
\]
This isomorphism is natural, which means that \( i_L^{S^1,+} \circ D^{S^1}_L = D^{S^1} \circ (i^C_L)_* \) holds. Given that \( c_k^{GH}(C) = \inf\limits\{L > 0 \mid i_L^{S^1,+} \neq 0\} \) and \( s_k(C) = \inf\{L > 0 \mid (i^C_L)_* \neq 0\} \), we conclude from the naturality of the limit isomorphism that \( c^{GH}_k(C) = s_k(C) \). Since this equality holds for a non-degenerate, strongly convex domain, and since these domains are Hausdorff-dense in bounded convex domains, the continuity of \( c^{GH}_k \) and \( s_k \) implies that Theorem \hyperlink{TheoremA}{A} holds.

\textbf{Outlook.} From Theorem \hyperlink{TheoremA}{A}, it follows that we have a minimax representation of Gutt-Hutchings capacities. Hence, one can expect that it is possible to find a numerical algorithm for calculating these capacities in the convex setting. This was already accomplished by Göing-Jaeschke for the Ekeland-Hofer-Zehnder (EHZ) capacity in \cite{Goi98} using $\Psi_C$. Moreover, Haim-Kislev in \cite{Hai19} introduced a combinatorial formula for calculating the EHZ capacity of convex polytopes by minimizing the functional $\Psi_C$ with a suitable class of loops (See the work of Chaidez and Hutchings \cite{CH21} for a more dynamical approach). Thus, one can also hope that an analogous formula can be found for the higher capacities.

\textbf{Acknowledgements.} I am deeply grateful to my advisor, Alberto Abbondandolo, for introducing me to the open problems concerning higher capacities and Clarke's duality, careful oversight of this paper's development, and invaluable guidance throughout numerous discussions.

This project is supported by the SFB/TRR 191 ‘Symplectic Structures in Geometry, Algebra and Dynamics,’ funded by the DFG (Projektnummer 281071066 - TRR 191).

\section{Floer complex of a Hamiltonian}\label{sectionfloer}

On $\mathbb{R}^{2n}$, we choose the standard Liouville form 
\[\lambda_0=\frac{1}{2}\sum_{j=1}^n (x_jdy_j-y_jdx_j),\] 
and the standard symplectic form
\[\omega_0\defeq d\lambda_0=\sum_{j=1}^n dx_j\wedge dy_j.\] 
We denote by 
\[J_0:\mathbb{R}^{2n}\to  \mathbb{R}^{2n}, \ \ \ J_0(x,y)=(-y,x)\] 
the standard complex structure, which corresponds to multiplication by $i$ under the identification $ \mathbb{R}^{2n}\cong \mathbb{C}^n$ given by $(x,y)\mapsto x+iy$. The Euclidean scalar product, the symplectic form $\omega_0$, and $J_0$ are related by the identity 
\[\langle u,v\rangle =\omega_0( u,J_0 v).\]

Our sign convention for the Hamiltonian vector field $X_H$ associated with a smooth Hamiltonian 
\[H:\mathbb{R}^{2n}\longrightarrow \mathbb{R}\] 
is 
\[i_{X_H}\omega_0=-dH.\]
In the paper \cite{AK22}, Abbondandolo and Kang introduced a chain isomorphism between the Floer complex of $H$ and the Morse complex of $\psi_{H^*}: M \to \mathbb{R}$, where $\psi_{H^*}$ is the reduced dual functional. In the first part of this paper, we will adjust this isomorphism to align with our sign conventions and extend this work by constructing an isomorphism for specific types of Hamiltonians $H:\mathbb{T}\times \mathbb{R}^{2n}\times S^{2N+1}\to \mathbb{R}$, where $\mathbb{T}=\mathbb{R}/\mathbb{Z}$.

\large\textbf{Non-parameterized case }\normalsize

For a time-dependent Hamiltonian 
\[H:\mathbb{T}\times \mathbb{R}^{2n}\to \mathbb{R},\] 

the action functional is defined as 
\[\PhiH: \Honehalf\to \mathbb{R}, \ \ \ \PhiH(x)\defeq\int\limits_{\mathbb{T}}x^*\lambda_0 - \int\limits_{\mathbb{T}} H_t(x(t)) \, dt.\]
Elements $x\in \Honehalf$ can be expressed as 
\[x(t)=\sum\limits_{k\in \mathbb{Z}} e^{2\pi k t J_0} \hat{x}(k), \ \ \ \hat{x}(k) \in \mathbb{R}^{2n},\] 
where it holds that 
\[\sum |k| |\hat{x}(k)|^2 < \infty,\] 
with the scalar product given by 
\[\langle x, y \rangle_{1/2} = \langle \hat{x}(0), \hat{y}(0) \rangle + 2\pi \sum_{k\in\mathbb{Z}} |k| \langle \hat{x}(k), \hat{y}(k) \rangle.\]
\begin{remark}
    For $\PhiH$ to be well-defined, one must assume some growth conditions on $H$. The polynomial growth of the $k$-th derivative of $H$ with respect to the spatial variable is sufficient to ensure that $\PhiH\in C^k$. This follows from the Sobolev embedding theorem. In fact, more is true, as there are better inequalities than those from Sobolev embeddings. For instance, if one assumes that the second derivative of $H$ has polynomial growth, then $\PhiH$ is smooth. More on this topic can be found in \cite{Abb01}.  
\end{remark}

Let's recall some definitions from \cite{AK22}.

\textbf{Linear growth of Hamiltonian vector field}: 
The Hamiltonian vector field $X_{H_t}$ of the Hamiltonian $H\in C^\infty(\mathbb{T}\times \mathbb{R}^{2n})$ is said to have linear growth with respect to the spatial variable if there exists a positive number $c>0$ such that: 
\[|X_{H_t}(x)|\leq c(1+|x|), \ \ \ (t,x)\in \mathbb{T}\times \mathbb{R}^{2n}.\]

\textbf{Non-resonance at infinity of Hamiltonian}: The Hamiltonian $H\in C^\infty(\mathbb{T}\times \mathbb{R}^{2n})$ is said to be non-resonant at infinity if there exist positive numbers $\varepsilon>0$ and $r> 0$ such that for every smooth curve $x:\mathbb{T}\to\mathbb{R}^{2n}$ satisfying
\[\|\dot{x}-X_{H_\cdot}(x)\|_{L^2(\mathbb{T})}<\varepsilon,\] 
there holds $\|x\|_{L^2(\mathbb{T})}\leq r$.

We say that the smooth loop 
\[x:\mathbb{T}\to \mathbb{R}^{2n}\]
is a $1$-periodic orbit of $H$ if it satisfies the equation $\dot{x}(t)=X_H(x(t))$. Such a $1$-periodic orbit is said to be non-degenerate if $d\varphi_{X_H}(x(0))$ is non-degenerate, meaning that $1$ is not an eigenvalue of $d\varphi_{X_H}(x(0))$, where $\varphi_{X_H}\defeq\varphi^1_{X_H}$ is the time-$1$ flow of $X_H$.

\textbf{Non-degenerate Hamiltonian:} The Hamiltonian $H\in C^\infty(\mathbb{T}\times \mathbb{R}^{2n})$ is said to be non-degenerate if all of its $1$-periodic orbits are non-degenerate.

\textbf{Admissible Hamiltonians:} We will denote by $\mathcal{H}_F$ the set of all non-degenerate smooth Hamiltonians
\[H:\mathbb{T}\times \mathbb{R}^{2n}\to \mathbb{R}\]
that are non-resonant at infinity and have a bounded second derivative with respect to the spatial variable. For these Hamiltonians, the Floer complex is well-defined.

\begin{remark}
    The Floer complex is well-defined under weaker assumptions, specifically if we require $X_H$ to have linear growth instead of the boundedness of the second derivative, as shown in \cite{AK22}. For aesthetic reasons, we work here with more restrictive conditions for Hamiltonians.
\end{remark}

Let $\mathcal{J}(\mathbb{R}^{2n},\omega_0)$ be the space of $\omega_0$-compatible complex structures on $\mathbb{R}^{2n}$. For $H \in \mathcal{H}_F$ and every $x, y \in P(H)$, where $P(H)$ is the set of $1$-periodic orbits of $H$, and $J \in C^0([-\infty, +\infty] \times \mathbb{T}, \mathcal{J}(\mathbb{R}^{2n}, \omega_0))$, we have the moduli space 
\[\modulifloerhat(x,y;H,J)\]
defined as the space of smooth maps $u: \mathbb{R} \times \mathbb{T} \to \mathbb{R}^{2n}$ that satisfy the equation 
\begin{equation}\label{floernonpareq}
  \partial_s u + J(s,t,u(s,t)) (\partial_t u - X_{H_t}(u)) = 0
\end{equation}  
and the asymptotic conditions
\begin{equation*}
  \begin{array}{lcl}
    &\lim\limits_{s \to -\infty} u(s, \cdot) = y,\\
    &\lim\limits_{s \to +\infty} u(s, \cdot) = x.
  \end{array}
\end{equation*}

\begin{remark}\label{notation}
    By usual arguments, it holds that $P(H) = \crit(\PhiH)$. Also, equation \eqref{floernonpareq} represents the $L^2$-gradient equation of the action functional for our sign conventions. Therefore, we are counting trajectories from $x$ to $y$, which motivates putting $x$ first and $y$ second in the definition of the moduli space $\modulifloerhat(x, y; H, J)$.
\end{remark}

From \cite[Proposition A.1]{AK22}, we have regularity properties of $u$. Indeed, let $u$ satisfy \eqref{floernonpareq} and apriori belong to $H^1_{loc}(U, \mathbb{R}^{2n})$, where $U \subseteq \mathbb{R} \times \mathbb{T}$ is an open set. Assuming that $J$ is smooth and bounded, this proposition implies that $u \in C^\infty(U, \mathbb{R}^{2n})$.

Additionally, due to energy bounds, \cite[Proposition 3.1]{AK22} implies that $\modulifloerhat(x,y;H,J)$ is uniformly bounded in the $L^\infty$-norm. Then, by standard arguments, we conclude that $\modulifloerhat(x,y;H,J)$ is pre-compact in  
$C^\infty_{loc}(\mathbb{R} \times \mathbb{T}, \mathbb{R}^{2n})$.

For a generic choice of $J$, $\modulifloerhat(x,y;H,J)$ is a smooth manifold such that 
\[\dim \modulifloerhat(x,y;H,J) = \mu_{CZ}(x) - \mu_{CZ}(y),\]
where $\mu_{CZ}$ denotes the Conley-Zehnder index. See \cite{SZ92} for details.

Since $\modulifloerhat(x,y;H,J)$ has a free $\mathbb{R}$-action given by reparametrization 
\[s_0 \mapsto u(\cdot + s_0, \cdot),\] 
we conclude that 
\[\modulifloer(x,y;H,J) \defeq \modulifloerhat(x,y;H,J)/\mathbb{R}\]
is a smooth manifold and it holds that 
\[\dim \modulifloer(x,y;H,J) = \mu_{CZ}(x) - \mu_{CZ}(y) - 1.\]

\textbf{Floer complex}

Combining all of these results, we conclude that the Floer complex \[\{\C\F_*(H), \partial^F\}\]
is well-defined for a generic uniformly bounded, $\omega_0$-compatible almost complex structure $J(t,x)$ where $(t,x)\in \mathbb{T}\times \mathbb{R}^{2n}$. Indeed, $\C\F_*(H)$ is a $\mathbb{Z}_2$-vector space graded by the Conley-Zehnder index. The boundary operator \[\partial^F: \C\F_*(H) \to \C\F_{*-1}(H)\]
is defined by 
\[\partial^F(x) = \sum\limits_{\mu_{CZ}(y) = \mu_{CZ}(x) - 1} n_F(x, y) y,\]
where $n_F(x, y)$ denotes the parity of the set $\modulifloer(x, y; H, J)$.

By compactness arguments and the fact that $\modulifloer(x, y; H, J)$ is a $0$-dimensional manifold, we have that $n_F(x, y)$ is finite. Standard compactness and gluing arguments imply that 
\[\partial^F \circ \partial^F = 0.\]

Thus, we have a well-defined Floer complex $\{\C\F_*(H), \partial^F\}$ and Floer homology 
\[\H \F_*(H, J) = \H_*(\{\C\F(H), \partial^F\}).\] 
Changing the almost complex structure gives us a chain-isomorphic Floer complex.

On the Floer complex, we have a well-defined filtration with the action functional $\PhiH$. Indeed, as mentioned in Remark \ref{notation}, equation \eqref{floernonpareq} is the $L^2$-gradient equation of $\PhiH$, so it holds that 
\[\PhiH(x) - \PhiH(y) = \int\limits_{\mathbb{R} \times \mathbb{T}} \|\partial_s u\|_J^2 \, ds \, dt \geq 0.\]

Therefore, if we denote by $\C\F_k^{<a}(H)$ the vector space generated by all $1$-periodic orbits $x$ of Conley-Zehnder index $k$ such that $\PhiH(x) < a$, then $\partial^F$ maps $\C\F_k^{<a}(H)$ to $\C\F_{k-1}^{<a}(H)$.

\textbf{Continuation maps}

Let $(H_\alpha, J_\alpha)$ and $(H_\beta, J_\beta)$ be pairs for which the Floer complexes are well-defined and $H_\alpha \leq H_\beta$. Then we have a non-increasing homotopy $H^{\alpha\beta}_s$, $s \in \mathbb{R}$, i.e., a homotopy for which $\partial_s H^{\alpha\beta}_s \leq 0$, and a homotopy of almost complex structures $J^{\alpha\beta}_s$ such that 
\[(H^{\alpha\beta}_s, J^{\alpha\beta}_s) = (H_\beta, J_\beta), \quad s \leq T_0 - 1,\] 
and 
\[(H^{\alpha\beta}_s, J^{\alpha\beta}_s) = (H_\alpha, J_\alpha), \quad s \geq T_0.\]
Here $T_0 < -1$ is a fixed number.
For every $x^\alpha \in P(H_\alpha)$ and $y^\beta \in P(H_\beta)$, we have the moduli space 
\[\modulifloer(x^\alpha, y^\beta; H^{\alpha\beta})\] 
defined as the space of smooth maps $u: \mathbb{T} \times \mathbb{R} \to \mathbb{R}^{2n}$ which satisfy the equations
\begin{equation}\label{floernonparcont}
  \begin{array}{lcl}
    &\partial_s u + J^{\alpha\beta}_s(s, t, u(s, t)) (\partial_t u - X_{H^{\alpha\beta}_{s,t}}(u)) = 0
  \end{array}
\end{equation}  
and the asymptotic conditions 
\begin{equation}
  \begin{array}{lcl}
    &\lim\limits_{s \to -\infty} u(s, \cdot) = y^\beta, \\
    &\lim\limits_{s \to +\infty} u(s, \cdot) = x^\alpha.
  \end{array}
\end{equation}

Regularity for $H^1_{loc}$ of \eqref{floernonparcont} follows from \cite[Proposition A.1]{AK22} as in the case of the standard Floer equation. From the fact that $H_s$ is a non-increasing homotopy, it follows that $\Phi_{H_s}$ is increasing along $u \in \modulifloer(x^\alpha, y^\beta; H^{\alpha\beta})$, which implies that 
\[\|\partial_s u\|^2_{J^{\alpha\beta}_s} \leq \Phi_{H_\alpha}(x^\alpha) - \Phi_{H_\beta}(y^\beta).\] 
By \cite[Remark 3.2]{AK22}, we have the pre-compactness of the above-defined moduli spaces. For $(H^{\alpha\beta}_s, J^{\alpha\beta}_s)$ chosen generically, we have that $\modulifloer(x^\alpha, y^\beta; H^{\alpha\beta})$ is a smooth manifold of dimension $\mu_{CZ}(x^\alpha; H_\alpha) - \mu_{CZ}(y^\beta; H_\beta)$. Thus, we define a map 
\[i^F_{(H^{\alpha\beta}_s, J^{\alpha\beta}_s)}: \C\F_*(H_\alpha, J_\alpha) \to \C\F_*(H_\beta, J_\beta)\]
with 
\[i^F_{(H^{\alpha\beta}_s, J^{\alpha\beta}_s)}(x^\alpha) = \sum\limits_{\mu_{CZ}(y^\beta) = \mu_{CZ}(x^\alpha)} n_F^{\alpha\beta}(x^\alpha, y^\beta) y^\beta,\]
where $n_F^{\alpha\beta}(x^\alpha, y^\beta)$ is the parity of the set $\modulifloer(x^\alpha, y^\beta; H^{\alpha\beta})$. This is a chain map by standard gluing and compactness arguments. Using homotopies of homotopies, we obtain that a different homotopy $(H^{\alpha\beta}_s, J^{\alpha\beta}_s)$ yields a chain-homotopic map. Therefore, they induce the same map in homology, which we call the continuation map.

The map $i^F_{(H^{\alpha\beta}_s, J^{\alpha\beta}_s)}: \C\F_*(H_\alpha, J_\alpha) \to \C\F_*(H_\beta, J_\beta)$ preserves the filtration. Indeed, because $\Phi_{H^{\alpha\beta}_s}(u(s, \cdot))$ is increasing for $u \in \modulifloer(x^\alpha, y^\beta; H^{\alpha\beta})$, we have that $\Phi_{H_\alpha}(x^\alpha) \geq \Phi_{H_\beta}(y^\beta)$.

\large\textbf{$S^1$-equivariant case}

\normalsize

Here, we review the definition of the $S^1$-equivariant Floer complex for family Floer homology. Different approaches to defining the $S^1$-equivariant Floer complex can be found in \cite{BO17}. The definition of such a complex for family Floer homology in the Liouville setting, along with positive $S^1$-equivariant homology properties, is described in \cite{GH18}.      

Let $H \in C^\infty(\mathbb{T} \times \mathbb{R}^{2n} \times S^{2N+1})$. We will sometimes use the notation $H_{z,t}(x)\defeq H(t,x,z)$. Such a Hamiltonian produces a time-dependent parameterized Hamiltonian vector field $X_{H_{z,t}}$, defined by the identity 
\[i_{X_{H_{z,t}}}\omega_0 = -dH_{z,t}.\] 

\textbf{Uniform linear growth of parameterized Hamiltonian vector field:} 
The Hamiltonian vector field $X_{H_{z,t}}$ of the parameterized Hamiltonian $H \in C^\infty(\mathbb{T} \times \mathbb{R}^{2n} \times S^{2N+1})$ is said to have uniform linear growth with respect to the spatial variable if there exists a positive number $c > 0$ such that
\[|X_{H_{z,t}}(x)| \leq c(1 + |x|), \ \ (t,x,z) \in \mathbb{T} \times \mathbb{R}^{2n} \times S^{2N+1}.\]
Let $f: \mathbb{C}P^N \to \mathbb{R}$ be an arbitrary Morse function and $\widetilde{f}: S^{2N+1} \to \mathbb{R}$ its $S^1$-invariant lift.

Let metric on $S^{2N+1}$ be standard one.

\textbf{$\widetilde{f}$-non-resonance at infinity of parameterized Hamiltonian:} 
The parameterized Hamiltonian $H \in C^\infty(\mathbb{T} \times \mathbb{R}^{2n} \times S^{2N+1})$ is said to be non-resonant at infinity with respect to the $S^1$-invariant Morse function $\widetilde{f}: S^{2N+1} \to \mathbb{R}$ if there exist positive numbers $\varepsilon > 0$ and $r > 0$ such that for every smooth curve $x: \mathbb{T} \to \mathbb{R}^{2n}$ and every $z \in S^{2N+1}$ satisfying 
\begin{equation}\label{fnoneq}
\|\dot{x} - X_{H_{z,\cdot}}(x)\|_{L^2(\mathbb{T})} < \varepsilon, \ \ \ \|\nabla \widetilde{f}(z)\| < \varepsilon,    
\end{equation}
the bound $\|x\|_{L^2(\mathbb{T})} \leq r$ holds.

\begin{remark}
    An equivalent definition can be given for any metric on $S^{2N+1}$ due to the compactness of this space.
\end{remark}

On $S^{2N+1}$, we have a natural $S^1$-action given by 
\[\theta \cdot z = e^{2\pi i\theta}z, \ \ \ \theta \in \mathbb{T}, \ z \in S^{2N+1}.\]
\textbf{$S^1$-invariance of parameterized Hamiltonians:} We say that the parameterized Hamiltonian $H \in C^\infty(\mathbb{T} \times \mathbb{R}^{2n} \times S^{2N+1})$ is invariant under the $S^1$-action if 
\[H_{\theta \cdot z}(t + \theta, \cdot) = H_{z}(t, \cdot),\]
for all $\theta \in \mathbb{T}$, $t \in \mathbb{T}$, and $z \in S^{2N+1}$.

For such a Hamiltonian, the parameterized action functional 
\[\Phi_H: H^{1/2}(\mathbb{T}, \mathbb{R}^{2n}) \times S^{2N+1} \to \mathbb{R}, \ \ \ \Phi_H(x, z) = \int_{\mathbb{T}} x^* \lambda_0 - \int_0^1 H_{z,t}(x(t)) \, dt\]
is invariant under the diagonal $S^1$-action on \[H^{1/2}(\mathbb{T}, \mathbb{R}^{2n}) \times S^{2N+1},\] 
where the action on $H^{1/2}(\mathbb{T}, \mathbb{R}^{2n})$ is given by reparametrization, i.e.,
\[\theta \cdot x = x(\cdot - \theta), \ \ \ x \in H^{1/2}(\mathbb{T}, \mathbb{R}^{2n}), \ \theta \in \mathbb{T}.\]
Let 
\[f_N: \mathbb{C}P^N \to \mathbb{R}, \ \ \ f_N([z_0: \cdots : z_N]) = \frac{\sum_{j=0}^N (j+1) |z_j|^2}{\sum_{j=0}^N |z_j|^2}\] 
be the standard Morse function on $\mathbb{C}P^N$, and let 
\[\widetilde{f}_N: S^{2N+1} \to \mathbb{R}, \ \ \ \widetilde{f}_N(z) = \sum_{j=0}^N (j+1) |z_j|^2\] 
be the pullback of $f_N$ to $S^{2N+1}$. Notice that $\widetilde{f}_N$ is a $S^1$-invariant Morse function.

We denote by $g_N$ the standard metric on $S^{2N+1}$. This metric is $S^1$-invariant. We define the subset $\mathcal{H}_F^{S^1,N}(\widetilde{f}_N, g_N)$ of Hamiltonians $H \in C^\infty(\mathbb{T} \times \mathbb{R}^{2n} \times S^{2N+1})$ that are admissible with respect to the pair $(\widetilde{f}_N, g_N)$ as follows.

\textbf{$\widetilde{f}_N$-admissible $S^1$-invariant Hamiltonians:} $ H \in \mathcal{H}_F^{S^1,N}(\widetilde{f}_N, g_N)$ if it satisfies:

\begin{enumerate}
    \item[(F1)] $H$ is $S^1$-invariant.
    \item[(F2)] For every $v \in \text{crit}(\widetilde{f}_N)$, the Hamiltonian $H_v$ is non-degenerate.
    \item[(F3)] $H$ is non-increasing along the gradient flow lines of $\widetilde{f}_N$.
    \item[(F4)] The partial derivatives of $H$ of the second order involving the $x$ and $z$ coordinates are bounded.
    \item[(F5)] $H$ is $\widetilde{f}_N$-non-resonant at infinity.
\end{enumerate}

Let $P(H, \widetilde{f}_N)$ be the set of pairs $(x, v)$ where $x \in P(H_v)$ and $v \in \text{crit}(\widetilde{f}_N)$. This set is invariant under the $S^1$-action, and the action on it is free. The $S^1$-orbit of $(x, v)\in P(H, \widetilde{f}_N)$ is denoted by $S_{(x,v)}$.

For $(x, v) \in P(H, \widetilde{f}_N)$, we associate an index 
\[\mu(x, v) \defeq \mu_{CZ}(x; H_v) = \mu_{CZ}(x; H_v) + \ind(v; \widetilde{f}_N) - N.\] 
This index is constant along the orbits due to the $S^1$-invariance of $H$ (property (F1)).

We say that the family of almost complex structures $J \in C^\infty(\mathbb{T} \times \mathbb{R}^{2n} \times S^{2N+1}; \mathcal{J}(\mathbb{R}^{2n}, \omega_0))$ is $S^1$-invariant if 
\[J_{\theta \cdot z}^{t + \theta} = J_z^t, \ \ \ \theta \in \mathbb{T}, \ t \in \mathbb{T}, \ z \in S^{2N+1}.\]

Let $J$ be a uniformly bounded $S^1$-invariant almost complex structure. For any two orbits $S_{(x,v)}, S_{(y,w)} \subset P(H, \widetilde{f}_N)$, we define the moduli space 
\[\hat{\mathcal{M}}^{S^1,N}_F(S_{(x,v)}, S_{(y,w)}; H, \widetilde{f}_N, J)\]
as the space of pairs $(u, z)$ consisting of smooth maps $u: \mathbb{R} \times \mathbb{T} \to \mathbb{R}^{2n}$ and $z: \mathbb{R} \to S^{2N+1}$ that satisfy the system of equations
\begin{equation}\label{floerred}
  \begin{array}{lcl}
    & \partial_s u + J_{z(s)}(s, t, u(s, t))(\partial_t u - X_{H_{z(s)}}(u)) = 0, \\
    & \dot{z}(s) - \nabla \widetilde{f}_N(z) = 0 
  \end{array}
\end{equation}
and the asymptotic conditions
\begin{equation*}
  \begin{array}{lcl}
    & \lim\limits_{s \to -\infty} (u(s), z(s)) \in S_{(y, w)}, \\
    & \lim\limits_{s \to +\infty} (u(s), z(s)) \in S_{(x, v)}.
  \end{array}
\end{equation*}

\begin{remark}
    The second equation of the system \eqref{floerred} is the gradient equation of $\widetilde{f}_N$ instead of the negative gradient equation, as in \cite{GH18}. The difference is that in \cite{GH18}, the first equation represents the negative $L^2$-gradient equation of the action functional, while in our case, it represents the $L^2$-gradient equation. Therefore, system \eqref{floerred} is natural for our sign conventions.
\end{remark}
Let $(u, z) \in H^1_{\text{loc}}(U' \times U'', \mathbb{R}^{2n}) \times H^1_{\text{loc}}(U', S^{2N+1})$ be a weak solution of the system \eqref{floerred}, where $U' \subset \mathbb{R}$ and $U'' \subset \mathbb{T}$ are open sets.

Since $z$ is a priori continuous, by standard bootstrapping for ODEs, we conclude that $z$ is smooth. On the other hand, $u$ is not a priori continuous. Therefore, $J$ is not a priori continuous, so one cannot use standard techniques to prove the regularity of $u$. However, given uniform linear growth of $X_{H_{z,t}}$, using \cite[Lemma A.2]{AK22}, we can improve the regularity of $u$ to $W^{1,q}_{\text{loc}}(U' \times U'', \mathbb{R}^{2n})$ for some $q > 2$, as in \cite{AK22}. This implies that $u$ is continuous. Then, using standard bootstrapping techniques, we obtain the following result.

\begin{proposition}\label{regularityfloer}
    Let $H \in \mathcal{H}_F^{S^1, N}(\widetilde{f}_N, g_N)$. Let $U'$ and $U''$ be open subsets of $\mathbb{R}$ and $\mathbb{T}$, respectively, such that $J$ is a uniformly bounded $\omega_0$-compatible almost complex structure smoothly depending on $(s, t, z) \in U' \times U'' \times S^{2N+1}$. If $(u, z) \in H^1_{\text{loc}}(U' \times U'', \mathbb{R}^{2n}) \times H^1_{\text{loc}}(U', S^{2N+1})$ is a weak solution of system \eqref{floerred}, then it is smooth.
\end{proposition}

Solutions of the equation
\begin{equation}\label{s1morsegradient}
    \dot{z}(s) - \nabla \widetilde{f}_N(z) = 0
\end{equation} 
are a priori contained in a compact region since $S^{2N+1}$ is compact. On the other hand, if we have uniform linear growth of $X_{H_{z,t}}$ and $\widetilde{f}_N$-non-resonance at infinity for $H \in C^\infty(\mathbb{T} \times \mathbb{R}^{2n} \times S^{2N+1})$, by a small modification of the arguments in \cite[Proposition 3.1]{AK22}, we obtain the following result.

\begin{proposition}\label{compactnessfloer}
    Let $H \in \mathcal{H}_F^{S^1, N}(\widetilde{f}_N, g_N)$. Let $I$ be an open unbounded interval, $I'$ an interval (possibly unbounded) such that $\overline{I'} \subset I$, and $J$ a uniformly bounded $\omega_0$-compatible almost complex structure on $\mathbb{R}^{2n}$ smoothly depending on $(s, t, z) \in I \times \mathbb{T} \times S^{2N+1}$. For every $E > 0$, there exists a positive constant $M = M(E) > 0$ such that every solution $(u, z) \in C^{\infty}(I \times \mathbb{T}, \mathbb{R}^{2n}) \times C^{\infty}(I, S^{2N+1})$ of system \eqref{floerred} with energy bound
    \[
\int_{I \times \mathbb{T}} \|\partial_s u\|_{J_{z(s)}}^2 \, ds \, dt \leq E
\]
satisfies
\[
\sup_{(s, t) \in I' \times \mathbb{T}} |u(s, t)| \leq M.
\]
\end{proposition}

\begin{proof}
    We use $|\cdot|$ to denote the standard norms on $\mathbb{R}^{2n}$ and $S^{2N+1}\subset \mathbb{R}^{2N+2}$. Since $J$ is uniformly bounded, the metric $g_J = \omega_0(\cdot, J \cdot)$ is uniformly equivalent to the standard Euclidean metric on $\mathbb{R}^{2n}$. Thus, we can assume that our energy bound has the form
\begin{equation}\label{uenergybound}
    \int_{I \times \mathbb{T}} |\partial_s u|^2 \, ds \, dt \leq E.
\end{equation}

Since $S^{2N+1}$ is compact, every solution $z$ of \eqref{s1morsegradient} satisfies
\[\int_{I} |\dot{z}(s)|^2 \, ds \leq E_{\widetilde{f}_N} \defeq \max \widetilde{f}_N - \min \widetilde{f}_N.\]
This inequality, together with \eqref{uenergybound}, imply that
\[\int_{I \times \mathbb{T}} |\partial_s u|^2 \, ds \, dt + \int_{I} |\dot{z}(s)|^2 \, ds \leq E + E_{\widetilde{f}_N}.\]
By property (F5), $H$ is $\widetilde{f}_N$-non-resonant at infinity. Therefore, we can choose $\varepsilon > 0$ as specified in the definition of $\widetilde{f}_N$-non-resonance at infinity for $H$. Let $\delta = \min \{\varepsilon, \frac{\varepsilon}{\|J\|_\infty}\}$ and
\[D = D(u, z) = \left\{ s \in \mathbb{R} \mid \|\partial_s u(s, \cdot)\|_{L^2(\mathbb{T})} + |\dot{z}(s)| < \delta \right\}.\]
From the Chebyshev inequality, we have
\[|I \setminus D| \leq \frac{1}{\delta^2} \int_I \left(\|\partial_s u(s, \cdot)\|_{L^2(\mathbb{T})} + |\dot{z}(s)|\right)^2 \, ds = \frac{2}{\delta^2} \left(\int_{I \times \mathbb{T}} |\partial_s u|^2 \, ds \, dt + \int_{I} |\dot{z}(s)|^2 \, ds\right) < L,\]
where
\[L = L(E)\defeq 2 \frac{E + E_{\widetilde{f}_N}}{\delta^2} + 1,\]
and $|I \setminus D|$ is the measure of the set $I \setminus D$.
Any interval $I'$ of length $L$ must intersect $D$ in at least one point $s_0\in I'$. For such a point, it holds
\begin{align}
    \nonumber & \|\partial_t u(s_0, \cdot) - X_{\cdot, z(s_0)}(u(s_0, \cdot))\|_{L^2(\mathbb{T})} \leq \|J\|_{\infty} \|\partial_s u(s_0, \cdot)\|_{L^2(\mathbb{T})} < \|J\|_{\infty} \delta \leq \varepsilon, \\
    \nonumber & |\nabla \widetilde{f}_N(z(s_0))| = |\dot{z}(s_0)| < \delta \leq \varepsilon.
\end{align}

Therefore, we have a positive number $r > 0$ such that
\[\|\partial_s u(s_0, \cdot)\|_{L^2(\mathbb{T})} \leq r.\]
Since property (F4) implies that $X_H$ has uniform linear growth, the remainder of the proof follows similarly to proof of \cite[Proposition 3.1]{AK22}.
\end{proof}

\begin{lemma}\label{filtrations1equfloer} Let $H\in \mathcal{H}_F^{S^1,N}(\widetilde{f}_N,g_N)$ and $(u,z)\in \hat{\mathcal{M}}^{S^1,N}_F(S_{(x,v)},S_{(y,w)};H,\widetilde{f}_N,J)$. Then it holds:

    \begin{enumerate}
        \item $\PhiH$ is non-decreasing along the $(u,z)$.
        \item $E(u)\defeq\int\limits_{\mathbb{R}\times \mathbb{T}} \|\partial_su\|_J^2dsdt\leq \PhiH(x,v)-\PhiH(y,w)$.
    \end{enumerate}
\end{lemma}

\begin{proof}
    Differentiating $\PhiH(u(s,\cdot),z(s))$ with respect to $s$ we get 
    \begin{gather*}
        \frac{\partial}{\partial s}\PhiH(u(s,\cdot),z(s))=\frac{\partial\PhiH}{\partial x}(\partial_su(s,\cdot),z(s))+\frac{\partial\PhiH}{\partial z}(u(s,\cdot),\partial_sz(s))\\
        =\int_\mathbb{T} \|u(s,t)\|_J^2dt-\int_\mathbb{T}g_N(\nabla_zH(u(s,t)),\nabla \widetilde{f}_N(z(s)))dt\geq 0.
    \end{gather*}

The last inequality holds due to the property (F3) since we have that 
\[g_N(\nabla_zH(u(s,t)),\nabla \widetilde{f}_N(z(s)))\leq 0\]
for every $(s,t)\in \mathbb{R}\times \mathbb{T}$. This proves the claim (1).

Integrating the both sides of equality with respect to $s$ we get \[\PhiH(x,v)-\PhiH(y,w)=\int_{\mathbb{R}\times \mathbb{T}} \|u(s,t)\|_J^2dsdt-\int_{\mathbb{R}\times \mathbb{T}}g_N(\nabla_zH(u(s,t)),\nabla \widetilde{f}_N(z(s)))dsdt.\]
Since 
\[\int_{\mathbb{R}\times \mathbb{T}}g_N(\nabla_zH(u(s,t)),\nabla \widetilde{f}_N(z(s)))dsdt\leq 0,\]

we get that the second claim also holds.

\end{proof}

The claim (2) of the previous Lemma gives us uniform energy bound on 
\[\hat{\mathcal{M}}^{S^1,N}_F(S_{(x,z)},S_{(y,w)};H,\widetilde{f}_N, J).\]
Consequently, Proposition \ref{compactnessfloer} implies uniform boundness of \[\hat{\mathcal{M}}^{S^1,N}_F(S_{(x,z)},S_{(y,w)};H,\widetilde{f}_N,J)\] in $L^\infty$-norm
 and pre-compactness in
 \[C^\infty_{loc}(\mathbb{R}\times \mathbb{T},\mathbb{R}^{2n})\times C^\infty_{loc}(\mathbb{R},S^{2N+1})\]
by a standard bubbling-off and elliptic bootstrap argument.
We have that for a generic $S^1$-invariant almost complex structure $J$, the space \[\hat{\mathcal{M}}^{S^1,N}_F(S_{(x,v)},S_{(y,w)};H,\widetilde{f}_N,J)\] 
is a smooth manifold with 
\[\dim\hat{\mathcal{M}}^{S^1,N}_F(S_{(x,v)},S_{(y,w)};H,\widetilde{f}_N,J)=\mu(x,z)-\mu(y,w)+1.\]
Transversality is a consequence of the fact that the system of equations \eqref{floerred} reduces to a continuation equation, as shown in \cite{BO17}. For results concerning dimension, see \cite{BO13a,BO13b} and Lemma \ref{uppertriangular}. 

For transversality results in the $S^1$-equivariant setting for Borel construction (see \cite{BO17}), see \cite{BO10}, and for results concerning dimension, see \cite{BO13a} and \cite{BO13b}.

Since we have a free $\mathbb{R}\times \mathbb{T}$ action on this space we conclude that \[\mathcal{M}^{S^1,N}_F(S_{(x,v)},S_{(y,w)};H,\widetilde{f}_N,J)\defeq\hat{\mathcal{M}}^{S^1,N}_F(S_{(x,v)},S_{(y,w)};H,\widetilde{f}_N,J)/\mathbb{R}\times \mathbb{T}\]
is a smooth manifold of dimension $\mu(x,z)-\mu(y,w)-1.$

\textbf{$S^1$-equivariant Floer complex}

 From the previous arguments we have a well-defined $S^1$-equivariant Floer complex \[\{\C\F^{S^1,N}_*(H,\widetilde{f}_N),\partial^{F,S^1}\}\] 
for a generic $S^1$-invariant, uniformly bounded, $\omega_0$-compatible almost complex structure $J(t,x,z)$ where $(t,x,z)\in \mathbb{T}\times \mathbb{R}^{2n}\times S^{2N+1}$.
 
 Indeed, $\C\F^{S^1,N}_*(H,\widetilde{f}_N)$ is a $\mathbb{Z}_2$-vector space generated by critical orbits of $P(H,\widetilde{f}_N)$. Grading of this complex is given by 
 \[|S_{(x,v)}|\defeq\mu(x,v)+N\] The boundary operator 
 \[\partial^{F,S^1}:\C\F^{S^1,N}_*(H,\widetilde{f}_N)\to \C\F^{S^1,N}_{*-1}(H,\widetilde{f}_N)\] 
 is defined as \[\partial^{F,S^1}(S_{(x,v)})=\sum n^{S^1}_F(S_{(x,v)},S_{(y,w)})S_{(y,w)},\]
where the sum goes over all $S_{(y,w)}$ such that $|S_{(y,w)}|=|S_{(x,v)}|-1$ and $n^{S^1}_F(S_{(x,v)},S_{(y,w)})$ is parity of the set $\mathcal{M}^{S^1,N}_F(S_{(x,v)},S_{(y,w)};H,\widetilde{f}_N,J)$. When $|S_{(y,w)}|=|S_{(x,v)}|-1$ by compactness results we conclude that $\mathcal{M}^{S^1,N}_F(S_{(x,v)},S_{(y,w)};H,\widetilde{f}_N,J)$ is compact $0$-dimensional manifold and therefore $\mathcal{M}^{S^1,N}_F(S_{(x,v)},S_{(y,w)};H,\widetilde{f}_N,J)$ is a finite set which implies that $\partial^{F,S^1}$ is well defined. The standard gluing and compactness argument gives us that \[\partial^{F,S^1}\circ \partial^{F,S^1}=0.\]
Thus, we have well-defined $S^1$-equivariant Floer homology 
\[\H\F^{S^1,N}_*(H,\widetilde{f}_N;J)=\H_*(\{\C\F^{S^1,N}_*(H,\widetilde{f}_N),\partial^F\}).\]
We get chain isomorphic $S^1$-equivariant Floer chain complexes and, therefore, the same homology for different choices of almost complex structures.

Filtration of this complex by $\Phi_H$ is well-defined because of claim (1) of Lemma \ref{filtrations1equfloer}. In other words by this property we have that $\partial^{F,S^1}$ maps $\C\F^{S^1,N,<a}_*(H)$ to $\C\F^{S^1,N,<a}_{*-1}(H)$. 

\textbf{Continuation maps } 

For pairs $(N_\alpha, H_\alpha)$ and $(N_\beta,H_\beta)$, where $H_\alpha\in \mathcal{H}_F^{S^1,N_\alpha}(\widetilde{f}_{N_\alpha},g_{N_\alpha})$ and $H_\beta\in\mathcal{H}_F^{S^1,N_\beta}(\widetilde{f}_{N_\beta},g_{N_\beta})$ we have the relation \[(N_\alpha, H_\alpha)\leq (N_\beta,H_\beta)\]
given by \[N_\alpha\leq N_\beta, \ \ \ H_\alpha\leq i_{N_\beta,N_\alpha}^*H_\beta\]

where $i_{N_\beta,N_\alpha}: S^{2N_\alpha+1}\to S^{2N_\beta+1}, \ \ \ i_{N_\beta,N_\alpha}(z)=(z,0)$.

Then, in a similar way as in the non-parameterized case, we can define continuation maps by non-increasing homotopy $(H^{\alpha\beta}_s,J^{\alpha\beta}_s)$ from $i_{N_\beta,N_\alpha}^*H_\beta$ to $H_\alpha$. Generically, this induces a chain map \[i^{F,S^1}_{(H^{\alpha\beta}_s,J^{\alpha\beta}_s)}:\C\F^{S^{1},N_\alpha}_*(H_\alpha,\widetilde{f}_{N_\alpha};J_\alpha)\to C\F^{S^{1},N_\beta}_*(H_\beta,\widetilde{f}_{N_\beta};J_\beta),\]

which is unique up to chain homotopy, and therefore induces a continuation map in homology. Due to non-increasing property of the homotopy, $i^{F,S^1}_{(H^{\alpha\beta}_s,J^{\alpha\beta}_s)}$ preserves the action filtration.

\section{Morse complex of the dual action functional}\label{sectionmorse}

In this section, we examine the ($S^1$-equivariant) Morse complex associated with the reduced dual action functional for a particular type of admissible ($S^1$-invariant) Hamiltonians.

\large\textbf{Non-parameterized case} 

\normalsize

Let
\[\Honenull \defeq \{x \in \Hone \mid \int_\mathbb{T} x(t) \, dt = 0\}.\]
This space is equipped with the natural projection
\[\mathbb{P} : \Hone \to \Honenull, \quad \mathbb{P}(x) = x - \int_\mathbb{T} x(t) \, dt.\]
\textbf{Quadratic convexity:} The smooth Hamiltonian
\[H : \mathbb{T} \times \mathbb{R}^{2n} \to \mathbb{R}\]
is said to be quadratically convex if there exist positive constants $\underline{h} > 0$ and $\overline{h} > 0$ such that
\[\underline{h} I \leq \nabla^2 H_t(x) \leq \overline{h} I, \quad x \in \mathbb{R}^{2n}, \, t \in \mathbb{T},\]
where $I$ is the identity isomorphism.

Let $H$ be quadratically convex. Then, the dual action functional is defined as
\[\PsiH : \Honenull \to \mathbb{R}, \quad \PsiH(x) = -\int_\mathbb{T} x^* \lambda_0 + \int_\mathbb{T} H_t^*(-J_0 \dot{x}(t)) \, dt,\]
where
\[H^*_t(x) = \max_y \left(\langle x, y \rangle - H_t(y) \right)\]
is the Fenchel conjugate of $H$.

Unlike $\PhiH$, the dual action functional is not $C^2$ except when $H$ is quadratic. However, when $H$ is quadratically convex, $\PsiH$ is $C^1$ and twice Gateaux-differentiable, with the Gateaux second derivative being a continuous symmetric bilinear map of finite Morse index.

From \cite[Lemma 5.1]{AK22}, we know that $P(H)$ and $\crit(\PsiH)$ are in one-to-one correspondence, with the bijection given by
\[\mathbb{P}|_{P(H)} : P(H) \to \crit(\PsiH),\]
and it holds that
\[\PhiH(x) = \PsiH(x_0),\]
where $x_0 = \mathbb{P}(x)$. Additionally, from \cite[Proposition 1.1]{AK22} and \cite[Proposition 5.3]{AK22}, we have
\begin{equation}\label{nullityindex}
    \nullity(x_0; \PsiH) = \dim \ker(I - \varphi^1_{X_H}(x)), \quad \ind(x_0; \PsiH) = \mu_{CZ}(x; H).
\end{equation}

The issue that $\PsiH$ is not $C^2$ can be addressed as in \cite{AK22}. Elements $x \in \Hone$ can be written in the form of a Fourier series, i.e.,
\[x(t) = \sum_{k \in \mathbb{Z}} e^{2\pi k t J_0} \hat{x}(k), \quad \hat{x}(k) \in \mathbb{R}^{2n},\]
such that
\[\sum |k|^2 |\hat{x}(k)|^2 < \infty.\]

Thus, for every $l \in \mathbb{N}$, the space $\Hone$ can be decomposed as
\[\Hone = \mathbb{H}_0 \oplus \mathbb{H}_l \oplus \mathbb{H}^l,\]
where
\[\mathbb{H}_l = \{x \in \Hone \mid \hat{x}(k) = 0 \text{ if } k < 1 \text{ or } k > l\},\]
and
\[\mathbb{H}^l = \{x \in \Hone \mid \hat{x}(k) = 0 \text{ if } 0 \leq k \leq l\},\]
with $\mathbb{H}_0 \cong \mathbb{R}^{2n}$ denoting the space of constant loops.

This decomposition is orthogonal. We denote the orthogonal projection to $\mathbb{H}_l$ by
\[\mathbb{P}_l : \Hone \to \mathbb{H}_l.\]
It is clear that
\[\Honenull = \mathbb{H}_l \oplus \mathbb{H}^l.\]
As shown in \cite{AK22}, for sufficiently large $l \in \mathbb{N}$, we obtain a saddle point reduction of $\PsiH$, i.e., for $2\pi (l+1) > \overline{h}$ (where $\overline{h}$ is the constant from the definition of quadratic convexity of $H$), the second differential of the function
\[\mathbb{H}^l \to \mathbb{R}, \quad y \mapsto \PsiH(x + y)\]
is bounded below by a coercive quadratic form. Thus, this map is strictly convex and has a unique non-degenerate critical point, which is a global minimizer. This defines a map
\[Y_l : \mathbb{H}_l \to \mathbb{H}^l,\]
where $Y_l(x)$ is the global minimizer of the previously mentioned function. We denote by
\[\Gamma_l : \mathbb{H}_l \to \Honenull, \quad \Gamma_l(x) = (x, Y_l(x)),\]
the graph of this map.

We can now define the $l$-reduction of the dual functional as
\[\psilH : \mathbb{H}_l \to \mathbb{R}, \quad \psilH(x) = \PsiH(\Gamma_l(x)).\]

This method was also used in \cite{Vit89}.
\begin{remark}
    In \cite{AK22}, the reduced dual functional is defined as the restriction of $\PsiH$ to the manifold $\Gamma_l(\mathbb{H}_l)$. This does not make any difference in the context of the Morse complex. However, defining this functional on $\mathbb{H}_l$ is more suitable for us since it simplifies the definition of the continuation maps for Morse complexes.    
\end{remark}

From \cite[Proposition 6.2]{AK22}, we have that if $H$ is quadratically convex, $Y_l$ is smooth and takes values in $C^\infty(\mathbb{T}, \mathbb{R}^{2n})$, with
\[\Gamma_l(\mathbb{H}_l) \subset \Honenull\]
being a smooth submanifold of dimension $2nl$. Moreover, $\psilH$ is a smooth function, and $x_0 \in \crit(\PsiH)$ if and only if $x_l \defeq \mathbb{P}_l(x_0) \in \crit(\psilH)$. Furthermore, we have the identity
\[\nullity(x_l; \psilH) = \nullity(x_0; \PsiH), \quad \ind(x_l; \psilH) = \ind(x_l; \PsiH).\]

\begin{remark}\label{nullityindexred}
Since $\mathbb{P}_l \circ \mathbb{P}=\mathbb{P}_l$, from the previous conclusions, we have that
\[\mathbb{P}_l|_{P(H)} : P(H) \to \crit(\psilH)\]
is a bijection and it holds that $\PhiH(x) = \psilH(x_l)$, where $x_l = \mathbb{P}_l(x)$. Additionally, from the above equalities and \eqref{nullityindex}, it follows that
\[\nullity(x_l; \psilH) = \dim\ker(I - \varphi^1_H(x)), \quad \ind(x_l; \psilH) = \mu_{CZ}(x; H) - n.\]    
\end{remark}

\begin{remark}\label{psnonparameterized}
    If we also assume that $H$ is non-resonant at infinity, then from \cite[Proposition 6.2]{AK22} we have that the pair $(\PsiH, \langle \cdot, \cdot \rangle_{\Honenull})$ satisfies the Palais-Smale condition. Therefore, for the pullback $g_{H^1_0}$ of the standard metric on $\Honenull$ to $\mathbb{H}_l$ via the map $\Gamma_l$, the pair $(\psilH, g_{H^1_0})$ satisfies the Palais-Smale condition. Since $\Gamma_l$ is bilipschitz, we conclude that $g_{H^1_0}$ is equivalent to the standard metric on $\mathbb{H}_l$. Thus, for every metric $g$ on $\mathbb{H}_l$ that is equivalent to the standard one, the pair $(\psilH, g)$ satisfies the Palais-Smale condition.
\end{remark}

\textbf{Dually admissible Hamiltonians:} We define the family $\mathcal{H}_\Theta$ to be the family of all non-degenerate Hamiltonians $H: \mathbb{T} \times \mathbb{R}^{2n} \to \mathbb{R}$ that are quadratically convex and non-resonant at infinity. 

\textbf{Morse Complex}

Since $H$ is non-degenerate, from Remark \eqref{nullityindexred}, we conclude that $\psilH$ is a smooth Morse function. Additionally, since $H$ is non-resonant at infinity, it has finitely many critical points. Therefore, for a generic metric $g$ on $\mathbb{H}_l$ that is equivalent to the standard metric, the pair $(\psilH, g)$ is a Morse-Smale pair. Such a pair will satisfy the Palais-Smale condition due to Remark \ref{psnonparameterized}.

Therefore, generically, the Morse complex
\[\{CM_*(\psilH), \partial^M\},\]
is well-defined. The space $CM_*(\psilH)$ is a $\mathbb{Z}_2$-vector space whose generators are the critical points of $\psilH$, and it is graded by the Morse index of these critical points. The boundary operator
\[\partial^M: CM_*(\psilH) \to CM_{*-1}(\psilH)\]
is defined by
\[\partial^M(x_l) = \sum_{\ind(y_l) = \ind(x_l) - 1} n_M(x_l, y_l) y_l,\]
where $n_M(x_l, y_l)$ is the parity of the set of non-parametrized negative gradient flow lines of $\psilH$ going from $x_l$ to $y_l$, i.e., the space
\[\mathcal{M}_M(x_l, y_l; \psilH, g) \defeq \hat{\mathcal{M}}_M(x_l, y_l; \psilH, g) / \mathbb{R} = W^u(x_l, -\nabla_g \psilH) \cap W^s(y_l, -\nabla_g \psilH) / \mathbb{R}.\]
This is well-defined due to compactness arguments, and it holds that
\[\partial^M \circ \partial^M = 0\]
by standard compactness and gluing results. The Morse complex is uniquely defined up to chain isomorphism. We have a natural filtration on this complex by $\psilH$.

\textbf{Continuation Maps}

Let $H_\alpha, H_\beta \in \mathcal{H}_\Theta$ be such that $H_\alpha \leq H_\beta$. Then it follows that $H^*_\beta \leq H^*_\alpha$, and therefore $\Psi_{H_\beta^*} \leq \Psi_{H_\alpha^*}$. For large enough $l \in \mathbb{N}$, we have that the $l$-saddle point reduction exists for both functionals. From the definition of the reduction, it is clear that
\[\psi^l_{H_\beta^*} \leq \psi^l_{H_\alpha^*}.\]

Let $g_\alpha$ and $g_\beta$ be metrics uniformly equivalent to the standard one on $\mathbb{H}_l$ such that the Morse complexes for $(\psi^l_{H_\alpha^*}, g_\alpha)$ and $(\psi^l_{H_\alpha^*}, g_\beta)$ are well-defined.

We define a non-increasing homotopy $(\psi^{{\alpha\beta}}_s, g^{\alpha\beta}_s)$, i.e., a homotopy $\psi^{\alpha\beta}_s$ such that $\partial_s \psi^{\alpha\beta}_s \leq 0$, where
\[(\psi^{\alpha\beta}_s, g^{\alpha\beta}_s) = (\psi^l_{H_\alpha^*}, g_\alpha), \ \ s \leq T_0 - 1,\]
and
\[(\psi^{\alpha\beta}_s, g^{\alpha\beta}_s) = (\psi^l_{H_\beta^*}, g_\beta), \ \ s \geq T_0,\]
where $T_0 < -1$ is fixed.

For $x^\alpha_l \in \crit(\psi^l_{H_\alpha^*})$, we define the unstable set of the homotopy $(\psi^{\alpha\beta}_s, g^{\alpha\beta}_s)$ to be
\[W^u(x^\alpha_l; (\psi^{\alpha\beta}_s, g^{\alpha\beta}_s)) = \{y \in \mathbb{H}_l \mid \dot{\gamma}(s) = -\nabla_{g^{\alpha\beta}_s} \psi^{\alpha\beta}_s(\gamma(s)), \ \ \gamma(s_0) = y, \ \ s_0 \in \mathbb{R}, \ \ \lim_{s \to -\infty} \gamma(s) = x^\alpha_l\},\]
and for $y_l^\beta \in \crit(\psi^l_{H_\beta^*})$, we define the stable manifold of the homotopy $(\psi^{\alpha\beta}_s, g^{\alpha\beta}_s)$ to be
\[W^s(y_l^\beta; (\psi^{\alpha\beta}_s, g^{\alpha\beta}_s)) = \{y \in \mathbb{H}_l \mid \dot{\gamma}(s) = -\nabla_{g^{\alpha\beta}_s} \psi^{\alpha\beta}_s(\gamma(s)), \ \ \gamma(s_0) = y, \ \ s_0 \in \mathbb{R}, \ \ \lim_{s \to +\infty} \gamma(s) = y_l^\beta\}.\]

Now, for $x_l^\alpha \in \crit(\psi_{H_\alpha^*})$ and $y_l^\beta \in \crit(\psi_{H_\beta^*})$, we define the moduli space
\[\mathcal{M}_M(x_l^\alpha, y_l^\beta; \psi^{\alpha\beta}) = W^u(x_l^\alpha; (\psi^{\alpha\beta}_s, g^{\alpha\beta}_s)) \cap W^s(y_l^\beta; (\psi^{\alpha\beta}_s, g^{\alpha\beta}_s)).\]

For a generic choice of the homotopy, $\mathcal{M}_M(x_l^\alpha, y_l^\beta; \psi^{\alpha\beta})$ is a smooth manifold of dimension $\ind(x_l^\alpha; \psi^l_{H_\alpha^*}) - \ind(x_l^\beta; \psi^l_{H_\beta^*})$.

Therefore, we define the chain map
\[i^M_{(\psi^{\alpha\beta}_s, g^{\alpha\beta}_s)}: CM_*(\psi_{H^*_\alpha}, g_\alpha) \to CM_*(\psi_{H^*_\beta}, g_\beta)\]
as
\[i^M_{(\psi^{\alpha\beta}_s, g^{\alpha\beta}_s)}(x_l^\alpha) = \sum_{\ind(y_l^\beta; \psi^l_{H^*_\beta}) = \ind(x_l^\alpha; \psi^l_{H^*_\alpha})} n_M(x_l^\alpha, y_l^\beta) y_l^\beta,\]
where $n_M(x_l^\alpha, y_l^\beta)$ denotes the parity of the set $\mathcal{M}_M(x_l^\alpha, y_l^\beta; \psi^{\alpha\beta})$. This map is well-defined due to compactness arguments, and it is a chain map by compactness and gluing. Moreover, it preserves the filtration because $\psi^{\alpha\beta}_s$ is a non-increasing homotopy. Using homotopies of homotopies, we see that such a chain map is unique up to chain homotopy and, therefore, induces the same map in homology called continuation map.

\large\textbf{$S^1$-equivariant case}

\normalsize

Now, we will discuss the $S^1$-equivariant Morse complex associated with the reduced dual functional corresponding to $S^1$-invariant admissible Hamiltonians. 

\textbf{Uniform Quadratic Convexity:} The parameterized smooth Hamiltonian 
\[H:\mathbb{T}\times \mathbb{R}^{2n}\times S^{2N+1} \to \mathbb{R}\] 
is said to be uniformly quadratically convex if there exist positive numbers $\underline{h} > 0$ and $\overline{h} > 0$ such that 
\[\underline{h}I \leq\nabla^2 H_{z,t}(x)\leq \overline{h}I, \ \ \ t\in \mathbb{T}, x\in \mathbb{R}^{2n}, z\in S^{2N+1},\]
where $I$ is the identity isomorphism.

For such a Hamiltonian, the parameterized Fenchel conjugate \[H^*_{z,t}(x)=\max\limits_y(\langle x,y\rangle -H_{z,t}(y))\]
 is smooth.

\begin{lemma}\label{parfenchelconjugatelemma}
Assume that $H^*$ is the parameterized, time-dependent Fenchel conjugate of $H$ (and that both are of sufficient regularity). Then it holds:
\begin{gather*}
H^*_{z,t}(\nabla_xH_{z,t}(x)) + H_{z,t}(x) = \langle \nabla_xH_{z,t}(x), x \rangle, \\
\nabla_xH^*_{z,t} \circ \nabla_xH_{z,t} = id, \ \ \ \nabla_z H^*_t(\nabla_xH_{t,z}(x),z) = -\nabla_z H_{t}(x,z),    
\end{gather*}
where $\nabla_z$ denotes the gradient operator on $S^{2N+1}$ with respect to some metric $g$.
\end{lemma}

\begin{proof}

From the definition of $H^*$, we have
\begin{equation}\label{parfencelconjeq}
H^*_{z,t}(x) = \langle x,y \rangle - H_{z,t}(y),    
\end{equation}
for $y = \nabla_x H^{-1}_{z,t}(x)$. Therefore, we have $x = \nabla_x H_{z,t}(y)$, which confirms that the first identity holds. Differentiating both sides of \eqref{parfencelconjeq} with respect to spatial variable gives 
\[dH_{z,t}^*(x) = \langle x, \cdot \rangle + \langle y, \cdot \rangle - dH_{z,t}(y).\]
Subtracting $y = \nabla_x H^{-1}_{z,t}(x)$ one gets 
\[dH_{z,t}^*(x) = \langle x, \cdot \rangle + \langle \nabla_x H^{-1}_{z,t}(x), \cdot \rangle - \langle x, \cdot \rangle = \langle \nabla_x H_{z,t}^{-1}(x), \cdot \rangle,\]
and therefore, we have 
\[\nabla_xH^*_{z,t} \circ \nabla_xH_{z,t} = id.\] 
On the other hand identity \[H^*_t(\nabla_x H_{z,t}(x),z) = \langle \nabla_x H_{z,t}(x), x \rangle - H_t(x,z)\]
holds. Differentiating both sides of this equation with respect to $z$, we obtain
\[\frac{\partial{H^*_t}}{\partial x}(\nabla_x H_{z,t}(x),z) \left[ \frac{\partial \nabla_x H_{z,t}}{\partial z}(x) \right] + \frac{\partial{H^*_t}}{\partial z}(\nabla_x H_{z,t}(x),z) = \left\langle \frac{\partial \nabla_x H_{z,t}}{\partial z}(x), x \right \rangle - \frac{\partial H_t}{\partial z}(x,z).\]
Combining this with the identity 
\[\frac{\partial{H^*_t}}{\partial x}(\nabla_x H_{z,t}(x),z) \left[ \frac{\partial \nabla_x H_{z,t}}{\partial z}(x) \right] = \left\langle \frac{\partial \nabla_x H_{z,t}}{\partial z}(x), x \right\rangle,\]
we get 
\[\nabla_z H^*_t(\nabla_xH_{t,z}(x),z) = -\nabla_z H_{t}(x,z).\]
 \end{proof}

\textbf{$\widetilde{f}_N$-Dually admissible $S^1$-invariant Hamiltonians:} $H \in \mathcal{H}_\Theta^{S^1,N}(\widetilde{f}_N,g_N)$ if it satisfies:
\begin{enumerate}
    \item[(H1)] $H$ is $S^1$-invariant.
    \item[(H2)] For every $v \in \crit(\widetilde{f}_N)$, the Hamiltonian $H_v$ is non-degenerate.
    \item[(H3)] $H$ is non-increasing along the gradient flow lines of $\widetilde{f}_N$.
    \item[(H4)] $H$ is uniformly quadratically convex, and the partial derivatives $\frac{\partial^2 H}{\partial x \partial z}$ and $\frac{\partial^2 H}{\partial z^2}$ are bounded.
    \item[(H5)] $H$ is $\widetilde{f}_N$-non-resonant at infinity.
\end{enumerate}
Let $H \in \mathcal{H}_\Theta^{S^1,N}(\widetilde{f}_N,g_N)$. Since $H$ is uniformly quadratically convex (by (H4)), we have a smooth parameterized Fenchel conjugate $H^*$ of $H$ and the parameterized dual action functional
\[\PsiH: \Honenull \times S^{2N+1} \to \mathbb{R}, \quad \PsiH(x) = -\int\limits_\mathbb{T} x^* \lambda_0 + \int\limits_\mathbb{T} H^*_{z,t}(-J_0 \dot{x}(t)) \, dt.\] We define the set of critical points of the pair $(\PsiH, \widetilde{f}_N)$ as
\[\crit(\PsiH, \widetilde{f}_N) \defeq \{ (x_0, v) \in \Honenull \times S^{2N+1} \mid x_0 \in \crit(\Psi_{H^*_v}), \, v \in \crit(\widetilde{f}_N) \},\]
and their Morse index as
\[\ind(x_0, v) \defeq \ind(x_0; \Psi_{H^*_v}) + \ind(v; \widetilde{f}_N), \quad (x_0, v) \in \crit(\PsiH, \widetilde{f}_N).\]

As we will see, $\PsiH$ is $S^1$-invariant, and therefore, on $\crit(\PsiH, \widetilde{f}_N)$ we have a free $S^1$-action. The $S^1$-orbit of $(x_0, v)\in \crit(\PsiH, \widetilde{f}_N)$ is denoted by $S_{(x_0,v)}$ and the index is defined by
\[\ind(S_{(x_0, v)})\defeq\ind(x_0, v).\] 

\begin{proposition}\label{correspondancecriticals1equi}
    Let $H \in \mathcal{H}_\Theta^{S^1,N}(\widetilde{f}_N, g_N)$.
    \begin{enumerate}
        \item The dual action functional $\PsiH$ is $S^1$-invariant, and the orbits of $P(H, \widetilde{f}_N)$ are in one-to-one correspondence with the orbits of $\crit(\PsiH, \widetilde{f}_N)$. This bijection is given by
        \[\mathbb{P} \times \text{id}_{S^1} \big|_{P(H, \widetilde{f}_N)} : P(H, \widetilde{f}_N) \to \crit(\PsiH, \widetilde{f}_N).\]
        Under this map, it holds that $\PhiH(x, z) = \PsiH(x_0, z)$ where $x_0 = \mathbb{P}(x)$. 
        \item For all $v \in \crit(\widetilde{f}_N)$, the dual functional $\Psi_{H^*_v}$ is non-degenerate functional with finite Morse index. Moreover, it holds
        \[\ind(S_{(x_0, v)}) = |S_{(x, v)}| - n, \quad (x, v) \in P(H, \widetilde{f}_N).\]
    \end{enumerate}
\end{proposition}

\begin{proof}

\textbf{Claim} (1):
From property (H1), we have that
\[H_{\theta \cdot z}(t+\theta, \cdot) = H_z(t, \cdot),\]
for all $\theta \in \mathbb{T}$, $t \in \mathbb{T}$, and $z \in S^{2N+1}$. Since, by definition, we have
\[H^*_z(t, x) = \max_{y \in \mathbb{R}^{2n}} \left(x \cdot y - H_z(t, y) \right),\]
it follows that
\[H_{\theta \cdot z}^*(t + \theta, x) = \max_{y \in \mathbb{R}^{2n}} \left(x \cdot y - H_{\theta \cdot z}(t + \theta, y) \right) = \max_{y \in \mathbb{R}^{2n}} \left(x \cdot y - H_z(t, y) \right) = H^*_z(t, x),\]
for all $\theta, t \in \mathbb{T}$ and $z \in S^{2N+1}$. Thus, $H^*$ is $S^1$-invariant. On the other hand, we have
\begin{gather*}
    \PsiH(\theta \cdot (x, z)) = \PsiH(x(\cdot - \theta), \theta \cdot z) = -\frac{1}{2} \int\limits_0^1 J_0 \dot{x}(s - \theta) \cdot x(s - \theta) ds+ \int\limits_0^1 H^*_{s, \theta \cdot z}(-J_0 \dot{x}(s - \theta))  ds\\
    \stackrel{t = s - \theta}{=} -\frac{1}{2} \int\limits_0^1 J_0 \dot{x}(t) \cdot x(t) \, dt + \int\limits_0^1 H^*_{t + \theta, \theta \cdot z}(-J_0 \dot{x}(t))  dt\\
    = -\frac{1}{2} \int\limits_0^1 J_0 \dot{x}(t) \cdot x(t) \, dt + \int\limits_0^1 H^*_{t, z}(-J_0 \dot{x}(t))  dt = \PsiH(x, z).
\end{gather*}

Therefore, $\PsiH$ is $S^1$-invariant. From \cite[Lemma 5.1]{AK22}, we have that $\mathbb{P} \times \text{id}_{S^1} \big|_{P(H, \widetilde{f}_N)}$ is a bijection, and  $\PhiH(x, v) = \PsiH(x_0, v)$. 

\textbf{Claim} (2):
From \cite[Proposition 1.1]{AK22}, \cite[Proposition 5.3]{AK22}, and property (H2), we conclude that $\Psi_{H^*_v}$ is non-degenerate functional with finite Morse index when $v \in \crit(\widetilde{f}_N)$. Since it holds
\[|S_{(x, v)}| = \mu_{CZ}(x; H_v) + \ind(v; \widetilde{f}_N),\]
and
\[\ind(S_{(x_0, v)}) = \ind(x_0; \Psi_{H^*_v}) + \ind(v; \widetilde{f}_N),\]
again from \cite[Proposition 1.1]{AK22} and \cite[Proposition 5.3]{AK22}, we obtain the desired result.

\end{proof}

\begin{proposition}\label{pscondition}
 Let $H \in \mathcal{H}_\Theta^{S^1,N}(\widetilde{f}_N, g_N)$. The pair $(\PsiH, \widetilde{f}_N)$ satisfies the P-S condition. More precisely, if $(x_k, z_k) \in \Honenull \times S^{2N+1}$ is a sequence such that $(d \Psi_{H^*_{z_k}}(x_k), d \widetilde{f}_N(z_k))$ converges strongly in $(\Honenull)^* \times T^*S^{2N+1}$ to zero when $k \to \infty$, then $(x_k,z_k)$ has a convergent subsequence.
\end{proposition}

\begin{proof}
The condition from the proposition is equivalent to the condition that
\begin{equation}\label{convpseudogradient}
    (\nabla_{H^1} \Psi_{H^*_{z_k}}(x_k), \nabla \widetilde{f}_N(z_k)) \to 0,
\end{equation}
as $k \to \infty$.  

We have that
\[\nabla_{H^1} \Psi_{H^*_{z_k}}(x_k) = \Pi(-J_0(x_k - \nabla_x H^*_{z_k, t}(-J_0 \dot{x}_k))) \defeq y_k,\]
where
\[\Pi : L^2(\mathbb{T}, \mathbb{R}^{2n}) \to \Honenull, \quad (\Pi v)(t) \defeq \int\limits_0^t v(s) \, ds - \int\limits_\mathbb{T} \left(\int\limits_0^t v(s) \, ds \right) dt.\]
Therefore, it holds that
\[-J_0(x_k - \nabla_x H^*_{z_k, t}(-J_0 \dot{x}_k)) = \dot{y}_k,\]
where $\dot{y}_k \to 0$ in $L^2(\mathbb{T}, \mathbb{R}^{2n})$ as $k \to \infty$ due to \eqref{convpseudogradient}. Rewriting this equality and then using Lemma \ref{parfenchelconjugatelemma}, we get
\begin{equation}\nonumber
    -J_0 \dot{x}_k = \nabla_x H_{t, z_k} (x_k - J_0 \dot{y}_k).
\end{equation}
Multiplying this by $J_0$, we obtain
\begin{equation}\label{rewriten}
    \dot{x}_k = X_{H_{z_k}} (x_k - J_0 \dot{y}_k).
\end{equation}

Since $H$ is globally uniformly quadratic, $\nabla_x H_{t, z_k}$, and therefore $X_{H_{t, z}}$, is globally Lipschitz continuous with respect to the spatial variable. Thus, using \eqref{rewriten}, we have
\begin{equation}\label{nonrxpart}
\begin{array}{cc}
     & \|\dot{x}_k - X_{H_{z_k, \cdot}}(x_k)\|_{L^2(\mathbb{T}, \mathbb{R}^{2n})} \\
     & \leq \overline{h} \|J_0 \dot{y}_k\|_{L^2(\mathbb{T}, \mathbb{R}^{2n})} = \overline{h} \|\dot{y}_k\|_{L^2(\mathbb{T}, \mathbb{R}^{2n})} \to 0
\end{array}
\end{equation}

Then, using the fact that
\[\|\nabla \widetilde{f}_N(z_k)\| \to 0,\]
convergence in \eqref{nonrxpart} and the fact that $H$ is $\widetilde{f}_N$-non-resonant at infinity, we conclude that $x_k$ is uniformly bounded in $L^2$. 

Since $S^{2N+1}$ is compact, $z_k$ has a convergent subsequence. To show that $x_k$ also has a convergent subsequence in $\Honenull$, since we have uniform boundedness of the sequence $x_k$, we can use the same methods as in the proof of \cite[Proposition 6.2]{AK22} to prove the claim.
\end{proof}

Similar to the case of the non-parametrized dual functional, $\PsiH$ is not $C^2$. Thus, for the Morse complex to be well-defined, we can use saddle point reduction as in \cite{AK22} to achieve sufficient regularity. As in the non-parameterized case, we will denote the $l$-saddle point reduction as $\psilH$. Critical points of a pair $(\psilH,\widetilde{f}_N)$ and their Morse index are defined in the same way as for the pair $(\PsiH,\widetilde{f}_N)$.

\begin{proposition}\label{sadlepointreductionS1} Let $H\in \mathcal{H}_\Theta^{S^1,N}(\widetilde{f}_N,g_N)$. If $l\in \mathbb{N}$ is large enough, the following holds.
    
    \begin{enumerate}
    \item For every $x\in \mathbb{H}_l$ and $z\in S^{2N+1}$ the restriction of $\PsiH$  to $\{x\}\times \mathbb{H}^l\times \{z\}$ has a unique critical point $(x,Y_{l,N}(x,z),z)$, which is a non-degenerate global minimizer of this restriction.
    \item The map $Y_{l,N}:\mathbb{H}_l\times S^{2N+1} \to \mathbb{H}^l$ takes values in $C^\infty(\mathbb{T},\mathbb{R}^{2n})$ and is smooth with respect to the $C^k$-norm on the target for any $k\in \mathbb{N}$. If we denote by 
    \[\Gamma_{l,N}:\mathbb{H}_l\times S^{2N+1}\to \Honenull\times S^{2N+1}, \ \ \ \Gamma_{l,N}(x,z)=(x,Y_{l,N}(x,z),z)\]
    the graph of $Y_{l,N}$, the manifold $\Gamma_{l,N}(\mathbb{H}_l\times S^{2N+1})$
    is a smooth $(2nl+2N+1)-$dimensional submanifold of $\Honenull\times S^{2N+1}$ such that $\Gamma_{l,N}(\mathbb{H}_l\times S^{2N+1})\subset C^\infty(\mathbb{T},\mathbb{R}^{2n})\times S^{2N+1}$.
    \item The restricted dual action functional defined as \[\psilH:\mathbb{H}_l\times S^{2N+1}\to \mathbb{R}, \ \ \ \psilH(x,z)=\PsiH(\Gamma_{l,N}(x,z))\] is smooth.
    \item The map $Y_{l,N}$ is $S^1$-equivariant. Consequently, the manifold $\Gamma_{l,N}(\mathbb{H}_l\times S^{2N+1})$ and $\psilH$ are $S^1$-invariant. 
    \item The map 
    \[\mathbb{P}_l\times id_{S^1}|_{\crit(\PsiH,\widetilde{f}_N)}:\crit(\PsiH,\widetilde{f}_N)\to \crit(\psilH,\widetilde{f}_N)\]
    is a bijection. Moreover it holds that 
    \[\PsiH(x_0,v)=\psilH(x_l,v),\]
    where $x_l=\mathbb{P}_l(x_0)$.
    \item For all $v\in \crit(\widetilde{f}_N)$, the function $\psi^l_{H_v^*}$ is non-degenerate. Moreover for $S_{(x_l,v)}\subset \crit(\psilH,\widetilde{f}_N)$ it holds 
    \[\ind(S_{(x_l,v)})=\ind(S_{(x_0,v)}).\]
    \item $\psilH$ is non-inceasing along flow lines of $-\nabla \widetilde{f}_N$.
\end{enumerate}

\end{proposition}

\begin{proof} Statments $(1)-(3)$ have the same proof as \cite[Proposition 6.2.]{AK22}. 

\textbf{Claim} (4):  From the fact that $Y(x,y)$ is a critical point of the function 
\[\mathbb{H}^l\to \mathbb{R},\ \ \ y\mapsto \psilH(x+y,z)\]
we have that $y=Y(x,z)$ if and only if $\nabla_{H^1}\Psi_{H_z^*}(x+y)\perp \mathbb{H}^l$ which is equivalent to the condition $\nabla_{H^1}\Psi_{H_z^*}(x+y)\in \mathbb{H}_l$. Thus, we have 
\begin{equation}\label{conditionreduction}
  \Pi(-J_0(x+y-\nabla_xH^*_{z, t}(-J_0(\dot{x}+\dot{y}))))=u,  
\end{equation}
for some $u\in \mathbb{H}_l$, where 
\[\Pi:L^2(\mathbb{T},\mathbb{R}^{2n})\to \mathbb{H}_1,\ \ \  (\Pi v)(t)\defeq\int\limits_0^tv(s)ds-\int\limits_\mathbb{T}\left(\int\limits_0^tv(s)ds\right)dt.\]
Differentiating \eqref{conditionreduction} we get
\begin{equation}\label{conditionreductiondifer}
    -J_0(x+y-\nabla_xH^*_{z, t}(-J_0(\dot{x}+\dot{y})))=\dot{u}.
\end{equation}
We want to show that if $y=Y_{l,N}(x,z)$ then $\theta\cdot y=Y_{l,N}(\theta\cdot(x,z))$ for all $\theta\in \mathbb{T}$. Putting  $\theta\cdot(x,z)=(x(\cdot-\theta),\theta \cdot z)$ and $\theta \cdot y=y(\cdot -\theta)$ in left side of identity \eqref{conditionreductiondifer} we have  
\begin{align*}
    &-J_0(x(s-\theta)+y(s-\theta)-\nabla_xH^*_{\theta\cdot z, s}(-J_0(\dot{x}(s-\theta)+\dot{y}(s-\theta)))\\
    &\stackrel{t=s-\theta}{=}-J_0(x(t)+y(t)-\nabla_xH^*_{\theta\cdot z, t+\theta}(-J_0(\dot{x}(t)+\dot{y}(t)))=\dot{u}(t)\\
    &=-J_0(x(t)+y(t)-\nabla_xH^*_{z,t}(-J_0(\dot{x}(t)+\dot{y}(t)))\\ 
    &\stackrel{t=s-\theta}{=}\dot{u}(s-\theta)
\end{align*}

where the second equality comes from $S^1$-invariance of $H^*$. Applying $\Pi$ to both sides of equality we get that \eqref{conditionreduction} holds for $\theta\cdot(x,z)$ and $\theta\cdot y$. Hence, $Y_{l,N}$ is $S^1$-equivariant and $\Gamma_{l,N}(\mathbb{H}_l\times S^{2N+1})$ and $\psilH$ are invariant under the $S^1$-action.

\textbf{Claim} (5): From claim (d) of \cite[Proposition 6.3]{AK22} and the definition of $\psilH$ the claim trivially follows.

\textbf{Claim} (6): From the claim (d) of \cite[Proposition 6.3]{AK22} and the claim (2) of Proposition \ref{correspondancecriticals1equi} it follows that $\psi^l_{H^*_v}$ is non-degenerate for $v\in \crit(\widetilde{f}_N)$. Again from the claim (d) of \cite[Proposition 6.3]{AK22} we have that $\ind(S_{(x_l,v)})=\ind(S_{(x_0,v)})$. 

\textbf{Claim} (7): From the Proposition \ref{parfenchelconjugatelemma} we have that $\nabla_z H^*_t(\nabla_xH_{t,z}(x),z)=-\nabla_z H_{t}(x,z)$. Since $H$ is non-decreasing along the flow lines of $-\nabla \widetilde{f}_N$, due to (H3), we conclude that $H^*$ is non-increasing along flow lines of $-\nabla \widetilde{f}_N$ so the same holds for $\PsiH$.
    
Let $\nu$ be flow of $-\nabla \widetilde{f}_N$. Then we have that \[\frac{d}{ds}\psi^l_{H^*_{\nu(s)}(x)}=\frac{d}{ds}(\Psi_{H^*_{\nu(s)}}(x+Y(x,\nu(s)),\nu(s)))\]
\[=\frac{d\Psi_{H^*_{\nu(s)}}(x+Y(x,\nu(s)))}{dy}\left[\frac{dY}{dz}(x,\nu(s))[-\nabla \widetilde{f}_N(z(s))]\right]\]
\[+\frac{d \Psi_{H^*_{\nu(s)}(x)}}{dz}(x+Y_{l,N}(x,z))[-\nabla \widetilde{f}_N(z(s))]\]

The first term is equal to $0$ since $Y(x,\nu(s))$ is critical point of \[\mathbb{H}^l\to \mathbb{R}, \ \ \ y\mapsto \Psi_{H_{\nu(s)}^*}(x+y).\]
The second term is lesser or equal than 0, being that $\PsiH$ is non-increasing along flow lines of $-\nabla \widetilde{f}_N$. Therefore conclusion follows.

\end{proof}

\textbf{$S^1$-equivariant Morse vector fields}

Let $M$ be a finite-dimensional manifold equipped with a smooth free $S^1$-action. Let $X$ be an $S^1$-equivariant smooth vector field. We define the set of rest points of $X$ as 
\[\rest(X)=\{p\in M\mid X(p)=0\}.\]
On this set, we have a natural $S^1$-action. We define the spectrum of a point $p\in \rest(X)$, denoted by $\sigma(DX(p))$, as the set of eigenvalues of the Jacobian $DX(p)$ at $p$. We define the negative and positive spectrum, respectively, as 
\[\sigma^-(DX(p))=\{z\in \sigma(DX(p))\mid \text{Re}(z)<0\}\]
and 
\[\sigma^+(DX(p))=\{z\in \sigma(DX(p))\mid \text{Re}(z)>0\}.\]
We say that the critical orbit $S_{p}=\mathbb{T}\cdot p \subset \rest(X)$ is hyperbolic if the following conditions hold:
\begin{itemize}
    \item $\dim\ker(DX(p))=1,$
    \item $\sigma(DX(p))\backslash\{0\}\cap i\mathbb{R}=\emptyset.$
\end{itemize}
If all critical orbits of $X$ are hyperbolic, we say that $X$ is an $S^1$-equivariant Morse vector field.
A $C^1$-function $f:M\to \mathbb{R}$ is an $S^1$-invariant Lyapunov function for $X$ if the following conditions hold:
\begin{itemize}
    \item $f$ is $S^1$-invariant,
    \item $df(p)[X(p)]<0$ for all $p\in M\backslash \rest(X)$.
\end{itemize}
See \cite{AM06} for more details on Morse vector fields and Lyapunov functions.  
Let $g: S^{2N+1}\to \mathcal{G}(\mathbb{H}_l)$ be a smooth family of metrics on $\mathbb{H}_l$. We say that the family $g_z$ is $S^1$-invariant if for all $\theta \in \mathbb{T}$, the following holds:
\[g_{\theta\cdot z} \circ \theta_*=g_z.\]

\begin{proposition}\label{vectorfieldvg}
    Let $H\in \mathcal{H}_\Theta^{S^1,N}(\widetilde{f}_N,g_N)$ and let $l\in \mathbb{N}$ be large enough such that the $l$-reduced dual functional $\psilH$ exists. Let $g: S^{2N+1}\to \mathcal{G}(\mathbb{H}_l)$ be a smooth $S^1$-invariant family of metrics on $\mathbb{H}_l$, uniformly equivalent to the standard one. Then we define the following vector field 
    \[V_g(x,z)=(\nabla_{g_z}\psilH(x,z),\nabla_{\widetilde{f}_N}(z)),\]
    for which the following claims hold:
\begin{enumerate}
    \item If $(x_k, z_k) \in \mathbb{H}_l \times S^{2N+1}$ is a sequence such that $V_g(x_k,z_k)$ converges to zero as $k \to \infty$, then $(x_k,z_k)$ has a convergent subsequence.
    \item $V_g$ is a non-degenerate $S^1$-equivariant Morse vector field such that \[\rest(V_g)=\crit(\psilH,\widetilde{f}_N).\] Additionally, for $(x_l,v)\in \rest(V_g)$, the number of elements of $\sigma^-(DV_g(x_l,v))$, counting multiplicities, is $\ind(x_l,v)$.
    \item $\psilH+\widetilde{f}_N$ is a $S^1$-invariant Lyapunov function for $-V_g$.  
    \item For every $(x_l,v)\in \rest(V_g)$, the sets $W^u(S_{(x_l,v)};-V_g)$ and $W^s(S_{(x_l,v)};-V_g)$ are smooth $S^1$-invariant embedded submanifolds of $\mathbb{H}_l\times S^{2N+1}$ such that 
    \[\dim W^u(S_{(x_l,v)};-V_g) =\ind(S_{(x_l,v)})+1\]   
    and 
    \[ \codim W^s(S_{(x_l,v)};-V_g) =\ind(S_{(x_l,v)}).\]
    \item $\psilH$ is non-increasing along the flow of $-V_g$.
\end{enumerate}
\end{proposition}

\begin{proof}
\textbf{Claim} (1): Since $Y_{l,N}$ is uniformly bilipschitz, we have that the family of pullback metrics induced by the family of maps 
\[\mathbb{H}_l\to \Honenull, \ \ \ x\mapsto (x,Y(x,z)), \ \ \ z\in S^{2N+1}\] 
is equivalent to the standard one. Therefore, the claim follows from Proposition \ref{pscondition}.

\textbf{Claim} (2): Since $g$ is $S^1$-invariant, it follows that $V_g$ is $S^1$-equivariant. 

Since it holds that $(d\psi^l_{H^*_v}(x_l),d\widetilde{f}_N(v))=0$ if and only if $(x_l,v)\in \crit(\psilH,\widetilde{f}_N)$, it's clear that 
\[\rest(V_g)=\crit(\psilH,\widetilde{f}_N).\]
For $(x_l,v)\in \rest(V_g)$, we have that 
\[D(x_l,v)=\begin{bmatrix}
\nabla_x \nabla_{g_z}\psilH(x_l,v) & \nabla_z \nabla_{g_z}\psilH(x_l,v) \\
0 & \nabla_z^2\widetilde{f}_N(z) 
\end{bmatrix}.\]

From this, we have that $\text{rank}(D(x_l,v))\geq \text{rank}(\nabla_x \nabla_{g_z}\psilH(x_l,v))+\text{rank}(\nabla_z^2\widetilde{f}_N(z))$. Since $\nabla_x \nabla_{g_z}\psilH(x_l,v)$ is non-degenerate and $\nabla_z^2\widetilde{f}_N(z)$ has a $1$-dimensional cokernel, we conclude from the previous inequality that $\codim (\text{im}(D(x_l,v))) \leq 1.$ Therefore, it must be that 
\[\dim \ker(D(x_l,v))\leq 1.\] 
Since $V_g$ is $S^1$-equivariant, we conclude that 
\[\dim \ker(D(x_l,v))\geq 1\] also holds. Therefore, 
\[\dim\ker(D(x_l,v))=1.\] 
Moreover, since $D(x_l,v)$ is a block upper diagonal matrix, we have that the characteristic polynomial of $D(x_l,v)$ is a product of the characteristic polynomial of $\nabla_x \nabla_{g_z}\psilH(x_l,v)$ and the characteristic polynomial of $\nabla_z^2\widetilde{f}_N(z)$. Since both of these matrices are symmetric, we conclude that $\sigma(D(x_l,v))\subset \mathbb{R}$ and it follows that 
\[\sigma(D(x_l,v))\backslash \{0\}\cap i\mathbb{R}=\emptyset.\]
Therefore, $V_g$ is an $S^1$-equivariant Morse vector field. 
Additionally, from the equality 
\[\det(D(x_l,v)-\lambda I)=\det(\nabla_x \nabla_{g_z}\psilH(x_l,v)-\lambda I)\det(\nabla_z^2\widetilde{f}_N(z)-\lambda I),\] 
we have that 
\[\sigma^-(D(x_l,v))=\sigma^-(\nabla_x \nabla_{g_z}\psilH(x_l,v))\cup \sigma^-(\nabla_z^2\widetilde{f}_N(z))\]
holds, counting multiplicities. Therefore, the number of elements of $\sigma^-(D(x_l,v))$, counting multiplicities, is $\ind(x_l,v)$. 

\textbf{Claim} (3): We have that 
\[d(\psilH+\widetilde{f}_N)(x,z)[-V_g(x,z)]=-\|\psi^l_{H_z}(x)\|^{2}+\partial_z\psilH(x,z)(-\nabla \widetilde{f}_N(z))-\|\nabla \widetilde{f}_N(z)\|^2.\]
From claim (7) of Proposition \ref{sadlepointreductionS1}, we have that 
\[\partial_z\psilH(x,z)(-\nabla \widetilde{f}_N(z))\leq 0, \ \ \ (x,z)\in \mathbb{H}_l\times S^{2N+1}.\]
Additionally, if 
\[(x,z)\notin \rest(V_g)=\crit(\psilH,\widetilde{f}_N),\]
it follows that either $\|\psi^l_{H_z}(x)\|^{2}$ or $\|\nabla \widetilde{f}_N(z)\|^2$ is strictly positive. Therefore, we conclude that 
\[d(\psilH+\widetilde{f}_N)(x,z)[-V_g(x,z)]<0, \ \ \  (x,z)\notin \rest(V_g).\]
\textbf{Claim} (4): From claims (1), (2), and (3), we conclude that $W^u(S_{(x_l,v)};-V_g)$ and $W^s(S_{(x_l,v)};-V_g)$ are smooth $S^1$-invariant embedded submanifolds of $\mathbb{H}_l\times S^{2N+1}$. From the second part of claim (2), we conclude that 
\[\dim W^u(S_{(x_l,v)};-V_g) =\ind(S_{(x_l,v)})+1\]
and 
\[\codim W^s(S_{(x_l,v)};-V_g) =\ind(S_{(x_l,v)})\]
holds.

\textbf{Claim} (5): We have that 
\[d(\psilH)(x,z)[-V_g(x,z)]=-\|\psi^l_{H_z}(x)\|^{2}-\partial_z\psilH(x,z)(\nabla \widetilde{f}_N(z)).\]
Since claim (7) of Proposition \ref{sadlepointreductionS1} gives us that 
\[\partial_z\psilH(x,z)(-\nabla \widetilde{f}_N(z))\leq 0,\]
we conclude that 
\[d(\psilH)(x,z)[-V_g(x,z)]\leq 0, \ \ \ (x,z)\in \mathbb{H}_l\times S^{2N+1}.\]
\end{proof}

\begin{remark}\label{nullityindexreds1equi}

    Let $H \in \mathcal{H}_\Theta^{S^1,N}(\widetilde{f}_N, g_N)$. 
    
    Combining propositions \ref{correspondancecriticals1equi} and \ref{sadlepointreductionS1}, we conclude that critical orbits $P(H, \widetilde{f}_N)$ and critical orbits $\crit(\psilH, \widetilde{f}_N)$ are in one-to-one correspondence given by 
    \[(\mathbb{P}_l \times id_{S^{2N+1}})|_{P(H, \widetilde{f}_N)}: P(H, \widetilde{f}_N) \to \crit(\psilH, \widetilde{f}_N)\]

    such that it holds \[\PhiH(x, v) = \psilH(x_l, v) = \PsiH(x_0, v)\]
    where $x_0 = \mathbb{P}(x)$ and $x_l = \mathbb{P}_l(x) = \mathbb{P}_l(x_0)$. Additionally, it holds that \[\ind(S_{(x_l, v)}) = |S_{(x, v)}| - n.\]
\end{remark}

\textbf{$S^1$-equivariant Morse complex}

Let $H \in \mathcal{H}_\Theta^{S^1,N}(\widetilde{f}_N, g_N)$ and $g: S^{2N+1} \to \mathcal{G}(\mathbb{H}_l)$ be an $S^1$-invariant family of metrics which is uniformly equivalent to the standard one. 
For $S_{(x_l, v)}$, $S_{(y_l, w)} \subset \crit(\psilH, \widetilde{f}_N)$, we define a moduli space \[\hat{\mathcal{M}}_M^{S^1,N}(S_{(x_l, v)}, S_{(y_l, w)}; \psilH, \widetilde{f}_N, g) = W^u(S_{(x_l, v)}; -V_g) \cap W^s(S_{(y_l, w)}; -V_g)\]
This makes sense since from claim (2) of Proposition \ref{vectorfieldvg}, we have that \[\rest(V_g) = \crit(\psilH, \widetilde{f}_N).\]
Moreover, from claim (4) of the same proposition, we have that \[W^u(S_{(x_l, v)}; -V_g)\]
and \[W^s(S_{(y_l, w)}; -V_g)\] 
are smooth $S^1$-invariant embedded submanifolds.  
Since the system of equations
\[(\partial_s u, \partial_s z) = -V_g(u, z) = -(\nabla_{g_z}\psilH(u, z), \nabla \widetilde{f}_N(z))\]
reduces to a continuation equation, as in the case of family Floer homology (see \cite{BO17}), transversality of intersections can be achieved for generic $g$. In this situation, from claim (4) of Proposition \ref{vectorfieldvg}, we have that 
\[\hat{\mathcal{M}}_M^{S^1,N}(S_{(x_l, v)}, S_{(y_l, w)}; \psilH, \widetilde{f}_N, g)\] 
is a smooth manifold of dimension $\ind(S_{(x_l, v)}) - \ind(S_{(y_l, w)}) + 1$. Since we have a free $\mathbb{R} \times \mathbb{T}$ action, we have that 
\[\mathcal{M}_M^{S^1,N}(S_{(x_l, v)}, S_{(y_l, w)}; \psilH, \widetilde{f}_N, g) = \hat{\mathcal{M}}_M^{S^1,N}(S_{(x_l, v)}, S_{(y_l, w)}; \psilH, \widetilde{f}_N, g) / \mathbb{R} \times \mathbb{T}\]
is a smooth manifold with \[\dim\mathcal{M}_M^{S^1,N}(S_{(x_l, v)}, S_{(y_l, w)}; \psilH, \widetilde{f}_N, g) = \ind(S_{(x_l, v)}) - \ind(S_{(y_l, w)}) - 1.\]
Therefore, by standard compactness and gluing arguments, we have a well-defined $S^1$-equivariant Morse chain complex \[\{\C\M_*^{S^1,N}(H, \widetilde{f}_N), \partial^{M, S^1}\}.\]
Here $\C\M_*^{S^1,N}(H, \widetilde{f}_N)$ is a $\mathbb{Z}_2$ vector space generated by critical orbits of $\crit(H, \widetilde{f}_N)$ and graded by their index. The boundary operator 
\[\partial^{M, S^1}: \C\M_*^{S^1,N}(H, \widetilde{f}_N) \to \C\M_{*-1}^{S^1, N}(H, \widetilde{f}_N)\]
is defined as 
\[\partial^{M, S^1}(S_{(x_l, v)}) = \sum n^{S^1}_M(S_{(x_l, v)}, S_{(y_l, w)}) S_{(y_l, w)},\]
where the sum is over critical orbits $S_{(y_l, w)}$ such that
\[\ind(S_{(y_l, w)}) = \ind(S_{(x_l, v)}) - 1,\]
and $n^{S^1}_M(S_{(x_l, v)}, S_{(y_l, w)})$ is parity of the set $\mathcal{M}_M^{S^1, N}(S_{(x_l, v)}, S_{(y_l, w)}; \psilH, \widetilde{f}_N, g)$.
Thus, we have a well-defined $S^1$-equivariant Morse complex, unique up to a chain isomorphism and Morse homology 
\[\H\M_*^{S^1, N}(\psilH, \widetilde{f}_N; g) = H_*(\{\C\M_*^{S^1, N}(H, \widetilde{f}_N), \partial^{M, S^1}\}).\]
 From claim (5) of Proposition \ref{vectorfieldvg}, it follows that the filtration of the complex by $\psilH$ is well-defined.

\textbf{Continuation maps}

Let $H_\alpha \in \mathcal{H}_\Theta^{S^1, N_\alpha}(f_{N_\alpha}, g_{N_\alpha})$ and $H_\beta \in \mathcal{H}_\Theta^{S^1, N_\beta}(f_{N_\beta}, g_{N_\beta})$. If it holds $(N_\alpha, H_\alpha) \leq (N_\beta, H_\beta)$, then we have that $N_\alpha \leq N_\beta$ and $H_\alpha \leq i_{N_\beta, N_\alpha}^* H_\beta$. Therefore, it holds that $i_{N_\beta, N_\alpha}^* H^*_\beta \leq H^*_\alpha$ and we have that $i_{N_\beta, N_\alpha}^* \Psi_{H_\beta^*} \leq \Psi_{H_\alpha^*}$. For large enough $l \in \mathbb{N}$ the saddle point reduction exists for both dual action functionals, and it holds \[i_{N_\beta, N_\alpha}^* \psi_{H_\beta^*} \leq \psi_{H_\alpha^*}.\]

Having the fixed $S^1$-invariant family of metrics, $g_\alpha$ and $g_\beta$ for which the $S^1$-equivariant Morse complexes of \[(\psi_{H_\alpha^*}, f_{N_\alpha}; g_\alpha), \ \ (\psi_{H_\beta^*}, f_{N_\beta}; g_\beta)\] 
are well-defined, we can take non-increasing homotopy $(\psi^{\alpha\beta}_s, g^{\alpha\beta}_s)$, i.e., homotopy from $\psi_{H_\alpha^*}$ to $i_{N_\beta, N_\alpha}^* \psi_{H_\beta^*}$. This, as in the non-parameterized case, defines unstable and stable manifolds of homotopy. Moduli spaces of homotopy are defined as intersections of these, quotiented with $\mathbb{T}$. Generically, this defines a chain map 
\[i^{M, S^1}_{(\psi^{\alpha\beta}_s, g^{\alpha\beta}_s)}: \C\M_*^{S^1, N_\alpha}(\psi^{l}_{H_\alpha^*}, \widetilde{f}_{N_\alpha}; g_\alpha) \to \C\M_*^{S^1, N_\beta}(\psi^{l}_{H_\beta^*}, \widetilde{f}_{N_\beta}; g_\beta),\]
unique up to chain homotopy. This map preserves the action filtration and induces a continuation map in homology.

\section{Cauchy-Riemann type linear operators on negative half-cylinders}

Elements $x \in \Honehalf$ can be expressed as
\[x(t) = \sum_{k \in \mathbb{Z}} e^{2\pi k t J_0} \hat{x}(k), \quad \hat{x}(k) \in \mathbb{R}^{2n}, \]
where it holds
\[ \sum_{k \in \mathbb{Z}} |k| |\hat{x}(k)|^2 < \infty \]
with the given scalar product \[ \langle x, y \rangle_{1/2} = \langle \hat{x}(0), \hat{y}(0) \rangle + 2\pi \sum_{k \in \mathbb{Z}} |k| \langle \hat{x}(k), \hat{y}(k) \rangle.
\]
Thus, we have the splitting
\[
\Honehalf = \mathbb{H}_- \oplus \mathbb{H}_0 \oplus \mathbb{H}_+,
\]
where
\[
\mathbb{H}_- = \{x \in \Honehalf \mid \hat{x}(k) = 0, \, k \geq 0\},
\]
\[
\mathbb{H}_+ = \{x \in \Honehalf \mid \hat{x}(k) = 0, \, k \leq 0\},
\]
and $\mathbb{H}_0 \cong \mathbb{R}^{2n}$ is the space of constant loops.

This splitting is orthogonal, and therefore, we have well-defined corresponding orthogonal projectors $\mathbb{P}_+$, $\mathbb{P}_-$, and $\mathbb{P}_0$.

\large\textbf{Non-parameterized operators}
\normalsize

We define the space $\mathcal{A}(2n)$ to be the space of pairs $\alpha = (j, s)$ where $j \in C(\mathbb{T}, \mathcal{J}(\mathbb{R}^{2n}, \omega_0))$ and $s \in C(\mathbb{T}, \mathcal{L}(\mathbb{R}^{2n}))$ is a loop into the space of symmetric matrices  with respect to $j$ (i.e., $s(t)$ symmetric with respect to the metric $g_t = \omega_0(\cdot, j(t)\cdot)$) such that the path
\[
\Psi : [0,1] \to \operatorname{Sp}(2n)
\]
defined by
\[
\dot{\Psi}(t) = j(t) s(t) \Psi(t), \quad \Psi(0) = \unit
\]
is non-degenerate, i.e., $\Psi(1)$ has no eigenvalue $1$.

We denote
\[
\mu_{CZ}(s) \defeq \mu_{CZ}(\Psi).
\]

Let $\alpha = (j, s) \in \mathcal{A}(2n)$ be an arbitrary element.

With $\mathcal{D}_-(\alpha)$ we denote the space of operators of the form
\[
H^1((-\infty, 0) \times \mathbb{T}, \mathbb{R}^{2n}) \to L^2((-\infty, 0) \times \mathbb{T}, \mathbb{R}^{2n}) \times \mathbb{H}_+,
\]
\[
u \mapsto (\partial_s u + J(s, t)(\partial_t u + A(s, t) u), \mathbb{P}_+ u(0, \cdot)),
\]
where

\begin{enumerate}
    \item $J \in C([-\infty, 0] \times \mathbb{T}, \mathcal{J}(\mathbb{R}^{2n}, \omega_0))$ where $J(-\infty, t) = j(t)$ for every $t \in \mathbb{T}$ and $J|_{[-1, 0] \times \mathbb{T}} = J_0.$
    \item $S \in C([-\infty, 0] \times \mathbb{T}, \mathcal{L}(\mathbb{R}^{2n}))$ such that $S(-\infty, t) = s(t)$ for every $t \in \mathbb{T}$.
\end{enumerate}

In the above expression, $u(0, \cdot)$ denotes the trace of $u$ on the boundary $\{0\} \times \mathbb{T}$ of the half-cylinder. Since $u \in H^1$, it follows that $u(0, \cdot) \in \Honehalf$.

The following proposition is a variant of \cite[Proposition 8.2]{AK22}.

\begin{proposition}\label{cauchyriemannonp}
    Let $\alpha = (j, s) \in \mathcal{A}(2n)$. Then every operator in $\mathcal{D}_-(\alpha)$ is Fredholm of index $n - \mu_{CZ}(s)$.
\end{proposition}

\large\textbf{Parameterized operators with block upper-triangular asymptote}
\normalsize

We define the space $\mathcal{A}^u(2n,m)$ as the space of pairs $\widetilde{\alpha}=(j,\widetilde{s})$, where $j\in  C(\mathbb{T},\mathcal{J}(\mathbb{R}^{2n},\omega_0))$ and $\widetilde{s}\in C(\mathbb{T},\mathcal{L}(\mathbb{R}^{2n+m}))$ is a loop of block upper-triangular matrices such that
\[\widetilde{s}(t)\defeq\begin{pmatrix}
s(t) & r(t)\\
 0   & \tau 
\end{pmatrix}\]
where $s\in C(\mathbb{T},\mathcal{L}(\mathbb{R}^{2n}))$ is a loop into the space of symmetric matrices with respect to $j$ (i.e., $s(t)$ is symmetric with respect to the metric $g_t = \omega_0(\cdot, j(t)\cdot)$). The following properties hold:

\begin{itemize}
    \item The path
    \[\Psi :[0,1] \to \operatorname{Sp}(2n)\]
    defined by
    \[\dot{\Psi}(t)=j(t)s(t)\Psi(t), \ \ \ \Psi(0)=\unit\]
    is non-degenerate, i.e. $\Psi(1)$ has no eigenvalue equal to 1.
    \item The matrix $\tau$ is hyperbolic. 
\end{itemize}

To a loop of block upper triangular matrices $\widetilde{s}$ (for which $\tau$ does not need to be hyperbolic), we associate the index
\[\mu(\widetilde{s})\defeq\mu_{CZ}(\Psi)+\frac{1}{2}\sign(\tau).\]

For an arbitrarily chosen map $\widetilde{S}\in C([-\infty,0]\times \mathbb{T},\mathcal{L}(\mathbb{R}^{2n}\times \mathbb{R}^m))$, where
\[\widetilde{S}(s,t)\defeq\begin{pmatrix}
S(s,t) & R(s,t)\\
0 & T(s) 
\end{pmatrix},\]
we have the multiplication operator

\[\widetilde{S} \begin{pmatrix}
u\\
\lambda
\end{pmatrix}=(S(s,t)u+R(s,t)\lambda(s),T(s)\lambda(s)).\]

Now, for an arbitrarily chosen $\widetilde{\alpha}=(j,\widetilde{s})\in \mathcal{A}^u(2n,m)$, we denote by $\mathcal{D}^{m,u}_-(\widetilde{\alpha})$ the space of operators that have the form
\[H^1((-\infty,0)\times \mathbb{T},\mathbb{R}^{2n})\times H^1((-\infty,0),\mathbb{R}^{m}) \to L^{2}((-\infty,0)\times \mathbb{T},\mathbb{R}^{2n})\times L^{2}((-\infty,0),\mathbb{R}^{m})\times \mathbb{H}_+\times \mathbb{R}^m,\]
\[(u,\lambda)\mapsto\left(\left(\begin{pmatrix}
\partial_s+J(s,t)\partial_t& 0\\
0 & \partial_s
\end{pmatrix}+\widetilde{S}\right)\begin{pmatrix}
u\\
\lambda 
\end{pmatrix}
, (\mathbb{P}_+(u(0,\cdot)),\lambda(0))\right),\]

such that the following holds:

\begin{enumerate}
    \item $J\in C([-\infty,0]\times \mathbb{T},\mathcal{J}(\mathbb{R}^{2n},\omega_0))$, where $J(-\infty,t)=j(t)$ for every $t\in \mathbb{T}$, with the condition $J|_{[-1,0]\times \mathbb{T}}=J_0.$
    \item $\widetilde{S}\in C([-\infty,0]\times \mathbb{T},\mathcal{L}(\mathbb{R}^{2n}))$ such that $\widetilde{S}(-\infty,t)=\widetilde{s}(t)$ for every $t\in \mathbb{T}$.
\end{enumerate}

In the above expression, $u(0,\cdot)$ denotes the trace of $u$ on the boundary $\{0\} \times \mathbb{T}$ of the half-cylinder. Since $u\in H^1$, it follows that $u(0,\cdot) \in \Honehalf$, and $\lambda$ is continuous. Thus, such an operator is well-defined.

\begin{proposition}\label{cauchyriemmanpar}
    Let $\widetilde{\alpha}=(j,\widetilde{s})\in \mathcal{A}^u(2n,m)$. Then every operator in $\mathcal{D}^{m,u}_-(\widetilde{\alpha})$ is Fredholm of index $n-\mu(\widetilde{s})-\frac{m}{2}$.
\end{proposition}

To prove this, we need the following lemma.

\begin{lemma}\label{uppertriangular}
    Let $X, Y, V,$ and $W$ be Banach spaces. Let $L_S:X \to V$ and $L_T:Y \to W$ be Fredholm operators. If $L_R:Y \to V$ is any bounded operator, then the operator
    \[
    L:X \times Y \to V \times W, \quad L(x, y) = \begin{pmatrix}
        L_S(x) + L_R(y) \\
        L_T(y)
    \end{pmatrix}
    \]
    is Fredholm, with $\ind(L) = \ind(L_S) + \ind(L_T)$.
\end{lemma}

\begin{proof}

A bounded operator $F: X \to Y$ between Banach spaces is Fredholm if and only if there exists a bounded operator $D: Y \to X$ such that

\[
F\circ D=A+K, \quad D\circ F= A'+K',
\]
where $A$ and $A'$ are invertible, and $K$ and $K'$ are compact.

Therefore, there exist $D_S:V \to X$ and $D_T:W \to Y$ such that
\[
L_S \circ D_S = A_S + K_S, \quad D_S\circ L_S = A_S' + K_S',
\]
and
\[
L_T \circ D_T = A_T + K_T, \quad D_T \circ L_T = A_T' + K_T',
\]
where $A_S$, $A_S'$, $A_T$, and $A_T'$ are invertible, and $K_S$, $K_S'$, $K_T$, and $K_T'$ are compact operators.

If we define 
\[
D:V \times W \to X \times Y, \quad D = D_S \oplus D_T,
\]
we have
\[
L \circ D = \begin{pmatrix}
    A_S + L_R \circ D_T \\
    A_T
\end{pmatrix} + \begin{pmatrix}
    K_S \\
    K_T
\end{pmatrix},
\]
and
\[
D \circ L = \begin{pmatrix}
    A_S' + D_S \circ L_R \\
    A_T'
\end{pmatrix} + \begin{pmatrix}
    K_S' \\
    K_T'
\end{pmatrix}.
\]

Since the operators
\[
\begin{pmatrix}
    A_S + L_R \circ D_T \\
    A_T
\end{pmatrix}, \quad \begin{pmatrix}
    A_S' + D_S \circ L_R \\
    A_T'
\end{pmatrix}
\]
are invertible, and the operators 
\[\begin{pmatrix}
    K_S \\
    K_T
\end{pmatrix}, \quad \begin{pmatrix}
    K_S' \\
    K_T'
\end{pmatrix}\]
are compact, we conclude that $L$ must be Fredholm.

To calculate the index of such an operator, we define a path of Fredholm operators:
\[
L_\varepsilon = \begin{pmatrix}
    L_S(x) + (1 - \varepsilon)L_R(y) \\
    L_T(y)
\end{pmatrix}, \quad \varepsilon \in [0, 1].
\]

From the previous part of the proof, this is indeed a path of Fredholm operators. Moreover, we have $L_0 = L$ and $L_1 = L_S \oplus L_T$. Hence, it follows that
\[
\ind(L) = \ind(L_S \oplus L_T) = \ind(L_S) + \ind(L_T).
\]
\end{proof}

\begin{proof}[Proof of Proposition ~\ref{cauchyriemmanpar}]
Let $L\in \mathcal{D}^{m,u}_-(\widetilde{\alpha})$ be arbitrary. We define the operator
\begin{align*}
    &L_S: H^1((-\infty,0) \times \mathbb{T}, \mathbb{R}^{2n}) \to L^2((-\infty,0) \times \mathbb{T}, \mathbb{R}^{2n}) \times \mathbb{H}_+, \\   
    &L_S(u) = (\partial_s u + J(s,t) \partial_t u + S(s,t) u, \mathbb{P}_+(u(0, \cdot))),
\end{align*}
the operator 
\begin{align*}
    &L_T: H^1((-\infty,0), \mathbb{R}^m) \to L^2((-\infty,0), \mathbb{R}^m) \times \mathbb{R}^{2n}, \\   
    &L_T(\lambda) = (\partial_s \lambda + T(s) \lambda, \lambda(0)),
\end{align*}
and the operator
\begin{align*}
    &L_R: H^1((-\infty,0), \mathbb{R}^m) \to L^2((-\infty,0) \times \mathbb{T}, \mathbb{R}^{2n}) \times \mathbb{H}_+, \\
    &L_R(\lambda) = (R(s,t)\lambda(s), 0).
\end{align*}

The operator $L_S$ is Fredholm of index $n-\mu(s)$ due to Proposition \ref{cauchyriemannonp}, $L_T$ is Fredholm of index $-\frac{1}{2}\sign(\overline{\tau}) - \frac{m}{2}$ by elementary ODE considerations, and $L_R$ is obviously a bounded operator. Since
\[
L(u,\lambda)=\begin{pmatrix}
 L_S(u)+L_R(\lambda) \\
 L_T(\lambda)
\end{pmatrix},
\]
from the previous lemma, we get the desired result.
\end{proof}

\section{Half-cylinder moduli spaces}

This section will address negative half-cylinder moduli spaces for both the non-parameterized and the $S^1$-equivariant cases. Such spaces are needed to prove the theorem from \cite{AK22} and its version in the $S^1$-equivariant category.
 
 \large\textbf{Non-parameterized case}
\normalsize

As discussed in Section \ref{sectionmorse}, for sufficiently large $l \in \mathbb{N}$, the map $Y_l: \mathbb{H}_l \to \mathbb{H}^l$ is well-defined as the unique global minimizer of the function 
\[\mathbb{H}^l \to \mathbb{R}, \quad y \mapsto \PsiH(x + y).\]

Moreover, the graph of this function, denoted by $\Gamma_l$, provides a manifold $\Gamma_l(\mathbb{H}_l)$ of dimension $2nl$ consisting of smooth loops. The projection of this manifold onto the positive part, $\mathbb{P}_+(\Gamma_l(\mathbb{H}_l))$, is an embedding and therefore also a manifold of dimension $2nl$ consisting of smooth loops. Specifically, $\mathbb{P}_+(\Gamma_l(\mathbb{H}_l))$ is the positive graph of the function $Y_l$, i.e., the graph of the function $\mathbb{P}_+ \circ Y_l$ within $\mathbb{H}_+$, consisting of smooth loops. In order to simplify the notation, we denote by
\[\Gamma_{l,+}: \mathbb{H}_l \to \mathbb{H}_+, \quad \Gamma_{l,+}(x) = (x, (\mathbb{P}_+ \circ Y_l)(x)),\]
 the graph of $\mathbb{P}_+ \circ Y_l$.

Let $H \in \mathcal{H}_\Theta$. We assume that $J \in C([-\infty,0] \times \mathbb{T}, \mathcal{J}(\mathbb{R}^{2n}, \omega_0))$ where $J|_{[-1,0] \times \mathbb{T}} = J_0$, and that $g$ is some metric uniformly equivalent to the standard one on $\mathbb{H}_l$. For every $x_l \in \crit(\psilH)$ and $y \in P(H)$, we define the moduli space
\[\mathcal{M}_\Theta(x_l, y; H, J, g)\]
as the space of smooth maps $u: (-\infty,0] \times \mathbb{T} \to \mathbb{R}^{2n}$ such that
\[\partial_s u + J(s,t,u(s,t)) (\partial_t u - X_{H_t}(u)) = 0,\]
and
\[\lim_{s \to -\infty} u(s, \cdot) = y, \quad \mathbb{P}_+(u(0, \cdot)) \in \Gamma_{l,+}(W^u(x_l; -\nabla_g \psilH)).\]

This space can be viewed as the preimage of the operator defined as
\[\overline{\partial}_{H}: y + H^1((-\infty,0] \times \mathbb{T}, \mathbb{R}^{2n}) \to L^2((-\infty,0] \times \mathbb{T}, \mathbb{R}^{2n}) \times \mathbb{H}_+,\]
\[\overline{\partial}_{H}(u) = (\partial_s u + J(s,t,u(s,t)) (\partial_t u - X_{H_t}(u)), \mathbb{P}^+(u(\cdot,0))).\]

Here, $u(\cdot,0)$ denotes the trace operator at the boundary of the half-cylinder. Operator $\overline{\partial}_{H}$ is smooth, which follows from quadratic convexity (see \cite{AK22}).

If $\overline{\partial}_H(u) \in \{0\} \times \Gamma_{l,+}(W^u(\mathbb{P}_l(x); -\nabla_g \psilH))$, then $u(0, \cdot) = v + w$ where $v \in C^\infty(\mathbb{T}, \mathbb{R}^{2n})$ and $w \in \mathbb{H}_-$. By the same arguments as in \cite[Proposition 7.2]{AK22}, we conclude that $u \in C^\infty((-\infty,0] \times \mathbb{T}, \mathbb{R}^{2n})$ and that $u$ converges to $y$ as $s \to -\infty$, which implies that
\[\overline{\partial}_H^{-1}(\{0\} \times \Gamma_{l,+}(W^u(\mathbb{P}_l(x_l); -\nabla_g \psilH))) \subseteq \mathcal{M}_\Theta(x_l, y; H, J, g).\]
Since the other inclusion is trivial, we conclude the following.

\begin{proposition}\label{regularityhalfnonpar}\cite{AK22}
    Let $H \in \mathcal{H}_\theta$. Then
    \[\mathcal{M}_\Theta(x_l, y; H, J, g) = \overline{\partial}_H^{-1}(\{0\} \times \Gamma_{l,+}(W^u(\mathbb{P}_l(x_l); -\nabla_g \psilH))).\]
\end{proposition}

\begin{proposition}\label{energyboundsnonparhalf}\cite{AK22}
    Let $u \in \mathcal{M}_\Theta(x_l, y; H, J, g)$. Then
    \[\int_{(-\infty,0] \times \mathbb{T}} |\partial_s u|^2_J \, ds \, dt \leq \psilH(x_l) - \PhiH(y).\]
    Moreover, the equality $\psilH(x_l) = \PhiH(y)$ holds if and only if $x = y$, in which case the moduli space $\mathcal{M}_\Theta(x_l, x; H, J, g)$ consists of the unique half-cylinder $u(s, \cdot) = x$.
\end{proposition}

For the proof of this proposition, see \cite[Proposition 7.1]{AK22}.

\large\textbf{$S^1$-equivariant case} 
\normalsize
    
Let $H \in \mathcal{H}_\Theta^{S^1,N}(\widetilde{f}_N, g_N)$ and let $l \in \mathbb{N}$ be sufficiently large such that the reduced dual functional
\[\psilH: \mathbb{H}_l \times S^{2N+1} \to \mathbb{R}\]
is defined. We denote
\[\Gamma_{l,N,+}: \mathbb{H}_l \times S^{2N+1} \to \mathbb{H}_+ \times S^{2N+1}, \quad \Gamma = (\mathbb{P}_+ \times \text{id}_{S^{2N+1}}) \circ \Gamma_{l,N},\]
as the graph of the map $\mathbb{P}_+ \circ Y_{l,N}$. We assume that $J \in C([-\infty,0] \times \mathbb{T} \times S^{2N+1}, \mathcal{J}(\mathbb{R}^{2n}, \omega_0))$ is $S^1$-invariant with $J|_{[-1,0] \times \mathbb{T} \times S^{2N+1}} = J_0$. Additionally, we assume that $g: S^{2N+1} \to \mathcal{G}(\mathbb{H}_l)$ is an $S^1$-invariant family of metrics uniformly equivalent to the standard one.

For $S_{(x_l, v)} \subset \crit(\psilH, \widetilde{f}_N)$ and $S_{(y, w)} \subset P(H, \widetilde{f}_N)$, we define the moduli space
\[\hat{\mathcal{M}}^{S^1,N}_\Theta(S_{(x_l, v)}, S_{(y, w)}; H, \widetilde{f}_N, J, g)\]
as the space of pairs $(u, z)$ consisting of smooth maps $u: (-\infty,0] \times \mathbb{T} \to \mathbb{R}^{2n}$ and $z: (-\infty,0] \to S^{2N+1}$ that satisfy the system of equations
\begin{equation}\label{floerredhalf}
  \begin{array}{lcl}
    &\partial_su+J_{z(s)}(s,t,u(s,t))(\partial_tu-X_{H_{z(s)}}(u))=0\\
    &\dot{z}(s)-\nabla \widetilde{f}_N(z)=0 
  \end{array}
 \end{equation}  
and the conditions 
\begin{equation}
  \begin{array}{lcl}
    &\lim\limits_{s\to -\infty} (u(s),z(s))\in S_{(y,w)},\\
    & (\mathbb{P}_+(u(0,\cdot)),z(0))\in \Gamma_{l,N,+}(W^u(S_{(x_l,v)},-V_g)).
  \end{array}
  \end{equation} 

where $V_g=(\nabla_g \psilH,\nabla \widetilde{f}_N).$

Similar to the non-parameterized case, we have an operator defined as
\[\overline{\partial}_{H,N}: S_{(y, w)} + H^1((-\infty,0] \times \mathbb{T}, \mathbb{R}^{2n}) \times H^1((-\infty,0], S^{2N+1})\]
\[\to L^2((-\infty,0] \times \mathbb{T}, \mathbb{R}^{2n}) \times L^2((-\infty,0], z^* TS^{2N+1}) \times \mathbb{H}_+ \times S^{2N+1},\]
\[\overline{\partial}_{H,N}(u, z) = (\partial_s u + J_{z(s)}(s,t,u(s,t))(\partial_t u - X_{H_{z(s)}}(u)), \partial_s z - \nabla \widetilde{f}_N(z), \mathbb{P}_+(u(0, \cdot)), z(0)).\]

This operator is smooth due to the property (H4) of the family $\mathcal{H}_\Theta^{S^1,N}(\widetilde{f}_N, g_N)$. We aim to show that
\[\overline{\partial}_{H,N}^{-1} (\{0\} \times \Gamma_{l,N,+}(W^u(S_{(x_l, v)}, -V_g))) = \hat{\mathcal{M}}^{S^1,N}_\Theta(S_{(x_l, v)}, S_{(y, w)}; H, \widetilde{f}_N, J, g).\]

To prove this, we will use the following lemma.

\begin{lemma}\label{mixedderivativesham}
  Let $\Lambda$ be any finite-dimensional manifold. Let $H \in C^\infty(\mathbb{T} \times \mathbb{R}^{2n} \times \Lambda)$ be a parameterized Hamiltonian, $u \in C^\infty((0,1) \times \mathbb{T}, \mathbb{R}^{2n})$, and $\lambda \in C^\infty((0,1), \Lambda)$. The following holds:

  Let $h \geq 0$ and $k \geq 0$ be integers such that $h + k \geq 1$. Then
  \begin{equation}
    \nonumber\partial_s^h \partial_t^k (\nabla_x H_t(u, \lambda)) = \nabla_x^2 H_t(u, \lambda) \partial_s^h \partial_t^k u + p,
  \end{equation}
  where $p$ is a polynomial mapping of the partial derivatives $\partial_s^i \partial_t^j u$ and derivatives $\partial_s^r \lambda$ with $0 \leq i \leq h$, $0 \leq j \leq k$, $0 \leq i + j \leq h + k - 1$, and $0 \leq r \leq h$. The coefficients of $p$ are of the form $A(t, u(s,t), \lambda(s))$, where $A$ is smooth.
\end{lemma}

The proof of this lemma is argued by induction in the same way as the proof of \cite[Lemma 7.5]{AK22}.

\begin{proposition}\label{regularityhalfs1}
    Let $H \in \mathcal{H}_\Theta^{S^1,N}(\widetilde{f}_N, g_N)$. Then 
    \begin{equation}\nonumber
      \overline{\partial}_{H,N}^{-1} (\{0\} \times \Gamma_{l,N,+}(W^u(S_{(x_l, v)}, -V_g))) = \hat{\mathcal{M}}^{S^1,N}_\Theta(S_{(x_l, v)}, S_{(y, w)}; H, \widetilde{f}_N, J, g).
    \end{equation}
\end{proposition}

\begin{proof}
    The inclusion $\hat{\mathcal{M}}^{S^1,N}_\Theta(S_{(x_l, v)}, S_{(y, w)}; H, \widetilde{f}_N, J, g) \subseteq \overline{\partial}_{H,N}^{-1} (\{0\} \times \Gamma_{l,N,+}(W^u(S_{(x_l, v)}, -V_g)))$ is trivial. On the other hand, from Proposition \ref{regularityfloer} we have that if $(u, z) \in \overline{\partial}_{H,N}^{-1} (\{0\} \times \Gamma_{l,N,+}(W^u(S_{(x_l, v)}), -V_g))$, then $(u, z) \in C^\infty((-\infty,0) \times \mathbb{T}, \mathbb{R}^{2n}) \times C^\infty((-\infty,0), S^{2N+1})$. Moreover, since $z$ is a solution of the ODE
    \[\partial_s z = \nabla \widetilde{f}_N(z),\]
     $\nabla \widetilde{f}_N$ is smooth and $S^{2N+1}$ is compact, $z$ can be uniquely and smoothly extended to $(-\infty,0]$. Therefore, $z \in C^\infty((-\infty,0], S^{2N+1})$. Additionally, due to the non-degeneracy of $S_{(y, w)}$ by standard arguments, we have that

\[\lim_{s \to -\infty} (u(s), z(s)) \in S_{(y, w)}.\]

    Therefore,in order to show that the inclusion
 \[\overline{\partial}_{H,N}^{-1} (\{0\} \times \Gamma_{l,N,+}(W^u(S_{(x_l, v)}, -V_g))) \subseteq \hat{\mathcal{M}}^{S^1,N}_\Theta(S_{(x_l, v)}, S_{(y, w)}; H, \widetilde{f}_N, J, g)
\]
    holds, it is sufficient to show that $u \in C^\infty((-\infty,0] \times \mathbb{T}, \mathbb{R}^{2n})$. Since we already have that $u \in C^\infty((-\infty,0) \times \mathbb{T}, \mathbb{R}^{2n})$, it is sufficient to show that $u \in H^k((-1,0) \times \mathbb{T}, \mathbb{R}^{2n})$ for every $k \in \mathbb{N}$ (this follows by the Sobolev embedding theorem, which implies that $u \in C^\infty([-1,0] \times \mathbb{T}, \mathbb{R}^{2n})$). Together with the previous conclusions, this gives us the desired result.

    Since $J_z(t,x) = J_0$ when $(t,x,z) \in [-1,0] \times \mathbb{T} \times S^{2N+1}$ and $z \in H^k((-1,0), S^{2N+1})$ for every $k \in \mathbb{N}$, the rest of the proof follows by induction using the previous lemma, in the same way as in the proof of \cite[Proposition 7.2]{AK22}.
\end{proof}

\begin{remark}\label{weighted operator}
Let $d > 0$ be smaller than the spectral gap of the asymptotic operator $D_{(x,v)}$ for every $(x,v) \in \crit(H,\widetilde{f}_N)$. Then we can define the weighted operator 
\[
\overline{\partial}^d_{H,N}: S_{(y,w)} + H^1((-\infty,0] \times \mathbb{T}, \mathbb{R}^{2n}; e^{-ds} ds \, dt) \times H^1((-\infty,0], S^{2N+1}; e^{-ds} ds \, dt)
\]
\[
\to L^2((-\infty,0] \times \mathbb{T}, \mathbb{R}^{2n}; e^{-ds} ds \, dt) \times L^2((-\infty,0], z^* TS^{2N+1}; e^{-ds} ds \, dt) \times \mathbb{H}_+ \times S^{2N+1}
\]
\[
\overline{\partial}^d_{H,N}(u, z) = \left(\partial_s u + J_{z(s)}(s,t,u(s,t))(\partial_t u - X_{H_{z(s)}}(u)), \, \partial_s z - \nabla \widetilde{f}_N(z), \, \mathbb{P}_+(u(0,\cdot)), \, z(0)\right)
\]
It will still hold that 
\[
(\overline{\partial}^d_{H,N})^{-1}(\{0\} \times \Gamma_{l,N,+}(W^u(S_{(x_l,v)}, -V_g))) = \hat{\mathcal{M}}^{S^1,N}_\Theta(S_{(x_l,v)}, S_{(y,w)}; H, \widetilde{f}_N, J, g)
\]
since exponential decay is faster than $e^{-ds}$.
\end{remark}

The asymptotic operator $D_{(x,v)}$ is the linearization of the operator 
\[
(u,z) \mapsto \left( \partial_s u + J_{z}(s,t,u(s,t))(\partial_t u - X_{H_{z}}(u)), \partial_s z - \nabla \widetilde{f}_N(z) \right),
\]
when $(u(s,t),z(s)) = (x(t),v)$, and the domain of this map is $H^1(\mathbb{T}, \gamma^* \mathbb{R}^{2n}) \times T_v S^{2N+1}$ (i.e., $(v(s,t),w(s)) = (\gamma(t),l)$). Therefore, we have the map
\[
D_{(x,v)}: H^1(\mathbb{T}, \gamma^* \mathbb{R}^{2n}) \times T_v S^{2N+1} \to L^2(\mathbb{T}, \gamma^* \mathbb{R}^{2n}) \times T_v S^{2N+1}.
\]

The bijectivity of these maps is equivalent to the non-degeneracy of points. Due to $S^1$-invariance and the non-degeneracy assumption, for every orbit $S_{(x,v)} \subset P(H,\widetilde{f}_N)$, it is possible to find $d > 0$ such that it is smaller than the spectral gap of all asymptotic operators in the orbit. Since there are finitely many orbits, such a $d > 0$ exists as we have chosen it.

From \cite[Proposition 5.2]{AK22} that for $x \in \Hone$ and $y \in \mathbb{H}_- \oplus \mathbb{H}_0$ it holds that 
\begin{equation}\label{phipsiinequality}
    \PhiH(x+y) \leq \PsiH(\mathbb{P}(x)) - \frac{1}{2}\|\mathbb{P}_-y\|^2_{H^{1/2}}
\end{equation} 
Equality holds if and only if $x + y$ is a critical point of $\PsiH$. This  implies energy bounds on moduli spaces of $S^1$-invariant half-cylinders.

\begin{proposition}\label{energyboundsS1half}
Let $(u, z) \in \hat{\mathcal{M}}^{S^1,N}_\Theta(S_{(x_l,v)}, S_{(y,w)}; H, \widetilde{f}_N, J, g)$. Then
\[
\int_{(-\infty,0] \times \mathbb{T}} |\partial_s u|^2_{J_{z(s)}} \, ds \, dt \leq \psilH(x_l, v) - \PhiH(y, w),
\]
Moreover, $\psilH(x_l, v) = \PhiH(y, w)$ if and only if $S_{(x, v)} = S_{(y, w)}$, in which case the moduli space consists of the unique orbit of constant cylinder maps, i.e., $(u(s, \cdot), z(s)) = \theta \cdot (x, v)$ for $\theta \in \mathbb{T}$.

\begin{proof}
It holds that 
\[
\int_{(-\infty,0] \times \mathbb{T}} \|u(s,t)\|_J^2 \, ds \, dt - \int_{(-\infty,0] \times \mathbb{T}} g_N(\nabla_z H(u(s,t)), \nabla \widetilde{f}_N(z(s))) \, ds \, dt
\]
\[
= \PhiH(u(0, \cdot), z(0)) - \PhiH(y, w).
\]
Since $H$ is non-increasing along the flow of $\nabla \widetilde{f}_N$, we have that
\begin{equation}\label{inequalityenergyboundshalf}
    \int_{(-\infty,0] \times \mathbb{T}} \|u(s,t)\|_J^2 \, ds \, dt \leq \PhiH(u(0, \cdot), z(0)) - \PhiH(y, w).
\end{equation}
Additionally
\[
u(0, \cdot) = u_1 + u_2,
\]
holds, where $(u_1, z(0)) \in \Gamma_{l,N,+}(W^u(S_{(x_l, v)}, -V_g)) = (\mathbb{P}^+ \times \text{id}_{S^{2N+1}})(\Gamma_{l,N}(W^u(S_{(x_l, v)}; -V_g)))$ and $u_2 \in \mathbb{H}_- \oplus \mathbb{H}_0$.

Rearranging this sum we get that 
\[
u(0, \cdot) = \Tilde{u}_1 + \Tilde{u}_2
\]
where $(\Tilde{u}_1, z(0)) \in \Gamma_{l,N}(W^u(S_{x_l, v)}; -V_g))$ and $\Tilde{u}_2 \in \mathbb{H}_- \oplus \mathbb{H}_0$. We have that 
\begin{equation}\label{inequalityenergyboundshalfunstable}
\PhiH(\Tilde{u}_1 + \Tilde{u}_2, z(0)) \leq \PsiH(\Tilde{u}_1, z(0)) \leq \PsiH(x_0, v) = \psilH(x_l, v).
\end{equation}
The first inequality comes from \eqref{phipsiinequality}, the second inequality comes from the fact that $(\Tilde{u}_1, z(0)) \in \Gamma_{l,N}(W^u(S_{(x_l, v)}; -V_g))$, and the last equality comes from Remark \ref{nullityindexreds1equi}.

Therefore, combining \eqref{inequalityenergyboundshalf} and \eqref{inequalityenergyboundshalfunstable} we get 
\[
\int_{(-\infty,0] \times \mathbb{T}} \|u(s,t)\|_J^2 \, ds \, dt \leq \psilH(x_l, v) - \PhiH(y, w).
\]
If it holds that 
\[
\psilH(x_l, v) = \PhiH(y, w),
\] 
then all the inequalities in \eqref{inequalityenergyboundshalfunstable} must be equalities. The last inequality is an equality if and only if $(\Tilde{u}_1, z(0)) \in S_{(x_0, v)}$. Therefore, we conclude that $z(0) = \Tilde{v} \in S_v \subset \crit(\widetilde{f}_N)$. On the other hand, since we have that $(\Tilde{u}_1, z(0)) \in S_{(x_0, v)}$, from the fact that the second inequality of \eqref{inequalityenergyboundshalfunstable} is equality if and only if $u(0)$ is the critical point of $\Phi_{H^*_{\Tilde{v}}}$, we conclude that $(u(0), z(0)) = (\Tilde{x}, \Tilde{v}) \in S_{(x, v)}$. Since $(u, z)$ satisfies the system of equations \eqref{floerredhalf}, it follows that $z(s) = \Tilde{v}$ for all $s \in (-\infty,0]$. Since equality must also hold for \eqref{inequalityenergyboundshalf}, we conclude that $u(s, \cdot) = \Tilde{x}$ for all $s \in (-\infty,0]$. Combining all of this, we have that 
\[
\lim_{s \to -\infty} (u(s, \cdot), z(s)) = (\Tilde{x}, \Tilde{z}) \in S_{(y, w)},
\]
which now implies that $S_{(x, v)} = S_{(y, w)}$ and that the moduli space consists of the unique orbit of constant cylinder maps. If $S_{(x, v)} = S_{(y, w)}$, the equality obviously holds.
\end{proof}
\end{proposition}

\section{Dimension of half-cylinder moduli spaces}\label{sectiondimension}

\large\textbf{Non-parameterized case}
\normalsize

Let $ u \in \mathcal{M}_\Theta(x_l,y;H,J,g) $. According to Proposition \ref{cauchyriemannonp}, $ D\overline{\partial}_H(u)$ is a linear Fredholm operator with
\begin{equation}\label{indexfredholmnonparametrezied}
    \ind(D\overline{\partial}_H(u))=n-\mu_{CZ}(x;H).
\end{equation}
When $ x \neq y $, for generic $ J $ and $ g $, the map $ \overline{\partial}_H $ is transverse to the manifold \[ \{0\}\times \Gamma_{l,+}(W^u(x_l;-\nabla_g\psilH)).\]When $ x = y $, $ u(s,\cdot)=x $ is the trivial half-cylinder. Transversality is ensured by the following result.

\begin{proposition}[\cite{AK22}]\label{transversalitynonparconstant}
   Let $ u \in \mathcal{M}_\Theta(x_l,x;H,J,g) $. The linearized Fredholm operator $ D\overline{\partial}_{H}(u) $ is transverse to the space $ \{0\}\times T_{\mathbb{P}_+(u(0,\cdot))}\Gamma_{l,+}(W^u(x_l;-\nabla_g\psilH)) $.
\end{proposition}

\begin{proof}

As argued in Proposition \ref{energyboundsnonparhalf}, $ \mathcal{M}_\Theta(x_l,x;H,J,g) $ consists of the unique half-cylinder $ u(s,\cdot)=x $. The Floer equations can be written in the form
\begin{equation}\label{fnpar}
     \partial_s u = \nabla\Phi_H(u)
\end{equation}
where $ \nabla\Phi_H $ is given with respect to the $ L^2 $ metric. Since $ u(s,\cdot)=x $, linearizing the equation \eqref{fnpar} we get
\[
\partial_s v = \nabla^2\Phi_H(x)v.
\]

Therefore, the operator $ D\overline{\partial}_H(u) $ is of the form
\[
D\overline{\partial}_H(u)(v) = (\partial_sv - \nabla^2\Phi_H(x)v, \mathbb{P}_+(v(0,\cdot))).
\]

To show that $ D\overline{\partial}_H(u) $ is transverse to $ \{0\}\times T_{\mathbb{P}_+(x)}\Gamma_{l,+}(W^u(x_l;-\nabla_g\psilH)) $, it is sufficient to show that
\begin{equation}\label{intersectiontrivialnonpar}
    \text{im}(D\overline{\partial}_H(u)) \cap \left(\{0\}\times T_{\mathbb{P}_+(x)}\Gamma_{l,+}(W^u(x_l;-\nabla_g\psilH))\right) = \{0\}.
\end{equation}

Indeed, if this condition is satisfied, then $ \dim \ker(D\overline{\partial}_H(u)) = 0 $, which implies from \eqref{indexfredholmnonparametrezied} that
\begin{equation}\label{codimensionconstantnonpar}
    \codim(\text{im}(D\overline{\partial}_H(u))) = \mu_{CZ}(x) - n.
\end{equation}

On the other hand, we have
\begin{equation}\label{dimensionconstantunstablenonpar}
    \dim\left(\{0\}\times T_{\mathbb{P}_+(x)}\Gamma_{l,+}(W^u(x_l;-\nabla_g\psilH))\right) = \ind(x_l) = \mu_{CZ}(x) - n
\end{equation}
where the last equality follows from Remark \ref{nullityindexred}. Combining \eqref{intersectiontrivialnonpar}, \eqref{codimensionconstantnonpar}, and \eqref{dimensionconstantunstablenonpar}, we see that $ \text{im}(D\overline{\partial}_H(u)) $ and $ \{0\}\times T_{\mathbb{P}_+(x)}\Gamma_{l,+}(W^u(x_l;-\nabla_g\psilH)) $ are transverse.

To show that \eqref{intersectiontrivialnonpar} holds, we need to show that if
\[
\nu \in \text{im}(D\overline{\partial}_H(u)) \cap \left(\{0\}\times T_{\mathbb{P}_+(x)}\Gamma_{l,+}(W^u(x_l;-\nabla_g\psilH))\right)
\]
then $ \nu = 0 $. We do this by showing that the function
\[
\varphi(s) \defeq \|\nu(s,\cdot)\|_{L^2}^2
\]
is identically zero, as shown in \cite[Proposition 8.3]{AK22}.

\end{proof}

Due to transversality arguments and \eqref{indexfredholmnonparametrezied}, we have that $ \mathcal{M}_\Theta(x_l,y;H,J,g) $ is generically a smooth manifold of dimension
\[
\ind(D\overline{\partial}_H(u)) + \dim\left(\{0\}\times i_{+,l}(W^u(x_l;g))\right) = \ind(D\overline{\partial}_H(u)) + \dim W^u(x_l;-\nabla_g\psilH).
\]

Therefore, we have
\begin{equation}\label{dimensionhalfcylnonpar}
  \dim \mathcal{M}_\Theta(x_l,y;H,J,g)  = \ind(x_l; \psilH) + n - \mu_{CZ}(y;H).
\end{equation}

\large\textbf{$S^1$-equivariant case }
\normalsize

If we choose $d > 0$ as in Remark \ref{weighted operator}, then for the moduli space
\[\hat{\mathcal{M}}^{S^1,N}_\Theta(S_{(x_l,v)},S_{(y,w)};H,\widetilde{f}_N,J,g),\]
we have
\[ (\overline{\partial}^d_{H,N})^{-1}(\{0\} \times i_{l,N,+}(W^u(S_{(x_l,v)}, -V_g))) = \hat{\mathcal{M}}^{S^1,N}_\Theta(S_{(x_l,v)},S_{(y,w)};H,\widetilde{f}_N,J,g), \]
where $\overline{\partial}^d_{H,N}$ is the weighted operator described in the same Remark.

We can view the Hilbert manifold
\[ S_{(y,w)} + H^1((-\infty,0] \times \mathbb{T}, \mathbb{R}^{2n}; e^{-ds} \, ds \, dt) \times H^1((-\infty,0], S^{2N+1}; e^{-ds} \, ds \, dt) \]
as the $S^1$-orbit of the affine space
\[ (\Tilde{y},\Tilde{w}) + H^1((-\infty,0] \times \mathbb{T}, \mathbb{R}^{2n}; e^{-ds} \, ds \, dt) \times H^1((-\infty,0], S^{2N+1}; e^{-ds} \, ds \, dt), \]
where $(\Tilde{y},\Tilde{w}) \in S_{(y,w)}$.

Thus, for $(u,z) \in \mathcal{M}_\Theta(S_{(x_l,v)},S_{(y,w)};H,\widetilde{f}_N,J,g)$ such that
\[ \lim_{s \to -\infty} (u(s, \cdot), z(s)) = (\Tilde{y}, \Tilde{w}), \]
the linearization of the operator $\overline{\partial}^d_{H,N}$ at $u$ can be understood as a map
\[ D\overline{\partial}^d_{H,N}(u,z): H^1(u^*T\mathbb{R}^{2n}; e^{-ds} \, ds \, dt) \oplus H^1(z^*S^{2N+1}; e^{-ds} \, ds \, dt) \oplus V_{(\Tilde{y},\Tilde{w})} \]
\[ \to L^2(u^*T\mathbb{R}^{2n}; e^{-ds} \, ds \, dt) \oplus L^2(z^*S^{2N+1}; e^{-ds} \, ds \, dt) \oplus \mathbb{H}_+ \oplus T_{z(0)}S^{2N+1}. \]

Here, $V_{(\Tilde{y},\Tilde{w})}$ is generated by $\beta(s)(\dot{\Tilde{y}}, X^{S^1}_{\Tilde{w}})$, where $(\dot{\Tilde{y}}, X^{S^1}_{\Tilde{w}})$ is the infinitesimal generator of the $S^1$-action, and $\beta:(-\infty,0] \to [0,1]$ is a smooth cutoff function that is 1 near $-\infty$.

\begin{proposition}
    Let $(u,z) \in \hat{\mathcal{M}}^{S^1,N}_\Theta(S_{(x_l,v)},S_{(y,w)};H,\widetilde{f}_N,J,g)$. Then the linearized operator $D\overline{\partial}^d_{H,N}(u,z)$ is Fredholm of index $n - |S_{(y,w)}|$.
\end{proposition}

\begin{proof}
     By choosing some $p > 1$ and conjugating the restriction of $ D\overline{\partial}^d_{H,N}(u,z)$ to \[H^1(u^*T\mathbb{R}^{2n}; e^{-ds} \, ds \, dt) \oplus H^1(z^*S^{2N+1}; e^{-ds} \, ds \, dt)\] by $e^{-\frac{d}{p}s}$, we obtain an operator $\overline{L} \in \mathcal{D}^{m,u}_-(\widetilde{\alpha})$, where $\widetilde{\alpha} = (j, \widetilde{s})$ and $\widetilde{s}$ is a perturbation of a block upper triangular matrix by the matrix $\frac{d}{p}\unit$. By our choice of $d > 0$, this is a Fredholm operator.

    Indeed, by definition, the asymptotic operator at $-\infty$ is
    \[ \overline{D}_{(\Tilde{y},\Tilde{w})} = D_{(\Tilde{y},\Tilde{w})} + \frac{d}{p} \unit. \]
    
    Since $d > 0$ is smaller than the spectral gap of $D_{(\Tilde{y},\Tilde{w})}$, the asymptotic operator is non-degenerate with index
    \[ \mu(\Tilde{y},\Tilde{w}; \overline{D}) = \mu(\Tilde{y},\Tilde{w}; D) + \frac{1}{2}. \]
    
    The identity $\mu(\Tilde{y},\Tilde{w}; D) = \mu(\Tilde{y},\Tilde{w})$ holds since the asymptote of the matrix part of the linearized operator $D \overline{\partial}^{F,d}_{H,N}(u,z)$ is loop of upper-triangular matrices, which on the diagonal have loop of symmetric matrices $\{\nabla_J^2 H_{\Tilde{w},t}(\Tilde{y}(t)), -\nabla_z^2 \widetilde{f}_N(\Tilde{w})\}$ and \[\mu_{CZ}(\nabla_J^2 H_{\Tilde{w},t}(\Tilde{y}(t))) + \frac{1}{2} \text{sign}(-\nabla_z^2 \widetilde{f}_N(\Tilde{w})) = \mu_{CZ}(\Tilde{y};H_{\Tilde{w}}) + \ind(\Tilde{w};\widetilde{f}_N) - N = \mu(\Tilde{y},\Tilde{w}).\]  By Proposition \ref{cauchyriemmanpar}, this implies that $\overline{L}$ is Fredholm of index $n - \mu(\Tilde{y},\Tilde{w}) - N - 1 = n - \mu(y,w) - N - 1$.

    Since the restriction of $ D\overline{\partial}^d_{H,N}(u,z)$ to a codimension 1 subspace is conjugate to $\overline{L}$, it follows that $D\overline{\partial}^d_{H,N}(u,z)$ is also Fredholm with
    \[ \ind( D\overline{\partial}^d_{H,N}(u,z)) = \ind(\overline{L}) + 1 = n - \mu(y,w) - N. \]

    Given that $|S_{(y,w)}| = \mu(y,w) + N$, we conclude that the claim holds.
\end{proof}

\begin{proposition}
 Let $(u,z)\in \hat{\mathcal{M}}^{S^1,N}_\Theta(S_{(x_l,v)},S_{(x,v)};H,\widetilde{f}_N,J,g)$. Then the linearized operator $D\overline{\partial}^d_{H,N}(u,z)$ is tranverse to $\{0\}\times T_{(\mathbb{P}_+(u(0,\cdot)),z(0))}\Gamma_{l,N,+}(W^u(S_{(x_l,v)};-V_g)).$    
\end{proposition}
\begin{proof}

In Proposition \ref{energyboundsS1half}, it was shown that $\hat{\mathcal{M}}^{S^1,N}_\Theta(S_{(x_l,v)},S_{(x,v)};H,\widetilde{f}_N,J,g)$ consists of a unique orbit of constant half-cylinders. Let $(u(s,\cdot),z(s))=(\Tilde{x},\Tilde{v})\in S_{(x,v)}$. As argued before, we can understand $D\overline{\partial}^d_{H,N}(u,z)$ as  
\[D\overline{\partial}^d_{H,N}(u,z): H^1(u^*T\mathbb{R}^{2n};e^{-ds}dsdt)\oplus H^1(z^*S^{2N+1};e^{-ds}dsdt)\oplus V_{(\Tilde{x},\Tilde{v})},\] 
\[\to L^2(u^*T\mathbb{R}^{2n};e^{-ds}dsdt)\oplus L^2(z^*S^{2N+1};e^{-ds}dsdt)\oplus \mathbb{H}_+\oplus T_{z(0)}S^{2N+1}.\]

Let $L$ denote the restriction of this operator to \[H^1(u^*T\mathbb{R}^{2n};e^{-ds}dsdt)\oplus H^1((z^*S^{2N+1};e^{-ds}dsdt).\]

We will show that $L$ is transverse to $\{0\}\times T_{(\mathbb{P}_+(u(0,\cdot)),z(0))}\Gamma_{l,N,+}(W^u(S_{(x_l,v)};-V_g)).$

From analogous arguments as in Proposition \ref{transversalitynonparconstant}, we have that
\begin{equation}\label{intersectiontransversallityS1}
    \text{im}(L)\cap \left(\{0\}\times T_{(\mathbb{P}_+(u(0,\cdot)),z(0))}\Gamma_{l,N,+}(W^u(S_{(x_l,v)};-V_g))\right)=0
\end{equation}
implies that $L$ transverse to $\{0\}\times T_{(\mathbb{P}_+(u(0,\cdot)),z(0))}\Gamma_{l,N,+}(W^u(S_{(x_l,v)};-V_g)).$

Let's show that \eqref{intersectiontransversallityS1} holds.

The family Floer system of equations can be written in the form 
\begin{equation}\label{floerconstants1equation}
        \begin{pmatrix}
        \partial_s u\\
        \partial_s z
    \end{pmatrix}=\begin{pmatrix}
        \nabla_{L^2}\Phi_{H_z}(u)\\
        \nabla_z \widetilde{f}_N(z)
    \end{pmatrix}.
\end{equation}

Since \[(u(s,\cdot),z(s))=(\Tilde{x},\Tilde{v})\in S_{(x,v)},\] by linearizing the equation \eqref{floerconstants1equation}, we get 
\begin{equation}\label{linearizedfloerconstants1equation}
         \begin{pmatrix}
        \partial_s \eta\\
        \partial_s \nu
    \end{pmatrix}=\begin{pmatrix}
        \nabla^2_{L^2}\Phi_{H_{\Tilde{v}}}(\Tilde{x}) & \nabla_z \nabla_{L^2}\Phi_{H_{\Tilde{v}}}(\Tilde{x})\\
        0 & \nabla^2_z \widetilde{f}_N(\Tilde{v})
    \end{pmatrix}\begin{pmatrix}
        \eta\\
        \nu
    \end{pmatrix}.
\end{equation}

Let \[(\eta,\nu)\in \text{im}(L)\cap \left(\{0\}\times T_{(\mathbb{P}_+(u(0,\cdot)),z(0))}\Gamma_{l,N,+}(W^u(S_{(x_l,v)};-V_g))\right).\]

From \eqref{linearizedfloerconstants1equation} we have that \[\partial_s \nu = \nabla^2_z \widetilde{f}_N(\Tilde{v}) \nu,\]

which implies that 
\begin{equation}\label{stablenegativefn}
    \nu(0)\in T_{\Tilde{v}}W^{u}(S_{\Tilde{v}};\nabla \widetilde{f_N})=T_{v'}W^{s}(S_{\Tilde{v}};-\nabla \widetilde{f_N}).
\end{equation}

since $\nu \in H^1((-\infty,0))$. On the other hand, we have that  \[(\eta(0,\cdot),\nu(0))=(\eta_0,\nu(0))+\eta_1+\eta_2\] where 
       $(\eta_0,\nu(0))\in T_{(\mathbb{P}_+(\Tilde{x},\Tilde{v})}\Gamma_{l,N,+}(W^u(S_{(x_l,v)};-V_g))$, $\eta_1\in \mathbb{R}^{2n}$, and $\eta_2\in \mathbb{H}^-_{1/2}$.

Since $V_g=(\nabla_{g_z} \psi^l_H,\nabla \widetilde{f}_N)$, this implies that 
\begin{equation}\label{unstablenegativefn}
    \nu(0)\in T_{\Tilde{v}}W^u(S_{\Tilde{v}},-\nabla \widetilde{f}_N).
\end{equation}

From \eqref{stablenegativefn} and \eqref{unstablenegativefn} it follows that \[\nu(0)\in \mathbb{R}\cdot X^{S^1}_{\Tilde{v}},\]

where $X^{S^1}_{\Tilde{v}}$ is the infinitesimal generator of the $S^1$-action on $T_{\Tilde{v}}S^{2N+1}$. Since it holds \[\ker(\nabla^2_z \widetilde{f}_N(\Tilde{v})) = \mathbb{R}\cdot X^{S^1}_{\Tilde{v}},\]

we conclude, from the equation \[\partial_s \nu = \nabla^2_z \widetilde{f}_N(\Tilde{v}) \nu,\]

and uniqueness of solutions of ODEs that $\nu(s)=\nu(0)\in \mathbb{R}\cdot X^{S^1}_{\Tilde{v}}$ for $s\in (-\infty,0]$. Since $\nu \in H^1$, the previous conclusion implies that $\nu=0$. 

Now, the first equation of the system \eqref{linearizedfloerconstants1equation} reduces to \[\partial_s \eta = \nabla^2_{L^2}\Phi_{H_{\Tilde{v}}}(\Tilde{x}) \eta.\]

Since $\Phi_{H_{\Tilde{v}}}$ is non-degenerate, by the same arguments as in \cite[Proposition 8.3]{AK22}, we show that \[\varphi(s) = \|\eta(s)\|^2_{L^2}, \ \ \ s\in (-\infty,0]\]

is identically equal to zero, which implies that $\eta=0$. This now implies that \eqref{intersectiontransversallityS1} holds, which, by the discussion at the beginning of the proof, gives us that $L$ is transverse to $\{0\}\times T_{(\mathbb{P}_+(u(0,\cdot)),z(0))}\Gamma_{l,N,+}(W^u(S_{(x_l,v)};-V_g))$. Because $L$ is a restriction of $D\overline{\partial}^d_{H,N}(u,z)$, the same conclusion holds for $D\overline{\partial}^d_{H,N}(u,z)$.  
\end{proof}

The dimension of the unstable manifold $W^u(S_{(x_l,v)},-V_g)$ is $\ind(S_{(x_l,v)}) + 1$ due to claim (4) of Proposition \ref{vectorfieldvg}. Therefore, from the previous two propositions and Remark \ref{weighted operator}, we conclude that for generic $J$ and $g$ we have that $\hat{\mathcal{M}}^{S^1,N}_\Theta(S_{(x_l,v)},S_{(y,w)};H,\widetilde{f}_N,J,g)$ is a smooth manifold and it holds 
\[\dim\hat{\mathcal{M}}^{S^1,N}_\Theta(S_{(x_l,v)},S_{(y,w)};H,\widetilde{f}_N,J,g)=\ind(S_{(x_l,v)}) + n - |S_{(y,w)}| + 1.\]

Since the $S^1$-action on $\hat{\mathcal{M}}^{S^1,N}_\Theta(S_{(x_l,v)},S_{(y,w)};H,\widetilde{f}_N,J,g)$ is free, we have that 
\[\mathcal{M}^{S^1,N}_\Theta(S_{(x_l,v)},S_{(y,w)};H,\widetilde{f}_N,J,g)\defeq\hat{\mathcal{M}}^{S^1,N}_\Theta(S_{(x_l,v)},S_{(y,w)};H,\widetilde{f}_N,J,g)/\mathbb{T},\]
is a smooth manifold with 
\begin{equation}\label{dimensionhalfcyls1equi}
\dim\mathcal{M}^{S^1,N}_\Theta(S_{(x_l,v)},S_{(y,w)};H,\widetilde{f}_N,J,g) = \ind(S_{(x_l,v)}) + n - |S_{(y,w)}|.
\end{equation}

\section{Compactness of half-cylinder moduli space}

\large\textbf{Non-parameterized case}
\normalsize

\cite[Proposition 9.1]{AK22} and \cite[Proposition 9.2]{AK22} imply the following result.

\begin{proposition}\label{compactnessnonpar}
    $\mathcal{M}_\Theta(x_l,y;H,J,g)$ is bounded in the $L^\infty$-norm and pre-compact in $C^\infty_{loc}((-\infty,0]\times \mathbb{T},\mathbb{R}^{2n})$.
\end{proposition}

The only case when $\mathcal{M}_\Theta(x_l,y;H,J,g)$ loses compactness is when trajectories experience potential breaking.

Let $J$ and $g$ be chosen generically.

\begin{proposition}\label{chainmapnonpar}
    Let $\ind(x_l)=\mu_{CZ}(z)-n+1$. Then $\mathcal{M}_\Theta(x_l,z;H,J,g)$ is a one-dimensional manifold, and it holds
    \begin{align}
        \nonumber\partial \mathcal{M}_\Theta(x_l,z;H,J,g)=&\bigcup_{\ind(y_l)=\ind(x_l)-1}\mathcal{M}_M(x_l,y_l;\psi_{H},g)\times \mathcal{M}_\Theta(y_l,z;H,J,g)\\
        \nonumber&\cup\bigcup_{\mu_{CZ}(y)=\mu_{CZ}(z)+1}\mathcal{M}_\Theta(x_l,y;H,J,g) \times \mathcal{M}_F(y,z;H,J)
    \end{align}
    where all the manifolds on the right-hand side are $0$-dimensional.
\end{proposition}

\begin{proof}
    Alternatively, we can understand the space ${M}_\Theta(x_l,z;H,J,g)$ analytically as the space of pairs $(\gamma,u)$ where $\gamma:(-\infty,0] \to \mathbb{H}_l$ and $u:(-\infty,0]\times \mathbb{T} \to \mathbb{R}^{2n}$ are smooth maps satisfying the equations
    \begin{align}
        \nonumber& \dot{\gamma}(s)=-\nabla_g \psilH(\gamma(s))\\
        \nonumber& \partial_su+J(s,t,u(s,t))(\partial_tu-X_{H_t}(u))=0
    \end{align}
    asymptotic conditions
     \begin{align}
        \nonumber& \lim\limits_{s\to -\infty}\gamma(s)=x_l\\
        \nonumber& \lim\limits_{s\to -\infty}u(s,\cdot)=z
    \end{align} 
    and boundary condition
    \[u(0,\cdot) = \Gamma_{l,+}(\gamma(0)) + \omega, \ \ \ \omega \in \mathbb{H}_- \oplus \mathbb{H}_0.\]

   Since we have already proven compactness and dimension results, this now reduces to standard compactness and gluing arguments as done, for instance, for conormal bundles in \cite{Dju17}.  
\end{proof}

\large\textbf{$S^1$-equivariant case} 
\normalsize

\begin{lemma}
    Let $(u,z)\in \hat{\mathcal{M}}^{S^1,N}_\Theta(S_{(x_l,v)},S_{(y,w)};H,\widetilde{f}_N,J,g)$ and \[(u(0),z(0))=(u_0+u_1+u_2,z(0))\] where $(u_0,z(0))\in \Gamma_{l,N}(W^u(S_{(x_l,v)};-V_g))$ and $u_1\in \mathbb{H}_0$ and $u_2\in \mathbb{H}_-$. Then it holds:
    \begin{enumerate}
        \item $u_0$ belongs to a compact set of $C^\infty(\mathbb{T},\mathbb{R}^{2n})$ which depends only on $S_{(x_l,v)}$ and $S_{(y,w)}$.
        \item $\|u_2\|^2_{1/2}\leq 2(\psilH(x_l,v)-\PhiH(y,w))$.
        \item $u_1$ belongs to a compact subset of $\mathbb{R}^{2n}$ which depends only on $S_{(x_l,v)}$ and $S_{(y,w)}$.
    \end{enumerate}
\end{lemma}

\begin{proof}
    From claim (1) of Lemma \ref{filtrations1equfloer} and $S^1$-invariance we have that
    \begin{equation}
       \nonumber \Phi_H(y,w)\leq \Phi_H(u(0,\cdot),z(0))=\Phi_H(u_0+u_1+u_2,z(0))
    \end{equation}

    Using this inequality, \cite[Proposition 5.2]{AK22} and the fact that \[(u_0,z(0))\in \Gamma_{l,N}(W^u(x_l;-V_g))\] we conclude that
    \begin{equation}\label{compactnesss1equboundaryinequality}
        \Phi_H(y,w)\leq \PsiH(u_0,z(0))-\frac{1}{2}\|u_2\|^2_{1/2}\leq \PsiH(x_0,v)-\frac{1}{2}\|u_2\|^2_{1/2}=\psilH(x_l,v)-\frac{1}{2}\|u_2\|^2_{1/2}
    \end{equation}

    From the first inequality in \eqref{compactnesss1equboundaryinequality} we have that $(u_0,z(0))$ belongs to the set
    \[\{\Psi_{H^*}\geq \Phi_H(y,w)\}\cap \Gamma_{l,N}(W^u(S_{(x_l,v)};-V_g)))\] 
    which is pre-compact in $C^\infty(\mathbb{T},\mathbb{R}^{2n})\times S^{2N+1}$ so we proved (1).

    Also from \eqref{compactnesss1equboundaryinequality} immediately follows (2).

    Using uniform quadratic convexity of $H$ we get that there are constants $a>0$ and $b>0$ such that
    \[H_{z,t}(x)\geq a|x|^2-b\] for all $(t,x,z)\in \mathbb{T}\times \mathbb{R}^{2n}\times S^{2N+1}$. Then we have that
    \begin{align}
        \nonumber&\int\limits_{\mathbb{T}}H_{z(0)}(u_0+u_1+u_2)dt\geq a\int\limits_{\mathbb{T}}|u_0+u_1+u_2|^2dt-b\\
        \nonumber&\geq a\int\limits_{\mathbb{T}}(|u_1|-|u_0|-|u_2|)^2dt-b\geq a|u_1|^2-2a|u_1|\int\limits_{\mathbb{T}}(|u_0|+|u_2|)dt-b.
    \end{align}

    From (1) and (2) we have that \[\int\limits_{\mathbb{T}}(|u_0|+|u_2|)dt\] is uniformly bounded by some positive constant $c>0$, which depends only on $S_{(x_l,v)}$ and $S_{(y,w)}$. Thus, it holds
    \[\int\limits_{\mathbb{T}}H_{z(0),t}(u_0+u_1+u_2)dt\geq a|u_1|^2-2a|u_1|c-b=a(|u_1|-c)^2-ac^2-b.\]

    From there, we have that
    \begin{align}
        \nonumber&\PhiH(x,v)\leq \Phi_H(u_0+u_1+u_2,z(0))=\frac{1}{2}\int\limits_{\mathbb{T}}\langle (u_0+u_2),J_0(u_0+u_2)dt-\int\limits_{\mathbb{T}}H_{z(0),t}(u_0+u_1+u_2)dt\\
        \nonumber&\leq \|u_0+u_2\|_{1/2}^2-a(|u_1|-c)^2+ac^2+b.
    \end{align}

    Since $\|u_0+u_2\|_{1/2}$ is uniformly bounded from (1) and (2), from the last inequality, we conclude that (3) also holds.
\end{proof}

\begin{proposition}
    The space $\hat{\mathcal{M}}^{S^1,N}_\Theta(S_{(x_l,v)},S_{(y,w)};H,\widetilde{f}_N,J,g)$ is bounded in the $L^\infty$-norm and pre-compact in $C^\infty_{loc}((-\infty,0]\times \mathbb{T},\mathbb{R}^{2n})\times C_{loc}^\infty((-\infty,0],S^{2N+1})$. 
\end{proposition}

\begin{proof}
    We already have that $z\in L^\infty((-\infty,0],S^{2N+1})$ and $z\in C_{loc}^\infty((-\infty,0],S^{2N+1})$ since $z$ is a solution of the gradient equation 
    \[\dot{z}=\widetilde{f}_N(z).\]

    On the other hand, since we have uniform energy bounds on space \[\hat{\mathcal{M}}^{S^1,N}_\Theta(S_{(x_l,v)},S_{(y,w)};H,\widetilde{f}_N,J,g),\] Proposition \ref{compactnessfloer} guarantees that for every $\sigma<0$ we have that uniform $L^\infty$-bound on 
    \[\{u|_{(-\infty,\sigma]\times \mathbb{T}}\mid (u,z) \in \hat{\mathcal{M}}^{S^1,N}_\Theta(S_{(x_l,v)},S_{(y,w)};H,\widetilde{f}_N,J,g)\}.\] By standard arguments, we conclude that for $\sigma_1<\sigma_2<0$ we have that 
    \[\{u|_{[\sigma_1,\sigma_2]\times \mathbb{T}}\mid (u,z) \in \hat{\mathcal{M}}^{S^1,N}_\Theta(S_{(x_l,v)},S_{(y,w)};H,\widetilde{f}_N,J,g)\}\]
    is pre-compact in $C^\infty([\sigma_1,\sigma_2]\times \mathbb{T},\mathbb{R}^{2n})$.

    Therefore, it's sufficient to show that 
    \[\{u|_{[-1,0]}\mid (u,z) \in \hat{\mathcal{M}}^{S^1,N}_\Theta(S_{(x_l,v)},S_{(y,w)};H,\widetilde{f}_N,J,g)\}\] 
    is uniformly bounded in $L^\infty([-1,0]\times \mathbb{T},\mathbb{R}^{2n})$ and pre-compact in $C^\infty([-1,0]\times \mathbb{T},\mathbb{R}^{2n})$. We can show that by demonstrating uniform bounds of $u$ in $H^k((-1,0)\times \mathbb{T},\mathbb{R}^{2n})$ for every $k\in \mathbb{N}$. Given that we have uniform bounds for $z\in H^k((-1,0),\mathbb{R}^{2n})$, the result follows from the same arguments as in \cite[Proposition 7.3]{AK22} using the previous lemma and Lemma \ref{mixedderivativesham}.
\end{proof}

\begin{remark}\label{chainmaps1equi}
    The analogous proposition to Proposition \ref{chainmapnonpar} holds for the $S^1$-equivariant case.    
\end{remark}

\section{Moduli spaces of homotopy half-cylinders}

As shown in the proof of Proposition \ref{chainmapnonpar}, half-cylinder moduli spaces can be described analytically as pairs $(u, \gamma)$ satisfying specific conditions. The moduli spaces introduced here can be analyzed similarly, following the approach in \cite{Dju17}. The main difference in our setting is the need to address compactness results, as the ambient space is $\mathbb{R}^{2n}$. While we discuss compactness briefly, detailed proofs are provided in \cite{Dju17}.  These moduli spaces are essential for proving the functionality property of the chain isomorphism in \cite{AK22} and are analogous to those in the $S^1$-equivariant case. Hence, we will focus only on the non-parameterized version.

Let $(H_\alpha, J_\alpha)$ and $(H_\beta, J_\beta)$ be two Hamiltonians for which the Floer complex is well-defined, and suppose that $H_\alpha \leq H_\beta$. We define a parameterized family of non-increasing homotopies $(H^{\alpha\beta}_{T,s}, J^{\alpha\beta}_{T,s})$ such that

\[
(H^{\alpha\beta}_{T,s}, J^{\alpha\beta}_{T,s}) = 
\begin{cases}
    (H_\beta, J_\beta), & s \leq T - 1, \\
    (H_\alpha, J_\alpha), & s \geq T,
\end{cases}
\]

for $T \leq T_0 < -1$.

Let $l \in \mathbb{N}$ be large enough so that the $l$-saddle point reduction exists for $\Psi_{H^*_\alpha}$ and $\Psi_{H^*_\beta}$. For $x_l^\alpha \in \operatorname{Crit}(\psi_{H^*_\alpha})$ and $z^\beta \in P(H_\beta)$, we define the moduli space
\[
\mathcal{M}(x_l^\alpha, \psi_{H^*_\alpha}; z^\beta, H^{\alpha\beta}_{T,s})
\]
as the space of smooth maps $u: (-\infty, 0] \times \mathbb{T} \to \mathbb{R}^{2n}$ such that
\[
\partial_s u + J^{\alpha\beta}_{T,s} ( \partial_t u - X_{H^{\alpha\beta}_{T,s}}(u) ) = 0
\]
and it holds
\[
\lim\limits_{s \to -\infty} u(s, \cdot) = z^\beta, \quad \mathbb{P}_+(u(0, \cdot)) \in \Gamma_{l,+}\left(W^u\left(x_l^\alpha; -\nabla_{g_\alpha} \psi_\alpha\right)\right),
\]

where $T \leq T_0$. This space depends on $J$, which we will not write in the moduli space notation to simplify. We will write an element $u$ as a pair $(T, u)$ to emphasize the parameter associated with $u$. Arbitrary $(T, u)$ have a uniform energy bound given by
\[
\psi^l_{H_\alpha^*}(x_l^\alpha) - \Phi_{H_\beta}(z^\beta).
\]
This result follows from the fact that the family of homotopies is non-increasing. Therefore, we obtain the result by reasoning as in Proposition \ref{energyboundsS1half}. On the other hand, for $s \leq T - 1$, we have $H^{\alpha\beta}_{T,s} = H_\beta$, and for $s \geq T$, it holds that $H^{\alpha\beta}_{T,s} = H_\alpha$. Since $H_\alpha$ and $H_\beta$ are non-resonant at infinity and have uniform linear growth of their Hamiltonian vector fields, and given the uniform energy bound and the fact that the length of the interval on which $H^{\alpha\beta}_{T,s}$ is $s$-dependent is $1$ for every $T \leq T_0$, it follows that for arbitrary $(T, u)$ and $\sigma < 0$, we have a uniform $L^\infty$-bound on $u|_{(-\infty, \sigma] \times \mathbb{T}}$.

This can be shown by a slight modification of the proof in \cite[Proposition 3.1]{AK22}. Then, by the standard arguments, we conclude that the set
\[
\left\{ u|_{[\sigma_1, \sigma_2] \times \mathbb{T}} \mid u \in \mathcal{M}(x_l^\alpha, \psi^l_{H^*_\alpha}; z^\beta, H^{\alpha\beta}_{T,s}) \right\}
\]
is pre-compact in $C^\infty([\sigma_1, \sigma_2] \times \mathbb{T}, \mathbb{R}^{2n})$ for $\sigma_1 < \sigma_2 < 0$. Since $H^{\alpha\beta}_{T,s} = H_\alpha$ for every $T \leq T_0$ and $s \in [-1,0]$, we conclude from \cite[Proposition 9.3]{AK22} that there is a uniform $L^\infty$-bound and a uniform $H^k$-bound on
\[
\left\{ u|_{[-1, 0] \times \mathbb{T}} \mid u \in \mathcal{M}(x_l^\alpha, \psi_{H^*_\alpha}; z^\beta, H^{\alpha\beta}_{T,s}) \right\}.
\]
Thus, this space is pre-compact in $C^\infty([-1, 0] \times \mathbb{T}, \mathbb{R}^{2n})$. Combining these results, we conclude that the moduli space $\mathcal{M}(x_l^\alpha, \psi^l_{H^*_\alpha}; z^\beta, H^{\alpha\beta}_{T,s})$ is uniformly bounded in the $L^\infty$-norm and pre-compact in $C^\infty_{\text{loc}}((-\infty, 0] \times \mathbb{T}; \mathbb{R}^{2n})$. For a generic choice of $(H^{\alpha\beta}_{T,s}, J^{\alpha\beta}_{T,s})$, the moduli space $\mathcal{M}(x_l^\alpha, \psi_{H^*_\alpha}; z^\beta, H^{\alpha\beta}_{T,s})$ is a smooth, pre-compact manifold with 
\[
\dim  \mathcal{M}(x_l^\alpha, \psi_{H^*_\alpha}; z^\beta, H^{\alpha\beta}_{T,s}) = \operatorname{ind}(x_l^\alpha; \psi^l_{H_\alpha^*}) + n - \mu_{CZ}(z^\beta; H_\beta) + 1.
\]

\begin{proposition}\label{homotopyfloer}
    Let $(H^{\alpha\beta}_{T,s}, J^{\alpha\beta}_{T,s})$ be a generic family of non-increasing homotopies. For $x_l^\alpha \in \crit(\psi_{H^*_\alpha})$ and $z^\beta \in P(H_\beta)$ such that $\ind(x_l^\alpha;\psi^l_{H_\alpha^*})=\mu_{CZ}(z^\beta;H_\beta)-n$ the moduli space 
    \[\mathcal{M}(x_l^\alpha, \psi_{H^*_\alpha};  H^{\alpha\beta}_{T,s},z^\beta)\]
    is a one-dimensional smooth manifold such that  \begin{align}
        \nonumber \partial\mathcal{M}(x_l^\alpha, \psi_{H^*_\alpha};z^\beta, H^{\alpha\beta}_{T,s})=&\mathcal{M}(x_l^\alpha, \psi_{H^*_\alpha};z^\beta, H^{\alpha\beta}_{T_0,s})\\
        \nonumber& \cup \bigcup\limits_{y^\alpha}\mathcal{M}_\Theta(x_l^\alpha,y^\alpha; H_\alpha,J_\alpha, g_\alpha)\times \mathcal{M}_F(y^\alpha,z^\beta; H^{\alpha\beta})\\
        \nonumber&\cup \bigcup\limits_{y^\alpha_l}\mathcal{M}_M(x_l^\alpha,y_l^\alpha; \psi^l_{H_\alpha},g_\alpha)\times \mathcal{M}(y_l^\alpha, \psi^l_{H^*_\alpha};z^\beta, H^{\alpha\beta}_{T,s})\\
        \nonumber&\cup \bigcup\limits_{y^\beta} \mathcal{M}(x_l^\alpha, \psi_{H^*_\alpha};y^\beta, H^{\alpha\beta}_{T,s})\times \mathcal{M}_F(y^\beta,z^\beta;H_\beta,J_\beta)
 \end{align}

  where all the manifolds on the right side are $0$-dimensional compact manifolds. 

\end{proposition}

Manifold $\mathcal{M}(x_l^\alpha, \psi_{H^*_\alpha};z^\beta, H^{\alpha\beta}_{T_0,s})$ is a space of all smooth maps $u: (-\infty,0]\times \mathbb{T}\to \mathbb{R}^{2n}$ 
such that 
\[\partial_su+J(\partial_tu-X_{H^{\alpha\beta}_{T_0,s}}(u))=0\]
and it holds 
\[\lim\limits_{s\to -\infty}u(s,\cdot)=z^\beta, \ \ \ \mathbb{P}_+(u(0,\cdot)) \in \Gamma_{l,+}(W^u(x_l^\alpha;-\nabla_{g_\alpha}\psi^l_{H^*_\alpha}))\]
Now, we define a similar space for the homotopy of reduced dual functionals.

Let $(\psi^{\alpha\beta}_{T,s}, g^{\alpha\beta}_{T,s})$ be a family of non-increasing homotopies such that

\[
(\psi^{\alpha\beta}_{T,s}, g^{\alpha\beta}_{T,s}) = 
\begin{cases}
    (\psi^l_{H_\alpha^*}, g_\alpha), & \text{if } s \leq T - 1, \\
    (\psi^l_{H_\beta^*}, g_\beta), & \text{if } s \geq T
\end{cases}
\]

for $T \leq T_0 < 1$. Here, $g^{\alpha\beta}_{T,s}$ is a smooth family of metrics. For $x_l^\alpha \in \operatorname{Crit}(\psi_{H_\alpha^*})$, we have
\begin{gather*}
W^u(x_l^\alpha; (\psi^{\alpha\beta}_{T,s}, g^{\alpha\beta}_{T,s}))\\
=\left\{ y \in \mathbb{H}^l \mid \dot{\gamma}(s) = -\nabla_{g^{\alpha\beta}_{T,s}} \psi^{\alpha\beta}_{T,s}(\gamma(s)), \ \gamma(s_0) = y, \ s_0 \in \mathbb{R}, \ \lim_{s \to -\infty} \gamma(s) = x_l^\alpha \right\}.
\end{gather*}
For $x_l^\alpha \in \operatorname{Crit}(\psi^l_{H_\alpha^*})$ and $z^\beta \in P(H_\beta)$, we define the moduli space
\[
\mathcal{M}(x_l^\alpha, \psi^{\alpha\beta}_{T,s}; z^\beta, H_\beta)
\]
as the space of all smooth maps $u: (-\infty, 0] \times \mathbb{T} \to \mathbb{R}^{2n}$ such that
\[
\partial_s u + J_\beta (\partial_t u - X_{H_\beta}(u)) = 0
\]
and
\[
\lim_{s \to -\infty} u(s, \cdot) = z^\beta, \quad \mathbb{P_+}(u(0, \cdot)) \in \Gamma_{l,+} \left(W^u(x_l^\alpha; (\psi^{\alpha\beta}_{T,s}, g_{T,s}))\right),
\]
for $T \leq T_0$. Since the family of homotopies is non-increasing, we conclude that elements of
\[
\mathcal{M}(x_l^\alpha, \psi^{\alpha\beta}_{T,s}; z^\beta, H_\beta)
\]
have a uniform energy bound given by
\[
\psi^l_{H_\alpha^*}(x_l^\alpha) - \Phi_{H_\beta}(z^\beta).
\]
Since $H_\beta$ is non-resonant at infinity and its Hamiltonian vector field grows linearly, \cite[Proposition 9.1]{AK22} states that for $\sigma < 0$, we have a uniform $L^\infty$-bound on the space
\[
\left\{ u \big|_{(-\infty, \sigma] \times \mathbb{T}} \mid u \in \mathcal{M}(x_l^\alpha, \psi^{\alpha\beta}_{T,s}; z^\beta, H_\beta) \right\}
\]
and that the space
\[
\left\{ u \big|_{[\sigma_1, \sigma_2] \times \mathbb{T}} \mid u \in \mathcal{M}(x_l^\alpha, \psi^{\alpha\beta}_{T,s}; z^\beta, H_\beta) \right\}
\]
is pre-compact for $\sigma_1 < \sigma_2 < 0$. Since the family of homotopies $\psi^{\alpha\beta}_{T,s}$ is non-increasing, the claims of \cite[Lemma 9.2]{AK22} hold for elements of $\mathcal{M}(x_l^\alpha, \psi^{\alpha\beta}_{T,s}; z^\beta, H_\beta)$. Therefore, using similar arguments as in \cite[Proposition 9.3]{AK22}, we conclude that there is a uniform $L^\infty$-bound on
\[
\left\{ u \big|_{[-1, 0] \times \mathbb{T}} \mid u \in \mathcal{M}(x_l^\alpha, \psi^{\alpha\beta}_{T,s}; z^\beta, H_\beta) \right\}
\]
and that it is pre-compact in $C^\infty([-1, 0] \times \mathbb{T}, \mathbb{R}^{2n})$. Hence, all of this implies a uniform bound in the $L^\infty$-norm and pre-compactness in $C^\infty_{\text{loc}}((-\infty, 0] \times \mathbb{T}; \mathbb{R}^{2n})$ for $\mathcal{M}(x_l^\alpha, \psi^{\alpha\beta}_{T,s}; z^\beta, H_\beta)$.

\begin{proposition}\label{homotopymorse}
    Let $(\psi^{\alpha\beta}_{T,s},g^{\alpha\beta}_{T,s})$ be a generic family of non-increasing homotopies. For $x_l^\alpha \in \crit(\psi_{H^*_\alpha})$ and $z^\beta \in P(H_\beta)$ such that $\ind(x_l^\alpha;\psi^l_{H_\alpha^*})=\mu_{CZ}(z^\beta;H_\beta)-n$ the moduli space 
    \[\mathcal{M}(x_l^\alpha,\psi^{\alpha\beta}_{T,s};z^\beta, H_\beta)\] 
    is a one-dimensional smooth manifold such that  \begin{align}
        \nonumber \partial\mathcal{M}(x_l^\alpha,\psi^{\alpha\beta}_{T,s};z^\beta, H_\beta)=&\mathcal{M}(x_l^\alpha,\psi^{\alpha\beta}_{T_0,s};z^\beta, H_\beta)\\
        \nonumber& \cup \bigcup\limits_{y_l^\beta}\mathcal{M}_M(x_l^\alpha,y_l^\beta;\psi^{\alpha\beta})\times\mathcal{M}_\Theta(y_l^\beta,z^\beta;H_\beta,J_\beta,g_\beta)\\
        \nonumber&\cup \bigcup\limits_{y^\alpha_l}\mathcal{M}_M(x_l^\alpha,y_l^\alpha; \psi^l_{H^*_\alpha},g_\alpha)\times \mathcal{M}(y_l^\alpha,\psi^{\alpha\beta}_{T,s};z^\beta, H_\beta)\\
        \nonumber&\cup \bigcup\limits_{y^\beta} \mathcal{M}(x_l^\alpha,\psi^{\alpha\beta}_{T,s};y^\beta, H_\beta) \times \mathcal{M}_F(y^\beta,z^\beta;H_\beta,J_\beta)
 \end{align}
where all the manifolds on the right side are $0$-dimensional compact manifolds. 
\end{proposition}
Here, the manifold $\mathcal{M}(x_l^\alpha, \psi^{\alpha\beta}_{T_0, s}; z^\beta, H_\beta)$ is defined as the space of all smooth maps $u: (-\infty, 0] \times \mathbb{T} \to \mathbb{R}^{2n}$ such that
\[
\partial_s u + J_\beta (\partial_t u - X_{H_\beta}) = 0
\]
and
\[
\lim_{s \to -\infty} u(s, \cdot) = z^\beta, \quad \mathbb{P}_+(u(0, \cdot)) \in \Gamma_{l,+}(W^u(x_l; (\psi^{\alpha\beta}_{T_0, s}, g^{\alpha\beta}_{T_0, s}))).
\]
Let $(H^{\alpha\beta}_s, J^{\alpha\beta}_s)$ be a non-increasing homotopy from $H_\beta$ to $H_\alpha$ such that
\[
(H^{\alpha\beta}_s, J^{\alpha\beta}_s) = (H_\beta, J_\beta), \quad s \leq T_0 - 1,
\]
and
\[
(H^{\alpha\beta}_s, J^{\alpha\beta}_s) = (H_\alpha, J_\beta), \quad s \geq T_0.
\]

Let $(\psi^{\alpha\beta}_s, g^{\alpha\beta}_s)$ be a non-increasing homotopy such that
\[
(\psi^{\alpha\beta}_s, g^{\alpha\beta}_s) = (\psi_{H^*_\alpha}, g_\alpha), \quad s \leq T_0 - 1,
\]
and
\[
(\psi^{\alpha\beta}_s, g^{\alpha\beta}_s) = (\psi_{H^*_\beta}, g_\beta), \quad s \geq T_0.
\]

Here, we assume that $T_0 < 1$.

We define a family of non-increasing homotopies $H^{\alpha\beta}_{\delta, s}$ and $\psi^{\alpha\beta}_{\delta, s}$ where $\delta\in [0,1]$ such that
\[
(H^{\alpha\beta}_{\delta, s}, \psi^{\alpha\beta}_{\delta, s}) = (H^{\alpha\beta}_s, \psi^l_{H^*_\beta}), \quad \delta = 0,
\]
and
\[
(H^{\alpha\beta}_{\delta, s}, \psi^{\alpha\beta}_{\delta, s}) = (H_\beta, \psi^{\alpha\beta}_s), \quad \delta = 1.
\]

We assume there exists $T_{\min}$ negative enough such that for all $\delta \in [0, 1]$ and all $s \leq T_{\min}$, it holds that $H^{\alpha\beta}_{\delta, s} = H_{\beta}$ and $\psi^{\alpha\beta}_{\delta, s} = \psi^l_{H^*_\alpha}$.

We define the moduli space
\[
\mathcal{M}(x_l^\alpha, \psi^{\alpha\beta}_{\delta, s}; z^\beta, H_{\delta, s}^{\alpha\beta})
\]
to be the space of smooth maps $u: (-\infty, 0] \times \mathbb{T} \to \mathbb{R}^{2n}$ such that
\[
\partial_s u + J^{\alpha\beta}_{\delta,s}(\partial_t u - X_{H^{\alpha\beta}_{\delta, s}}(u)) = 0
\]
and
\[
\lim_{s \to -\infty} u(s, \cdot) = z^\beta, \quad \mathbb{P}_+(u(0, \cdot)) \in \Gamma_{l,+}(W^u(x_l^\alpha; (\psi_{s, t}^{\alpha\beta}, g_{s, t}^{\alpha\beta}))).
\]
From the same arguments as before, we have that
\[
\psi^l_{H_\alpha^*}(x_l^\alpha) - \Phi_{H_\beta}(z^\beta)
\]
is a uniform energy bound for $\mathcal{M}(x_l^\alpha, \psi^{\alpha\beta}_{\delta, s}; z^\beta, H_{\delta, s}^{\alpha\beta})$. Since for $s \leq T_{\min}$, we have $H_{\delta, s} = H_{\beta}$ independently of $\delta$, and $H_\beta$ is non-resonant at infinity with its Hamiltonian vector field having linear growth, we conclude, by a slight modification of \cite[Proposition 3.1]{AK22}, that due to the uniform energy bound, the space $\mathcal{M}(x_l^\alpha, \psi^{\alpha\beta}_{\delta, s}; z^\beta, H_{\delta, s}^{\alpha\beta})$ is uniformly bounded in the $L^\infty$-norm, except within an arbitrarily small neighborhood of $0$. 

On the other hand, arguing as for the second moduli space introduced in this section, we conclude that there is an $L^\infty$-bound on
\[
\{u|_{[-1,0] \times \mathbb{T}} \mid u \in \mathcal{M}(x_l^\alpha, \psi^{\alpha\beta}_{\delta, s}; z^\beta, H_{\delta, s}^{\alpha\beta})\}
\]
and that it is pre-compact in $C^\infty([-1,0] \times \mathbb{T}, \mathbb{R}^{2n})$.

Combining these two results, we again obtain the desired compactness result. Therefore, the following statement holds.

\begin{proposition}\label{homotopymixed}
    Let $(H^{\alpha\beta}_{\delta,s},J^{\alpha\beta}_{\delta,s})$  and $(\psi^{\alpha\beta}_{\delta,s},g^{\alpha\beta}_{\delta,s})$ be generic families of non-increasing homotopies. For $x_l^\alpha \in \crit(\psi_{H^*_\alpha})$ and $z^\beta \in P(H_\beta)$ such that $\ind(x_l^\alpha;\psi^l_{H_\alpha^*})=\mu_{CZ}(z^\beta;H_\beta)-n$ the moduli space 
    \[\mathcal{M}(x_l^\alpha,\psi^{\alpha\beta}_{\delta,s};z^\beta, H_{\delta,s}^{\alpha\beta})\] 
    is a one-dimensional smooth manifold such that  \begin{align}
        \nonumber \partial\mathcal{M}(x_l^\alpha,\psi^{\alpha\beta}_{\delta,s};z^\beta, H_{\delta,s}^{\alpha\beta})=&\mathcal{M}(x_l^\alpha,\psi^{\alpha\beta}_{s};z^\beta, H_\beta)\cup\mathcal{M}(x_l^\alpha,\psi^l_{H_\beta^*};z^\beta, H_{s}^{\alpha\beta})\\
        \nonumber&\cup\bigcup\limits_{y_l^\alpha}\mathcal{M}_M(\mathbb{P}_l(x_l^\alpha,y_l^\alpha; \psi_{H_\alpha^*};g_\alpha)\times \mathcal{M}(y_l^\alpha,\psi^{\alpha\beta}_{\delta,s};z^\beta, H_{\delta,s}^{\alpha\beta})\\
        \nonumber&\cup\bigcup\limits_{y^\beta} \mathcal{M}(x_l^\alpha,\psi^{\alpha\beta}_{\delta,s};y^\beta, H_{\delta,s}^{\alpha\beta}) \times \mathcal{M}_F(y^\beta,z^\beta;H_\beta,J_\beta)
         \end{align}
 where all the manifolds on the right side are $0$-dimensional compact manifolds. 

\end{proposition}

\section{Chain complex isomorphism}\label{sectionchain}

\begin{theorem}\label{theoremE1}
    Let $H$ be a Hamiltonian in $\mathcal{H}_\Theta$. \begin{enumerate}
        \item There is a chain complex isomorphism
    \[\Theta: \C\M_{*}(\psi^l_H) \to \C\F_{*+n}(H)\]
    from the Morse complex of the $l$-reduced dual action functional $\psi^l_H$ to the Floer complex of $H$. This isomorphism preserves the action filtrations.
        \item This isomorphism is natural. Let $H_\alpha, H_\beta \in \mathcal{H}_\Theta$ with $H_\alpha \leq H_\beta$, and let $l \in \mathbb{N}$ be large enough such that the $l$-saddle point reduction exists for both $\Psi_{H^*_\alpha}$ and $\Psi_{H^*_\beta}$. Let
        \[\Theta_\alpha: \C\M_{*}(\psi^l_{H^*_\alpha},g_\alpha) \to \C\F_{*+n}(H_\alpha, J_\alpha)\]
        and
        \[\Theta_\beta: \C\M_{*}(\psi^l_{H^*_\beta},g_\beta) \to \C\F_{*+n}(H_\beta, J_\beta)\]
        be the aforementioned chain isomorphisms, and let $(H^{\alpha\beta}_s, J^{\alpha\beta}_s)$ and $(\psi^{\alpha\beta}_s, g^{\alpha\beta}_s)$ be non-increasing homotopies such that the chain maps $i^F_{(H^{\alpha\beta}_s, J^{\alpha\beta}_s)}$ and $i^M_{(\psi^{\alpha\beta}_s, g^{\alpha\beta}_s)}$ are well-defined. Then $i^F_{(H^{\alpha\beta}_s, J^{\alpha\beta}_s)} \circ \Theta_\alpha$ and $\Theta_\beta \circ i^M_{(\psi^{\alpha\beta}_s, g^{\alpha\beta}_s)}$ are chain-homotopic.
    \end{enumerate} 
\end{theorem}

\begin{proof}

The proof of part (1) of this theorem is already given in \cite{AK22}. For the sake of completeness, we will rewrite the proof following the arguments from \cite{AK22}.

\textbf{Claim} (1): From Sections \ref{sectionfloer} and \ref{sectionmorse}, we know that the chain complexes 
\[\bigl\{\C\F_*(H), \partial^F\bigr\}, \ \ \bigl\{\C\M_{*}(\psi^l_{H^*}), \partial^M\bigr\},\]
are generically well-defined and unique up to a chain isomorphism. Here, $l \in \mathbb{N}$ is large enough such that the $l$-saddle point reduction of $\Psi_{H^*}$ exists. For a generic choice of a uniformly bounded, almost complex structure $J$ and a metric $g \in \mathcal{G}(\mathbb{H}_l)$ that is uniformly equivalent to the standard one, we obtain a manifold structure on the moduli spaces 
\[\mathcal{M}_\Theta(x_l, y; H, J, g),\]
for all critical points $x_l \in \mathrm{crit}(\psi_{H^*})$ and all $y \in P(H)$. See Section \ref{sectiondimension} for details.

From identity \eqref{dimensionhalfcylnonpar}, we have:
\[\dim\mathcal{M}_\Theta(x_l, y; H, J, g) = \mathrm{ind}(x_l; \psi_{H^*}) + n - \mu_{CZ}(y; H).\]

The grading of $\C\F_*(H)$ is given by the Conley-Zehnder index, while the grading of $\C\M_{*}(\psilH)$ is provided by the Morse index. 

This implies, by Remark \ref{nullityindexred}, that the generators of $\C\F_{k+n}(H)$ and the generators of $\C\M_{k}(\psilH)$ are in bijective correspondence via the projection $\mathbb{P}_l$. Thus, we naturally define 
\[\Bigl\{\Theta_k: \C\M_{k}(\psilH) \to \C\F_{k+n}(H)\Bigr\}_{k \in \mathbb{Z}}\]
by setting
\[\Theta_k x_l = \sum\limits_{y} n_\Theta(x_l, y) y,\]
where the sum is taken over all $y \in P(H)$ such that $\mu_{CZ}(y; H) = \mathrm{ind}(x_l; \psi_{H^*}) + n$, and $n_\Theta(x_l, y)$ is the parity of the set $\mathcal{M}_\Theta(x_l, y; H, J, g)$. This is well-defined since Remark \ref{nullityindexred} implies that $\mathcal{M}_\Theta(x_l, y; H, J, g)$ is a $0$-dimensional manifold, which, by Proposition \ref{compactnessnonpar}, is compact and therefore finite. Proposition \ref{chainmapnonpar} implies that $\Theta$ is a chain map.

Proposition \ref{energyboundsnonparhalf} shows that $\mathcal{M}_\Theta(x_l, y; H, J, g)$ is empty if $\psilH(x_l) < \Phi_H(y)$. Therefore, for any $c \in \mathbb{R}$, $\Theta$ maps the subcomplex $\C\M^{<c}_{*}(\psilH)$ to the subcomplex $\C\F^{<c}_{*+n}(H)$, preserving the filtration.

Now, let $k \in \mathbb{Z}$ be arbitrary. We will show that $\Theta_k$ is an isomorphism. Let $x^1, \dots, x^p$ be all Hamiltonian orbits generating $\C\F^{<c}_{k+n}(H)$ such that 
\[\Phi_H(x^1) \leq \dots \leq \Phi_H(x^p).\]
The corresponding basis of $\C\M^{<c}_{k}(\psilH)$ is $(x_l^1, \dots, x_l^p)$ due to Remark \ref{nullityindexred}. In these bases, $\Theta^{<c}_k$ is an upper-triangular $p \times p$ matrix with entries 1 on the diagonal. Indeed, by Proposition \ref{energyboundsnonparhalf}, we know that the set $\mathcal{M}_\Theta(x^i_l, x^j; H, J, g)$ is empty for $i < j$ and contains exactly one element if $i = j$. Thus, $\Theta^{<c}_k$ is an isomorphism.

\textbf{Claim} (2): Let $\Theta_\alpha$, $\Theta_\beta$, $i^F_{(H^{\alpha\beta}_s, J^{\alpha\beta}_s)}$, and $i^M_{(\psi^{\alpha\beta}_s, g^{\alpha\beta}_s)}$ be the chain maps from the claim. The map $i^F_{(H^{\alpha\beta}_s, J^{\alpha\beta}_s)} \circ \Theta_\alpha$ counts the number of elements in 
\[\bigcup\limits_{y^\alpha} \mathcal{M}_\Theta(x_l^\alpha, y^\alpha; H_\alpha, J_\alpha, g_\alpha) \times \mathcal{M}_F(y^\alpha, z^\beta; H^{\alpha\beta}).\]  
If $H^{\alpha\beta}_{T_0,s} = H^{\alpha\beta}_{s}$ and we define the chain map 
\[\chi: \C\M_{*}(\psi^l_{H^*_\alpha}, g_\alpha) \to \C\F_{*+n}(H_\beta, J_\beta), \ \ \ \chi(x_l^\alpha) = \sum\limits_{z^\beta} n_\chi(x_l^\alpha, z^\beta) z^\beta,\]
where $n_\chi(x_l^\alpha, z^\beta)$ is the number of trajectories in $\mathcal{M}(x_l^\alpha, \psi_{H^*_\alpha}; z^\beta, H^{\alpha\beta}_{s})$, then by Proposition \ref{homotopyfloer}, we have that $i^F_{(H^{\alpha\beta}_s, J^{\alpha\beta}_s)} \circ \Theta_\alpha$ and $\chi$ are chain homotopic. Similarly, if $\psi^{\alpha\beta}_{T_0,s} = \psi^{\alpha\beta}_{s}$ and we define the chain map 
\[\Upsilon: \C\M_{*}(\psi^l_{H^*_\alpha}, g_\alpha) \to \C\F_{*+n}(H_\beta, J_\beta), \ \ \ \Upsilon(x_l^\alpha) = \sum\limits_{z^\beta} n_\Upsilon(x_l^\alpha, z^\beta) z^\beta,\]
where $n_\Upsilon(x_l^\alpha, z^\beta)$ is the number of trajectories in $\mathcal{M}(x_l^\alpha, \psi^{\alpha\beta}_{s}; z^\beta, H_\beta)$, then by Proposition \ref{homotopymorse}, we have that $\Theta_\beta \circ i^M_{(\psi^l_s, g_s)}$ and $\Upsilon$ are chain homotopic. Finally, by Proposition \ref{homotopymixed}, we conclude that $\chi$ and $\Upsilon$ are chain homotopic, completing the proof.
\end{proof}

\begin{theorem}\label{theoremE2}
   Let $H$ be an $S^1$-invariant Hamiltonian in $\mathcal{H}_\Theta^{S^1,N}(\widetilde{f}_N, g_N)$. 
    \begin{enumerate}
        \item There is a chain complex isomorphism
        \[
        \Theta^{S^1}: \C\M^{S^1,N}_{*}(\psi^l_{H^*}, \widetilde{f}_N) \to \C\F^{S^1,N}_{*+n}(H, \widetilde{f}_N),
        \]
        from the $S^1$-equivariant Morse complex of the pair $(\psi^l_{H^*}, \widetilde{f}_N)$ to the $S^1$-equivariant Floer complex of the pair $(H, \widetilde{f}_N)$. This isomorphism preserves the action filtrations.
        \item This isomorphism is natural. Let $(N_\alpha, H_\alpha) \leq (N_\beta, H_\beta)$, where $H_\alpha \in \mathcal{H}_\Theta^{S^1, N_\alpha}(\widetilde{f}_{N_\alpha}, g_{N_\alpha})$ and $H_\beta \in \mathcal{H}_\Theta^{S^1, N_\beta}(\widetilde{f}_{N_\beta}, g_{N_\beta})$. Let $l \in \mathbb{N}$ be large enough such that the $l$-saddle point reduction exists for both $\Psi_{H^*_\alpha}$ and $\Psi_{H^*_\beta}$. Let
        \[
        \Theta^{S^1}_\alpha: \C\M^{S^1,N_\alpha}_{*}(\psi^l_{H^*_\alpha}, \widetilde{f}_{N_\alpha}; g_\alpha) \to \C\F_{*+n}^{S^1, N_\alpha}(H_\alpha, \widetilde{f}_{N_\alpha}, J_\alpha)
        \]
        and
        \[
        \Theta^{S^1}_\beta: \C\M^{S^1,N_\beta}_{*}(\psi^l_{H^*_\beta}, \widetilde{f}_{N_\beta}; g_\beta) \to \C\F_{*+n}^{S^1, N_\beta}(H_\beta, \widetilde{f}_{N_\beta}, J_\beta)
        \]
        be the aforementioned chain isomorphisms and let $(H^{\alpha\beta}_s, J^{\alpha\beta}_s)$ and $(\psi^{\alpha\beta}_s, g^{\alpha\beta}_s)$ be non-increasing homotopies such that the chain maps $i^{F,S^1}_{(H^{\alpha\beta}_s, J^{\alpha\beta}_s)}$ and $i^{M,S^1}_{(\psi^{\alpha\beta}_s, g^{\alpha\beta}_s)}$ are well-defined. Then $i^{F,S^1}_{(H^{\alpha\beta}_s, J^{\alpha\beta}_s)} \circ \Theta^{S^1}_\alpha$ and $\Theta^{S^1}_\beta \circ i^{M,S^1}_{(\psi^{\alpha\beta}_s, g^{\alpha\beta}_s)}$ are chain homotopic.
    \end{enumerate}
\end{theorem}

\begin{proof}
    Proof is completely analogus to the proof of the Previous theorem.
\end{proof}

\section{In between Clarke's duality}\label{sectionautonomusclarke}

In this section, we will compare the singular ($S^1$-equivariant) homology of sublevel sets of Clarke's dual functional introduced in \cite{Cla79} with the singular ($S^1$-equivariant) homology of sublevel sets of the dual action functional of the form 
\[
\PsiH(x) = -\frac{1}{2} \int\limits_\mathbb{T} \langle J_0 x(t), \dot{x}(t) \rangle \, dt + \int\limits_\mathbb{T} H^*(-J_0 \dot{x}) \, dt
\]
for particular types of autonomous Hamiltonians.

\textbf{Clarke's dual functional associated with a convex set}

Let $C \subset \mathbb{R}$ be a bounded convex domain whose interior contains the origin. We define the function 
\[
\HC: \mathbb{R}^{2n} \to \mathbb{R}
\]
to be a positively 2-homogeneous function such that $\HC|_{\partial C}=1$. With 
\[
\HdC: \mathbb{R}^{2n} \to \mathbb{R}
\]
we denote its Fenchel conjugate. We introduce the action 
\[
\mathcal{A}: H^1_0(\mathbb{T}, \mathbb{R}^{2n}) \to \mathbb{R}, \quad \mathcal{A}(x) = \frac{1}{2} \int\limits_\mathbb{T} \langle J_0 x(t), \dot{x}(t) \rangle \, dt
\]
and the functional 
\[
\mathcal{H}_C: H^1_0(\mathbb{T}, \mathbb{R}^{2n}) \to \mathbb{R}, \quad \mathcal{H}_C(x) = \int\limits_\mathbb{T} \HC^*(-J_0 \dot{x}(t)) \, dt.
\]
Using these, Clarke's dual functional on the positive open cone $\mathcal{A}^{-1}(0, +\infty)$ is defined as:
\[
\widetilde{\Psi}_C: \mathcal{A}^{-1}(0, +\infty) \to \mathbb{R}, \quad \widetilde{\Psi}_C(x) \defeq \frac{\mathcal{H}_C(x)}{\mathcal{A}(x)}.
\]

The 0-homogeneity of $\widetilde{\Psi}_C$ naturally follows from the positive 2-homogeneity of $\mathcal{A}$ and $\mathcal{H}$. The regularity of $\widetilde{\Psi}_C$ matches that of $\mathcal{H}_C$. Thus, $\widetilde{\Psi}_C$ is not generally $C^1$, but it is locally Lipschitz continuous and possesses a notion of weak critical points (see \cite{Cla83, AO14}). Given that both $\mathcal{H}_C$ and $\mathcal{A}$ are $S^1$-invariant and $\widetilde{\Psi}_C$ is 0-homogeneous, it follows that $\widetilde{\Psi}_C$ is invariant under $\mathbb{C}_* = \mathbb{C} \backslash \{0\}$ action. Consequently, if $N$ is an $S^1$-invariant hypersurface of $\mathcal{A}^{-1}(0, +\infty)$ that is transversal to the radial direction and $\mathbb{R}_+ N = \mathcal{A}^{-1}(0, +\infty)$, then the inclusion
\[
\{x \in N \mid \widetilde{\Psi}_C(x) < L\} \to \{\widetilde{\Psi}_C(x) < L\}
\]
is an $S^1$-equivariant homotopy equivalence for every $L \in \mathbb{R}$. For this reason, the Clarke's dual functional in this form is typically considered as a restriction of $\widetilde{\Psi}_C$ to $\mathcal{H}^{-1}(1)$ or $\mathcal{A}^{-1}(1)$. Here, we define 
\[
\Lambda \defeq \mathcal{A}^{-1}(1)
\]
and
\[
\PsiC: \Lambda \to \mathbb{R}, \quad \PsiC \defeq \widetilde{\Psi}_C|_\Lambda=\mathcal{H}_C|_{\Lambda}.
\]
Therefore, it holds that
\[
\PsiC(x) = \int\limits_\mathbb{T} \HdC(-J_0 \dot{x}(t)) \, dt.
\]

The functional $\PsiC$ is invariant under the $S^1$-action on $\Lambda$, and therefore, we will sometimes refer to critical orbits, with the usual notation $S_x \defeq S^1 \cdot x$, where $x \in \mathrm{crit}(\PsiC)$.

\textbf{Strongly convex domain:} We say that $C \subset \mathbb{R}^{2n}$ is a strongly convex domain if it is the closure of a bounded open convex set with a smooth boundary that has positive sectional curvature everywhere. 

If the interior of $C$ contains the origin, the strong convexity of $C$ is equivalent to the condition that the second derivative of $\HC$ is positive definite at every point $x \neq 0$. Since $\HC$ is twice differentiable on $\mathbb{R}^{2n} \setminus \{0\}$ and its second derivative is positively 0-homogeneous, we conclude that for strongly convex domains,
\[
\underline{h}_C I \leq \nabla^2 \HC(x) \leq \overline{h}_C I, \quad x \neq 0,
\]
where $\underline{h}_C > 0$ and $\overline{h}_C > 0$ are positive constants. Thus, $\HC$ is quadratically convex in this setting, except for the lack of twice differentiability at the origin. Moreover, the same conclusion holds for $\HdC$, as it is the Fenchel conjugate of $\HC$.

Let $C$ be a strongly convex domain whose interior contains the origin. Then $(C, \lambda_0)$ is a Liouville domain, where $\lambda_0$ is the standard Liouville form. The Reeb vector field $R$ on $\partial C$ is characterized by the equations
\[
i_R \lambda_0 = 1, \quad i_R \omega_0 = 0.
\]
We define the spectrum $\sigma(\partial C)$ as the set of periods of closed Reeb orbits (i.e., the action of the Reeb orbits). For a more general definition of the spectrum, see the introduction or \cite{HZ94, AO14}.

Alternatively, the orbits can be understood as solutions of the equation
\[
y: \mathbb{R}/T\mathbb{Z} \to \partial C, \quad \dot{y}(t) = J_0 \nabla \HC(y(t)), \quad T > 0,
\]
since the Reeb vector field has the form \[R=J_0\nabla \HC.\]
Reparametrizing these orbits, we can represent the set of Reeb orbits as
\[
\mathcal{R}(\partial C)\defeq \{y: \mathbb{T} \to \partial C \mid \dot{y} = T J_0 \nabla \HC(y(t)), \quad T > 0\}.
\]
The spectrum of $\partial C$, denoted by $\sigma(\partial C)$, is given by
\[
\sigma(\partial C) = \{\mathcal{A}(y) \mid y \in \mathcal{R}(\partial C)\},
\]
since the action $\mathcal{A}$ is independent of parametrization up to a sign.

Since $\HdC$ is quadratically convex, $\PsiC \in C^{1,1}$ when $C$ is a strongly convex domain. The correspondence between Reeb orbits and the critical points of $\PsiC$ is described by the following lemma.

\begin{lemma}\label{bijecthom}
    A point $x \in \crit(\PsiC)$ if and only if
    \[
    \nabla \HdC(-J_0 \dot{x}(t)) = \PsiC(x) x(t) + \beta,
    \]
    where $\beta = \int_\mathbb{T} \nabla \HdC(-J_0 \dot{x}(t)) \, dt$. Moreover, there is a bijection between closed Reeb orbits and $\crit(\PsiC)$ given by the map
    \[
    \mathcal{P}: \mathcal{R}(\partial C) \to \crit(\PsiC), \quad \mathcal{P}(y) = \frac{1}{\sqrt{\mathcal{A}(y)}} \pi(y),
    \]
    where $\pi(y) = y - \int_\mathbb{T} y(t) \, dt$. The inverse map is given by
    \[
    \mathcal{P}^{-1}: \crit(\PsiC) \to \mathcal{R}(\partial C), \quad \mathcal{P}^{-1}(x) = \frac{1}{\sqrt{\PsiC(x)}} (\PsiC(x) x(t) + \beta).
    \]
    Additionally, it holds that $\mathcal{A}(y) = \PsiC(\mathcal{P}(y))$.
\end{lemma}

The proof of this lemma can be found in \cite{Cla79,HZ94,AO08}.

\textbf{Non-degenerate domain:} We say that $C$ is non-degenerate if every closed Reeb orbit $y: \mathbb{R}/T\mathbb{Z} \to \partial C$ is transversally non-degenerate. This means that the differential of the return map $d\varphi^T_{R}(y(0))$ of the flow of the Reeb vector field $R$, restricted to the contact distribution $\ker(\lambda_0)$, does not have $1$ as an eigenvalue.

\begin{remark}\label{tansvnondeg}
This assumption on $C$ is not particularly strong in the sense that every bounded convex domain can be Hausdorff-approximated by convex domains of this type. In the case where $C$ is non-degenerate, $\sigma(\partial C)$ is discrete. Furthermore, for every period $T \in \sigma(\partial C)$, there exists a finite number of $S^1$-families of closed Reeb orbits with period $T$.
\end{remark}

From now on, we will assume that $C \subset \mathbb{R}^{2n}$ is a non-degenerate, strongly convex domain whose interior contains the origin.

\textbf{Admissible$^*$ approximating functions:} Let $\varepsilon = \frac{1}{2} \min \sigma(\partial C)$. We define the set $\mathcal{F}^*_{\operatorname{lin}}(C)$ as the collection of smooth functions $\varphi: \mathbb{R}_{\geq 0} \to \mathbb{R}$ that satisfy the following properties:

\begin{enumerate}
    \item There exists $s_\varepsilon > 0$ such that $\varphi(s) = \delta s - \zeta_\varepsilon$ for $s \in [0, s_\varepsilon]$, where $0 < \zeta_\varepsilon < \varepsilon$ and $0 < \delta < \frac{\varepsilon}{4}$.
    \item There exists $s_\eta > s_\varepsilon$ such that $\varphi(s) = \eta s - \zeta_\eta$ for $s \in [s_\eta, +\infty)$, where $\eta \notin \sigma(\partial C)$ and $\eta > \min \sigma(\partial C)$.
    \item $\varphi''(s) > 0$ for $s \in (s_\varepsilon, s_\eta)$.
    \item For the unique $s_{\min} \in (s_\varepsilon, s_\eta)$, where $\varphi'(s_{\min}) = T_{\min} = \min \sigma(\partial C)^{<\eta} = \min \sigma(\partial C)$, it holds that
    \[
    \varphi'(s_{\min}) s_{\min} - \varphi(s_{\min}) > \varepsilon.
    \]
    \item For the unique $s_{\max} \in (s_\varepsilon, s_\eta)$, where $\varphi'(s_{\max}) = T_{\max} = \max \sigma(\partial C)^{<\eta}$, it holds that
    \[
    \varphi'(s_{\max}) s_{\max} - \varphi(s_{\max}) < \eta.
    \]
\end{enumerate}

\begin{remark}\label{critical values}
   Notice that, due to conditions (1), (2), and (3), $\varphi$ is a strictly increasing convex function and is strictly convex on $(s_\varepsilon, s_\eta)$. The map $\varphi': (s_\varepsilon, s_\eta) \to (\delta, \eta)$ is an increasing bijection. Since $\eta > \min \sigma(\partial C)$, there exist unique values $s_{\min}$ and $s_{\max}$ determined by conditions (4) and (5). Moreover, the function $\varphi'(s) s - \varphi(s)$ is non-decreasing, so
\begin{equation}\label{boundsphiepsiloneta}
    \varepsilon < \varphi'(s) s - \varphi(s) < \eta \quad \text{for } s \in [s_{\min}, s_{\max}].
\end{equation}
\end{remark}

Clearly, the set $\mathcal{F}^*_{\operatorname{lin}}(C)$ is non-empty for any $\eta \notin \sigma(\partial C)$, $\eta > \sigma(\partial C)$ we choose. We denote $\varphi_\eta \in \mathcal{F}^*_{\operatorname{lin}}(C)$ to emphasize that the slope of $\varphi$ is $\eta$ at infinity and write $\Heta$ when $\Heta = \varphi_\eta \circ \HC$ for some $\varphi_\eta \in \mathcal{F}^*_{\operatorname{lin}}(C)$.

First, we will outline some essential properties of the Hamiltonian $\Heta$ and its Fenchel conjugate $\Hdeta$, as well as the regularity properties of $\PhiHeta$ and $\PsiHeta$. Then we will state the main proposition of this chapter.

\begin{lemma}\label{ham}
    Let $\varphi_\eta \in \mathcal{F}^*_{\operatorname{lin}}(C)$ and $H_\eta \defeq \varphi_\eta \circ \HC$, with $H_\eta^*$ denoting its Fenchel conjugate. Then, the following holds:
    \begin{enumerate}
        \item $H_\eta$ and hence $H^*_\eta$ are quadratically convex, except for the lack of twice-differentiability at the origin.
        \item For all $x \in \mathbb{R}^{2n}$, it holds that
        \[
        \nabla H_\eta^*(x) = k \nabla \HdC(x),
        \]
        where $k$ is the unique solution of the equation $\varphi'_\eta(k^2 \HdC(x)) k = 1$. This implies that $k \in \left[\frac{1}{\eta}, \frac{1}{\delta}\right]$.
        \item $\PhiHeta$ and $\PsiHeta$ are continuously differentiable $S^1$-invariant functionals.
    \end{enumerate}
\end{lemma}

\begin{proof}

\textbf{Claim }(1): Since $\varphi_\eta \in \mathcal{F}^*_{\operatorname{lin}}(C)$, for small $r > 0$ and $x \in B(0; r) \setminus \{0\}$, it holds that $\nabla^2 H_\eta(x) = \delta \nabla^2 \HC(x)$. On the other hand, for large enough $R > 0$, for $x \in \overline{B(0, R)}^c$, it holds that $\nabla^2 H_\eta(x) = \eta \nabla^2 \HC(x)$. Thus, we have
\begin{equation}\label{quadcon}
    \underline{h} I \leq \nabla^2 H_\eta(x) \leq \overline{h} I, \quad x \in B(0; r) \setminus \{0\} \cup \overline{B(0, R)}^c
\end{equation}
for some positive constants $\underline{h}, \overline{h} > 0$. Since $H_\eta$ is twice continuously differentiable on $\mathbb{R}^{2n} \setminus \{0\}$ with $\nabla^2 H_\eta(x)$ being positive definite and $\overline{B(0, R)} \setminus B(0, r)$ is compact, \eqref{quadcon} implies that there exist positive numbers $\underline{h}_\eta > 0$ and $\overline{h}_\eta > 0$ such that
\[
\underline{h}_\eta I \leq \nabla^2 H_\eta(x) \leq \overline{h}_\eta I, \quad x \in \mathbb{R}^{2n} \setminus \{0\}.
\]
Therefore, $H_\eta$ is quadratically convex except for the lack of twice-differentiability at the origin. Since $\nabla^2 H^*_\eta(\nabla H_\eta(x)) = \nabla^2 H_\eta(x)^{-1}$, the same holds for $H^*_\eta$.

\textbf{Claim }(2): From the properties of $\varphi_\eta$, we have that for every $x \in \mathbb{R}^{2n}$, the function
\[ k \mapsto \varphi'_\eta(k^2 \HdC(x)) k \]
is strictly increasing, with value $0$ at $k = 0$ and $+\infty$ as $k \to +\infty$. Hence, there is a unique solution to the equation
\begin{equation}\label{rightinversesolution}
    \varphi'_\eta(k^2 \HdC(x)) k = 1.
\end{equation}
Therefore, the map  
\begin{equation}\label{rightinverse}
    x \mapsto k(x) \nabla \HdC(x), \quad x \in \mathbb{R}^{2n},
\end{equation}  
where $k(x)$ is the solution of \eqref{rightinversesolution}, is well-defined. We will show that the map \eqref{rightinverse} is the right inverse of $\nabla H_\eta$. Since $\nabla H_\eta$ is bijective and $\nabla H_\eta^* = \nabla H_\eta^{-1}$, this will imply that
\begin{equation}\label{rightinversefinal}
    \nabla H_\eta^*(x) = k(x) \nabla \HdC(x), \quad x \in \mathbb{R}^{2n}.
\end{equation}

By the definition of $H_\eta$, we have $\nabla H_\eta(x) = \varphi_\eta'(\HC(x)) \nabla \HC(x)$. Thus,
\begin{equation}\label{rightinversecomputation}
\begin{array}{lcl}
 & \nabla H_\eta(k \nabla \HdC(x)) = \varphi_\eta'(\HC(k \nabla \HdC(x))) \nabla \HC(k \nabla \HdC(x)) \\
 & = \varphi_\eta'(k^2 \HC(\nabla \HdC(x))) k x = \varphi_\eta'(k^2 \HdC(x)) k x = x.
\end{array}
\end{equation}

The second equality follows from the $2$-homogeneity of $\HC$, the $1$-homogeneity of $\nabla \HC$, and the fact that $\nabla \HC \circ \nabla \HdC = \text{id}$. The third equality is a consequence of the identity
\begin{equation}\label{rightinversethird}
    \HC(\nabla \HdC(x)) = \HdC(x).
\end{equation}
Indeed, we have
\[\HC(\nabla \HdC(x)) = \langle \nabla \HdC(x), x \rangle - \HdC(x)\]
from the Fenchel conjugate formula, and
\[\langle \nabla \HdC(x), x \rangle = 2 \HdC(x)\]
from the $2$-homogeneity of $\HdC$, which together implies the identity \eqref{rightinversethird}. The last equality in \eqref{rightinversecomputation} holds due to \eqref{rightinversesolution}. Hence, the map \eqref{rightinverse} is the right inverse, and by the previous discussion, we conclude that \eqref{rightinversefinal} holds. Additionally, since $\delta \leq \varphi_\eta' \leq \eta$, it follows that $k \in \left[\frac{1}{\eta}, \frac{1}{\delta}\right]$.

\textbf{Claim} (3): From claim (1), it follows that $\nabla H_\eta$ must be Lipschitz continuous. Therefore, $\nabla H_\eta$ has linear growth, and $H_\eta$ has quadratic growth. By the Dominated Convergence Theorem and Sobolev embeddings $H^{1/2}(\mathbb{T}, \mathbb{R}^{2n}) \to L^p(\mathbb{T}, \mathbb{R}^{2n})$ where $p \geq 1$, we conclude that $\PhiHeta$ is $C^1$. Similarly, $\nabla \Hdeta$ has linear growth. Again, by the Dominated Convergence Theorem, we conclude that $\PsiHeta$ is $C^1$. $S^1$-invariance follows from the fact that $\Heta$ and therefore $\Hdeta$ are independent of time.

\end{proof}

Let $H_{\eta} = \varphi_\eta \circ \HC$. We denote by $T_\eta \in \sigma(\partial C)$ the smallest element in the spectrum greater than $\eta$. If $H_{\eta} = \varphi_\eta \circ \HC \leq \varphi_\nu \circ \HC = H_{\nu}$, then $\eta \leq \nu$ and $T_\eta \leq T_\nu$. For $L_\eta \in (\eta, T_\eta)$ and $L_\nu \in (\nu, T_\nu)$ where $L_\eta \leq L_\nu$, we use the notation $i^H_{L_\nu, L_\eta}$ for the inclusion
\begin{equation}\label{incH}
    i^{H}_{L_\nu, L_\eta}: (\{\Psi_{H_\eta^*} < L_\eta\}, \{\Psi_{H_\eta^*} < \varepsilon\}) \to (\{\Psi_{H_{\nu}^*} < L_\nu\}, \{\Psi_{H_\nu^*} < \varepsilon\})
\end{equation}
and $i^C_{L_\nu, L_\eta}$ for the inclusion
\[
i^C_{L_\nu, L_\eta}: \{\PsiC < L_\eta\} \to \{\PsiC < L_\nu\}.
\]

If $H_\eta \leq H_\nu$, then $\Psi_{H_\nu^*} \leq \Psi_{H_\eta^*}$, which implies that
\[
(\{\Psi_{H_\eta^*} < L_\eta\}, \{\Psi_{H_\eta^*} < \varepsilon\}) \subseteq (\{\Psi_{H_{\nu}^*} < L_\eta\}, \{\Psi_{H_\nu^*} < \varepsilon\}) \subseteq (\{\Psi_{H_{\nu}^*} < L_\nu\}, \{\Psi_{H_\nu^*} < \varepsilon\}).
\]

This justifies \eqref{incH}.
 
\begin{proposition}\label{isomorphismclarkeduality}
    Let $C$ be a non-degenerate strongly convex domain whose interior contains the origin, and let $A$ be an arbitrary abelian group. Let $\varphi_\eta \in \mathcal{F}^*_{\operatorname{lin}}(C)$, with $T_\eta$ being a minimal element of $\sigma(\partial C)$ greater than $\eta$, and let $H_\eta = \varphi_\eta \circ \HC$. The following holds: 
    \begin{enumerate}
        \item[$(1.1)$] For every $L \in (\eta, T_\eta)$ there exists an isomorphism 
        \[
        D^{H^*_\eta}_L: \H_*(\{\PsiC < L\}; A) \to \H_{*+1}(\{\Psi_{H_\eta^*} < L\}, \{\Psi_{H_\eta^*} < \varepsilon\}; A).
        \]

        \item[$(1.2)$] This isomorphism is natural. Let $H_{\eta} \leq H_{\nu}$. Then for every $L_\eta \in (\eta, T_\eta)$ and $L_\nu \in (\nu, T_\nu)$ such that $L_\eta \leq L_\nu$, the following diagram commutes:
        \[
        \begin{tikzcd}
         \H_*(\{\PsiC < L_\nu\}; A) \arrow[r, "D^{H^*_\nu}_{L_\nu}"] & \H_{*+1}(\{\Psi_{H_{\nu}^*} < L_\nu\}, \{\Psi_{H_{\nu}^*} < \varepsilon\}; A) \\
         \H_*(\{\PsiC < L_\eta\}; A) \arrow[u, "(i^C_{L_\nu, L_\eta})_*"] \arrow[r, "D^{H^*_\eta}_{L_\eta}"] & \H_{*+1}(\{\Psi_{H_{\eta}^*} < L_\eta\}, \{\Psi_{H_{\eta}^*} < \varepsilon\}; A) \arrow[u, "(i^H_{L_\nu, L_\eta})_*"]
        \end{tikzcd}
        \]
    \end{enumerate}
\end{proposition}

\begin{remark}\label{isomorphismclarkedualityS1}
    The same statement holds if one replaces the singular homology with the $S^1$-equivariant homology. Moreover, the proof of the proposition mentioned above is identical for both the singular and the $S^1$-equivariant case. Notice that all sets mentioned in the above proposition are indeed $S^1$-invariant due to Claim (3) of Lemma \ref{ham} and $S^1$-invariance of $\PsiC$.
    \end{remark}

Note that this proposition excludes the case $\eta < \min \sigma(\partial C)$ due to the definition of $\mathcal{F}^*_{\operatorname{lin}}(C)$. While the proposition would still hold in this case, it would require separate treatment. Since 
\[
SH^{+<L}_*(C) = 0, \quad L < \min \sigma(\partial C),
\]
and
\[
H_*(\{\PsiC < L\}) = 0, \quad L < \min \sigma(\partial C),
\]
with the same applying to the $S^1$-equivariant case, there is already a trivial canonical isomorphism between these homologies for $L < \min \sigma(\partial C)$. Thus, introducing such Hamiltonians is unnecessary for proving theorems \hyperlink{TheoremC1}{C.1} and \hyperlink{TheoremC2}{C.2}. As noted earlier, the proofs are identical in both singular and $S^1$-equivariant categories. Hence, we use "homology" to refer to both singular and $S^1$-equivariant homology. All constructions will be $S^1$-invariant or $S^1$-equivariant as needed, and for simplicity, we use singular homology notation. Additionally, we omit specifying coefficients in the lemmas and proofs, as the results hold for any abelian group $A$. Finally, all sets in the following statements are non-empty. This will be evident from the context, and explicitly stating it would only lengthen the text unnecessarily.

\begin{remark}\label{boundedpalaissmale}
  The functional $\PsiC$ satisfies a version of the Palais-Smale condition. Namely, for every sequence $x_n \in \Lambda$ such that  
\[
\|d\PsiC(x_n)\|_{(H^1_0)^*} \to 0, \quad \PsiC(x_n) \leq L,
\]  
there exists a convergent subsequence. Here, $L \geq \min \PsiC$ is arbitrary.

Let $V$ be a 1-homogeneous extension of $\nabla \PsiC$ to the positive cone $\{\mathcal{A} > 0\}$, defined by  
\[
V(x) = \sqrt{\mathcal{A}(x)} \nabla \PsiC\left(\frac{x}{\sqrt{\mathcal{A}(x)}}\right), \quad x \in \{\mathcal{A} > 0\}.
\]  
Since $\mathbb{R}_+ \crit(\PsiC) = \crit(\widetilde{\Psi}_C)$, the vector field $V$ is a pseudo-gradient for $\widetilde{\Psi}_C$ that is tangent to the level sets of $\mathcal{A}$. Moreover, since $V$ is a 1-homogeneous extension of $\nabla \PsiC$, it satisfies the Palais-Smale condition on each level set $\Lambda_a = \{x \in H^1_0(\mathbb{T}, \mathbb{R}^{2n}) \mid \mathcal{A}(x) = a\}$ where $a>0$ since $\nabla \PsiC$ satisfies the Palais-Smale condition.

\end{remark}

\begin{remark}\label{stable}
Let $\phi^t: \{\mathcal{A} > 0\} \to \{\mathcal{A} > 0\}$ be the flow of $-V$. This flow is complete since $V$ is locally Lipschitz and has linear growth on each level set of $\mathcal{A}$. Moreover, $\phi^t$ preserves the level sets of $\mathcal{A}$ since $V$ is tangent to these level sets.

If we assume that $C$ is a non-degenerate strongly convex domain, then according to Lemma \ref{bijecthom} and Remark \ref{tansvnondeg}, the critical orbits $S_x = S^1 \cdot x$ for $x \in \crit(\PsiC)$ are isolated. Consequently, all critical orbits of $\widetilde{\Psi}_C$ on the fixed level set of $\mathcal{A}$ are also isolated. Moreover, since $\widetilde{\Psi}_C$ is bounded from below, and $V$ satisfies the Palais-Smale condition (see Remark \ref{boundedpalaissmale}) on each level set of $\mathcal{A}$, for each $x \in \{\mathcal{A} > 0\}$, there exists a unique critical orbit $S_{x^-} \subset \Lambda_{\mathcal{A}(x)}\cap \crit(\widetilde{\Psi
}_C)$ such that, for every $S^1$-invariant neighborhood $U$ of $S_{x^-}$, there exists $t_0 > 0$ such that for all $t \geq t_0$, $\phi^t(x) \in U$.

\end{remark}

\begin{remark}\label{decPsiHeta}
Notice that the function $\PsiHeta$ decreases along the flow of $\phi^t$ if $x \in \{\mathcal{A} > 0\}$ is not a stationary point of $\phi^t$. Indeed, we have
\[
d\PsiHeta(x)[-V(x)] = d\mathcal{A}(x)[V(x)] - \int_0^1 \langle \nabla H_\eta^*(-J_0 \dot{x}(t)), -J_0 \frac{d}{dt}V(x)(t) \rangle \, dt.
\]

The first term is zero because $V(x) \in T \Lambda_{\mathcal{A}(x)}$. Therefore, we have
\begin{align}
    \nonumber &d\PsiHeta(x)[-V(x)] = -\int_0^1 \Bigl\langle \nabla H_\eta^*(-J_0 \dot{x}(t)), -J_0 \frac{d}{dt}V(x)(t) \Bigr\rangle \, dt \\
    \nonumber &= -\int_0^1 k(t) \Bigl\langle \nabla \HdC(-J_0 \dot{x}(t)), -J_0 \frac{d}{dt}V(x)(t) \Bigr\rangle \, dt \leq -\frac{1}{\eta} \int_0^1 \Bigl\langle \nabla \HdC(-J_0 \dot{x}(t)), -J_0 \frac{d}{dt}V(x)(t) \Bigr\rangle \, dt.
\end{align}

The equality and inequality follow from claim (2) of the previous lemma. On the other hand, by the 1-homogeneity of $V$, and the fact that $V$ is the gradient of $\PsiC$ on $\Lambda$, we have

\begin{align*}
     &\int_0^1 \left\langle \nabla \HdC(-J_0 \dot{x}(t)), -J_0 \frac{d}{dt}V(x)(t) \right\rangle \, dt\\
     &=\mathcal{A}(x) \int_0^1 \left\langle \nabla \HdC\left(\frac{-J_0 \dot{x}(t)}{\sqrt{\mathcal{A}(x)}}\right), -J_0 \frac{d}{dt}V\left(\frac{x}{\sqrt{\mathcal{A}(x)}}\right)(t) \right\rangle \, dt \\
     &= \mathcal{A}(x) d\PsiC\left(\frac{x}{\sqrt{\mathcal{A}(x)}}\right)\left[\nabla\PsiC\left(\frac{x}{\sqrt{\mathcal{A}(x)}}\right)\right] = \mathcal{A}(x) \left\|\nabla\PsiC\left(\frac{x}{\sqrt{\mathcal{A}(x)}}\right)\right\|^2_{H^1_0} = \|V(x)\|^2_{H^1_0}.
\end{align*}
Therefore, we have
\[
d\PsiHeta(x)[-V(x)] = -\int_0^1 \langle \nabla H_\eta^*(-J_0 \dot{x}(t)), -J_0 \frac{d}{dt}V(x)(t) \rangle \, dt \leq -\frac{1}{\eta} \|V(x)\|^2_
{H^1}< 0
\]
when $x$ is not a stationary point of $\phi^t$.
\end{remark}

\begin{lemma}\label{Psiheta}
    Let $H_\eta \defeq \varphi_\eta \circ \HC$ where $\varphi_\eta \in F^*(C)$. The following statements hold:

    \begin{enumerate}
        \item Non-trivial critical orbits of $\Psi_{H_\eta}$ are in one-to-one correspondence with critical orbits of $\Psi_C$ that have values smaller than $\eta$. This relation is given by the bijection 
        \[
        \crit(\Psi_{H_\eta}) \setminus \{0\} \to \crit(\Psi_C)^{<\eta}, \quad x \mapsto \frac{x}{\sqrt{\mathcal{A}(x)}}.
        \]
        Additionally, for all $x \in \crit(\Psi_{H_\eta}) \setminus \{0\}$, it holds that 
        \[
        \varepsilon < \Psi_{H_\eta^*}(x) < \eta
        \]
        and $\Psi_{H_\eta}(0) < \varepsilon$.
        
        \item There exists $r > 0$ such that 
        \[
        d\Psi_{H_\eta}(x)[x] > 0, \quad x \in (0, r] \Lambda
        \]
        and 
        \[
        \Psi_{H_\eta}(x) < \varepsilon, \quad x \in (0, r] \{\Psi_C < T_\eta\}.
        \]
        
        \item $\Psi_{H_\eta}$ is radially increasing on $\mathbb{R}_+ \{\Psi_C > \eta\}$ and radially non-decreasing on $\{\mathcal{A} \leq 0\}$.
        
        \item For every $x \in \mathbb{R}_+ \{\Psi_C < T_\eta\}$, the following claims hold:
        \begin{enumerate}
            \item[$($4$.a)$] There exists $\lambda' \in \mathbb{R}_+$ such that $\Psi_{H_\eta^*}(\lambda' x) > \varepsilon$. In particular, if the action of $x$ is strictly larger than the action of all non-trivial critical points of $\PsiHeta$, then $\lambda' \in (0,1)$.
            \item[$($4$.b)$] The derivative of the function $\lambda \mapsto \Psi_{H_\eta}(\lambda x)$, where $\lambda \in \mathbb{R}_+$, can exhibit three types of behavior:
            \begin{itemize}
                \item It is positive for all $\lambda \in \mathbb{R}_+$.
                \item There exists $\lambda_{\max} \in \mathbb{R}_+$ such that it is positive on $(0, \lambda_{\max})$ and zero on $[\lambda_{\max}, +\infty)$. This can only occur if $x \in \mathbb{R}_+ \{\Psi_C = \eta\}$.
                \item There exists $\lambda_{\max} \in \mathbb{R}_+$ such that it is positive on $(0, \lambda_{\max})$ and negative on $(\lambda_{\max}, +\infty)$. Moreover, in this case, the derivative of the function $\lambda \mapsto \Psi_{H_\eta}(\lambda x)$ is decreasing on $(\lambda_{\max}, +\infty)$.
            \end{itemize}
            
            \item[$($4$.c)$] In the special case where $x \in \crit(\Psi_{H_\eta}) \setminus \{0\}$, the derivative of $\lambda \mapsto \Psi_{H_\eta}(\lambda x)$ is positive on $(0, 1)$ and negative on $(1, +\infty)$.
        \end{enumerate}
        
        \item For every non-trivial critical orbit $S_x$ of $\Psi_{H_\eta}$, there exists a $S^1$-invariant neighborhood $U_{S_x} \subset \Lambda_{\mathcal{A}(x)}$ of $S_x$ such that $\mathbb{R}_+ U_{S_x} \subset \{\Psi_{H_\eta} < \eta\}$ and for some $R_{S_x} > 0$, it holds that $[R_{S_x}, +\infty) U_{S_x} \subset \{\Psi_{H_\eta} < 0\}$. 
    
    \end{enumerate}
\end{lemma}

\begin{proof}
\

   \textbf{Claim} (1): Given that $H_\eta = \varphi_\eta \circ \HC$, we see that $0$ is a critical point and that non-trivial critical points of $\PhiHeta$ take the form 
\begin{equation}\label{hamiltonianorbit}
\dot{y}(t) = J_0 \varphi_\eta'(\HC(y(t))) \nabla \HC(y(t)),
\end{equation}
where $\HC(y(t))=s_y$ is constant and $\varphi_\eta'(s_y) \in \sigma(\partial C)$. Since $\varphi_\eta'$ is a non-decreasing function with $\varphi_\eta': (s_\varepsilon, s_\eta) \to (\delta, \eta)$ being a increasing bijection (see Remark \ref{critical values}), it follows that $\crit(\PhiHeta) \setminus \{0\}$ and $\mathcal{R}(\partial C)^{<\eta}$ are in bijective correspondence via 
\[
R: \crit(\PhiHeta) \setminus \{0\} \to \mathcal{R}(\partial C)^{<\eta}, \quad R(y) = \frac{1}{\sqrt{\int_\mathbb{T} \HC(y(t)) \, dt}} y.
\]

Using the $2$-homogeneity of $\mathcal{A}$, we get that $\sqrt{\mathcal{A}(R(y))} = \frac{1}{\sqrt{\int_\mathbb{T} \HC(y(t)) \, dt}} \sqrt{\mathcal{A}(y)}$. Therefore, by Lemma \ref{bijecthom}, the map 
\[
\mathcal{P} \circ R: \crit(\PhiHeta) \setminus \{0\} \to \crit(\PsiC)^{<\eta}, \quad (\mathcal{P} \circ R)(y) = \frac{1}{\sqrt{\mathcal{A}(y)}} \pi (y)
\]
is a bijection. Additionally, from \cite[Lemma 5.1.]{AK22}, we know that 
\[
\pi_{\crit} \defeq \pi|_{\crit(\PhiHeta)}: \crit(\PhiHeta) \to \crit(\PsiHeta)
\]
is a bijection with $\pi(0) = 0$. Hence, 
\[
\pi_{\crit,+} \defeq \pi|_{\crit(\PhiHeta) \setminus \{0\}}: \crit(\PhiHeta) \setminus \{0\} \to \crit(\PsiHeta) \setminus \{0\}
\]
is also a bijection. The composition map 
\[
\mathcal{P}_\eta \circ R \circ \pi_{\crit,+}^{-1}: \crit(\PsiHeta) \setminus \{0\} \to \crit(\PsiC)^{<\eta}, \quad (\mathcal{P}_\eta \circ R \circ \pi_{\crit,+}^{-1})(x) = \frac{1}{\sqrt{\mathcal{A}(x)}} x
\]
is therefore a bijection as well. This completes the proof of the first part of (1). 

Next, consider a non-zero critical point $y$ of $\Phi_H$ with $\HC(y(t)) = s_y$. From \eqref{hamiltonianorbit} and the expression 
\[
\PhiHeta(y) = \frac{1}{2} \int_\mathbb{T} \langle\dot{y}(t) , J_0 y(t) \rangle \, dt - \int_\mathbb{T} \varphi_\eta(\HC(y(t))) \, dt,
\]
we have 
\[
\Phi_H(y) = \varphi_{\eta}'(s_y) s_y - \varphi_\eta(s_y).
\]
Since $\varphi_\eta': (s_\varepsilon, s_\eta) \to (\delta, \eta)$ is an increasing bijection and $\varphi_\eta'(s_y) \in \sigma(\partial C)$, it follows that $s_y \in [s_{\min}, s_{\max}]$ (for the definitions of $s_{\min}$ and $s_{\max}$, see the definition of $\mathcal{F}^*_{\operatorname{lin}}(C)$). This implies, by \eqref{boundsphiepsiloneta}, that 
\begin{equation}\label{epsiloneta}
\varepsilon < \PhiHeta(y)=\varphi_{\eta}'(s_y) s_y - \varphi_\eta(s_y) < \eta.    
\end{equation}
According to \cite[Lemma 5.1.]{AK22}, we have $\PhiHeta(y) = \PsiHeta(\pi(y))$ for $y$ being a critical point of $\PhiHeta$. Combining this with the fact that $y \neq 0$ was arbitrary and that $\pi_{\crit,+}$ is a bijection, \eqref{epsiloneta} implies that for every $x \in \crit(\Psi_{H^*_\eta}) \setminus \{0\}$, it holds 
\[
\varepsilon < \PsiHeta(x) < \eta.
\]

Moreover, $\PsiHeta(0) = \PhiHeta(0) = -\varphi_\eta(0) = \zeta_\varepsilon < \varepsilon$. This completes the proof of (1).

\textbf{Claim }(2): To prove the first part of this claim, we will use the fact that the minimal value of the functional 
\[
\Lambda \to \mathbb{R}, \quad x \to \int_0^1 \sqrt{\HdC(-J_0 \dot{x}(t))} \, dt
\]
is $\sqrt{T_{\min}}$, where $T_{\min} = \min \sigma(\partial C)$. If we view $S^1$ as $\mathbb{T} = \mathbb{R} / \mathbb{Z}$ rather than $\mathbb{R} / 2\pi \mathbb{Z}$, then \cite[Proposition 2.1]{AO08} tells us that  
\[
\sqrt{T_{\min}} = c_{EHZ}(C)^{\frac{1}{2}} = \frac{1}{2} \min_{x \in \Lambda} \int_\mathbb{T} h_C(-J_0 \dot{x}(t)) \, dt
\]
where $h_C$ is the support function of $C$ defined by 
\[
h_C(x) = \sup \{\langle u, x \rangle \mid u \in K \}.
\]

On the other hand, $\HdC(x) = \frac{1}{4} h_C^2(x)$. Thus, 
\[
\sqrt{\HdC(x)} = \frac{1}{2} h_C(x).
\]
This implies that 
\begin{equation}\label{squareroot}
\int_0^1 \sqrt{\HdC(-J_0 \dot{x}(t))} \, dt \geq \sqrt{T_{\min}}, \quad x \in \Lambda.
\end{equation}
Let $s_0 > 0$ be such that $\varphieta'(s_0^2) < \frac{T_{\min}}{8}$. Such an $s_0$ always exists since $\varphi_\eta'(s) = \delta$ near $0$ and $\delta < \frac{\varepsilon}{4}$ where $\varepsilon = \frac{1}{2} \min \sigma(\partial C)=\frac{T_{\min}}{2}$. 

We claim that 
\begin{equation}\label{increasing}
\frac{d \PsiHeta}{d \lambda}(\lambda x) > 0, \quad x \in \Lambda, \quad 0 < \lambda < s_0 \frac{\sqrt{T_{\min}}}{2}
\end{equation}
holds. For any $t$ where $\dot{x}(t)$ is defined and any $\lambda > 0$, by the claim (2) of Lemma \ref{ham}, we have 
\[
\nabla H^*_\eta(-\lambda J_0 \dot{x}(t)) = k(t, \lambda) \nabla \HdC(-\lambda J_0 \dot{x}(t))
\]
where $k(t, \lambda)$ satisfies 

\begin{equation}\label{ktlambda}
\varphi_\eta'( k^2 \HdC(-\lambda J_0 \dot{x}(t))) k=\varphi_\eta'(\lambda^2 k^2 \HdC(-J_0 \dot{x}(t))) k = 1.    
\end{equation}

Let 
\begin{equation}\label{setBclaim2}
B \defeq \left\{ t \in \mathbb{T} \mid \sqrt{\HdC(-J_0 \dot{x}(t))} \geq \frac{\sqrt{T_{\min}}}{4} \right\}.
\end{equation}
 We have that 
\begin{equation}\label{kinequalitysetB}
k(t, \lambda) \geq \frac{2}{\sqrt{T_{\min}} \sqrt{\HdC(-J_0 \dot{x}(t))}}, \quad 0 < \lambda < s_0 \frac{\sqrt{T_{\min}}}{2}, \quad t \in B.    
\end{equation}
Suppose otherwise. Then 
\[
\varphi'_\eta(\lambda^2 k^2 \HdC(-J_0 \dot{x}(t))) k < \varphi'_\eta(s_0^2) \frac{2}{\sqrt{T_{\min}} \sqrt{\HdC(-J_0 \dot{x}(t))}}
\]
From the definition of $B$, we have 
\[
\frac{2}{\sqrt{T_{\min}} \sqrt{\HdC(-J_0 \dot{x}(t))}} \leq \frac{8}{T_{\min}}
\]
and from the choice of $s_0$, we know $\varphi'_\eta(s_0^2) < \frac{T_{\min}}{8}$. Thus, it follows that 
\[
\varphi'_\eta(\lambda^2 k^2 \HdC(-J_0 \dot{x}(t))) k < 1
\]
which contradicts \eqref{ktlambda}. Therefore, \eqref{kinequalitysetB} holds.

Since $x \in \Lambda$, we have 
\[
\int_\mathbb{T} \langle J_0 x(t), \dot{x}(t) \rangle \, dt = 2 \mathcal{A}(x) = 2.
\]
Hence, for $x\in \Lambda$ and $0 < \lambda < s_0 \frac{\sqrt{T_{\min}}}{2}$, we can write 
\begin{align*}
 & \frac{d \PsiHeta}{d \lambda}(\lambda x) = -\lambda \int_\mathbb{T} \langle J_0 x(t), \dot{x}(t) \rangle \, dt + \int_\mathbb{T} \langle \nabla H_\eta^*(-\lambda J_0 \dot{x}(t)), -J_0 \dot{x}(t) \rangle \, dt \\
 & \geq -2\lambda + \int_B \langle \nabla H_\eta^*(-\lambda J_0 \dot{x}(t)), -J_0 \dot{x}(t) \rangle \, dt = \lambda \left(-2 + \int_B k(t, \lambda) \langle \nabla \HdC(-J_0 \dot{x}(t)), -J_0 \dot{x}(t) \rangle \, dt \right).
\end{align*}

The inequality follows from the fact that 
\[
\langle \nabla H_\eta^*(x), x \rangle \geq 0
\]
for all $x \in \mathbb{R}^{2n}$ since $H_\eta$ is convex, and the last equality is due to claim (2) of Lemma \ref{ham}.

Using the $2$-homogeneity of $\HdC$ and \eqref{kinequalitysetB} we get 
\begin{align*}
 & \frac{d \PsiHeta}{d \lambda}(\lambda x) \geq \lambda \left(-2 + 2 \int_B k(t, \lambda) \HdC(-J_0 \dot{x}(t)) \, dt \right) \geq \lambda \left(-2 + \frac{4}{\sqrt{T_{\min}}} \int_B \sqrt{\HdC(-J_0 \dot{x}(t))} \, dt \right) \\
 & = \lambda \left(-2 + \frac{4}{\sqrt{T_{\min}}} \int_\mathbb{T} \sqrt{\HdC(-J_0 \dot{x}(t))} \, dt - \frac{4}{\sqrt{T_{\min}}} \int_{B^c} \sqrt{\HdC(-J_0 \dot{x}(t))} \, dt \right).    
\end{align*}

From the definition of $B$ (see \eqref{setBclaim2}) and the fact that $|B^c| \leq 1$, where $|\cdot|$ is the measure of the set and $B^c$ is the complement of $B$ in $\mathbb{T}$, we have 
\[
\int_{B^c} \sqrt{\HdC(-J_0 \dot{x}(t))} \, dt \leq \frac{\sqrt{T_{\min}}}{4}.
\]
Combining this with the previous calculations and \eqref{squareroot}, we obtain
\begin{align}
    \nonumber \frac{d \PsiHeta}{d \lambda}(\lambda x) &\geq \lambda \left(-2 + \frac{4}{\sqrt{T_{\min}}} \int_\mathbb{T} \sqrt{\HdC(-J_0 \dot{x}(t))} \, dt - \frac{4}{\sqrt{T_{\min}}} \int_{B^c} \sqrt{\HdC(-J_0 \dot{x}(t))} \, dt \right) \\
    \nonumber &\geq \lambda \left(-2 + 4 - 1 \right) = \lambda > 0.
\end{align}

Therefore, \eqref{increasing} holds, which implies that for every $0 < r < s_0 \frac{\sqrt{T_{\min}}}{2}$, we have $d \PsiHeta(x)[x] > 0$ when $x \in (0, r] \Lambda$.

On the other hand, for $x \in \{\PsiC < T_\eta\}$, it holds that
\begin{align}
     \nonumber & \frac{d \PsiHeta}{d \lambda}(\lambda x) \leq \int_\mathbb{T} \langle \nabla H^*_\eta(-\lambda J_0 \dot{x}(t)), -J_0 \dot{x}(t) \rangle \, dt = 2 \lambda \int_\mathbb{T} k(t, \lambda) \HdC(-J_0 \dot{x}(t)) \, dt \\
     \nonumber & \leq \frac{2\lambda}{\delta} \int_\mathbb{T} \HdC(-J_0 \dot{x}(t)) \, dt \leq \frac{2\lambda}{\delta} T_\eta.  
\end{align}

The first inequality holds because $\frac{d}{d\lambda}\mathcal{A}(\lambda x) > 0$. The second is due to claim (2) of Lemma \ref{ham} ($k(t, \lambda) \leq \frac{1}{\delta}$), and the third follows from $x \in \{\PsiC < T_\eta\}$. Therefore, we have
\[
\PsiHeta(r x) = \PsiHeta(0) + \int_0^r \frac{d \PsiHeta}{d s}(\lambda x) \, ds \leq \PsiHeta(0) + \frac{r^2}{\delta} T_\eta.
\]

Since $\PsiHeta(0) < \varepsilon$, this implies that if $r$ is sufficiently small, $\PsiHeta < \varepsilon$ on $(0, r] \{\PsiC < T_\eta\}$. We choose $0 < r < s_0 \frac{\sqrt{T_{\min}}}{2}$ small enough to satisfy this condition. This completes the proof of (2).

\textbf{Claim }(3): To prove the first part of this claim, it suffices to show that $\PsiHeta(\lambda x)$ increases with $\lambda > 0$ for $x \in \{\PsiC > \eta\}$.

\begin{align}
 \nonumber \frac{d\PsiHeta}{d \lambda}(\lambda x) &= -\lambda \int\limits_\mathbb{T} \langle J_0 x(t), \dot{x}(t) \rangle \, dt + \int \limits_\mathbb{T} \langle \nabla H_\eta^*(-\lambda J_0 \dot{x}(t)), -J_0 \dot{x}(t) \rangle \, dt \\
 \nonumber &= \lambda \left(-2 + \int \limits_\mathbb{T} k(t, \lambda) \langle \nabla \HdC(-J_0 \dot{x}(t)), -J_0 \dot{x}(t) \rangle \, dt \right) \\
 \nonumber &= \lambda \left(-2 + 2 \int \limits_\mathbb{T} k(t, \lambda) \HdC(-J_0 \dot{x}(t)) \, dt \right)
\end{align}

In the second equality, we used claim (2) of Lemma \ref{ham} and the fact that $x \in \{\PsiC > \eta\} \subset \Lambda$, which implies $\mathcal{A}(x) = 1$. For the third equality, we applied the 2-homogeneity property of $\HdC$. According to Lemma \ref{ham}, $k(t, \lambda) \geq \frac{1}{\eta}$. Thus,

\begin{align}
 \nonumber \frac{d\PsiHeta}{d \lambda}(\lambda x) &\geq \lambda \left(-2 + \frac{2}{\eta} \int \limits_\mathbb{T} \HdC(-J_0 \dot{x}(t)) \, dt \right) \\
 \nonumber &= \lambda \left(-2 + \frac{2}{\eta} \PsiC(x) \right) > \lambda \left(-2 + \frac{2}{\eta} \eta \right) = 0
\end{align}

In the last inequality, we utilized the fact that $x \in \{\PsiC > \eta\}$. Hence, $\PsiHeta$ is radially increasing on $\mathbb{R}_+ \{\PsiC > \eta\}$. On the other hand, if $x \in \{\mathcal{A} \leq 0\}$, then for any $\lambda \geq 0$, we have

\begin{align}
    \nonumber \frac{d\PsiHeta}{d \lambda}(\lambda x) &= -\lambda \int\limits_\mathbb{T} \langle J_0 x(t), \dot{x}(t) \rangle \, dt + \int \limits_\mathbb{T} \langle \nabla H_\eta^*(-\lambda J_0 \dot{x}(t)), -J_0 \dot{x}(t) \rangle \, dt \\
    \nonumber &\geq \int \limits_\mathbb{T} \langle \nabla H_\eta^*(-\lambda J_0 \dot{x}(t)), -J_0 \dot{x}(t) \rangle \, dt \geq 0
\end{align}

Here, the first inequality follows from $x \in \{\mathcal{A} \leq 0\}$, and the second inequality is a consequence of the convexity of $H_\eta^*$, which ensures that $\langle \nabla H_\eta^*(-\lambda J_0 \dot{x}(t)), -J_0 \dot{x}(t) \rangle \geq 0$ for all $t$. Thus, we have proved claim (3).

\textbf{Claim} (4): Let $x \in \mathbb{R}_+ \{\PsiC < T_\eta\}$.

\textbf{Claim} (4.a):  By Remark \ref{stable}, there exists a unique critical orbit $S_{x^-} \subset \Lambda_{\mathcal{A}(x)}$ such that for every $S^1$-invariant neighborhood $U \subset \Lambda_{\mathcal{A}(x)}$ of $S_{x^-}$, there is a $t_0 > 0$ such that for all $t \geq t_0$, $\phi^t(x) \in U$, where $\phi^t$ denotes the flow of $-V$, and $V$ is the $1$-homogeneous pseudo-gradient vector field for $\widetilde{\Psi}_C$ defined in Remark \ref{boundedpalaissmale}. Since $\widetilde{\Psi}_C$ decreases along $\phi^t$ and there are no critical values between $\eta$ and $T_\eta$, we conclude that
\[
S_{x^-} \subset \crit(\widetilde{\Psi}_C)^{< T_{\eta}} = \crit(\widetilde{\Psi}_C)^{< \eta} = \mathbb{R}_+ \crit(\PsiC)^{< \eta}.
\]

By claim (1) of this lemma, there exists a unique $\lambda' > 0$ such that $\lambda' S_{x^-} \subset \crit(\PsiHeta) \setminus \{0\}$. Since $\phi^t$ is $1$-homogeneous, it follows that for every $S^1$-invariant neighborhood $U \subset \Lambda_{\mathcal{A}(\lambda' x)}$ of $\lambda' S_{x^-}$, there exists a $t_0 > 0$ such that for all $t \geq t_0$, $\phi^t(\lambda' x) \in U$. Given that $\PsiHeta(\lambda' x^-) > \varepsilon$ and $\PsiHeta$ is $S^1$-invariant, we can find an $S^1$-invariant neighborhood $U'$ of $\lambda' S_{x^-}$ such that $\PsiHeta > \varepsilon$ on $U'$. For large $t > 0$, $\phi^t(\lambda' x) \in U'$. From Remark \ref{decPsiHeta}, we know that $\PsiHeta$ is non-increasing along the flow of $\phi^t$. Therefore,
\[
\PsiHeta(\lambda' x) \geq \PsiHeta(\phi^t(\lambda' x)) > \varepsilon.
\]

In particular if the second condition of the claim is satisfied, then it holds \[\mathcal{A}(x)>\mathcal{A}(\lambda'  x^-)=\mathcal{A}(\lambda' x)=(\lambda')^2 \mathcal{A}(x)\]

Since $\mathcal{A}(x)>0$, this implies that $\lambda '<1$.

\textbf{Claim} (4.b): From claim (2) of this lemma, we have that
\begin{equation}\label{nearor}
  \frac{d\PsiHeta}{d\lambda}(\lambda x) > 0  
\end{equation}
for small $\lambda > 0$. Suppose there exists some $\lambda_0 > 0$ such that
\[
\frac{d\PsiHeta}{d\lambda}(\lambda_0 x) \leq 0.
\]

Let $\lambda_{\max} = \inf \{\lambda > 0 \mid \frac{d\PsiHeta}{d\lambda}(\lambda x) \leq 0 \}$. Since the derivative must be positive near zero, $\lambda_{\max} > 0$. By the definition of $\lambda_{\max}$, it follows that
\[
\frac{d\PsiHeta}{d\lambda}(\lambda x) > 0, \quad \lambda \in (0, \lambda_{\max}),
\]
and
\[
\frac{d\PsiHeta}{d\lambda}(\lambda_{\max} x) = \lambda_{\max} \left(-\int\limits_\mathbb{T} \langle J_0 x(t), \dot{x}(t) \rangle \, dt + 2 \int\limits_\mathbb{T} K(t, \lambda_{\max}) \HdC(-J_0 \dot{x}(t)) \, dt \right) = 0.
\]

Since $\lambda_{\max} > 0$, it must be that
\[
g(\lambda_{\max}) = -\int\limits_\mathbb{T} \langle J_0 x(t), \dot{x}(t) \rangle \, dt + 2 \int\limits_\mathbb{T} K(t, \lambda_{\max}) \HdC(-J_0 \dot{x}(t)) \, dt = 0.
\]

The function
\[
g(\lambda) = -\int\limits_\mathbb{T} \langle J_0 x(t), \dot{x}(t) \rangle \, dt + 2 \int\limits_\mathbb{T} K(t, \lambda) \HdC(-J_0 \dot{x}(t)) \, dt
\]
is non-increasing because for $\lambda_1 < \lambda_2$, we have $k(\lambda_2, t) \leq k(\lambda_1, t)$, which is a consequence of Claim (2) of Lemma \ref{ham}.

Let
\[
B = \{t \in \mathbb{T} \mid \HdC(-J_0 \dot{x}(t)) > 0, \, k(t, \lambda_0) < \eta\}.
\]

Assume $|B| = 0$, where $|\cdot|$ is the measure of the set. Since $k(\lambda, t)$ is non-increasing with respect to $\lambda$, for $\lambda > \lambda_{\max}$, we have $k(t, \lambda) = \frac{1}{\eta}$ on $B^c$ if $\HdC(-J_0 \dot{x}(t)) > 0$. Therefore,
\[
\int\limits_{B^c} K(t, \lambda) \HdC(-J_0 \dot{x}(t)) \, dt = \int\limits_{B^c} K(t, \lambda_{\max}) \HdC(-J_0 \dot{x}(t)) \, dt, \quad \lambda \geq \lambda_{\max}.
\]

Since $|B| = 0$, it follows that $g(\lambda) = g(\lambda_{\max})$ for $\lambda \geq \lambda_{\max}$, which implies that
\[
\frac{d\PsiHeta}{d\lambda}(\lambda x) = 0, \quad \lambda \geq \lambda_{\max}.
\]

Additionally, we have
\begin{align}
\nonumber & \int\limits_{\mathbb{T}} K(t, \lambda_{\max}) \HdC(-J_0 \dot{x}(t)) \, dt = \int\limits_{B^c} K(t, \lambda_{\max}) \HdC(-J_0 \dot{x}(t)) \, dt \\
\nonumber & = \frac{1}{\eta} \int\limits_{B^c} \HdC(-J_0 \dot{x}(t)) \, dt = \frac{1}{\eta} \int\limits_{\mathbb{T}} \HdC(-J_0 \dot{x}(t)) \, dt.
\end{align}

On the other hand, since $g(\lambda_{\max}) = 0$, it follows from the previous calculation that
\[
\frac{2}{\eta} \int\limits_{\mathbb{T}} \HdC(-J_0 \dot{x}(t)) = 2\mathcal{A}(x).
\]

Thus, we have $\PsiC \left(\frac{x}{\sqrt{\mathcal{A}(x)}} \right) = \eta$.

Assume now that $|B| > 0$. For $\lambda > \lambda_{\max}$, we have $k(\lambda, t) < k(\lambda_{\max}, t)$ for all $t \in B$. Since $|B| > 0$, it follows that
\[
\int\limits_\mathbb{T} K(t, \lambda) \HdC(-J_0 \dot{x}(t)) \, dt < \int\limits_\mathbb{T} K(t, \lambda_{\max}) \HdC(-J_0 \dot{x}(t)) \, dt.
\]

Therefore, $g(\lambda) < g(\lambda_{\max}) = 0$ for $\lambda > \lambda_{\max}$, so
\[
\frac{d\PsiHeta}{d\lambda}(\lambda x) < 0, \quad \lambda \in (\lambda_{\max}, +\infty).
\]

Let $\lambda_1 < \lambda_2$ where $\lambda_1, \lambda_2 \in (\lambda_{\max}, +\infty)$. Then,
\[
\frac{d\PsiHeta}{d\lambda}(\lambda_2 x) = \lambda_2 g(\lambda_2) < \lambda_1 g(\lambda_2) \leq \lambda_1 g(\lambda_1) = \frac{d\PsiHeta}{d\lambda}(\lambda_1 x),
\]
where the first inequality follows from $g(\lambda_2)$ being negative, and the second from $g(\lambda)$ being non-increasing. Therefore, the derivative of $\PsiHeta(\lambda x)$ is decreasing on $(\lambda_{\max}, +\infty)$.

\textbf{Claim} (4.c): Let $x \in \crit(\PsiHeta) \setminus \{0\}$. Since $x$ is a critical point, it must be that
\[
\frac{d\PsiHeta}{d\lambda}(\lambda x) = 0
\]
for $\lambda_{\max} = 1$. Given that $x \in \mathbb{R}_+ \{\PsiC < \eta\}$, from (4.b) it follows that it must exhibit the third type of derivative behavior.

\textbf{Claim} (5): Let $S_x$ be an arbitrary non-trivial critical orbit of $\PsiHeta$. Consider an $S^1$-invariant neighborhood $U_{S_x} \subset \Lambda_{\mathcal{A}(x)} \cap \mathbb{R}_+\{\PsiC < T_\eta\}$ of $S_x$ such that $\PsiHeta < \eta$ on $[1 - \delta_x, 1 + \delta_x] U_{S_x}$, for some small $\delta_x > 0$. Such a neighborhood exists because $\PsiHeta(x) < \eta$. Additionally, by possibly shrinking $U_{S_x}$, we can ensure that
\[
\frac{d\PsiHeta}{d\lambda}((1 - \delta_x)y) > 0, \quad y \in U
\]
and
\[
\frac{d\PsiHeta}{d\lambda}((1 + \delta_x)y) < a, \quad y \in U,
\]
where $a < 0$ is some negative constant. This follows from the fact that
\[
\frac{d\PsiHeta}{d\lambda}((1 - \delta_x)x) > 0
\]
and
\[
\frac{d\PsiHeta}{d\lambda}((1 + \delta_x)x) < a
\]
for some negative $a < 0$, which follows from claim (4.c) of this Lemma. From (4.b), we know that $\PsiHeta$ increases radially on $(0, 1 - \delta_x] U_{S_x}$ and decreases radially on $[1 + \delta_x, +\infty) U_{S_x}$. Combining this with the fact that $\PsiHeta < \eta$ on $[1 - \delta_x, 1 + \delta_x] U_{S_x}$, we conclude that $\mathbb{R_+} U_{S_x} \subset \{\PsiHeta < \eta\}$. Additionally, from (4.c), we know that the derivative of $\PsiHeta(\lambda y)$ is decreasing for $y \in U_{S_x}$ and $\lambda \geq 1 + \delta_x$. Therefore, we have
\[
\frac{d\PsiHeta}{d\lambda}(\lambda y) < a, \quad y \in U_{S_x}, \ \lambda \geq 1 + \delta_x.
\]
Combining this with the fact that $\mathbb{R_+} U_{S_x} \subset \{\PsiHeta < \eta\}$, we conclude that there exists $R_{S_x} > 0$ such that $[R_{S_x}, +\infty) U_{S_x} \subset \{\PsiHeta < 0\}$, which completes the proof of this Lemma.

\end{proof}

\begin{lemma}\label{inclusioncutoff}
    Let $L \in (\eta, T_\eta)$. Then the inclusion map
    \[
    i^{\eta,C}_L: \left(\{\PsiHeta < L\} \cap \mathbb{R}_+ \{\PsiC < L\}, \{\PsiHeta < \varepsilon\} \cap \mathbb{R}_+ \{\PsiC < L\}\right) \to \left(\{\PsiHeta < L\}, \{\PsiHeta < \varepsilon\}\right)
    \]
    induces isomorphisms in singular and $S^1$-equivariant homology.
\end{lemma}

\begin{proof}

This proof follows from the Mayer–Vietoris sequence for pairs. We introduce the pair  
\begin{equation} \label{otherpair}
\left(\{\PsiHeta < L\} \cap ([r, +\infty) \{\PsiC \leq \eta\})^c, \{\PsiHeta < \varepsilon\} \cap ([r, +\infty) \{\PsiC \leq \eta\})^c\right)    
\end{equation}
where $r > 0$ is the number from statement (2) of Lemma \ref{Psiheta}, and $([r, +\infty) \{\PsiC \leq \eta\})^c$ denotes the complement of a given set within $H^1_0(\mathbb{T}, \mathbb{R}^{2n})$. The sets  
\[
\{\PsiHeta < L\} \cap \mathbb{R}_+ \{\PsiC < L\}, \quad \{\PsiHeta < \varepsilon\} \cap \mathbb{R}_+ \{\PsiC < L\}, 
\]
\[
\{\PsiHeta < L\} \cap ([r, +\infty) \{\PsiC \leq \eta\})^c, \quad \text{and} \quad \{\PsiHeta < \varepsilon\} \cap ([r, +\infty) \{\PsiC \leq \eta\})^c
\]
are all open in $\Honenull$. Moreover, we have the identity  
\[
([r, +\infty) \{\PsiC \leq \eta\})^c = (0, r) \{\PsiC \leq \eta\} \cup \mathbb{R}_+ \{\PsiC > \eta\} \cup \{\mathcal{A} \leq 0\}.
\]
Thus, by claims (2) and (3) of Lemma \ref{Psiheta}, we conclude that $\PsiHeta$ is radially non-decreasing on  
\[
([r, +\infty) \{\PsiC \leq \eta\})^c.
\]

By claim (1) of Lemma \ref{Psiheta},  
\[
0 \in \{\PsiHeta < \varepsilon\} \cap \left([r, +\infty) \{\PsiC \leq \eta\}\right)^c \subset \{\PsiHeta < L\} \cap \left([r, +\infty) \{\PsiC \leq \eta\}\right)^c,
\]
and since $\PsiHeta$ is radially non-decreasing on $\left([r, +\infty) \{\PsiC \leq \eta\}\right)^c$, the sets \[\{\PsiHeta < \varepsilon\} \cap \left([r, +\infty) \{\PsiC \leq \eta\}\right)^c\text{ and }\{\PsiHeta < L\} \cap \left([r, +\infty) \{\PsiC \leq \eta\}\right)^c\] are star-shaped with respect to $0$. This implies that $\{0\}$ is a strong deformation retract of both sets, and we obtain  
\begin{equation}\label{pairhomology}
    \H_*\left(\{\PsiHeta < L\} \cap \left([r, +\infty) \{\PsiC \leq \eta\}\right)^c, \{\PsiHeta < \varepsilon\} \cap \left([r, +\infty) \{\PsiC \leq \eta\}\right)^c\right) = 0.
\end{equation}

The union of the pairs
\[
\left(\{\PsiHeta < L\} \cap \mathbb{R}_+ \{\PsiC < L\}, \{\PsiHeta < \varepsilon\} \cap \mathbb{R}_+ \{\PsiC < L\}\right)
\]
and
\[
\left(\{\PsiHeta < L\} \cap \left([r, +\infty) \{\PsiC \leq \eta\}\right)^c, \{\PsiHeta < \varepsilon\} \cap \left([r, +\infty) \{\PsiC \leq \eta\}\right)^c\right)
\]
is precisely
\begin{equation}\label{union}
    \left(\{\PsiHeta < L\}, \{\PsiHeta < \varepsilon\}\right),
\end{equation}
and their intersection is
\begin{equation}\label{intersection}
  \begin{array}{lcl}
     &\left(\{\PsiHeta < L\} \cap \mathbb{R}_+ \{\PsiC < L\} \cap \left([r, +\infty) \{\PsiC \leq \eta\}\right)^c, \hspace{3cm} \right. \\
      &  \left.\{\PsiHeta < \varepsilon\} \cap \mathbb{R}_+\{\PsiC < L\} \cap \left([r, +\infty) \{\PsiC \leq \eta\}\right)^c\right)
  \end{array}
\end{equation}

Since $\PsiHeta$ is radially non-decreasing on $\mathbb{R}_+ \{\PsiC < L\}\cap \left([r, +\infty) \{\PsiC \leq \eta\}\right)^c$, we conclude that if 
\[
x \in \{\PsiHeta < \varepsilon\} \cap \mathbb{R}_+ \{\PsiC < L\} \cap \left([r, +\infty) \{\PsiC \leq \eta\}\right)^c,
\]
then 
\[
\lambda x \in \{\PsiHeta < \varepsilon\} \cap \mathbb{R}_+ \{\PsiC < L\} \cap \left([r, +\infty) \{\PsiC \leq \eta\}\right)^c
\]
for all $\lambda \in (0, 1]$. The same property holds for $\{\PsiHeta < L\} \cap \mathbb{R}_+ \{\PsiC < L\} \cap \left([r, +\infty) \{\PsiC \leq \eta\}\right)^c$. Combining this with the fact that
\[
(0, \frac{r}{2}] \{\PsiC < L\} \subset \{\PsiHeta < \varepsilon\} \cap \mathbb{R}_+ \{\PsiC < L\} \cap \left([r, +\infty) \{\PsiC \leq \eta\}\right)^c
\]
\[
\subset \{\PsiHeta < L\} \cap \mathbb{R}_+ \{\PsiC < L\} \cap \left([r, +\infty) \{\PsiC \leq \eta\}\right)^c,
\]
which follows from claim (2) of Lemma \ref{Psiheta} ($\PsiHeta<\varepsilon$ on $(0, r] \{\PsiC < T_\eta\}$), we conclude that $(0, \frac{r}{2}] \{\PsiC < L\}$ is a strong deformation retract of both spaces, and the retraction is given by the radial projection onto $\frac{r}{2} \{ \PsiC < L\}$ of points outside $(0, \frac{r}{2}] \{\PsiC < L\}$. Since both sets contain the same strong deformation retract, it follows that
\begin{equation*}
    \begin{array}{lcl}
     &\H_*\left(\{\PsiHeta < L\} \cap \mathbb{R}_+ \{\PsiC < L\} \cap \left([r, +\infty) \{\PsiC \leq \eta\}\right)^c,  \hspace{3cm} \right.\\
      & \left.\{\PsiHeta < \varepsilon\} \cap \mathbb{R}_+ \{\PsiC < L\} \cap \left([r, +\infty) \{\PsiC \leq \eta\}\right)^c\right)=0.
  \end{array}
\end{equation*}

Now that we have shown the homology of the pair \eqref{otherpair} is trivial, the homology of the intersection \eqref{intersection} is also trivial, and the union is \eqref{union}, it follows that for the pairs  
\[
\left(\{\PsiHeta < L\} \cap \mathbb{R}_+ \{\PsiC < L\}, \{\PsiHeta < \varepsilon\} \cap \mathbb{R}_+ \{\PsiC < L\}\right),
\]
and  
\[
\left(\{\PsiHeta < L\} \cap \left([r, +\infty) \{\PsiC \leq \eta\}\right)^c, \{\PsiHeta < \varepsilon\} \cap \left([r, +\infty) \{\PsiC \leq \eta\}\right)^c\right),
\]
the induced Mayer-Vietoris sequence takes the form  
\begin{align}
    \nonumber & 0 \to \H_*\left(\{\PsiHeta < L\} \cap \mathbb{R}_+ \{\PsiC < L\}, \{\PsiHeta < \varepsilon\} \cap \mathbb{R}_+ \{\PsiC < L\}\right) \\
    \nonumber & \xrightarrow{(i^{\eta,C}_L)_*} \H_*\left(\{\PsiHeta < L\}, \{\PsiHeta < \varepsilon\}\right) \to 0.
\end{align}

\end{proof}

Now, we will examine the topology of $\{\PsiHeta<L\}\cap \mathbb{R}_+ \{\PsiC<L\}$ and the topology of $\{\PsiHeta<\varepsilon\}\cap \mathbb{R}_+ \{\PsiC<L\}$.

\begin{lemma}\label{norml}
  Let $L \in (\eta, T_\eta)$. The map 
  \[
  N_L^{\eta,C}: \{\PsiHeta<L\}\cap \mathbb{R}_+ \{\PsiC<L\} \to \{\PsiC<L\}
  \]
  defined by 
  \[
  N_L^{\eta,C}(x) = \frac{x}{\sqrt{\mathcal{A}(x)}}
  \]
  induces an isomorphism in singular and $S^1$-equivariant homology.
\end{lemma}

\begin{proof}
    Note that this map is well-defined since 
    \[
    \{\PsiHeta<L\}\cap \mathbb{R}_+ \{\PsiC<L\} \subset \mathbb{R}_+ \{\PsiC<L\}.
    \]

    The map $N_L^{\eta,C}(x)$ filters through the space $\mathbb{R}_+ \{\PsiC<L\}$ in the following sense:

    \[
    \begin{tikzcd}[scale cd=0.9]
     & \mathbb{R}_+ \{\PsiC<L\} \arrow[dr,"N"] \\
    \{\PsiHeta<L\}\cap \mathbb{R}_+ \{\PsiC<L\} \arrow[ur,"i"] \arrow[rr,"N_L^{\eta,C}"] &&\{\PsiC<L\}
    \end{tikzcd}
    \]

    where $i$ is inclusion and $N$ is the retraction from $\mathbb{R}_+ \{\PsiC<L\}$ to $\{\PsiC<L\}$ given by 
    \[
    N(x) = \frac{x}{\sqrt{\mathcal{A}(x)}}.
    \]
    Since $\{\PsiC<L\}$ is a strong deformation retract of $\mathbb{R}_+ \{\PsiC<L\}$, it follows that $N$ induces an isomorphism in homology. Thus, it is sufficient to prove that the same claim holds for $i$.

    To prove this, we will show that for a small $\delta > 0$, the space $\{\PsiHeta \leq L - \delta\} \cap \mathbb{R}_+ \{\PsiC < L\}$ is a strong deformation retract of the spaces $\{\PsiHeta < L\} \cap \mathbb{R}_+ \{\PsiC < L\}$ and $\mathbb{R}_+ \{\PsiC < L\}$.

    Let $U_{S_x}$ be the neighborhoods from claim (5) of Lemma \ref{Psiheta}, and let $r>0$ be from claim (2) of the same lemma. From claims (2) and (5) of Lemma \ref{Psiheta}, it follows that $\PsiHeta<\eta$ on $(0,r]\{\PsiC<L\} \cup \bigcup_{S_x} \mathbb{R}_+ U_{S_x}$, where the union is taken over all non-trivial critical orbits of $\PsiHeta$. Therefore, we have  
\begin{equation}\label{intersectionflow}
\resizebox{.91\hsize}{!}{$\PsiHeta^{-1}([\eta,+\infty))\cap \mathbb{R}_+ \{\PsiC<L\} = \PsiHeta^{-1}([\eta,+\infty))\cap \left(\mathbb{R}_+ \{\PsiC<L\} \setminus \left((0,r]\{\PsiC<L\} \cup \bigcup_{S_x} \mathbb{R}_+ U_{S_x}\right)\right)$}
\end{equation}

    Let $V$ be the pseudo-gradient vector field of $\widetilde{\Psi}_C$ defined in Remark \ref{boundedpalaissmale}. This vector field satisfies the Palais-Smale condition on the level sets of $\mathcal{A}$, it is 1-homogeneous, and from claim (1) of Lemma \ref{Psiheta}, we have that 
\[
\mathbb{R}_+ \text{crit}(\PsiHeta) = \mathbb{R}_+ \text{crit}(\PsiC)^{<\eta} = \text{crit}(\widetilde{\Psi}_C)^{<T_\eta}.
\]
Hence, for $L \in (\eta, T_\eta)$, it holds that
\[
\|V(x)\|_{H^1} \geq a > 0, \quad x \in \mathbb{R}_+ \{\PsiC < L\} \setminus \left((0, r] \{\PsiC < L\} \cup \bigcup_{S_x} \mathbb{R}_+ U_{S_x}\right),
\]
for some positive constant $a > 0$. Therefore, from \eqref{intersectionflow}, we conclude that 
\[
\|V(x)\|_{H^1} \geq a, \quad x \in \PsiHeta^{-1}([\eta, +\infty)) \cap \mathbb{R}_+ \{\PsiC < L\},
\]
and Remark \ref{decPsiHeta} implies that 
\[
d\PsiHeta(x)[V(x)] = \int_0^1 \langle \nabla H_\eta^*(-J_0 \dot{x}(t)), -J_0 \frac{d}{dt} V(x)(t) \rangle \, dt \geq \frac{1}{\eta} \|V(x)\|_{H^1}^2 \geq \frac{1}{\eta} a^2,
\]
when $x \in \PsiHeta^{-1}([\eta,+\infty)) \cap \mathbb{R}_+ \{\PsiC<L\}$. 

We can now define the vector field 
\[
Y(x) = \frac{-V(x)}{d\PsiHeta(x)[V(x)]}
\]
on $\PsiHeta^{-1}((\eta,+\infty)) \cap \mathbb{R}_+ \{\PsiC<L\}$. Since $V$ and $d\PsiHeta$ are locally Lipschitz and locally bounded, and  
\[
d\PsiHeta(x)[V(x)] = \int_0^1 \langle \nabla H_\eta^*(-J_0 \dot{x}(t)), -J_0 \frac{d}{dt} V(x)(t) \rangle \, dt
\]
is bounded from below by a positive constant, one easily concludes that the vector field $Y$ is locally Lipschitz on $\PsiHeta^{-1}((\eta, +\infty)) \cap \mathbb{R}_+ \{\PsiC < L\}$. 

We therefore have a locally defined flow $\phi^t_Y$ of $Y$ on the open set $\PsiHeta^{-1}((\eta, +\infty)) \cap \mathbb{R}_+ \{\PsiC < L\}$, and this flow is continuous in its domain. Moreover, this flow is a reparameterization of the flow of $-V$, and it satisfies $\PsiHeta(\phi^t_Y(x)) = \PsiHeta(x) - t$. Thus, for a fixed $\delta > 0$ such that $L - 2\delta > \eta$, and for every $x \in \PsiHeta^{-1}((L - \frac{3}{2}\delta, +\infty)) \cap \mathbb{R}_+ \{\PsiC < L\}$, it must hold
\[
t_+(x) > \PsiHeta(x) - (L - 2\delta),
\]
where $t_+(x)$ is the maximal time for which $\phi^t_Y(x)$ is defined. 

We define the retraction 
    \[
    R: \mathbb{R}_+ \{\PsiC<L\} \to \{\PsiHeta \leq L-\delta\} \cap \mathbb{R}_+ \{\PsiC<L\}
    \]
    by 
    \[
    R(x) =
    \begin{cases}
    x, &  x \in \{\PsiHeta \leq L-\delta\} \cap \mathbb{R}_+ \{\PsiC<L\}, \\
    \phi^{(\PsiHeta(x)-(L-\delta))}_Y(x), &  x \in \mathbb{R}_+ \{\PsiC<L\} \setminus \{\PsiHeta < L-\delta\} \cap \mathbb{R}_+ \{\PsiC<L\}.
    \end{cases}
    \]

    Note that this is well-defined for the second case since $t_+(x) > \PsiHeta(x) - 2\delta$ and 
    \[
    \PsiHeta\left(\phi^{(\PsiHeta(x)-(L-\delta))}_Y(x)\right) = L - \delta,
    \]
    and on the overlap, $x \in \{\PsiHeta = L-\delta\} \cap \mathbb{R}_+ \{\PsiC<L\}$, so it holds that 
    \[
    \phi^{(\PsiHeta(x)-(L-\delta))}_Y(x) = \phi^0_Y(x) = x.
    \]
    Thus, $R$ is a well-defined function.

    In the separate cases, $R$ is continuous since the identity, $\phi_Y$, and $\PsiHeta$ are continuous. Since the sets in both cases are closed in $\mathbb{R}_+ \{\PsiC < L\}$, it follows that $R$ is continuous. Therefore, $\{\PsiHeta \leq L - \delta\} \cap \mathbb{R}_+ \{\PsiC < L\}$ is a retract of $\mathbb{R}_+ \{\PsiC < L\}$.

    We now define the homotopy 
    \[
    K: [0,1] \times \mathbb{R}_+ \{\PsiC<L\} \to   \mathbb{R}_+ \{\PsiC<L\}
    \]
    by 
    \[
    K(t,x) =
    \begin{cases}
    x, &  x \in \{\PsiHeta \leq L-\delta\} \cap \mathbb{R}_+ \{\PsiC<L\}, \\
    \phi^{t(\PsiHeta(x)-(L-\delta))}_Y(x), & x \in \mathbb{R}_+ \{\PsiC<L\} \setminus \{\PsiHeta < L-\delta\} \cap \mathbb{R}_+ \{\PsiC<L\}.
    \end{cases}
    \]

  Based on the same arguments, we conclude that $K$ is well-defined and continuous. Let us denote by $i_{L-\delta}: \{\PsiHeta \leq L-\delta\} \cap \mathbb{R}_+ \{\PsiC < L\} \to \mathbb{R}_+ \{\PsiC < L\}$ the inclusion. It is clear that the following holds:

\begin{itemize}
    \item $K_0 = \text{id},$
    \item $K_1 = i_{L-\delta} \circ R,$
    \item $K_t(x) = x, \quad x \in \{\PsiHeta \leq L - \delta\} \cap \mathbb{R}_+ \{\PsiC < L\}$.
\end{itemize}

Hence, $\{\PsiHeta \leq L - \delta\} \cap \mathbb{R}_+ \{\PsiC < L\}$ is a strong deformation retract of $\mathbb{R}_+ \{\PsiC < L\}$. 

Similarly, using the flow of $\phi_Y$, we show that $\{\PsiHeta \leq L - \delta\} \cap \mathbb{R}_+ \{\PsiC < L\}$ is a strong deformation retract of $\{\PsiHeta < L\} \cap \mathbb{R}_+ \{\PsiC < L\}$. Since both $\{\PsiHeta < L\} \cap \mathbb{R}_+ \{\PsiC < L\}$ and $\mathbb{R}_+ \{\PsiC < L\}$ contain the same strong deformation retract, it follows that the inclusion
\[
i: \{\PsiHeta < L\} \cap \mathbb{R}_+ \{\PsiC < L\} \to \mathbb{R}_+ \{\PsiC < L\}
\]
induces an isomorphism in homology, which concludes the proof.

\end{proof}

\begin{lemma}\label{UminusUplus}
    Assume that $L\in (\eta,T_\eta)$. Define  
\[
U^-_{\eta,L} = \{x \in \{\PsiHeta < \varepsilon\} \cap \mathbb{R}_+ \{\PsiC < L\} \mid \lambda x \in \{\PsiHeta < \varepsilon\} \cap \mathbb{R}_+ \{\PsiC < L\}, \text{ for all } \lambda \in (0,1]\}
\]
and  
\[
U^+_{\eta,L} = \{x \in \{\PsiHeta < \varepsilon\} \cap \mathbb{R}_+ \{\PsiC < L\} \mid \lambda x \in \{\PsiHeta < \varepsilon\} \cap \mathbb{R}_+ \{\PsiC < L\}, \text{ for all } \lambda \in [1,+\infty)\}.
\]
Now, let the map  
\[
N_\varepsilon^{\eta,C}: \{\PsiHeta < \varepsilon\} \cap \mathbb{R}_+ \{\PsiC < L\} \to \{\PsiC < L\}
\]
be defined by  
\[
N_\varepsilon^{\eta,C}(x) = \frac{x}{\sqrt{\mathcal{A}(x)}}.
\]

The following holds:

    \begin{enumerate}
        \item $U^-_{\eta,L}$ and $U^+_{\eta,L}$ are disjoint $S^1$-invariant components of $\{\PsiHeta < \varepsilon\} \cap \mathbb{R}_+ \{\PsiC < L\}$ such that
        \[
        \{\PsiHeta < \varepsilon\} \cap \mathbb{R}_+ \{\PsiC < L\} = U^-_{\eta,L} \cup U^+_{\eta,L}.
        \]
        \item The map $N_\varepsilon^{\eta,C}$, when restricted to $U^-_{\eta,L}$ or $U^+_{\eta,L}$, induces isomorphisms in singular and $S^1$-equivariant homology.
        \item If $H_\eta \leq H_\nu$, then for $L_\eta \leq L_\nu$ such that $L_\eta \in (\eta, T_\eta)$ and $L_\nu \in (\nu, T_\nu)$, the inclusion
        \[
        i^{H,\varepsilon}_{L_\nu,L_\eta}: \{\PsiHeta < \varepsilon\} \cap \mathbb{R}_+ \{\PsiC < L_\nu\} \to \{\PsiHnu < \varepsilon\} \cap \mathbb{R}_+ \{\PsiC < L_\eta\}
        \]
        maps $U^-_{\eta,L_\eta}$ to $U^-_{\nu,L_\nu}$ and $U^+_{\eta,L_\eta}$ to $U^+_{\nu,L_\eta}$.
    \end{enumerate}

\end{lemma}

\begin{proof}
    \textbf{Claim }(1): Since $\PsiHeta$ and $\PsiC$ are $S^1$-invariant, it is clear that $U^-_{\eta,L}$ and $U^+_{\eta,L}$ are also $S^1$-invariant. The first thing we will show is that  
\begin{equation}\label{Uminusderivariverepresentation}
    U^-_{\eta,L} = \{x \in \{\PsiHeta < \varepsilon\} \cap \mathbb{R}_+ \{\PsiC < L\} \mid d\PsiHeta(x)[x] > 0\}   
\end{equation}
and  
\begin{equation}\label{Uplusderivativerepesentation}
    U^+_{\eta,L} = \{x \in \{\PsiHeta < \varepsilon\} \cap \mathbb{R}_+ \{\PsiC < L\} \mid d\PsiHeta(x)[x] < 0\}.   
\end{equation}

Indeed, if $d\PsiHeta(x)[x] > 0$, then, according to all three possibilities in claim (4.b) of Lemma \ref{Psiheta}, the function $\PsiHeta(\lambda x)$ must be increasing on $(0,1]$. Hence, for all $\lambda \in (0,1]$, we have $\PsiHeta(\lambda x) \leq \PsiHeta(x) < \varepsilon$, implying that $x \in U^-_{\eta,L}$. On the other hand, if $d\PsiHeta(x)[x] \leq 0$, then from statements (4.a) and (4.b) of Lemma \ref{Psiheta}, there exists some $\lambda \in (0,1]$ such that $\PsiHeta(\lambda x) > \varepsilon$. Hence, $x \notin U^-_{\eta,L}$. This proves \eqref{Uminusderivariverepresentation}.

Similarly, if $d\PsiHeta(x)[x] < 0$, then from claim (4.b) of Lemma \ref{Psiheta}, we conclude that $\PsiHeta(\lambda x)$ is decreasing for $\lambda \in [1,+\infty)$ and that $\PsiHeta(\lambda x) \leq \PsiHeta(x) < \varepsilon$. Conversely, if $d\PsiHeta(x)[x] \geq 0$, then from claims (4.a) and (4.b) of Lemma \ref{Psiheta}, there exists some $\lambda \in [1,+\infty)$ such that $\PsiHeta(\lambda x) > \varepsilon$. Thus, \eqref{Uplusderivativerepesentation} also holds.

Representations \eqref{Uminusderivariverepresentation} and \eqref{Uplusderivativerepesentation} imply that $U^-_{\eta,L}$ and $U^+_{\eta,L}$ are open in $\{\PsiHeta < \varepsilon\} \cap \mathbb{R}_+ \{\PsiC < L\}$ and disjoint. Therefore we only need to show that 
    \begin{equation}\label{disjointunionUminusUplus}
      \{\PsiHeta < \varepsilon\} \cap \mathbb{R}_+ \{\PsiC < L\} = U^-_{\eta,L} \cup U^+_{\eta,L}.
    \end{equation}

    Let $x \in \{\PsiHeta < \varepsilon\} \cap \mathbb{R}_+ \{\PsiC < L\} \subset \mathbb{R}_+ \{\PsiC < T_\eta\}$ be arbitrary. From claim (4.b) of Lemma \ref{Psiheta}, we know that if $d\PsiHeta(x)[x] = 0$, then the function $\PsiHeta(\lambda x)$ has a global maximum at $1$. From claim (4.a) of  Lemma \ref{Psiheta}, we conclude that $\PsiHeta(x) > \varepsilon$, which is a contradiction. Therefore, it must be $d\PsiHeta(x)[x] > 0$ or $d\PsiHeta(x)[x] < 0$, which implies, by \eqref{Uminusderivariverepresentation} and \eqref{Uplusderivativerepesentation} that  \eqref{disjointunionUminusUplus} holds. 

    \textbf{Claim (2):} We begin by showing that $N_\varepsilon^{\eta,C}$, when restricted to $U^-_{\eta,L}$, induces an isomorphism in homology. From Claim (3) of Lemma \ref{Psiheta} and the definition of $U^-_{\eta,L}$, we have  
\[
(0,r] \{\PsiC < L\} \subset U^-_{\eta,L}.
\]  
Thus, we define a retraction  
\[
R^-: U^-_{\eta,L} \to (0,r] \{\PsiC < L\}
\]  
by radially projecting points outside $(0,r] \{\PsiC < L\}$ onto the boundary $r \{\PsiC < L\}$. Since for any $x \in U^-_{\eta,L}$, the point $\lambda x$ also belongs to $U^-_{\eta,L}$ for all $\lambda \in (0,1]$, it follows that $(0,r] \{\PsiC < L\}$ is a strong deformation retract of $U^-_{\eta,L}$. Consequently, $R^-$ induces an isomorphism in homology.  

The map $N_\varepsilon^{\eta,C}$ factors through $(0,r] \{\PsiC < L\}$ as follows:  
\[
\begin{tikzcd}[scale cd=0.9]
    & (0,r] \{\PsiC < L\} \arrow[dr,"N^-"] \\
    U^-_{\eta,L} \arrow[ur,"R^-"] \arrow[rr,"N_\varepsilon^{\eta,C}"] && \{\PsiC < L\}
\end{tikzcd}
\]  
where $N^-(x) = \frac{x}{\sqrt{\mathcal{A}(x)}}$. Since $N^-$ also induces an isomorphism in homology, we conclude that $N_\varepsilon^{\eta,C}$, when restricted to $U^-_{\eta,L}$, also induces an isomorphism as a composition map.  

Next, let $R_{\max}$ be the maximum of $\sqrt{\mathcal{A}({\widetilde{x}})}R_{S_{\widetilde{x}}}$ over all non-trivial critical orbits $S_{\widetilde{x}}$ of $\PsiHeta$, where $R_{\widetilde{x}}> 0$ is the number associated to $S_{\widetilde{x}}$ in claim (5) of Lemma \ref{Psiheta}. Since there are finitely many critical orbits of $\PsiHeta$, we conclude that $R_{\max} > 0$ is finite.  

We now prove that  
\begin{equation}\label{componentequalityepsilon}
U^+_{\eta,L}\cap \left[R_{\max},+\infty\right)  \{\PsiC < L\}  =  \{\PsiHeta<\varepsilon\}\cap\left[R_{\max},+\infty\right) \{\PsiC < L\}.   
\end{equation}  

Due to the choice of $R_{\max}>0$ and claim (1) of Lemma \ref{Psiheta}, it follows that $\PsiHeta < 0$ on $[R_{\max},+\infty) \cap \crit(\PsiC)^{<T_\eta}$. Therefore, using claim (4.c) of Lemma \ref{Psiheta}, we conclude that the action of all elements in $[R_{\max},+\infty) \cap \{\PsiC < L\}$ must be strictly larger than the action of all critical points of $\PsiHeta$. Consequently, claim (4.a) of Lemma \ref{Psiheta} implies that for an arbitrary $x\in [R_{\max},+\infty) \cap \{\PsiC < L\}$, there exists $\lambda'\in (0,1)$ such that $\PsiHeta(\lambda' x)<\varepsilon$.

In particular, for $x\in \{\PsiHeta<\varepsilon\}\cap \left[R_{\max},+\infty\right)  \{\PsiC < L\} $, since $\PsiHeta(\lambda' x) > \varepsilon$ for $\lambda'\in (0,1)$ and $\PsiHeta(x) < \varepsilon$, claim (4.b) of Lemma \ref{Psiheta} implies that $\PsiHeta(\lambda x)$ is decreasing on $[1,+\infty)$. This shows that $x \in U^+_{\eta,L}$. Since the reverse inclusion is trivial, we conclude that \eqref{componentequalityepsilon} holds.  

By the defining property of $U^+_{\eta,L}$, it follows that $ U^+_{\eta,L}\cap[R_{\max},+\infty)\{\PsiC < L\}$ is a strong deformation retract of $U^+_{\eta,L}$. Thus, by \eqref{componentequalityepsilon}, the set $\{\PsiHeta<\varepsilon\} \cap \left[R_{\max},+\infty\right) \cap \{\PsiC < L\}$ is a strong deformation retract of $U^+_{\eta,L}$, where the retraction  
\[
R^+: U^+_{\eta,L} \to \{\PsiHeta<\varepsilon\} \cap \left[R_{\max},+\infty\right) \cap \{\PsiC < L\} 
\]  
is given by radial projection onto the boundary $\{\PsiHeta<\varepsilon\} \cap R_{\max} \{\PsiC < L\}$. This yields the following commuting diagram:  
\[
\begin{tikzcd}[scale cd=0.9]
    & \{\PsiHeta<\varepsilon\} \cap \left[R_{\max},+\infty\right) \{\PsiC < L\}  \arrow[dr,"N^+"] \\
    U^+_{\eta,L} \arrow[ur,"R^+"] \arrow[rr,"N_\varepsilon^{\eta,C}"] && \{\PsiC < L\}
\end{tikzcd}
\]  

Let us denote by $V_{S_{\overline{x}}}=\frac{1}{\sqrt{\mathcal{A}(\widetilde{x})}}U_{S_{\widetilde{x}}} \subset \{\PsiC<L\}$ a neighborhood of the critical orbit $S_{\overline{x}}\subset \crit(\PsiC)$. Since, by the choice of $R_{\max}$, we have $\PsiHeta < 0$ on $[R_{\max},+\infty) V_{S_{\overline{x}}}$ for every critical orbit $S_{\overline{x}}\subset \crit(\PsiC)^{<T_\eta}$, one can show, using the same methods as in the proof of Lemma \ref{norml}, that for sufficiently small $\delta>0$, the space $\{\PsiHeta\leq\varepsilon-\delta\}\cap \left[R_{\max},+\infty\right) \{\PsiC < L\}$ is a strong deformation retract of $\{\PsiHeta<\varepsilon\}\cap \left[R_{\max},+\infty\right) \{\PsiC < L\}$ and $\left[R_{\max},+\infty\right) \{\PsiC < L\}$. This implies that $N^+$ induces an isomorphism in homology. Since $R^+$ also induces an isomorphism in homology, it follows that $N_\varepsilon^{\eta,C}$, when restricted to $U^+_{\eta,L}$, induces an isomorphism in homology as a composition of these maps.

\textbf{Claim} (3): Let $x \in U^-_{\eta, L_\eta}$. Then for every $\lambda \in (0,1]$, we have $\lambda x \in U^-_{\eta, L_\eta}$. Therefore, $i^{H, \varepsilon}_{\nu, \eta}(\lambda x) = \lambda x \in \{\PsiHnu < \varepsilon\} \cap \mathbb{R}_+ \{\PsiC < L_\eta\}$, so we conclude that $i^{H, \varepsilon}_{L_\nu, L_\eta}(x) \in U^-_{\nu, L_\nu}$. Similarly, we can show that $i^{H, \varepsilon}_{\nu, \eta}$ maps $U^+_{\eta, L_\eta}$ to $U^+_{\nu, L_\eta}$.

\end{proof}

\begin{remark}\label{inclusionles}
    Notice that from the previous two lemmas, it follows that the inclusion 
    \[
    i_L: \{\PsiHeta < \varepsilon\} \cap \mathbb{R}_+ \{\PsiC < L\} \hookrightarrow \{\PsiHeta < L\} \cap \mathbb{R}_+ \{\PsiC < L\}
    \]
    induces a surjective map in singular and $S^1$-equivariant homology. Indeed, by the previous lemma, we have that $N_\varepsilon^{\eta, C}$ induces a surjective map in homology. On the other hand, from Lemma \ref{norml}, we know that $N_L^{\eta, C}$ induces an isomorphism in homology. Now, from the commutative diagram  
    \[
    \begin{tikzcd}[scale cd=0.9]
     & \{\PsiHeta < L\} \cap \mathbb{R}_+ \{\PsiC < L\} \arrow[dr, "N_L^{\eta, C}"] \\
    \{\PsiHeta < \varepsilon\} \cap \mathbb{R}_+ \{\PsiC < L\} \arrow[ur, "i_L"] \arrow[rr, "N_\varepsilon^{\eta, C}"] && \{\PsiC < L\}
    \end{tikzcd}
    \]
    the conclusion follows.
\end{remark}

Now we have all the topological data we need to prove the Proposition \ref{isomorphismclarkeduality}.  

\begin{lemma}\label{algebraiclemma} 
\begin{enumerate}
    \item Let 
    \[\begin{tikzcd}
      0 \arrow[r] & A \arrow[r, "i"] & C^-\oplus C^+ \arrow[r, "p"] & B \arrow[r] & 0 
    \end{tikzcd}\]
    be a short exact sequence in an abelian category such that $p|_{C^-}$ and $p|_{C^+}$ are isomorphisms. Let $\pi^-:C^-\oplus C^+ \to C^-$ denote the projection. Then 
    \[D:B \to A, \quad D = (p \circ \pi^- \circ i)^{-1},\]
    is an isomorphism.
    
    \item Let 
    \[\begin{tikzcd}
      0 \arrow[r] & A_2 \arrow[r, "i_2"] & C^-_2\oplus C^+_2 \arrow[r, "p_2"] & B_2 \arrow[r] & 0 \\
      0 \arrow[r] & A_1 \arrow[u, "i_A"] \arrow[r, "i_1"] & C^-_1\oplus C^+_1 \arrow[u, "i_C"] \arrow[r, "p_1"] & B_1 \arrow[u, "i_B"] \arrow[r] & 0
    \end{tikzcd}\]
    be a commutative diagram of short exact sequences in an abelian category such that these sequences satisfy the properties of (1), and let $i_C$ preserve splitting, i.e., $i_C(C_1^-) \subseteq C_2^-$ and $i_C(C_1^+) \subseteq C_2^+$. Then for the isomorphisms 
    \[D_1: B_1 \to A_1\] 
    and 
    \[D_2: B_2 \to A_2\] 
    defined as in (1), the following diagram
    \[\begin{tikzcd}
       B_2 \arrow[r, "D_2"] & A_2 \\
       B_1 \arrow[u, "i_B"] \arrow[r, "D_1"] & A_1 \arrow[u, "i_A"]
    \end{tikzcd}\]
    commutes.
\end{enumerate}
\end{lemma}

\begin{proof}
\textbf{Claim }(1): Since
\[\begin{tikzcd}
  0 \arrow[r] & A \arrow[r, "i"] & C^-\oplus C^+ \arrow[r, "p"] & B \arrow[r] & 0 
\end{tikzcd}\]
is a short exact sequence, we have that 
\[i: A \to \ker(p)\] is an isomorphism. Thus, it is sufficient to show that 
\[p \circ \pi^-|_{\ker(p)}: \ker(p) \to B\]
is an isomorphism. We will do this by finding an inverse map.

We have that \[\ker(p) = \{(c^-, c^+) \in C^-\oplus C^+ \mid p(c^-) = -p(c^+)\}.\]

Since $p|_{C^-}$ and $p|_{C^+}$ are isomorphisms, we have that 
\[(p \circ \pi^-|_{\ker(p)})^{-1}: B \to \ker(p)\]
is given by
\[(p \circ \pi^-|_{\ker(p)})^{-1}(b) = (p|_{C^-}^{-1}(b), -p|_{C^+}^{-1}(b)).\]

By the definition of $\ker(p)$, one can easily check that this is a well-defined homomorphism. Additionally, it is clear that this is indeed an inverse map of $p \circ \pi^-|_{\ker(p)}$. Therefore, we have that $p \circ \pi^- \circ i: A \to B$ is an isomorphism, which implies that $D = (p \circ \pi^- \circ i)^{-1}: B \to A$ is indeed an isomorphism. 

\textbf{Claim }(2): Let 
\[\begin{tikzcd}
  0 \arrow[r] & A_2 \arrow[r, "i_2"] & C^-_2\oplus C^+_2 \arrow[r, "p_2"] & B_2 \arrow[r] & 0 \\
  0 \arrow[r] & A_1 \arrow[u, "i_A"] \arrow[r, "i_1"] & C^-_1\oplus C^+_1 \arrow[u, "i_C"] \arrow[r, "p_1"] & B_1 \arrow[u, "i_B"] \arrow[r] & 0
\end{tikzcd}\]
be the commutative diagram from the claim. We have that 
\[D_1 = (p_1 \circ \pi_1^- \circ i_1)^{-1}\]
and 
\[D_2 = (p_2 \circ \pi_2^- \circ i_2)^{-1}.\]
are isomorphisms. From the commutativity of the diagram and the fact that $i_C$ preserves splitting, we have 
\[i_B \circ (D_1)^{-1} = i_B \circ p_1 \circ \pi_1^- \circ i_1 = p_2 \circ i_C \circ \pi_1^- \circ i_1 = p_2 \circ \pi_2^- \circ i_C \circ i_1 = p_2 \circ \pi_2^- \circ i_2 \circ i_A = (D_2)^{-1} \circ i_A,\]
which implies that the diagram
\[\begin{tikzcd}
   B_2 \arrow[r, "D_2"] & A_2 \\
   B_1 \arrow[u, "i_B"] \arrow[r, "D_1"] & A_1 \arrow[u, "i_A"]
\end{tikzcd}\]
commutes.
\end{proof}

\begin{proof}[Proof of Proposition~\ref{isomorphismclarkeduality}]
\textbf{Claim }(1): From Remark \ref{inclusionles}, we have that the inclusion 
\[i_L: \{\PsiHeta < \varepsilon\} \cap \mathbb{R}_+ \{\PsiC < L\} \hookrightarrow \{\PsiHeta < L\} \cap \mathbb{R}_+ \{\PsiC < L\}\]
induces a surjection in homology. Therefore, the map 
\[j_L: (\{\PsiHeta < L\} \cap \mathbb{R}_+ \{\PsiC < L\},\emptyset) \to (\{\PsiHeta < L\} \cap \mathbb{R}_+ \{\PsiC < L\}, \{\PsiHeta < \varepsilon\} \cap \mathbb{R}_+ \{\PsiC < L\})\]
is trivial in homology. This implies that the long exact sequence for the pair 
\[(\{\PsiHeta < L\} \cap \mathbb{R}_+ \{\PsiC < L\}, \{\PsiHeta < \varepsilon\} \cap \mathbb{R}_+ \{\PsiC < L\})\]
gives us a short exact sequence 
\begin{multline*}
 0 \longrightarrow H_{* + 1}(\{\PsiHeta < L\} \cap \mathbb{R}_+ \{\PsiC < L\}, \{\PsiHeta < \varepsilon\} \cap \mathbb{R}_+ \{\PsiC < L\}) \xlongrightarrow[]{\partial_L} \\ \xlongrightarrow[]{\partial_L} H_*(\{\PsiHeta < \varepsilon\} \cap \mathbb{R}_+ \{\PsiC < L\}) \xlongrightarrow[]{i_L} H_*(\{\PsiHeta < L\} \cap \mathbb{R}_+ \{\PsiC < L\}) \rightarrow 0
\end{multline*}

From Lemma \ref{inclusioncutoff}, we have that 
\[i^{\eta,C}_L: \left(\{\PsiHeta < L\} \cap \mathbb{R}_+ \{\PsiC < L\}, \{\PsiHeta < \varepsilon\} \cap \mathbb{R}_+ \{\PsiC < L\}\right) \to \left(\{\PsiHeta < L\}, \{\PsiHeta < \varepsilon\}\right)\]
induces an isomorphism in homology. From Lemma \ref{norml}, we have that 
\[N_L^{\eta,C}: \{\PsiHeta < L\} \cap \mathbb{R}_+ \{\PsiC < L\} \to \{\PsiC < L\}\]
induces an isomorphism in homology. Composing $(i_L)_*$ with $(N_L^{\eta,C})_*$ gives us $(N_\varepsilon^{\eta,C})_*$. Additionally, from Lemma \ref{UminusUplus}, we have that 
\[H_*(\{\PsiHeta < \varepsilon\} \cap \mathbb{R}_+ \{\PsiC < L\}) = \H_*(U^-_{\eta,L}) \oplus \H_*(U^+_{\eta,L}).\] 
Combining all of this, we get the short exact sequence 
\[\begin{tikzcd}[scale cd=0.9]
  0 \arrow[r] & 
  \H_*(\{\PsiHeta < L\}, \{\PsiHeta < \varepsilon\}) \arrow[r, "\widetilde{\partial}_L"] & \H_*(U^-_{\eta,L}) \oplus \H_*(U^+_{\eta,L})  \arrow[r, "(N_\varepsilon^{\eta,C})_*"] & \H_*(\{\PsiC < L\}) \arrow[r] & 0 
\end{tikzcd}
\]
where $\widetilde{\partial}_L = \partial_L \circ (i^{\eta,C}_L)_*^{-1}$. 
Since from claim (2) of Lemma \ref{UminusUplus}, we have that $(N_\varepsilon^{\eta,C})_*$ restricted to $\H_*(U^-_{\eta,L})$ and $\H_*(U^+_{\eta,L})$ is an isomorphism, the claim (1) of Lemma \ref{algebraiclemma} implies that we have an isomorphism 
\[D^{H^*_\eta}_L: \H_*(\{\PsiC < L\}) \to \H_{*+1}(\{\Psi_{H_\eta^*} < L\}, \{\Psi_{H_\eta^*} < \varepsilon\}), \ \ \  L \in (\eta, T_\eta),\]
defined as 
\[D^{H^*_\eta}_L = ((N_\varepsilon^{\eta,C})_* \circ \pi^-_L \circ \widetilde{\partial}_L)^{-1},\] 
where 
\[\pi^-_L: \H_*(U^-_{\eta,L}) \oplus \H_*(U^+_{\eta,L}) \to \H_*(U^-_{\eta,L})\]
is the projection.

\textbf{Claim }(2): Let $H_\eta \leq H_\nu$. Let $L_\eta \in (\eta, T_\eta)$ and $L_\nu \in (\nu, T_\nu)$ where $L_\eta \leq L_\nu$. We have that the diagram 
\[\begin{tikzcd}
   \{\PsiHnu < L_\nu\} \cap \mathbb{R}_+ \{\PsiC < L_\nu\}  \arrow[r, "N_{L_\nu}^{\eta,C}"] & \{\PsiC < L_\nu\}  \\
  \{\PsiHeta < L_\eta\} \cap \mathbb{R}_+ \{\PsiC < L_\eta\} \arrow[u, ""] \arrow[r, "N_{L_\eta}^{\eta,C}"] & \{\PsiC < L_\eta\} \arrow[u, ""]
\end{tikzcd}
\]
commutes, where the vertical maps are inclusions. Also, we have that the diagram 
\[\begin{tikzcd}[scale cd=0.9]
   (\{\PsiHnu < L_\nu\} \cap \mathbb{R}_+ \{\PsiC < L_\nu\}, \{\PsiHnu < \varepsilon\} \cap \mathbb{R}_+ \{\PsiC < L_\nu\})  \arrow[r, "i^{\eta,C}_{L_\nu}"] & (\{\PsiHnu < L_\nu\}, \{\PsiHnu < \varepsilon\})   \\
  (\{\PsiHeta < L_\eta\} \cap \mathbb{R}_+ \{\PsiC < L_\eta\}, \{\PsiHeta < \varepsilon\} \cap \mathbb{R}_+ \{\PsiC < L_\eta\})  \arrow[r, "i^{\eta,C}_{L_\eta}"] \arrow[u, ""] & (\{\PsiHeta < L_\eta\}, \{\PsiHeta < \varepsilon\}) \arrow[u, ""]
\end{tikzcd}
\]
commutes, where the vertical maps are inclusions. Due to the commutativity of these two diagrams, the fact that $i^{\eta,C}_{L_\eta}$ and $i^{\eta,C}_{L_\nu}$ induce isomorphisms in homology, and the functoriality property of long exact sequences for topological pairs, we have that the diagram

\[\begin{tikzcd}[scale cd=0.9]
  0 \arrow[r] & 
  \H_*(\{\PsiHeta < L_\nu\}, \{\PsiHnu < \varepsilon\}) \arrow[r, "\widetilde{\partial}_{L_\nu}"] & \H_*(U^-_{\nu,L_\nu}) \oplus \H_*(U^+_{\nu,L_\nu}) \arrow[r, "(N_\varepsilon^{\nu,C})_*"] & \H_*(\{\PsiC < L_\nu\}) \arrow[r] & 0 \\
  0 \arrow[r] & 
  \H_*(\{\PsiHeta < L_\eta\}, \{\PsiHeta < \varepsilon\}) \arrow[u, "(i^H_{L_\nu,L_\eta})_*"] \arrow[r, "\widetilde{\partial}_{L_\eta}"] & \H_*(U^-_{\eta,L_\eta}) \oplus \H_*(U^+_{\eta,L_\eta}) \arrow[u, "(i^{H,\varepsilon}_{L_\nu,L_\eta})_*"] \arrow[r, "(N_\varepsilon^{\eta,C})_*"] & \H_*(\{\PsiC < L_\eta\}) \arrow[u, "(i^C_{L_\nu,L_\eta})_*"] \arrow[r] & 0.
\end{tikzcd}
\]
commutes. Since from claim (3) of Lemma \ref{UminusUplus} we have that $(i^{H,\varepsilon}_{L_\nu,L_\eta})_*$ preserves splitting, it follows from claim (2) Lemma \ref{algebraiclemma} that the diagram 
\[ \begin{tikzcd}
 \H_*(\{\PsiC < L_\nu\}) \arrow{r}{D^{H^*_\nu}_{L_\nu}} & \H_{*+1}(\{\Psi_{H_{\nu}^*} < L_\nu\}, \{\Psi_{H_{\nu}^*} < \varepsilon\}) \\%
 \H_*(\{\PsiC < L_\eta\}) \arrow{r}{D^{H^*_\eta}_{L_\eta}}  \arrow[swap]{u}{(i^C_{L_\nu,L_\eta})_*} & \H_{*+1}(\{\Psi_{H_{\eta}^*} < L_\eta\}, \{\Psi_{H_{\eta}^*} < \varepsilon\}) \arrow[swap]{u}{(i^H_{L_\nu,L_\eta})_*}
\end{tikzcd}
\]
commutes.
\end{proof}

\section{Positive ($S^1$-equivariant) symplectic homology and
autonomous Hamiltonians}\label{sectionautonomussymplectic}

In this section, we compare the positive ($S^1$-equivariant) symplectic homology of the convex domain with the singular ($S^1$-equivariant) homology of sublevel sets of a dual functional of the form
\[\PsiH(x) = -\frac{1}{2}\int\limits_\mathbb{T}\langle J_0 x(t), \dot{x}(t) \rangle \, dt + \int\limits_\mathbb{T} H^*(-J_0 \dot{x}) \, dt,\]
for particular types of Hamiltonians.

Let $C \subset \mathbb{R}^{2n}$ be a non-degenerate, strongly convex domain whose interior contains the origin (see Section \ref{sectionautonomusclarke} for the definitions of these properties).

\textbf{Admissible approximating functions:}

A smooth function $\varphi: \mathbb{R}_{\geq 0} \to \mathbb{R}$ belongs to $\mathcal{F}_{\operatorname{lin}}(C)$ if it satisfies:

\begin{enumerate}
    \item For some $0 < s_\varepsilon < 1$, $\varphi(s) = \delta s - \zeta_\varepsilon$ for $s \in [0, s_\varepsilon]$, where $0 < \zeta_\varepsilon < \varepsilon$ and $0 < \delta < \frac{\varepsilon}{4}$.
    \item For some $s_\eta > 1$, $\varphi(s) = \eta s - \zeta_\eta$ for $s \in [s_\eta, +\infty)$, where $\eta \notin \sigma(\partial C)$ and $\eta > \min \sigma(\partial C)$.
    \item $\varphi''(s) > 0$ for $s \in (s_\varepsilon, s_\eta)$, and $\varphi(1) < 0$ and $\varphi'(1) < \min \sigma(\partial C)$.
    \item For $s_{\max} \in (s_\varepsilon, s_\eta)$ such that $\varphi'(s_{\max}) = T_{\max} = \max \sigma(\partial C)^{<\eta}$, it holds that
    \[\varphi'(s_{\max}) s_{\max} - \varphi(s_{\max}) < \eta\]
    and
    \[\varphi'(s_{\max})(s_{\max} - 1) - \varphi(s_{\max}) < \min\left\{\frac{1}{\eta}, \frac{1}{3} \text{gap}(\eta)\right\},\]
    where $\text{gap}(\eta)$ is the minimal distance between two elements in $\sigma(\partial C)^{<\eta}$.
\end{enumerate}

We will emphasize the slope $\eta$ of $\varphi$ by writing $\varphi_\eta \defeq \varphi$.

\begin{remark}\label{admisiblefunctionsproperties}
    We claim that for $\varphi \in \mathcal{F}_{\operatorname{lin}}(C)$, the inequality
    \begin{equation}\label{graterthanepsilon}
        \varphi'(s) s - \varphi(s) > \varphi'(s), \quad s > 1
    \end{equation}
    holds. Indeed, from (2) and (3), we have that $g(s) = \varphi'(s)(s - 1) - \varphi(s)$ is a non-decreasing function for $s > 1$. On the other hand, from (3), we have that $g(1) > 0$. Therefore, \eqref{graterthanepsilon} holds. From (3), $\varphi': (s_\varepsilon, s_\eta) \to (\delta, \eta)$ is an increasing bijection. Thus, $s_{\max}$ is well-defined and unique. Additionally, $\varphi'(1) < \min \sigma(\partial C)$, which implies that there exists a unique $s_{\min} > 1$ such that $\varphi'(s_{\min}) = T_{\min} = \min \sigma(\partial C)^{<\eta} = \min \sigma(\partial C)$. From \eqref{graterthanepsilon}, it follows that $\varphi'(s_{\min}) s_{\min} - \varphi(s_{\min}) > T_{\min} = 2 \varepsilon$. Using this conclusion, the fact that $\varphi'(s) s - \varphi(s)$ is increasing, and (4), we conclude that
    \begin{equation}\label{nontrivialorbitsactionbounds}
       2\varepsilon < \varphi'(s) s - \varphi(s) < \eta, \quad s \in [s_{\min}, s_{\max}]
    \end{equation}
    Therefore, $\mathcal{F}_{\operatorname{lin}}(C) \subset \mathcal{F}^*_{\operatorname{lin}}(C)$.
    
    Moreover, the increasing property of the function $g(s)$, its positivity for $s > 1$, and (4) imply that
    \begin{equation}\label{admissiblegap}
        |\varphi'(s)(s - 1) - \varphi(s)| < \min\left\{\frac{1}{\eta}, \frac{1}{3} \text{gap}(\eta)\right\}, \quad s \in [s_{\min}, s_{\max}]
    \end{equation}  
\end{remark} 

Now, we state the main Proposition of this section.

\begin{proposition}\label{isomorphismautomussymplectic}
  Let $C$ be a non-degenerate strongly convex domain whose interior contains the origin. Let $\varphi_\eta \in \mathcal{F}_{\operatorname{lin}}(C)$ and let $T_\eta$ be the minimal element of $\sigma(\partial C)$ greater than $\eta$, and $H_\eta = \varphi_\eta \circ \HC$. The following holds:   

  \begin{enumerate}
    \item[$(1.1)$] For every $ L \in (\eta, T_\eta)$, there exists an isomorphism
    \[D^{H_\eta}_L: \H_{*}(\{\Psi_{H_\eta} < L\}, \{\Psi_{H^*_\eta} < \varepsilon\}; \mathbb{Z}_2) \to \S\H^{+,<L}_{*+n}(C; \mathbb{Z}_2).\]
    \item[$(1.2)$] This isomorphism is natural. Let $H_{\eta} \leq H_{\nu}$. Then, for every $L_\eta \in (\eta, T_\eta)$ and $L_\nu \in (\nu, T_\nu)$ such that $L_\eta \leq L_\nu$, the following diagram commutes:
    \[ 
    \begin{tikzcd}
      \H_{*}(\{\Psi_{H_{\nu}} < L_\nu\}, \{\Psi_{H_{\nu}} < \varepsilon\}; \mathbb{Z}_2) \arrow{r}{D^{H_\nu}_{L_\nu}} & \S\H^{+, < L_\nu}_{*+n}(C; \mathbb{Z}_2) \\
      \H_{*}(\{\Psi_{H_{\eta}} < L_\eta\}, \{\Psi_{H_{\eta}} < \varepsilon\}; \mathbb{Z}_2) \arrow{r}{D^{H_\eta}_{L_\eta}} \arrow[swap]{u}{(i^{H}_{L_\nu, L_\eta})_*} & \S\H^{+, < L_\eta}_{*+n}(C; \mathbb{Z}_2) \arrow[swap]{u}{i^{+}_{L_\nu, L_\eta}}
    \end{tikzcd}.
    \]
  \end{enumerate}
\end{proposition}

\begin{remark}\label{isomorphismautomussymplecticS1}
    The same claim holds if one replaces singular homology with $S^1$-equivariant homology and positive symplectic homology with positive $S^1$-equivariant symplectic homology. Unlike in the previous section, the proof of the $S^1$-equivariant case differs. We will emphasize these differences throughout the section.
\end{remark}

For definitions of ($S^1$-equivariant) symplectic homology and their positive versions, we refer the reader to \cite{Vit99, BO17, GH18}.

We define another family with $F_{\infty}(C)$ of smooth functions $\varphi: \mathbb{R} \to \mathbb{R}$ such that

\begin{enumerate}
    \item There exists $\varphi_\nu \in \mathcal{F}_{\operatorname{lin}}(C)$ such that $\varphi|_{[1, +\infty)} = \varphi_\nu|_{[1, +\infty)}$.
    \item $\varphi$ is increasing on $(0, +\infty)$ and $\varphi''(s) > 0$ for $s \in (0, s_\eta)$.
    \item $\varphi(s) = \varphi(0) > \varepsilon$ for $s \leq 0$.
\end{enumerate}

Notice that both families are non-empty for every $\eta \notin \sigma(\partial C)$.

Now, we define a family of admissible Morse-Bott Hamiltonians with respect to the family $\mathcal{F}_\infty(C)$ (denoted $\mathcal{H}_{MB}^\mathcal{F}(C)$) as the family of smooth Hamiltonians $H: \mathbb{R}^{2n} \to \mathbb{R}$ of the form
\[H = \varphi \circ H_C + f,\]
where $\varphi \in F_\infty(C)$ and $f: \mathbb{R}^{2n} \to \mathbb{R}$ is a $C^2$-small function such that it is supported inside $C$, and its second derivative is positive definite near the origin. Additionally, we require that $H$ is quadratically convex, $H|_{\overline{C}} < 0$, and $H(0) > -\varepsilon$.

\begin{remark}\label{cofinal}
    Notice that $\mathcal{H}_{MB}^{\mathcal{F}}(C)$ is a cofinal subset of $\mathcal{H}_{MB}(C)$ defined in \cite{GH18}. Moreover, $H \in \mathcal{H}^{\mathcal{F}}_{MB}(C)$ has a unique trivial critical orbit of $\Phi_H$ at $0$ and it holds that
    \[\Phi_H(0) < \varepsilon.\] Additionally, for non-trivial critical orbits $S_y$ of $\Phi_H$, from \eqref{nontrivialorbitsactionbounds}, we have that
    \[2 \varepsilon < \Phi_H(y) < \eta.\]
    Moreover, from \eqref{admissiblegap}, we have that
    \[|\Phi_H(y) - T_y| < \min\left\{\frac{1}{\eta}, \frac{1}{3} \text{gap}(\eta)\right\},\]
    where $T_y$ is the period of the corresponding closed Reeb orbit, which proves that $\mathcal{H}_{MB}^{\mathcal{F}}(C) \subseteq \mathcal{H}_{MB}(C)$.
\end{remark}

\large\textbf{Positive Symplectic homology}
\normalsize

As described in \cite{CFHW96} and \cite{BO09}, we can perturb $ H \in \mathcal{H}_{MB}(C) $ to a time-dependent Hamiltonian of the form

\[H'(t,x) = H(x) + \delta \sum_{S_y} \rho_{S_y}(x) f_{S_y}(t,x),\]

where the summation is taken over all non-trivial critical orbits. The functions $ \rho_{S_y} $ and $ f_{S_y} $ are defined locally near the given critical orbit, with $ \rho_{S_y} $ serving as a support function. Moreover, the following claim holds.

\begin{lemma}\label{MorseBottperturbation}
    For $H \in \mathcal{H}_{MB}(C)$ and a small enough $\delta > 0$, a time-dependent perturbation $H'$ of $H$ satisfies the following:

    \begin{enumerate}
        \item The constant orbits of $H'$ are the same as those of $H$.
        \item For every non-trivial critical orbit $S_y$ of $H$, there exists a pair of non-degenerate critical orbits of $H'$, denoted $(y_{\min}, y_{\max})$, which belong to $S_y$ and satisfy
        \[\mu_{CZ}(y_{\min}) = \mu_{CZ}(y), \quad \text{and} \quad \mu_{CZ}(y_{\max}) = \mu_{CZ}(y) + 1.\]
    \end{enumerate}

\end{lemma}

The proof of this lemma can be found in \cite{CFHW96} and \cite{BO09}.

Perturbations of this type appear in Morse theory for functions of Morse-Bott type (see \cite{BH09, BH13}). Instead of working with the Floer complex of perturbed Morse-Bott Hamiltonians, one can work with cascades and Morse-Bott moduli spaces (see \cite{BO09}). These two approaches are equivalent. Similar results can be obtained in the Morse case, as shown in \cite{BH13}.

A Hamiltonian $H' \in \mathcal{H}^{\mathcal{F}}(C)$ is an admissible perturbation of $H \in \mathcal{H}_{MB}^{\mathcal{F}}(C)$ if
\[
H'(t,x) = H(x) + \delta \sum_{S_y} \rho_{S_y}(x) f_{S_y}(t,x),
\]
and the following conditions hold:

\begin{itemize}
    \item Conditions (1) and (2) of Lemma \ref{MorseBottperturbation} are satisfied.
    \item The time-dependent perturbations are supported outside of $\overline{C}$.
    \item $H'$ is quadratically convex.
    \item For the pair $(y_{\min}, y_{\max})$ corresponding to $S_y$, it holds that
    \[|\Phi_{H'}(y_{\min}) - T_y| < \min\left\{\frac{1}{\eta}, \frac{1}{3}\text{gap}(\eta)\right\}, \quad \text{and} \quad |\Phi_{H'}(y_{\max}) - T_y| < \min\left\{\frac{1}{\eta}, \frac{1}{3}\text{gap}(\eta)\right\}.\]
    \item For every $y \in \crit(\Phi_H) \setminus \{0\}$, it holds that
    \[\varepsilon < \Phi_H(y) < \eta.\]
\end{itemize}

For $H \in \mathcal{H}^\mathcal{F}(C)$, we define the positive Floer complex as
\[\{\C\F_*^+(H), \partial^F\},\]

where
\[\C\F^+(H) \defeq \C\F(H) / \C\F^{<\varepsilon}(H).\]

This is well-defined since $\Phi_H$ is non-increasing along trajectories. The filtration on this complex by $\Phi_H$ is well-defined for the same reason. The homology of this complex is denoted by
\[\H\F^+_*(H, J)\defeq \H_*(\{\C\F_*^+(H), \partial^F\}).\]

This homology is independent of the almost complex structure.

If $(H_\alpha, J_\alpha)$ and $(H_\beta, J_\beta)$ are such that $H_\alpha, H_\beta \in \mathcal{H}^{\mathcal{F}}(C)$ and $H_\alpha \leq H_\beta$, then a generic non-increasing homotopy $(H^{\alpha\beta}_s, J^{\alpha\beta}_s)$ induces a chain map described in Section \ref{sectionfloer}. This chain map induces a continuation map
\[i^+_{H_\beta, H_\alpha}: \H\F^+_*(H_\alpha, J_\alpha) \to \H\F^+_*(H_\beta, J_\beta),\]
which gives us a directed system. Here, $\alpha$ and $\beta$ do not represent the slope.

The set $\mathcal{H}^{\mathcal{F}}(C)$ is cofinal in the set of admissible non-degenerate Hamiltonians (see \cite{GH18}). Therefore, we define the filtered positive symplectic homology as
\[SH^{+,<L}(C) = \varinjlim_{H \in \mathcal{H}^{\mathcal{F}}(C)} HF^{+,<L}(H).\]

\begin{remark}\label{comutationofdiagramfloermorse}
Since $\mathcal{H}^{\mathcal{F}}(C) \subset \mathcal{H}_\Theta$, for $H$ its corresponding $l$-reduced dual functional is $\psi^l_{H^*}$. We analogously define the positive Morse complex as the quotient complex
\[\C\M^+(\psi^l_{H^*}) \defeq \C\M(\psi^l_{H^*}) / \C\M^{<\varepsilon}(\psi^l_{H^*}).\]

The claim (1) of Theorem \ref{theoremE1} and the fact that corresponding critical points of functionals have the same action imply that the map induced from $\Theta$ in positive homology, denoted
\[\Theta^+_{\H}: \H\M_{*}^+(\psi^l_{H^*}, g) \to \H\F_{*+n}^+(H, J),\]
is an isomorphism that preserves action filtration.

Let $H_\alpha \leq H_\beta$ and $H_\alpha, H_\beta \in \mathcal{H}^\mathcal{F}(C)$. Assuming that the $l$-saddle point reduction exists for both dual functionals, we have from claim (2) of Theorem \ref{theoremE2} that the following diagram

\[
\begin{tikzcd}
\H\M_{*}^+(\psi^l_{H_\beta^*}, g_\beta) \arrow{r}{\Theta^+_{\beta, \H}} & \H\F_{*+n}^+(H_\beta, J_\beta) \\
\H\M_{*}^+(\psi^l_{H_\alpha^*}, g_\alpha) \arrow{r}{\Theta^+_{\alpha, \H}} \arrow[swap]{u}{i^{M,+}_{H_\beta^*, H_\alpha^*}} & \H\F_{*+n}^+(H_\alpha, J_\alpha) \arrow[swap]{u}{i^{F,+}_{H_\beta, H_\alpha}}
\end{tikzcd}
\]

commutes, where vertical maps are continuation maps.
    
\end{remark}

\begin{lemma}\label{positivesymplecticiso}
    Let $C$ be a non-degenerate, strongly convex domain whose interior contains the origin. Let $H_\eta \in \mathcal{H}^\mathcal{F}(C)$ be arbitrary, where $\eta$ is the slope at infinity and $T_\eta$ is the minimal element of $\sigma(\partial C)$ greater than $\eta$. The following holds:
    \begin{enumerate}
        \item For every $L \in (\eta, T_\eta)$, there exists an isomorphism
        \[D^{H_\eta}_L: \H_{*}(\{\Psi_{H^*_\eta} < L\}, \{\Psi_{H^*_\eta} < \varepsilon\}; \mathbb{Z}_2) \to SH^{+,<L}_{*+n}(C; \mathbb{Z}_2).\]
        \item This isomorphism is natural. Let $H_\eta \leq H_\nu$. For every $L_\eta \in (\eta, T_\eta)$ and $L_\nu \in (\nu, T_\nu)$ such that $L_\eta \leq L_\nu$, the following diagram

        \[
        \begin{tikzcd}
        \H_{*}(\{\Psi_{H_{\nu}} < L_\nu\}, \{\Psi_{H_{\nu}} < \varepsilon\}; \mathbb{Z}_2) \arrow{r}{D^{H_\nu}_{L_\nu}} & SH^{+,<L_\nu}_{*+n}(C; \mathbb{Z}_2) \\
        \H_{*}(\{\Psi_{H_{\eta}} < L_\eta\}, \{\Psi_{H_{\eta}} < \varepsilon\}; \mathbb{Z}_2) \arrow{r}{D^{H_\eta}_{L_\eta}} \arrow[swap]{u}{(i^{H}_{L_\nu, L_\eta})_*} & SH^{+,<L_\eta}_{*+n}(C; \mathbb{Z}_2) \arrow[swap]{u}{i^{+}_{L_\nu, L_\eta}}
        \end{tikzcd}
        \]

        commutes.
    \end{enumerate}
\end{lemma}

\begin{proof}

Let $H_\eta \in \mathcal{H}^\mathcal{F}(C)$ be arbitrary, and let $J$ be generically chosen. Since the slope of $H_\eta$ is $\eta$, it follows that the inclusion  
\[
\H\F^+_*(H_\eta, J) \to \S\H^{+, < \eta}_{*}(C)
\]
is an isomorphism. Moreover, for every critical point $y \in \text{crit}(\Phi_{H'_\eta})$, we have  
\[
\Phi_{H_\eta}(y) < \eta.
\]  

Combining this with the fact that there are no elements of $\sigma(\partial C)$ inside $[\nu, T_\nu)$, we conclude that  
\[
i^{\S\H^+, H_\eta}_L: \H\F^{+, < L}_*(H_\eta) \to \S\H^{+, < L}_{*}(C), \quad L \in (\eta, T_\eta)
\]
is an isomorphism.

Let $\psi^l_{H_\eta}$ be the $l$-reduced dual functional. For a generic $J$ and $g \in \mathcal{G}(\mathbb{H}_l)$ uniformly equivalent to the standard metric, we have, by Remark \ref{comutationofdiagramfloermorse}, that
\[\theta^{+, < L}_{\H} : \H\M^{+<L}_{*}(\psi^l_{H^*_\eta}, g) \to \H\F_{* + n}^{+, < L}(H_\eta, J), \quad L \in (\eta, T_\eta)\]
is an isomorphism. Additionally, by cellular filtration of the manifold, we have a natural isomorphism
\[h^L: \H_{*}(\{\psi^l_{H^*_\eta} < L\}, \{\psi^l_{H^*_\eta} < \varepsilon\}) \to \H\M^{+<L}_{*}(\psi^l_{H^*_\eta},g), \quad L \in (\eta, T_\eta).\]

Thus, for $L \in (\eta, T_\eta)$, we have the map
\[D^{H_\eta, l}_L: \H_{*}(\{\psi^l_{H^*_\eta} < L\}, \{\psi^l_{H^*_\eta} < \varepsilon\}) \to SH^{+, < L}_{* + n}(C), \quad D^{H_\eta, l} = i^{\S\H^+, H_\eta}_L \circ \theta^{+, < L}_{\H} \circ h^L,\]
which is an isomorphism.

For $H_\eta \in \mathcal{H}^\mathcal{F}(C)$, the map $D^{H_\eta, l}_L$ is defined for sufficiently large $l \in \mathbb{N}$, since the $l$-reduced dual functionals always exists if $l$ is large enough.

This map does not depend on the choice of $J$ and $g$. Indeed, let $J_1$ and $J_2$ be two different almost complex structures, and let $g_1$ and $g_2$ be two different metrics on $\mathbb{H}_l$ uniformly equivalent to the standard one. Then, the following diagrams
\[
\begin{tikzcd}
\H\F_*^{+, < L}(H_\eta, J_2) \arrow{r}{i^{\S\H^+, (H_\eta, J_2)}_L} & [3em] \S\H^{+, < L}_{*}(C) \\
\H\F_*^{+, < L}(H_\eta, J_1) \arrow{r}{i^{\S\H^+, (H_\eta, J_1)}_L} \arrow[swap]{u}{i^{F, +}_L} & \S\H^{+<L}_{*}(C) \arrow[u, equal]
\end{tikzcd},
\]
\[
\begin{tikzcd}
\H\M_{*}^{+, < L}(\psi^l_{H_\eta^*}, g_2) \arrow{r}{\Theta^{+, < L}_{2, \H}} & \H\F_{* + n}^{+, < L}(H_\eta, J_2) \\
\H\M_{*}^{+, < L}(\psi^l_{H_\eta^*}, g_1) \arrow{r}{\Theta^{+, < L}_{1, \H}} \arrow[swap]{u}{i^{M, +}_L} & \H\F_{* + n}^{+, < L}(H_\eta, J_1) \arrow[swap]{u}{i^{F, +}_L}
\end{tikzcd},
\]
\[
\begin{tikzcd}
\H_{*}(\{\psi^l_{H^*_\eta} < L\}, \{\psi^l_{H^*_\eta} < \varepsilon\}) \arrow{r}{h^{L, g_2}} & \H\M_{*}^{+, < L}(\psi^l_{H_\eta^*}, g_2) \\
\H_{*}(\{\psi^l_{H^*_\eta} < L\}, \{\psi^l_{H^*_\eta} < \varepsilon\}) \arrow{r}{h^{L, g_1}} \arrow[u, equal] & \H\M_{*}^{+, < L}(\psi^l_{H_\eta^*}, g_1) \arrow[swap]{u}{i^{M, +}_L}
\end{tikzcd}
\]
commute. The second diagram is precisely the diagram from Remark \ref{comutationofdiagramfloermorse} because all critical points have action smaller than $\eta$ and $L \in (\eta, T_\eta)$. Combining these three diagrams, we see that $D^{H_\eta, l}$ does not depend on the choice of $J$ and $g$.

Let $H_\eta, H_\nu \in \mathcal{H}^\mathcal{F}(C)$ be such that $H_\eta \leq H_\nu$. Let $l$ be large enough such that $D^{H_\eta,l}$ and $D^{H_\nu,l}$ are defined. Then, for every $L_\eta \in (\eta, T_\eta)$ and $L_\nu \in (\nu, T_\nu)$, we have the following commutative diagram:

\begin{equation}\label{commutativediagramisolreduction}
\begin{tikzcd}
 \H_{*}(\{\psi^l_{H_{\nu}} < L_\nu\}, \{\psi^l_{H_{\nu}} < \varepsilon\}; \mathbb{Z}_2) \arrow{r}{D^{H_\nu,l}_{L_\nu}}  & SH^{+ < L_\nu}_{*+n}(C; \mathbb{Z}_2)  \\%
\H_{*}(\{\psi^l_{H_{\eta}} < L_\eta\}, \{\psi^l_{H_{\eta}} < \varepsilon\}; \mathbb{Z}_2) \arrow{r}{D^{H_\eta,l}_{L_\eta}} \arrow[swap]{u}{(i^{H,l}_{L_\nu, L_\eta})_*} & SH^{+ < L_\eta}_{*+n}(C; \mathbb{Z}_2) \arrow[swap]{u}{i^{+}_{L_\nu, L_\eta}}
\end{tikzcd}   
\end{equation}

commutes, where $i^{H,l}_{L_\nu, L_\eta}$ is the inclusion of the corresponding pairs. This commutation follows from the diagrams:

\[ 
\begin{tikzcd}
 \H\F_*^{+,<L_\nu}(H_\nu, J_\nu) \arrow{r}{i^{\S\H^+, H_\nu}_{L_\nu}}  & \S\H^{+,<L_\nu}_{*}(C)  \\%
\H\F_*^{+,<L_\eta}(H_\eta, J_\eta) \arrow{r}{i^{\S\H^+, H_\eta}_{L_\eta}} \arrow[swap]{u}{i^{F,+}_{L_\nu, L_\eta}} & \S\H^{+<L_\eta}_{*}(C) \arrow[swap]{u}{i^{+}_{L_\nu, L_\eta}}
\end{tikzcd}
\]

\[ 
\begin{tikzcd}
 \H\M_{*}^{+,<L_\nu}(\psi^l_{H_\nu^*}, g_\nu) \arrow{r}{\Theta^{+,<L_\nu}_{\nu, \H}}  & \H\F_{*+n}^{+,<L_\nu}(H_\nu, J_\nu)  \\%
\H\M_{*}^{+,<L_\eta}(\psi^l_{H_\eta^*}, g_\eta) \arrow{r}{\Theta^{+,<L_\eta}_{\eta, \H}} \arrow[swap]{u}{i^{M,+}_{L_\nu, L_\eta}} & \H\F_{*+n}^{+,<L_\eta}(H_\eta, J_\eta) \arrow[swap]{u}{i^{F,+}_{L_\nu, L_\eta}}
\end{tikzcd}
\]

\[ 
\begin{tikzcd}
 \H_{*}(\{\psi^l_{H_\nu^*} < L_\nu\}, \{\psi^l_{H_\nu^*} < \varepsilon\}) \arrow{r}{h^{L_\nu}}  & \H\M_{*}^{+,<L_\nu}(\psi^l_{H_\nu^*}, g_\nu)  \\%
\H_{*}(\{\psi^l_{H_\eta^*} < L_\eta\}, \{\psi^l_{H_\eta^*} < \varepsilon\}) \arrow{r}{h^{L_\eta}} \arrow[swap]{u}{(i^{H, l}_{L_\nu, L_\eta})_*} & \H\M_{*}^{+,<L_\eta}(\psi^l_{H_\eta^*}, g_\eta) \arrow[swap]{u}{i^{M,+}_{L_\nu, L_\eta}}
\end{tikzcd}
\]

The second diagram is, as in the previous case, a diagram from Remark \ref{comutationofdiagramfloermorse}.

Now, let $H_\eta$ be a fixed Hamiltonian from the set $\mathcal{H}^\mathcal{F}(C)$. We define 

\[D^{H_\eta}_L: \H_{*}(\{\Psi_{H^*_\eta} < L\}, \{\Psi_{H^*_\eta} < \varepsilon\}; \mathbb{Z}_2) \to SH^{+,<L}_{*+n}(C; \mathbb{Z}_2), \quad L \in (\eta, T_\eta)\] 

as 

\[D^{H_\eta}_L = D^{H_\eta, l}_L \circ (\mathbb{P}^L_l)_*,\]

where $l \in \mathbb{N}$ is any number for which the $l$-reduced dual functional associated with $H_\eta$ exists, and 

\[\mathbb{P}_l^L: (\{\Psi_{H^*_\eta} < L\}, \{\Psi_{H^*_\eta} < \varepsilon\}) \to (\{\psi^l_{H^*_\eta} < L\}, \{\psi^l_{H^*_\eta} < \varepsilon\}),\]

is the restriction of the projection $\mathbb{P}_l$. The map $\mathbb{P}_l^L$ is well-defined and is a homotopy equivalence. Indeed, we have that 

\[\psi^l_{H_\eta^*}(x) = \Psi_{H_\eta^*}(\Gamma_l(x)),\]

where

\[\Gamma_{l}: \mathbb{H}_l \to \Honenull, \quad \Gamma_l(x) = (x, Y_l(x)),\]

and $Y_l(x)$ is the unique global minimizer of $\PsiHeta$ on $\{x\} \times \mathbb{H}_l$. Therefore, $\mathbb{P}_l^L$ is well-defined. On the other hand, $\PsiHeta$ is convex on $\{x\} \times \mathbb{H}^l$. Combining this with the fact that $Y_l(x)$ is the unique global minimizer, we conclude that $\mathbb{P}_l^L$ is a homotopy inverse of 

\[\Gamma_l^L: (\{\psi^l_{H^*_\eta} < L\}, \{\psi^l_{H^*_\eta} < \varepsilon\}) \to (\{\Psi_{H^*_\eta} < L\}, \{\Psi_{H^*_\eta} < \varepsilon\}).\]

Therefore, $\mathbb{P}_l^L$ is a homotopy equivalence and induces an isomorphism in homology.

Now, we need to show that the map $D^{H_\eta}_L$ is well-defined, i.e., that it does not depend on $l$.

Let $l_1, l_2 \in \mathbb{N}$ be two different numbers for which the $l_1$ and $l_2$-reduced dual functionals associated with $H_\eta$ exist. We can assume that $l_1 < l_2$. We have a natural orthogonal splitting of $\mathbb{H}_{l_2}$ given by 

\[\mathbb{H}_{l_2} = \mathbb{H}_{l_1} \oplus \mathbb{H}^{l_2, l_1},\]

where 

\[\mathbb{H}_{l_1} = \{x \in \Hone \mid \hat{x}(k) = 0 \text{ if } k < 1 \text{ or } k > l_1\},\]

and

\[\mathbb{H}^{l_2, l_1} = \{x \in \Hone \mid \hat{x}(k) = 0 \text{ if } k < l_1+1 \text{ or } k > l_2\}.\]

We use $x$ as a coordinate for $\mathbb{H}_{l_1}$ and $y$ as a coordinate for $\mathbb{H}^{l_2, l_1}$. With 

\[\Gamma_{l_2, l_1}: \mathbb{H}_{l_1} \to \mathbb{H}_{l_2}, \quad \Gamma_{l_2, l_1} = \mathbb{P}_{l_2} \circ \Gamma_{l_1} = (x, \mathbb{P}_{l_2} \circ Y_{l_1}(x)),\]

which is the graph of $\mathbb{P}_{l_2} \circ Y_{l_1}$. We want to show that the following diagram

\begin{equation}\label{diagraml1l2}
\begin{tikzcd}
 \H_{*}(\{\psi^{l_2}_{H_{\eta}} < L\}, \{\psi^{l_2}_{H_{\eta}} < \varepsilon\}; \mathbb{Z}_2) \arrow{r}{D^{H_\eta, l_2}_{L}}  & SH^{+, < L}_{*+n}(C; \mathbb{Z}_2)  \\%
 \H_{*}(\{\psi^{l_1}_{H_{\eta}} < L\}, \{\psi^{l_1}_{H_{\eta}} < \varepsilon\}; \mathbb{Z}_2) \arrow{r}{D^{H_\eta, l_1}_{L}} \arrow[swap]{u}{(\Gamma^L_{l_2, l_1})_*} & SH^{+, < L}_{*+n}(C; \mathbb{Z}_2) \arrow[u, equal]
\end{tikzcd}
\end{equation}

commutes, where 

\[\Gamma^L_{l_2, l_1}: (\{\psi^{l_1}_{H^*_\eta} < L\}, \{\psi^{l_1}_{H^*_\eta} < \varepsilon\}) \to (\{\psi^{l_2}_{H^*_\eta} < L\}, \{\psi^{l_2}_{H^*_\eta} < \varepsilon\}),\]

is a homotopy equivalence, with homotopy inverse being the map 

\[\mathbb{P}^L_{l_1, l_2}: (\{\psi^{l_2}_{H^*_\eta} < L\}, \{\psi^{l_2}_{H^*_\eta} < \varepsilon\}) \to (\{\psi^{l_1}_{H^*_\eta} < L\}, \{\psi^{l_1}_{H^*_\eta} < \varepsilon\}),\]

which is the restriction of $\mathbb{P}_{l_1}$.

For a fixed $x \in \mathbb{H}_{l_1}$, the map 

\[\mathbb{H}^{l_2, l_1} \to \mathbb{R}, \quad y \mapsto \psi^{l_2}_{H_\eta^*}(x + y),\]

has a unique minimizer, which is precisely $\mathbb{P}_{l_2} \circ Y_{l_1}(x) \in \mathbb{H}^{l_2, l_1}$. Additionally, $\psi^{l_2}_{H_\eta^*}$ is convex on $\{x\} \times \mathbb{H}^{l_2, l_1}$. Therefore, as before, we have that $\Gamma^L_{l_2, l_1}$ and $\mathbb{P}^L_{l_1, l_2}$ are such that

\begin{equation}\label{homotopyequivalencel1l2}
    \mathbb{P}^L_{l_1, l_2} \circ \Gamma^L_{l_2, l_1} = \text{id}, \quad \Gamma^L_{l_2, l_1} \circ \mathbb{P}^L_{l_1, l_2} \simeq \text{id}
\end{equation}

holds. From the previous discussion, it follows that

\begin{equation}\label{derivativezerotangent}
    \frac{\partial \psi^{l_2}_{H_\eta^*}}{\partial y}(\Gamma_{l_2, l_1}(x)) = 0, \quad x \in \mathbb{H}_{l_1}
\end{equation}

On the other hand, $\mathbb{H}^{l_2, l_1}$ is transverse to $\Gamma_{l_2, l_1}(\mathbb{H}_{l_1})$ because it is a graph. Therefore, we choose $g_2 \in \mathcal{G}(\mathbb{H}_{l_2})$ uniformly equivalent to the standard metric such that $\mathbb{H}^{l_2, l_1}$ is normal to $\Gamma_{l_2, l_1}(\mathbb{H}_{l_1})$. For such a metric, it follows from \eqref{derivativezerotangent} that $\nabla_{g_2} \psi^{l_2}_{H_\eta^*}$ is tangent to $\Gamma_{l_2, l_1}(\mathbb{H}_{l_1})$. We take $g_1 \in \mathcal{G}(\mathbb{H}_{l_1})$ such that $g_1 = (\Gamma_{l_2, l_1})^* g_2$. This metric will be equivalent to the standard one because $\Gamma_{l_2, l_1}$ is bilipschitz. Let $g_1$, $g_2$, and an almost complex structure $J$ be chosen such that the above requirements are met and the maps

\[\Theta_1: \C\M_{*}(\psi^{l_1}_{H_\eta^*}, g_1) \to \C\F_{*+n}(H_\eta, J)\] 

and 

\[\Theta_2: \C\M_{*}(\psi^{l_2}_{H_\eta^*}, g_2) \to \C\F_{*+n}(H_\eta, J)\] 

are well-defined.

Let $x_{l_1}$ be a critical point of $\psi^{l_1}_{H_\eta^*}$. Then $\Gamma_{l_1, l_2}(x_{l_1}) = x_{l_2}$ is a critical point of $\psi^{l_2}_{H_\eta^*}$. Moreover, the map 

\begin{equation}\label{bijectionl1l2}
\Gamma_{l_1, l_2}|_{\crit(\psi^{l_1}_{H_\eta^*})}: \crit(\psi^{l_1}_{H_\eta^*}) \to \crit(\psi^{l_2}_{H_\eta^*})
\end{equation}

is a bijection.

Being that $g_1 = (\Gamma_{l_2,l_1})^*g_2$ and by properties of $g_2$, it follows that
\begin{equation}\label{l1l2unstableequality}
    \Gamma_{l_2,l_1}(W^u(x_{l_1}; -\nabla_{g_1}\psi^{l_1}_{H_\eta^*})) = W^u(x_{l_2}; -\nabla_{g_2}\psi^{l_2}_{H_\eta^*}).
\end{equation}

Indeed, since $g_1 = (\Gamma_{l_2,l_1})^*g_2$, we have that $\Gamma_{l_2,l_1}(W^u(x_{l_1}; -\nabla_{g_1}\psi^{l_1}_{H_\eta^*}))$ is the unstable manifold at $x_{l_2}$ associated with the pair $(\psi^{l_2}_{H_\eta^*}|_{\Gamma_{l_2,l_1}(\mathbb{H}_{l_1})}, g_2|_{\Gamma_{l_2,l_1}(\mathbb{H}_{l_1})})$. Since $\nabla_{g_2}\psi^{l_2}_{H_\eta^*}$ is tangent to $\Gamma_{l_2,l_1}(\mathbb{H}_{l_1})$ and the map \eqref{bijectionl1l2} is a bijection, we conclude that
\[\Gamma_{l_2,l_1}(W^u(x_{l_1}; -\nabla_{g_1}\psi^{l_1}_{H_\eta^*})) \subseteq W^u(x_{l_2}; -\nabla_{g_2}\psi^{l_2}_{H_\eta^*}).\]

Additionally, $\Gamma_{l_2,l_1}(W^u(x_{l_1}; -\nabla_{g_1}\psi^{l_1}_{H_\eta^*}))$ and $W^u(x_{l_2}; -\nabla_{g_2}\psi^{l_2}_{H_\eta^*})$ have the same dimension since
\[\ind(x_l; \psi^l_{H_\eta^*}) = \mu_{CZ}(x; H_\eta) - n\]
for every $l$ for which $l$-saddle point reduction exists. Therefore, $\Gamma_{l_2,l_1}(W^u(x_{l_1}; -\nabla_{g_1}\psi^{l_1}_{H_\eta^*}))$ is an open connected subset of $W^u(x_{l_2}; -\nabla_{g_2}\psi^{l_2}_{H_\eta^*})$, which is also open and connected. On the other hand, $\Gamma_{l_2,l_1}(W^u(x_{l_1}; -\nabla_{g_1}\psi^{l_1}_{H_\eta^*}))$ is part of the graph $\Gamma_{l_2,l_1}(\mathbb{H}_{l_1})$, which is closed. Therefore, the unstable manifold $\Gamma_{l_2,l_1}(W^u(x_{l_1}; -\nabla_{g_1}\psi^{l_1}_{H_\eta^*}))$ is also closed in $W^u(x_{l_2}; -\nabla_{g_2}\psi^{l_2}_{H_\eta^*})$, which implies that \eqref{l1l2unstableequality} holds.

Therefore, $\Gamma_{l_2,l_1}$ induces a chain isomorphism  
\[
\gamma_{l_2,l_1}: \C\M_{*}(\psi^{l_1}_{H_\eta^*}, g_1) \to \C\M_{*}(\psi^{l_2}_{H_\eta^*}, g_2)
\]
given by  
\[
\gamma_{l_2,l_1}(x_{l_1}) = x_{l_2}.
\]

Moreover, \eqref{l1l2unstableequality} gives us that
\[\Gamma_{l_1,+}(W^u(x_{l_1}; -\nabla_{g_1}\psi^{l_1}_{H_\eta^*})) = \Gamma_{l_2,+}(W^u(x_{l_2}; -\nabla_{g_2}\psi^{l_2}_{H_\eta^*})).\]

Therefore,
\[\mathcal{M}_\Theta(x_{l_1}, y; H_\eta, J, g_1) = \mathcal{M}_\Theta(x_{l_2}, y; H_\eta, J, g_2),\]
which implies that the following diagram
\[
\begin{tikzcd}
 \C\M_{*}(\psi^{l_2}_{H_\eta^*}, g_2) \arrow{r}{\Theta_2} & \C\F_{*+n}(H_\eta, J) \\
 \C\M_{*}(\psi^{l_1}_{H_\eta^*}, g_1) \arrow{r}{\Theta_{1}} \arrow[swap]{u}{\gamma_{l_2,l_1}} & \H\F_{*+n}(H_\eta, J) \arrow[u, equal]
\end{tikzcd},
\]
commutes. If we denote by
\[\gamma^{+,<L}_{\H}: \H\M^{+,<L}_{*}(\psi^{l_1}_{H_\eta^*}, g_1) \to \H\M^{+,<L}_{*}(\psi^{l_2}_{H_\eta^*}, g_2)\]
the map induced from $\gamma_{l_2,l_1}$ in positive homology, we have that the diagram
\[
\begin{tikzcd}
 \H\M^{+,<L}_{*}(\psi^{l_2}_{H_\eta^*}, g_2) \arrow{r}{\Theta^{+,<L}_{2,\H}} & \H\F^{+,<L}_{*+n}(H_\eta, J) \\
 \H\M^{+,<L}_{*}(\psi^{l_1}_{H_\eta^*}, g_1) \arrow{r}{\Theta^{+,<L}_{1,\H}} \arrow[swap]{u}{\gamma^{+,<L}_{\H}} & \H\F^{+,<L}_{*+n}(H_\eta, J) \arrow[u, equal]
\end{tikzcd},
\]
commutes. Combining this with the commuting diagrams
\[
\begin{tikzcd}
 H_{*}(\{\psi^{l_2}_{H^*_\eta} < L\}, \{\psi^{l_2}_{H^*_\eta} < \varepsilon\}) \arrow{r}{h^{L, g_2}} & \H\M_{*}^{+,<L}(\psi^{l_2}_{H^*}, g_2) \\
 \H_{*}(\{\psi^{l_1}_{H^*_\eta} < L\}, \{\psi^{l_1}_{H^*_\eta} < \varepsilon\}) \arrow{r}{h^{L, g_1}} \arrow[swap]{u}{\Gamma^L_{l_2, l_1}} & \H\M_{*}^{+,<L}(\psi^{l_1}_{H^*}, g_1) \arrow[swap]{u}{\gamma^{+,<L}_{\H}}
\end{tikzcd},
\]
\[
\begin{tikzcd}
 \H\F_*^{+,<L}(H_\eta, J) \arrow{r}{i^{\S\H^+, H_\eta}_L} & \S\H^{+<L}_{*}(C) \\
 \H\F_*^{+,<L}(H_\eta, J_1) \arrow{r}{i^{\S\H^+, H_\eta}_L} \arrow[u, equal] & \S\H^{+<L}_{*}(C) \arrow[u, equal]
\end{tikzcd},
\]
we get that diagram \eqref{diagraml1l2} commutes. From there and \eqref{homotopyequivalencel1l2}, we get that
\[ D^{H'_\eta, l_2}_L \circ (\mathbb{P}^L_{l_2})_* = D^{H'_\eta, l_2}_L \circ (id)_* \circ (\mathbb{P}^L_{l_2})_* = D^{H'_\eta, l_2}_L \circ (\Gamma^L_{l_2, l_1} \circ \mathbb{P}^L_{l_1, l_2})_* \circ (\mathbb{P}^L_{l_2})_*\]
\[= D^{H'_\eta, l_2}_L \circ (\Gamma^L_{l_2, l_1})_* \circ (\mathbb{P}^L_{l_1, l_2} \circ \mathbb{P}^L_{l_2})_* = D^{H'_\eta, l_1}_L \circ (\mathbb{P}^L_{l_1})_*.\]

Thus, for $L \in (\eta, T_\eta)$, the definition of $D^{H'_\eta}_L = D^{H'_\eta, l}_L \circ (\mathbb{P}^L_{l})_*$ is independent of $l$, i.e., it is well-defined. Moreover, it is an isomorphism since both maps are isomorphisms.

Let $H_\eta, H_\nu \in \mathcal{H}^{\mathcal{F}}(C)$ such that $H_\eta \leq H_\nu$. Let $l \in \mathbb{N}$ be large enough such that $l$-saddle point reduction exists for  dual functionals associated with $H_\eta$ and $H_\nu$.

For $L_\eta \in (\eta, T_\eta)$ and $L_\nu \in (\nu, T_\nu)$ such that $L_\eta \leq L_\nu$, we have that the following diagram
\begin{equation}\label{lreductioncommutation}
    \begin{tikzcd}
     (\{\psi^{l}_{H^*_\nu} < L_\nu\}, \{\psi^{l}_{H^*_\nu} < \varepsilon\}) \arrow{r}{\mathbb{P}_l^{L_\nu}} & (\{\psi^{l}_{H^*_\nu} < L_\nu\}, \{\psi^{l}_{H^*_\nu} < \varepsilon\}) \\
     (\{\psi^{l}_{H^*_\eta} < L_\eta\}, \{\psi^{l}_{H^*_\eta} < \varepsilon\}) \arrow{r}{\mathbb{P}_l^{L_\eta}} \arrow[swap]{u}{i^{H}_{L_\nu, L_\eta}} & (\{\psi^{l}_{H^*_\nu} < L_\nu\}, \{\psi^{l}_{H^*_\nu} < \varepsilon\}) \arrow[swap]{u}{i^{H, l}_{L_\nu, L_\eta}}
    \end{tikzcd},
\end{equation}
commutes.

We already know that diagram \eqref{commutativediagramisolreduction}, i.e., the diagram
\[
\begin{tikzcd}
 \H_{*}(\{\psi^l_{H_{\nu}} < L_\nu\}, \{\psi^l_{H_{\nu}} < \varepsilon\}; \mathbb{Z}_2) \arrow{r}{D^{H_\nu, l}_{L_\nu}} & SH^{+<L_\nu}_{*+n}(C; \mathbb{Z}_2) \\
 \H_{*}(\{\psi^l_{H_{\eta}} < L_\eta\}, \{\psi^l_{H_{\eta}} < \varepsilon\}; \mathbb{Z}_2) \arrow{r}{D^{H_\eta, l}_{L_\eta}} \arrow[swap]{u}{(i^{H, l}_{L_\nu, L_\eta})_*} & SH^{+,<L_\eta}_{*+n}(C; \mathbb{Z}_2) \arrow[swap]{u}{i^{+}_{L_\nu, L_\eta}}
\end{tikzcd}
\]
commutes. Since $D^{H'_\nu}_{L_\eta} = D^{H'_\eta, l}_{L_\eta} \circ (\mathbb{P}_l^{L_\eta})$ and $D^{H'_\nu}_{L_\nu} = D^{H'_\nu, l}_{L_\nu} \circ (\mathbb{P}_l^{L_\nu})$, combining the diagram induced in homology from diagram \eqref{lreductioncommutation} and the commutativity of the last diagram, we conclude that the diagram of the claim (2) of this lemma commutes.

\end{proof}

\begin{corollary}\label{positivesymplecticautonomus}
    Let $C$ be a non-degenerate, strongly convex domain whose interior contains the origin. Let $H_\eta \in \mathcal{H}^\mathcal{F}_{MB}(C)$ be arbitrary, where $\eta$ is the slope at infinity and $T_\eta$ is the smallest element of $\sigma(\partial C)$ greater than $\eta$. The following holds:
    \begin{enumerate}
        \item For every $L \in (\eta, T_\eta)$, there exists an isomorphism
        \[D^{H_\eta}_L: \H_{*}(\{\Psi_{H^*_\eta} < L\}, \{\Psi_{H^*_\eta} < \varepsilon\}; \mathbb{Z}_2) \to SH^{+,<L}_{*+n}(C; \mathbb{Z}_2).\]
        \item This isomorphism is natural. Let $H_\eta \leq H_\nu$. For every $L_\eta \in (\eta, T_\eta)$ and $L_\nu \in (\nu, T_\nu)$ such that $L_\eta \leq L_\nu$, the following diagram

        \[
        \begin{tikzcd}
        \H_{*}(\{\Psi_{H_{\nu}} < L_\nu\}, \{\Psi_{H_{\nu}} < \varepsilon\}; \mathbb{Z}_2) \arrow{r}{D^{H_\nu}_{L_\nu}} & SH^{+,<L_\nu}_{*+n}(C; \mathbb{Z}_2) \\
        \H_{*}(\{\Psi_{H_{\eta}} < L_\eta\}, \{\Psi_{H_{\eta}} < \varepsilon\}; \mathbb{Z}_2) \arrow{r}{D^{H_\eta}_{L_\eta}} \arrow[swap]{u}{(i^{H}_{L_\nu, L_\eta})_*} & SH^{+,<L_\eta}_{*+n}(C; \mathbb{Z}_2) \arrow[swap]{u}{i^{+}_{L_\nu, L_\eta}}
        \end{tikzcd}
        \]

        commutes.
    \end{enumerate}
\end{corollary}

\begin{proof} \textbf{Claim} (1):
Let $H_\eta \in \mathcal{H}^{\mathcal{F}}_{MB}(C)$ be arbitrary, and let $\mathcal{H}^{\mathcal{F}}(C, H_\eta)$ denote the subset of perturbations $H_\eta'$ of $H_\eta$ such that $H'_\eta \in \mathcal{H}^{\mathcal{F}}(C)$ and
\begin{itemize}
    \item $H_\eta \leq H'_\eta$.
\end{itemize}

From this, it follows that $\Psi_{H'^*_\eta} \leq \Psi_{H^*_\eta}$, and therefore we have the inclusion
\[i^{H'^*_\eta, H^*_\eta}_L: (\{\Psi_{H^*_\eta} < L\}, \{\Psi_{H^*_\eta} < \varepsilon\}) \to (\{\Psi_{H'^*_\eta} < L\}, \{\Psi_{H'^*_\eta} < \varepsilon\}), \quad L \in (\eta, T_\eta).\]

This inclusion defines an isomorphism in homology. Therefore, we define
\[D^{H_\eta}_L: \H_{*}(\{\Psi_{H_\eta} < L\}, \{\Psi_{H^*_\eta} < \varepsilon\}; \mathbb{Z}_2) \to SH^{+,<L}_{*+n}(C; \mathbb{Z}_2), \quad L \in (\eta, T_\eta)\]

as
\[D^{H_\eta}_L = D^{H'_\eta} \circ (i^{H'^*_\eta, H^*_\eta}_L)_*, \quad H' \in \mathcal{H}^{\mathcal{F}}(C, H_\eta).\]

First, we need to show that $i^{H'^*_\eta, H^*_\eta}_L$ indeed induces an isomorphism in homology.

\textbf{Case 1:} $H'_\eta$ is close to $H_\eta$.

Let $H'_\eta \in \mathcal{H}^{\mathcal{F}}(C, H_\eta)$ be such that $H_\eta \geq H'_\eta - \varepsilon'$, where $\varepsilon' < \min\left\{\varepsilon - \Psi_{H^*_\eta}(0), \frac{T_\eta - \eta}{4}\right\}$. We will show that in this case $(i^{H'^*_\eta, H^*_\eta}_L)_*$ is an isomorphism for all $L \in (\eta, T_\eta)$. Notice that it is enough to show that the previous inclusion is an isomorphism for one $L_0 \in (\eta, T_\eta)$. Indeed, if $L \in (L_0, T_\eta)$, we have from the fact that both $\Psi_{H'^*_\eta}$ and $\Psi_{H^*_\eta}$ satisfy the PS condition and there are no critical values of either function in $[L_0, L]$ that inclusions
\[(\{\Psi_{H'^*_\eta} < L_0\}, \{\Psi_{H'^*_\eta} < \varepsilon\}) \to (\{\Psi_{H'^*_\eta} < L\}, \{\Psi_{H'^*_\eta} < \varepsilon\})\]
and
\[(\{\Psi_{H^*_\eta} < L_0\}, \{\Psi_{H^*_\eta} < \varepsilon\}) \to (\{\Psi_{H^*_\eta} < L\}, \{\Psi_{H^*_\eta} < \varepsilon\})\]
induce isomorphisms in homology. Now, from the commutativity of the diagram
\[
\begin{tikzcd}
(\{\Psi_{H'^*_\eta} < L_0\}, \{\Psi_{H'^*_\eta} < \varepsilon\}) \arrow{r}{} & (\{\Psi_{H'^*_\eta} < L\}, \{\Psi_{H'^*_\eta} < \varepsilon\}) \\
(\{\Psi_{H^*_\eta} < L_0\}, \{\Psi_{H^*_\eta} < \varepsilon\}) \arrow{r}{} \arrow[swap]{u}{i^{H'^*_\eta, H^*_\eta}_{L_0}} & (\{\Psi_{H^*_\eta} < L\}, \{\Psi_{H^*_\eta} < \varepsilon\}) \arrow[swap]{u}{i^{H'^*_\eta, H^*_\eta}_L}
\end{tikzcd}
\]
and the fact that $i^{H'^*_\eta, H^*_\eta}_{L_0}$ induces an isomorphism in homology, it follows that for all $L \in (L_0, T_\eta)$, the map $i^{H'^*_\eta, H^*_\eta}_L$ induces an isomorphism in homology. Similarly, we show that $i^{H^*_\eta, H'^*_\eta}_L$ induces an isomorphism in homology for all $L \in (\eta, L_0)$.

Let $L_0 \in (\eta, \frac{3T_\eta + \eta}{4})$. We will show that $i^{H'^*_\eta, H^*_\eta}_{L_0}$ is injective in homology.

From the relation $H_\eta \leq H'_\eta \leq H_\eta + \varepsilon'$, it follows that $\Psi_{H^*_\eta} - \varepsilon' \leq \Psi_{H'^*_\eta} \leq \Psi_{H^*_\eta}$. We have the following factorization
\[
\begin{tikzcd}[scale cd=0.85]
& (\{\Psi_{H'^*_\eta} < L_0\}, \{\Psi_{H'^*_\eta} < \varepsilon\}) \arrow[dr] \\
(\{\Psi_{H^*_\eta} < L_0\}, \{\Psi_{H^*_\eta} < \varepsilon\}) \arrow[ur, "i^{H'^*_\eta, H^*_\eta}_{L_0}"] \arrow[rr] && (\{\Psi_{H^*_\eta} < L_0 + \varepsilon'\}, \{\Psi_{H^*_\eta} < \varepsilon + \varepsilon'\})
\end{tikzcd}
\]
where all the maps are inclusions. We have that the inclusion
\[(\{\Psi_{H^*_\eta} < L_0\}, \{\Psi_{H^*_\eta} < \varepsilon\}) \to (\{\Psi_{H^*_\eta} < L_0 + \varepsilon'\}, \{\Psi_{H^*_\eta} < \varepsilon + \varepsilon'\})\]
induces an isomorphism in homology, which implies from the previous diagram that $i^{H'^*_\eta, H^*_\eta}_{L_0}$ is injective in homology.

Indeed, we have that the inclusion
\[\{\Psi_{H^*_\eta} < L_0\} \to \{\Psi_{H^*_\eta} < L_0 + \varepsilon'\}\]
induces an isomorphism in homology since $L_0 + \varepsilon' < T_\eta$,

and the inclusion
\[\{\Psi_{H^*_\eta} < \varepsilon\} \to \{\Psi_{H^*_\eta} < \varepsilon + \varepsilon'\}\]
also induces an isomorphism since $\Psi_{H^*_\eta}(0)<\varepsilon$, $\varepsilon + \varepsilon' < 2 \varepsilon$, and \[\Psi_{H^*_\eta}(x) > 2 \varepsilon, \quad x \in \crit(\Psi_{H^*_\eta}) \setminus \{0\}.\]

The last inequality follows from Remark \ref{cofinal}.

Therefore, by the five lemma, it follows that the inclusion
\[(\{\Psi_{H'^*_\eta} < L_0\}, \{\Psi_{H'^*_\eta} < \varepsilon\}) \to (\{\Psi_{H'^*_\eta} < L_0 + \varepsilon'\}, \{\Psi_{H'^*_\eta} < \varepsilon + \varepsilon'\})\]
indeed induces an isomorphism in homology.

Let $L_0 \in (\frac{3\eta + T_\eta}{4}, T_\eta)$. We will show that $i^{H'^*_\eta, H^*_\eta}_{L_0}$ is surjective in homology.

From the relation $H'_\eta - \varepsilon' \leq H_\eta \leq H'_\eta$, it follows that $\Psi_{H'^*_\eta} \leq \Psi_{H^*_\eta} \leq \Psi_{H'^*_\eta} + \varepsilon'$. We have the following factorization
\[
\begin{tikzcd}[scale cd=0.85]
& (\{\Psi_{H^*_\eta} < L_0\}, \{\Psi_{H^*_\eta} < \varepsilon\}) \arrow[dr, "i^{H'^*_\eta, H^*_\eta}_{L_0}"] \\
(\{\Psi_{H'^*_\eta} < L_0 - \varepsilon'\}, \{\Psi_{H'^*_\eta} < \varepsilon - \varepsilon'\}) \arrow[ur] \arrow[rr] && (\{\Psi_{H'^*_\eta} < L_0\}, \{\Psi_{H'^*_\eta} < \varepsilon\})
\end{tikzcd}.
\]

The inclusion
\[(\{\Psi_{H'^*_\eta} < L_0 - \varepsilon'\}, \{\Psi_{H'^*_\eta} < \varepsilon - \varepsilon'\}) \to (\{\Psi_{H'^*_\eta} < L_0\}, \{\Psi_{H'^*_\eta} < \varepsilon\})\]
induces an isomorphism in homology, which implies from the previous diagram that $i^{H'^*_\eta, H^*_\eta}_{L_0}$ is surjective in homology. Indeed, since $L_0 - \varepsilon' > \eta$ and $\varepsilon - \varepsilon' > \Psi_{H'^*_\eta}(0)$ (it holds that $\Psi_{H'^*_\eta}(0) = \Psi_{H^*_\eta}(0)$), we have that the inclusions
\[\{\Psi_{H'^*_\eta} < L_0 - \varepsilon'\} \to \{\Psi_{H'^*_\eta} < L_0\}, \quad \{\Psi_{H'^*_\eta} < \varepsilon - \varepsilon'\} \to \{\Psi_{H'^*_\eta} < \varepsilon\},\]
induce isomorphisms in homology, which implies that  
\[
(\{\Psi_{H'^*_\eta} < L_0 - \varepsilon'\}, \{\Psi_{H'^*_\eta} < \varepsilon - \varepsilon'\}) \to (\{\Psi_{H'^*_\eta} < L_0\}, \{\Psi_{H'^*_\eta} < \varepsilon\})
\]
induces an isomorphism in homology. Hence, the map \( i^{H'^*_\eta, H^*_\eta}_{L_0} \) is surjective in homology.

Combining the previous conclusions, we have that $i^{H'^*_\eta, H^*_\eta}_{L_0}$ induces an isomorphism in homology for $L_0 \in (\frac{3\eta + T_\eta}{4}, \frac{3T_\eta + \eta}{4})$. This implies, by the discussion at the beginning of the proof, that $i^{H'^*_\eta, H^*_\eta}_L$ induces an isomorphism in homology for every $L \in (\eta, T_\eta)$.

\textbf{Case 2:} Arbitrary $H'_\eta \in \mathcal{H}^{\mathcal{F},+}(C, H_\eta)$

Let $H_\eta'' \in \mathcal{H}^{\mathcal{F}}(C, H_\eta)$ be such that $H''_\eta \leq H'_\eta$ and
\[(i^{H_\eta''^*, H^*_\eta}_{L})_*, \quad L \in (\eta, T_\eta),\]
is an isomorphism.

Such an $H''_\eta$ always exists due to the definition of $\mathcal{H}^{\mathcal{F}}(C)$ and the previous case. Then we have a factorization
\[
\begin{tikzcd}[scale cd=0.9]
 &  (\{\Psi_{H''^*_\eta} < L\}, \{\Psi_{H''^*_\eta} < \varepsilon\}) \arrow[dr, "i^{H'^*_\eta, H''^*_\eta}_L"] \\
(\{\Psi_{H^*_\eta} < L\}, \{\Psi_{H^*_\eta} < \varepsilon\}) \arrow[ur, "i^{H_\eta^*, H''^*_\eta}_L"] \arrow[rr, "i^{H'^*_\eta, H^*_\eta}_L"] && (\{\Psi_{H'^*_\eta} < L\}, \{\Psi_{H'^*_\eta} < \varepsilon\})
\end{tikzcd}.
\]

On the other hand, from Lemma \ref{positivesymplecticiso}, we have that
\[(i^{H_\eta'^*, H''^*_\eta}_L)_*, \quad L \in (\eta, T_\eta)\]
is an isomorphism. From the last diagram, we have
\[(i^{H_\eta'^*, H^*_\eta}_L)_* = (i^{H_\eta'^*, H''^*_\eta}_L \circ i^{H''^*_\eta, H^*_\eta}_L)_* = (i^{H_\eta'^*, H''^*_\eta}_L)_* \circ (i^{H''^*_\eta, H^*_\eta}_L)_*\]
which implies that $i^{H_\eta'^*, H^*_\eta}_L$ induces isomorphism in homology for every $L \in (\eta, T_\eta)$.

Now, we need to show that $D^{H_\eta}_L$ is well-defined.

Let \[H_\eta', H_\eta'' \in \mathcal{H}^{\mathcal{F}}(C, H_\eta)\] be two different Hamiltonians such that $H_\eta' \leq H_\eta''$ holds. Let $L \in (\eta, T_\eta)$ be arbitrary. From the previous lemma, we have a diagram
\[
\begin{tikzcd}[scale cd=0.9]
 \H_{*}(\{\Psi_{H_{\eta}''^*} < L\}, \{\Psi_{H_{\eta}''^*} < \varepsilon\}; \mathbb{Z}_2) \arrow{r}{D^{H''_\eta}_L} & SH^{+,<L}_{*+n}(C; \mathbb{Z}_2) \\
 \H_{*}(\{\Psi_{H_{\eta}'^*} < L\}, \{\Psi_{H_{\eta}'^*} < \varepsilon\}; \mathbb{Z}_2) \arrow{r}{D^{H_\eta'}_L} \arrow[swap]{u}{(i^{H_\eta''^*, H_\eta'^*}_L)_*} & SH^{+,<L}_{*+n}(C; \mathbb{Z}_2) \arrow[u, equal]
\end{tikzcd}
\]
which commutes. Since it holds that
\[i^{H_\eta''^*, H_\eta^*}_L = i^{H_\eta''^*, H_\eta'^*}_L \circ i^{H_\eta'^*, H_\eta^*}_L,\]
we have that
\[D^{H''_\eta}_L \circ (i^{H''^*_\eta, H^*_\eta}_L)_* = D^{H''_\eta}_L \circ (i^{H_\eta''^*, H_\eta'^*}_L \circ i^{H_\eta'^*, H_\eta^*}_L)_*\]
\[= (D^{H''_\eta}_L \circ (i^{H_\eta''^*, H_\eta'^*}_L)_*) \circ (i^{H_\eta'^*, H_\eta^*}_L)_* = D^{H'_\eta}_L \circ (i^{H'^*_\eta, H^*_\eta}_L)_*.\]

Let now $H_\eta', H_\eta'' \in \mathcal{H}^{\mathcal{F},+}(C, H_\eta)$ be arbitrary. We can always find another perturbation $H_\eta''' \in \mathcal{H}^{\mathcal{F},+}(C, H_\eta)$ such that $H_\eta''' \leq H_\eta', H_\eta''$. From the previous discussion, we have that
\[D^{H''_\eta}_L \circ (i^{H''^*_\eta, H^*_\eta}_L)_* = D^{H'''_\eta}_L \circ (i^{H'''^*_\eta, H^*_\eta}_L)_* = D^{H'_\eta}_L \circ (i^{H'^*_\eta, H^*_\eta}_L)_*, \quad L \in (\eta, T_\eta).\]

This implies that
\[D^{H_\eta}_L = D^{H'_\eta}_L \circ (i^{H'^*_\eta, H_\eta}_L)_*, \quad H' \in \mathcal{H}^{\mathcal{F}}(C, H_\eta)\]
is independent of $H'_\eta$ and therefore, well-defined. Moreover, from the fact that $D^{H'_\eta}_L$ and $(i^{H'^*_\eta, H_\eta}_L)_*$ are isomorphisms, it follows that $D^{H_\eta}_L$ is an isomorphism.

\textbf{Claim }(2): Let now $H_\eta, H_\nu \in \mathcal{H}_{MB}^\mathcal{F}(C)$ such that $H_\eta \leq H_\nu$. We choose $H'_\eta \in \mathcal{H}^\mathcal{F}(C, H_\eta)$ and $H'_\nu \in \mathcal{H}^\mathcal{F}(C, H_\nu)$ such that $H'_\eta \leq H'_\nu$. We choose $L_\eta \in (\eta, T_\eta)$ and $L_\nu \in (\nu, T_\nu)$ such that $L_\eta \leq L_\nu$. Then we have the following commuting diagram
\[
\begin{tikzcd}[scale cd=0.9]
 (\{\Psi_{H^*_\nu} < L_\nu\}, \{\Psi_{H^*_\nu} < \varepsilon\}) \arrow{r}{i^{H'^*_\nu, H^*_\nu}_{L_\nu}} & (\{\Psi_{H'^*_\nu} < L_\nu\}, \{\Psi_{H'^*_\nu} < \varepsilon\}) \\
 (\{\Psi_{H^*_\eta} < L_\eta\}, \{\Psi_{H^*_\eta} < \varepsilon\}) \arrow{r}{i^{H'^*_\eta, H^*_\eta}_{L_\eta}} \arrow[swap]{u}{i^{H}_{L_\nu, L_\eta}} & (\{\Psi_{H'^*_\eta} < L_\eta\}, \{\Psi_{H'^*_\eta} < \varepsilon\}) \arrow[swap]{u}{i^{H'}_{L_\nu, L_\eta}}
\end{tikzcd}.
\]

This induces the commutative diagram
\[
\begin{tikzcd}[scale cd=0.9]
 \H_*(\{\Psi_{H^*_\nu} < L_\nu\}, \{\Psi_{H^*_\nu} < \varepsilon\}; \mathbb{Z}_2) \arrow{r}{(i^{H'^*_\nu, H^*_\nu}_{L_\nu})_*} & \H_*(\{\Psi_{H'^*_\nu} < L_\nu\}, \{\Psi_{H'^*_\nu} < \varepsilon\}; \mathbb{Z}_2) \\
 \H_*(\{\Psi_{H^*_\eta} < L_\eta\}, \{\Psi_{H^*_\eta} < \varepsilon\}; \mathbb{Z}_2) \arrow{r}{(i^{H'^*_\eta, H^*_\eta}_{L_\eta})_*} \arrow[swap]{u}{(i^H_{L_\nu, L_\eta})_*} & \H_*(\{\Psi_{H'^*_\eta} < L_\eta\}, \{\Psi_{H'^*_\eta} < \varepsilon\}; \mathbb{Z}_2) \arrow[swap]{u}{(i^{H'}_{L_\nu, L_\eta})_*}
\end{tikzcd}.
\]

Combining this diagram with the commutativity of the diagram from the previous lemma, i.e.,

\[
\begin{tikzcd}[scale cd=0.9]
 \H_{*}(\{\Psi_{H'_\nu} < L_\nu\}, \{\Psi_{H'_\nu} < \varepsilon\}; \mathbb{Z}_2) \arrow{r}{D^{H'_\nu}_{L_\nu}} & SH^{+, < L_\nu}_{*+n}(C; \mathbb{Z}_2) \\
 \H_{*}(\{\Psi_{H'_\eta} < L_\eta\}, \{\Psi_{H'_\eta} < \varepsilon\}; \mathbb{Z}_2) \arrow{r}{D^{H'_\eta}_{L_\eta}} \arrow[swap]{u}{(i^{H'}_{L_\nu, L_\eta})_*} & SH^{+, < L_\eta}_{*+n}(C; \mathbb{Z}_2) \arrow[swap]{u}{i^{+}_{L_\nu, L_\eta}}
\end{tikzcd}
\]

from the definitions of \( D^{H_\eta}_{L_\eta} \) and \( D^{H_\nu}_{L_\nu} \), the conclusion follows.

\end{proof}

\large\textbf{Positive $S^1$-equivariant symplectic homology}
\normalsize

Let $H_\eta \in \mathcal{H}_{MB}^{\mathcal{F}}(C)$. We say that $H'_\eta \in \mathcal{H}^{\mathcal{F},S^1,N}(C)$ is an $N$-admissible $S^1$-invariant perturbation of $H_\eta \in \mathcal{H}_{MB}^{\mathcal{F}}(C)$ if 
\[H'_\eta(t,x,z) = H_\eta(x) + \delta \left(\beta(z) \sum_{S_{y}} \rho_{S_y}(x) f_{S_y}(t - \tau_z, x) - \mu \rho(x) \widetilde{f}_N(z) \right).\]

Here, $\beta: S^{2N+1}\to \mathbb{R}$ is an $S^1$-invariant cutoff function equal to $1$ near $\crit(\widetilde{f}_N)$ and $0$ outside of $U$, which is an $S^1$-invariant neighborhood of $\crit(\widetilde{f}_N)$ for which we have a pre-chosen local section. The function 
\[\tau: U \to \mathbb{T},\]
is defined such that $\tau_z$ is the time at which $\tau_z^{-1}z$ belongs to the local section.

It holds that $\text{supp}(\rho_{S_y}) \cap \overline{C} = \emptyset$ and that 
\[\text{supp}(\rho) \subset V, \quad \rho|_B = 1,\]
where $B = \bigcup \text{supp}(\rho_{S_y})$ and $V$ is a bounded neighborhood of $B$ such that $V \cap \overline{C} = \emptyset.$

Moreover, we require the following conditions to hold: 

\begin{enumerate}
    \item $\mu > 0$ is sufficiently large such that $H'_\eta$ is non-increasing along the gradient flow of $\widetilde{f}_N$.
    \item $\delta > 0$ is sufficiently small such that
    \begin{itemize}
        \item $H'_\eta$ is uniformly quadratically convex.
        \item $H_\eta + \delta \sum_{S_{y}} \rho_{S_y} f_{S_y} \in \mathcal{H}^{\mathcal{F}}(C)$.
        \item The set $P(H'_\eta, \widetilde{f}_N) \setminus (\{0\} \times \crit(\widetilde{f}_N))$ satisfies
        \[\bigcup \left\{\theta \cdot (y_{\min}, v) \mid \theta \in \mathbb{T},\ v \in \crit(\widetilde{f}_N) \text{ and belongs to the local section}\right\}\]
        \[\cup \bigcup \left\{\theta \cdot (y_{\max}, v) \mid \theta \in \mathbb{T},\ v \in \crit(\widetilde{f}_N) \text{ and belongs to the local section}\right\},\]
        where $(y_{\min}, y_{\max})$ are critical pairs of $H_\eta + \delta \sum_{S_{y}} \rho_{S_y} f_{S_y}$.
        \item For all $(y, v) \in P(H'_\eta, \widetilde{f}_N) \setminus (\{0\} \times \crit(\widetilde{f}_N))$, we have 
        \[|\Phi_{H'_\eta}(y, v) - T_y| < \min\left\{\frac{1}{\eta}, \text{gap}(\eta)\right\},\]
        \item $\varepsilon < \Phi_{H'_\eta}(y, v) < \eta$, for $(y, v) \in P(H'_\eta, \widetilde{f}_N) \setminus (\{0\} \times \crit(\widetilde{f}_N))$.
    \end{itemize}
\end{enumerate}

As in the non-parameterized case, for all critical points $(0,v)\in \{0\} \times \crit(\widetilde{f}_N)$, it holds that
\[\Phi_{H'_\eta}(0, v) = \Phi_{H_\eta}(0) < \varepsilon,\]
since all perturbations are outside of $\overline{C}$.

For $H \in \mathcal{H}^{\mathcal{F},S^1,N}(C)\subset \mathcal{H}_\Theta^{S^1,N}(\widetilde{f}_N,g_N)$, we define the positive $S^1$-equivariant Floer complex denoted by
\[\{\C\F_*^{S^1, N, +}(H), \partial^{F, S^1}\},\]
as the quotient complex 
\[\C\F^{S^1, N, +}(H) \defeq \C\F^{S^1, N}(H) / \C\F^{S^1, N, < \varepsilon}(H).\]

This is well-defined since $\Phi_H$ is non-increasing along trajectories. The filtration on this complex by $\Phi_H$ is well-defined for the same reason. The homology of this complex is denoted by
\[\H\F_*^{S^1, N, +}(H, J).\]

This homology is independent of the choice of an almost complex structure.

Let $H_\alpha \in \mathcal{H}^{\mathcal{F},S^1,N_\alpha}(C)$ and $H_\beta \in \mathcal{H}^{\mathcal{F},S^1,N_\beta}(C)$ be such that $(N_\alpha, H_\alpha) \leq (N_\beta, H_\beta)$. 

Let $(H_\alpha, J_\alpha)$ and $(H_\beta, J_\beta)$ be such that the $S^1$-equivariant Floer complexes are well-defined.

Then a generic non-increasing homotopy $(H^{\alpha\beta}_s, J^{\alpha\beta}_s)$ induces a chain map unique up to chain homotopy. This chain map induces a continuation map 
\[i^{S^1, +}_{H_\beta, H_\alpha}: \H\F^{S^1, N_\alpha, +}_*(H_\alpha, J_\alpha) \to \H\F^{S^1, N_\beta, +}_*(H_\beta, J_\beta),\]
which gives us a directed system. Here, $\alpha$ and $\beta$ do not represent the slope.

The set $\mathcal{H}^{\mathcal{F}, S^1, N}(C)$ is cofinal in the family of $N$-admissible $S^1$-equivariant Hamiltonians (see \cite{GH18}). Thus, we have
\[\S\H^{S^1, N, +, <L}(C) = \varinjlim_{H \in \mathcal{H}^{\mathcal{F}, S^1, N}(C)} \H\F^{S^1, N, +, <L}(H),\]
and
\[\S\H^{S^1, +, <L}(C) = \varinjlim_{H \in \mathcal{H}^{\mathcal{F}, S^1, N}(C), N \in \mathbb{N}} \H\F^{S^1, +, <L}(H).\]

Since it is clear that 
$\mathcal{H}^{\mathcal{F}, S^1, N}(C) \subset \mathcal{H}^{S^1, N}_{\Theta}(\widetilde{f}_N, g_N)$, for the $l$-reduced dual functional $\psi^l_{H^*}$ associated with $H \in \mathcal{H}^{\mathcal{F}, S^1, N}(C)$, we can define positive $S^1$-equivariant Morse homology. Moreover, due to Theorem \ref{theoremE2}, an analogous claim to the one from Remark \ref{commutativediagramisolreduction} holds in the $S^1$-equivariant case. Thus, we can prove the next lemma in the same way as Lemma \ref{positivesymplecticiso}.

\begin{lemma}\label{positiveS1symplecticiso}
    Let $C$ be a non-degenerate, strongly convex domain whose interior contains the origin. Let $H_\eta\in \mathcal{H}^{\mathcal{F},S^1,N}(C)$ be arbitrary, where $\eta$ is the slope at infinity and $T_\eta$ is the minimal element of $\sigma(\partial C)$ greater than $\eta$. The following holds: 
    \begin{enumerate}
        \item For every $L\in (\eta,T_\eta)$, there exists an isomorphism
        \[
        D^{H_\eta,S^1,N}_L: \H^{S^1}_{*}(\{\Psi_{H^*_\eta}<L\},\{\Psi_{H^*_\eta}<\varepsilon\};\mathbb{Z}_2) \to SH^{S^1,N,+,<L}_{*+n}(C;\mathbb{Z}_2).
        \]
        \item This isomorphism is natural. Let $(N_1,H_{\eta})\leq (N_2,H_{\nu})$. For every $L_\eta \in (\eta, T_\eta)$ and $L_\nu \in (\nu, T_\nu)$ such that $L_\eta \leq L_\nu$, the following diagram 

        \[
        \begin{tikzcd}
        \H^{S^1}_{*}(\{\Psi_{H_{\nu}}<L_\nu\},\{\Psi_{H_{\nu}}<\varepsilon\};\mathbb{Z}_2) \arrow{r}{D^{H_\nu,S^1,N_2}_{L_\nu}}  & [3em] SH^{S^1,N_2,+,<L_\nu}_{*+n}(C;\mathbb{Z}_2)\\%
        \H^{S^1}_{*}(\{\Psi_{H_{\eta}}<L_\eta\},\{\Psi_{H_{\eta}}<\varepsilon\};\mathbb{Z}_2) \arrow{r}{D^{H_\eta,S^1,N_1}_{L_\eta}} \arrow[swap]{u}{(i^{H,N_2,N_1}_{L_\nu,L_\eta})_*} & SH^{S^1,N_1,+,<L_\eta}_{*+n}(C;\mathbb{Z}_2) \arrow[swap]{u}{i^{S^1,N_2,N_1,+}_{L_\nu,L_\eta}}
        \end{tikzcd}
        \]

        commutes.
    \end{enumerate}
\end{lemma}

\begin{corollary}\label{positiveS1symplecticautonomus}
    Let $C$ be a non-degenerate, strongly convex domain whose interior contains the origin. Let $H_\eta\in \mathcal{H}_{MB}^\mathcal{F}(C)$ be arbitrary, where $\eta$ is the slope at infinity and $T_\eta$ is the minimal element of $\sigma(\partial C)$ greater than $\eta$. The following holds:
    \begin{enumerate}
        \item[$(1.1)$] For every $L\in (\eta,T_\eta)$, there exists an isomorphism
        \[
        D^{H_\eta,S^1}_L: \H^{S^1}_{*}(\{\Psi_{H_\eta}<L\},\{\Psi_{H^*_\eta}<\varepsilon\};\mathbb{Z}_2) \to SH^{S^1,+,<L}_{*+n}(C;\mathbb{Z}_2).
        \]
        \item[$(1.2)$] This isomorphism is natural. Let $H_{\eta}\leq H_{\nu}$. Then for every $L_\eta\in (\eta, T_\eta)$ and $L_\nu \in (\nu, T_\nu)$ such that $L_\eta\leq L_\nu$, the following diagram 

        \[
        \begin{tikzcd}
        \H^{S^1}_{*}(\{\Psi_{H_{\nu}}<L_\nu\},\{\Psi_{H_{\nu}}<\varepsilon\};\mathbb{Z}_2) \arrow{r}{D^{H_\nu,S^1}_{L_\nu}}  &  SH^{S^1,+,<L_\nu}_{*+n}(C;\mathbb{Z}_2)\\%
        \H^{S^1}_{*}(\{\Psi_{H_{\eta}}<L_\eta\},\{\Psi_{H_{\eta}}<\varepsilon\};\mathbb{Z}_2) \arrow{r}{D^{H_\eta,S^1}_{L_\eta}} \arrow[swap]{u}{(i^{H}_{L_\nu,L_\eta})_*} & SH^{S^1,+,<L_\eta}_{*+n}(C;\mathbb{Z}_2) \arrow[swap]{u}{i^{S^1,+}_{L_\nu,L_\eta}}
        \end{tikzcd}
        \]

        commutes.
    \end{enumerate}
\end{corollary}

\begin{proof}

Let $H_\eta \in \mathcal{H}_{MB}^\mathcal{F}(C)$ and $N \in \mathbb{N}$ be arbitrary. Unlike in the case of the family $\mathcal{H}^\mathcal{F}(C)$, in $\mathcal{H}^{\mathcal{F},S^1,N}(C)$ we cannot find $H'_\eta$ such that
\[ H_\eta \leq H'_\eta. \]
On the other hand, we can always find $H'_\eta$ such that it is arbitrarily close to $H_\eta$ and
\[ H'_\eta \leq H_\eta. \]

Hence, we denote by $\mathcal{H}^{\mathcal{F},S^1,N}(C, H_\eta)$ the subset of $\mathcal{H}^{\mathcal{F},S^1,N}(C)$ consisting of $N$-admissible $S^1$-invariant perturbations $H'_\eta$ of $H_\eta$ such that
\begin{itemize}
    \item $H'_\eta \leq H_\eta$.
\end{itemize}
From this, it follows that $\Psi_{H^*_\eta} \leq \Psi_{H'^*_\eta}$, and therefore for $ L \in (\eta, T_\eta)$ we have an inclusion
\[i^{H^*_\eta, H'^*_\eta, N}_L: (\{\Psi_{H'^*_\eta} < L\}, \{\Psi_{H'^*_\eta} < \varepsilon\}) \to (\{\Psi_{H^*_\eta} < L\} \times S^{2N+1}, \{\Psi_{H^*_\eta} < \varepsilon\} \times S^{2N+1}).\]

We define
\[D^{H_\eta, S^1, N}_L: \H_*^{S^1}(\{\Psi_{H^*_\eta} < L\} \times S^{2N+1}, \{\Psi_{H^*_\eta} < \varepsilon\} \times S^{2N+1}) \to SH^{S^1, N, +, <L}_{*+n}(C; \mathbb{Z}_2)\]
as
\[D^{H_\eta, S^1, N}_L = D^{H'_\eta, S^1, N}_L \circ (i^{H'^*_\eta, H^*_\eta, N}_L)^{-1}_*, \quad H'_\eta \in \mathcal{H}^{\mathcal{F},S^1,N}(C, H_\eta).\]

where $D^{H_\eta, S^1, N}$ is the isomorphism from the previous lemma. Using analogous arguments as in Corollary \ref{positivesymplecticautonomus}, we show that the definition of $D^{H_\eta, S^1, N}_L$ is a well-defined isomorphism. Furthermore, if $H_\eta \leq H_\nu$ and $N_1 \leq N_2$, we have that the following diagram
\begin{equation}\label{NcomutativeS1equivariant}
\begin{tikzcd}
      \H^{S^1}_{*}(\{\Psi_{H_{\nu}} < L_\nu\} \times S^{2N_2+1}, \{\Psi_{H_{\nu}} < \varepsilon\} \times S^{2N_2+1}; \mathbb{Z}_2) \arrow{r}{D^{H_\nu, S^1, N_2}_{L_\nu}} & [0.5em] SH^{S^1, N_2, +, <L_\nu}_{*+n}(C; \mathbb{Z}_2) \\
\H^{S^1}_{*}(\{\Psi_{H_{\eta}} < L_\eta\} \times S^{2N_1+1}, \{\Psi_{H_{\eta}} < \varepsilon\} \times S^{2N_1+1}; \mathbb{Z}_2) \arrow{r}{D^{H_\eta, S^1, N_1}_{L_\eta}} \arrow[swap]{u}{(i^{H, N_2, N_1}_{L_\nu, L_\eta})_*} & SH^{S^1, N_1, +, <L_\eta}_{*+n}(C; \mathbb{Z}_2) \arrow[swap]{u}{i^{S^1, N_2, N_1, +}_{L_\nu, L_\eta}}
\end{tikzcd}
\end{equation}
commutes for $L_\eta \in (\eta, T_\eta)$ and $L_\nu \in (\nu, T_\nu)$ such that $L_\eta \leq L_\nu$. Taking $H_\eta = H_\nu$ and $L_\eta = L_\nu$ in the commutative diagram \eqref{NcomutativeS1equivariant}, we prove claim (1) by taking a direct limit. On the other hand, by taking the direct limit of the same diagram for $H_\eta$ and $H_\nu$, we prove claim (2).

\end{proof}

\begin{proof}[Proof of Proposition ~ \ref{isomorphismautomussymplectic}]

For $H_\eta = \varphi_\eta \circ \HC$ where $\varphi_\eta \in \mathcal{F}_{\operatorname{lin}}(C)$, we define the set
\[\mathcal{H}^{\mathcal{F}}_{MB}(C; H_\eta)\]
as the set of all $H'_\eta \in \mathcal{H}^{\mathcal{F}}_{MB}(C)$ with the slope $\eta$ such that
\begin{itemize}
    \item $H_\eta \leq H'_\eta$.
\end{itemize}

Then we have the inclusion
\[i^{H'^*_\eta, H^*_\eta}_L: (\{\Psi_{H^*_\eta} < L\}, \{\Psi_{H^*_\eta} < \varepsilon\}) \to (\{\Psi_{H'^*_\eta} < L\}, \{\Psi_{H'^*_\eta} < \varepsilon\}), \quad L \in (\eta, T_\eta).\]

This inclusion induces an isomorphism in homology. Indeed, we can find $H''_\eta \in \mathcal{H}^{\mathcal{F}}_{MB}(C)$ such that $H''_\eta \leq H_\eta$ and $H'''_\eta = \varphi'''_\eta \circ H_C$, where $\varphi'''_\eta \in \mathcal{F}_{\operatorname{lin}}(C)$ is such that $H_\eta' \leq H_\eta'''$.

From this, we have the following factorizations:

\[
\begin{tikzcd}
 &  (\{\Psi_{H^*_\eta} < L\}, \{\Psi_{H^*_\eta} < \varepsilon\})  \arrow[dr,"i^{H_\eta'^*,H^*_\eta}_{L}"] \\
(\{\Psi_{H''^*_\eta} < L\}, \{\Psi_{H''^*_\eta} < \varepsilon\}) \arrow[ur, "i^{H_\eta^*,H''^*_\eta}_{L}"] \arrow[rr,"i^{H_\eta'^*,H''^*_\eta}_{L}"] && (\{\Psi_{H'^*_\eta} < L\}, \{\Psi_{H'^*_\eta} < \varepsilon\}),
\end{tikzcd}
\]

\[
\begin{tikzcd}
 &  (\{\Psi_{H'^*_\eta} < L\}, \{\Psi_{H'^*_\eta} < \varepsilon\})  \arrow[dr,"i^{H_\eta'''^*,H'^*_\eta}_{L}"] \\
(\{\Psi_{H^*_\eta} < L\}, \{\Psi_{H^*_\eta} < \varepsilon\}) \arrow[ur, "i^{H_\eta'^*,H^*_\eta}_{L}"] \arrow[rr,"i^{H_\eta'''^*,H^*_\eta}_{L}"] && (\{\Psi_{H'''^*_\eta} < L\}, \{\Psi_{H'''^*_\eta} < \varepsilon\})
\end{tikzcd}
\]

From Corollary \ref{positivesymplecticautonomus} and Proposition \ref{isomorphismclarkeduality}, we have that $i^{H_\eta'^*,H''^*_\eta}_{L}$ and $i^{H_\eta'''^*,H^*_\eta}_{L}$ induce isomorphisms in homology. Therefore, $i^{H'^*_\eta, H^*_\eta}_L$ induces an isomorphism in homology.

For every $L \in (\eta, T_\eta)$, we define
\[D^{H_\eta}_L: \H_{*}(\{\Psi_{H_\eta} < L\}, \{\Psi_{H^*_\eta} < \varepsilon\}; \mathbb{Z}_2) \to SH^{+, <L}_{*+n}(C; \mathbb{Z}_2)\]
as
\[D^{H_\eta}_L = D^{H'_\eta}_L \circ (i^{H'^*_\eta, H^*_\eta}_L)_*, \quad H'_\eta \in \mathcal{H}^{\mathcal{F}}_{MB}(C; H_\eta).\]

To show that this is indeed a well-defined natural isomorphism, we use similar arguments as in the proof of Corollary \ref{positivesymplecticautonomus}.

\end{proof}
We prove the $S^1$-equivariant case by using Corollary \ref{positiveS1symplecticautonomus} and Remark \ref{isomorphismclarkedualityS1} in the same way. This justifies Remark \ref{isomorphismautomussymplecticS1}.

\section{Positive ($S^1$-equivariant) symplectic homology and Clarke's duality}

In this section, we will prove theorems \hyperlink{TheoremC1}{C.1} and \hyperlink{TheoremC2}{C.2}. First, we extract important properties of the family $\mathcal{F}_{\operatorname{lin}}(C)$.

\begin{enumerate}
    \item For every $L \notin \sigma(\partial C)$ such that $L > \min \sigma(\partial C)$, there exists $\varphi_\eta\in \mathcal{F}_{\operatorname{lin}}(C)$ such that $L \in (\eta, T_\eta)$.
    \item For every $L \notin \sigma(\partial C)$ such that $L > \min \sigma(\partial C)$ and every two functions $\varphi_{\eta}, \varphi_{\nu} \in \mathcal{F}_{\operatorname{lin}}(C)$ such that $L \in (\eta, T_\eta)$ and $L \in (\nu, T_\nu)$, there exists $\varphi_\mu \in \mathcal{F}_{\operatorname{lin}}(C)$ such that $\varphi_\eta, \varphi_\nu \leq \varphi_\mu$ and $L \in (\mu, T_\mu)$.
    \item For every $\min \sigma(\partial C) < L_1 < L_2$ which do not belong to the spectrum of $\partial C$, there exist $\varphi_{\eta}, \varphi_{\nu} \in \mathcal{F}_{\operatorname{lin}}(C)$ such that $\varphi_{\eta} \leq \varphi_{\nu}$ with $L_1 \in (\eta, T_{\eta})$ and $L_2 \in (\nu, T_{\nu})$.
\end{enumerate}

\begin{proof}[Proof of Theorem ~ \hyperlink{TheoremC1}{C.1}]

We define  
\[D_L: \H_*(\{\Psi_C < L\}) \to S\H^{+,<L}_{*+n+1}(C), \quad L \in \mathbb{R} \setminus \sigma(\partial C),\]

in the following way:
\[D_L = D^{H_\eta}_L \circ D^{H^*_\eta}_L\]

where $L \notin \sigma(\partial C)$ and $L > \min \sigma(\partial C)$. Hamiltonian $H_\eta = \varphi_\eta \circ H_C$ is such that $\varphi_\eta \in \mathcal{F}_{\operatorname{lin}}(C)$ and $L \in (\eta, T_\eta)$. Maps $D^{H^*_\eta}_L$ and $D^{H_\eta}_L$ are isomorphisms from propositions \ref{isomorphismclarkeduality} and \ref{isomorphismautomussymplectic} respectively (Note: $D^{H^*_\eta}_L$ is well-defined since $\mathcal{F}_{\operatorname{lin}}(C) \subset \mathcal{F}^*_{\operatorname{lin}}(C)$ by Remark \ref{admisiblefunctionsproperties}).   
For $L < \min \sigma(\partial C)$, there exists a unique isomorphism 
\[D_L = 0,\] 
since both homologies are zero.

From property (1) of the family $\mathcal{F}_{\operatorname{lin}}(C)$, we have that we can always find such $H_\eta$ from the definition of $D_L$. 

On the other hand, if $H_\eta \leq H_\nu$ are such that $L \in (\eta, T_\eta)$ and $L \in (\nu, T_\nu)$, we have commuting diagrams:

\[ 
\begin{tikzcd}
  \H_*(\{\Psi_C < L\}) \arrow{r}{D^{H^*_\nu}_{L}} & \H_{*+1}(\{\Psi_{H_{\nu}^*} < L\}, \{\Psi_{H_{\nu}^*} < \varepsilon\}) \\%
  \H_*(\{\Psi_C < L\}) \arrow{r}{D^{H^*_\eta}_{L}} \arrow[u, equal] & \H_{*+1}(\{\Psi_{H_{\eta}^*} < L\}, \{\Psi_{H_{\eta}^*} < \varepsilon\}) \arrow[swap]{u}{(i^H_{L})_*}
\end{tikzcd}
\]

\[ 
\begin{tikzcd}
  \H_{*+1}(\{\Psi_{H_{\nu}} < L\}, \{\Psi_{H_{\nu}} < \varepsilon\}) \arrow{r}{D^{H_\nu}_{L}} & SH^{+,<L}_{*+n+1}(C)\\%
  \H_{*+1}(\{\Psi_{H_{\eta}} < L\}, \{\Psi_{H_{\eta}} < \varepsilon\}) \arrow{r}{D^{H_\eta}_{L}} \arrow[swap]{u}{(i^{H}_{L})_*} & SH^{+,<L}_{*+n+1}(C) \arrow[u, equal]
\end{tikzcd}
\]

Combining these two diagrams, we have 
\[D^{H_\nu}_L \circ D^{H^*_\nu}_L = D^{H_\nu}_L \circ (i^H_{L})_* \circ D^{H^*_\eta} = D^{H_\eta}_L \circ D^{H^*_\eta}_L.\]

Therefore, due to property (2) of the family $\mathcal{F}_{\operatorname{lin}}(C)$, we conclude that the definition of $D_L$ is independent of $H_\eta$, which means it is well-defined. Since $D^{H_\eta}_L$ and $D^{H^*_\eta}_L$ are isomorphisms, $D_L$ is an isomorphism.

Let $\min \sigma(\partial C) < L_1 < L_2$ be any numbers that do not belong to the spectrum. Let $H_\eta = \varphi_\eta \circ H_C$ and $H_\nu = \varphi_\nu \circ H_C$ be Hamiltonians such that $\varphi_\eta$ and $\varphi_\nu$ are from property (3) of the family $\mathcal{F}_{\operatorname{lin}}(C)$. Then from propositions \ref{isomorphismclarkeduality} and \ref{isomorphismautomussymplectic}, we have the following diagrams:

\[ 
\begin{tikzcd}
  \H_*(\{\Psi_C < L_2\}) \arrow{r}{D^{H^*_\nu}_{L_2}} & \H_{*+1}(\{\Psi_{H_{\nu}^*} < L_2\}, \{\Psi_{H_{\nu}^*} < \varepsilon\}) \\%
  \H_*(\{\Psi_C < L_1\}) \arrow{r}{D^{H^*_\eta}_{L_1}} \arrow[swap]{u}{(i^C_{L_2, L_1})_*} & \H_{*+1}(\{\Psi_{H_{\eta}^*} < L_1\}, \{\Psi_{H_{\eta}^*} < \varepsilon\}) \arrow[swap]{u}{(i^H_{L_2, L_1})_*}
\end{tikzcd}
\]

\[ 
\begin{tikzcd}
  \H_{*+1}(\{\Psi_{H_{\nu}} < L_2\}, \{\Psi_{H_{\nu}} < \varepsilon\}) \arrow{r}{D^{H_\nu}_{L_2}} & SH^{+,<L_2}_{*+n+1}(C)\\%
  \H_{*+1}(\{\Psi_{H_{\eta}} < L_1\}, \{\Psi_{H_{\eta}} < \varepsilon\}) \arrow{r}{D^{H_\eta}_{L_1}} \arrow[swap]{u}{(i^{H}_{L_2, L_1})_*} & SH^{+,<L_1}_{*+n+1}(C) \arrow[swap]{u}{i^{+}_{L_2, L_1}}
\end{tikzcd}
\]

From these two diagrams and the definition of $D_{L_1}$ and $D_{L_2}$, we have that the diagram 

\[ 
\begin{tikzcd}
  \H_*(\{\Psi_C < L_2\}) \arrow{r}{D_{L_2}} & S\H^{+,<L_2}_{*+n+1}(C) \\%
  \H_*(\{\Psi_C < L_1\}) \arrow[swap]{u}{(i^C_{L_2, L_1})_*} \arrow{r}{D_{L_1}} & S\H^{+,<L_1}_{*+n+1}(C) \arrow[swap]{u}{i^{+}_{L_2, L_1}}
\end{tikzcd}
\]

commutes. If $L_1 < \min \sigma(\partial C)$, commutativity of the previous diagram is trivial since $D_{L_1} = 0$.

\end{proof}

Since propositions \ref{isomorphismclarkeduality} and \ref{isomorphismautomussymplectic} hold in the $S^1$-equivariant case (see Remarks \ref{isomorphismclarkedualityS1} and \ref{isomorphismautomussymplecticS1}), the proof of Theorem \hyperlink{TheoremC2}{C.2} is the same.

\begin{remark}\label{limitisomorphism}
    The isomorphism from Theorem \hyperlink{TheoremC2}{C.2} induces an isomorphism at the limit: 
    \[D^{S^1}: H_*^{S^1}(\Lambda) \to \S\H^{S^1,+}_{*+n+1}(C).\] 
    Indeed, since we can understand $H_*^{S^1}(\Lambda)$ as the direct limit over $L$ of $H_*^{S^1}(\{\Psi_C < L\})$, this isomorphism is defined as 
    \[D^{S^1}(x) = i^{S^1,+}_L \circ D^{S^1}_L(x), \quad x \in H_*^{S^1}(\{\Psi_C < L\}),\]
    where 
    \[i_L^{S^1,+}: \S\H^{S^1,+,<L}_{*+n+1}(W; \mathbb{F}) \to \S\H^{S^1,+}_{*+n+1}(W; \mathbb{F}),\] 
    is the inclusion. This isomorphism is natural, i.e., it holds that $i_L^{S^1,+} \circ D^{S^1}_L = D^{S^1} \circ (i^C_L)_*$ where 
    \[i^C_L: \{\Psi_C < L\} \to \Lambda\]
    is the inclusion. The same conclusion holds for the isomorphism from Theorem \hyperlink{TheoremC1}{C.1}.
\end{remark}

\section{Gutt-Hutchings capacities and Ekeland-Hofer spectral invariants}

In this section, we will prove Theorem \hyperlink{TheoremA}{A}. In the paper \cite{EH87}, Ekeland and Hofer introduced spectral invariants using Clarke's dual functional and the Fadell-Rabinowitz index. Both the spectral invariants and the Fadell-Rabinowitz\footnote{In this definition, the Fadell-Rabinowitz index is shifted by 1, aligning with the definition provided in \cite[Section 5]{FR78}.} index were defined in the introduction. We recall the definition of the spectral invariants.

\textbf{Ekeland-Hofer spectral invariants:} 

Let $C_0 \subset \mathbb{R}^{2n}$ be a bounded convex domain whose interior contains the origin. In Section \ref{sectionautonomusclarke}, we introduced Clarke's dual functional 
\[
\Psi_{C_0}:\Lambda \to \mathbb{R}, \quad \Psi_{C_0}(x) = \int\limits_\mathbb{T} \H^*_{C_0}(-J_0 \dot{x}(t)) \, dt.
\]

Now let $C \subset \mathbb{R}^{2n}$ be a bounded convex domain. We define the $k$-th Ekeland–Hofer spectral invariant as \[
s_k(C; \mathbb{F}) = \inf \{L > 0 \mid \indfr(\{\Psi_{C_0} < L\}; \mathbb{F}) \geq k\}, \quad k \geq 1,
\]
where $C_0$ is any translation of $C$ whose interior contains the origin. This definition does not depend on translation, and therefore the sequence of numbers $(s_k(C; \mathbb{F}))_{k \in \mathbb{N}}$ is well-defined.  

This sequence enjoys the following properties.

\begin{itemize}
    \item \textbf{Systole}: $s_1(C) = \min \PsiC = \min \sigma(\partial C)$.
    \item \textbf{Monotonicity:} The sequence $(s_k(C))_{k \in \mathbb{N}}$ is increasing with respect to $k$.
    \item \textbf{Spectrality}: Every $s_k(C)$ is a critical value of $\PsiC$ and therefore an element of $\sigma(\partial C)$.
    \item \textbf{Hausdorff-continuity}: For any sequence of bounded convex domains $C_i$ that converges to a bounded convex domain $C$ in the Hausdorff distance topology, it holds that $s_k(C_i) \to s_k(C)$.
\end{itemize}

\begin{lemma}\label{spectralinvariants}
    Let $\mathbb{F}$ be a field. Then 
    \[s_k(C; \mathbb{F}) = \inf \{L > 0 \mid (i^C_L)_* \neq 0\},\] 
    where $(i^C_L)_*: \H_{2k-2}^{S^1}(\{\Psi_{C_0} < L;\mathbb{F}\}) \to \H_{2k-2}^{S^1}(\Lambda;\mathbb{F}).$
\end{lemma}

\begin{proof}
    We have that 
    \[s_k(C) = \inf \{L > 0 \mid \indfr(\{\Psi_C < L\}) \geq k\},\]
    where $\indfr(\{\Psi_C < L\}) = \sup_{k \geq 0} \{k+1 \mid (\pi^L_2)^* e^k = 0\}$ and
    \[\pi^L_2: \{\Psi_C < L\} \times_{S^1} ES^1 \to BS^1, \quad \pi^L_2([x, y]) = [y]\]
    is the quotient projection. Since we have the following commuting diagram
     \[
\begin{tikzcd}[scale cd=0.85]
 &  H^*_{S^1}(\Lambda) \arrow[dl,swap, "(i^C_L)^*"]   \\
H^*_{S^1}(\PsiC<L)  && \arrow[ll, "(\pi^L_2)^*"] \arrow[ul,swap,"(\pi_2)^*"] H^*_{S^1}(BS^1),
\end{tikzcd}
\]  
    and $(\pi_2)^*$ is a ring isomorphism (since $\Lambda$ is $S^1$-contractible), we conclude that $s_k(C) = \inf \{L > 0 \mid (i^C_L)^* \neq 0\}$, where $(i^C_L)^*: \H_{S^1}^{2k-2}(\Lambda) \to \H_{S^1}^{2k-2}(\{\Psi_C < L\}).$ Since $\mathbb{F}$ is a field, by the universal coefficients theorem, we conclude that the claim holds.
\end{proof}

\textbf{Gutt-Hutchings capacities:}

In \cite{GH18}, Gutt and Hutchings extracted a sequence of capacities from positive $S^1$-equivariant homology. Let $W$ be a nice star-shaped domain, meaning the closure of a bounded open star-shaped set with a smooth boundary such that $\partial W$ is transverse to the radial vector field. By \cite[Remark 4.8]{GH18}, in this case we have
\[
c_k^{GH}(W; \mathbb{F}) = \inf\{L > 0 \mid i_L^{S^1,+} \neq 0\},
\]
where $i_L^{S^1,+} : \S\H^{S^1,+,<L}_{n-1+2k}(W; \mathbb{F}) \to \S\H^{S^1,+}_{n-1+2k}(W; \mathbb{F}).$ Here, $\S\H^{S^1,+,<L}_{n-1+2k}(W ;\mathbb{F})$ denotes the filtered positive $S^1$-equivariant homology, while the group $\S\H^{S^1,+}_{n-1+2k}(W; \mathbb{F})$ is the full positive $S^1$-equivariant homology. The map $i_L^{S^1,+}$ refers to the inclusion map. These capacities are extended continuously to bounded star-shaped domains that are not nice.

We list some of the properties of these capacities.

\begin{itemize}
    \item \textbf{Monotonicity:} The sequence $(c^{GH}_k(W))_{k\in \mathbb{N}}$ is monotone with respect to $k$.
    \item \textbf{Spectrality:} Every $c^{GH}_k(W)$ is an element of $\sigma(\partial W)$.
    \item \textbf{Hausdorff-continuity:} For any sequence of bounded star-shaped domains $W_i$ that converges to a bounded star-shaped domain $W$ in the Hausdorff distance topology, it holds that $c^{GH}_k(W_i)\to c^{GH}_k(W)$.
\end{itemize}

Moreover, being that they are capacities these invariants are monotonic with respect to symplectic embeddings. See \cite{GH18} for the definition of these capacities for general Liouville domains.

\begin{proof}[Proof of Theorem ~ \hyperlink{TheoremA}{A}]
    Let $C \subset \mathbb{R}^{2n}$ be a non-degenerate, strongly convex domain. By translation, we can assume that its interior contains the origin. 
    
    By Theorem \hyperlink{TheoremC2}{C.2} and Remark \ref{limitisomorphism}, we have
    \begin{align*}
      c^{GH}_k(C)&=\inf\limits\{L > 0 \mid i_L^{S^1,+} \neq 0\}= \inf \{L > 0 \mid (D^{S^1})^{-1} \circ i_L^{S^1,+} \circ D_L^{S^1} \neq 0\}\\
      &= \inf \{L > 0 \mid (D^{S^1})^{-1} \circ D^{S^1} \circ (i_L^C)_* \neq 0\}= \inf\{L > 0 \mid (i^C_L)_* \neq 0\}=s_k(C).
    \end{align*}
    
    The last equality follows from Lemma \ref{spectralinvariants}. Since both the Gutt-Hutchings capacities and the spectral invariants are continuous with respect to the Hausdorff distance topology, and non-degenerate strongly convex domains are dense in bounded convex domains, the theorem follows.
\end{proof}

\end{document}